%% file: MinimaxNLM_arxiv_R1.tex
\documentclass[final]{siamltexmm}

\usepackage{graphicx}
\usepackage[normalem]{ulem}
\usepackage{amssymb,amsmath}
\usepackage[latin1]{inputenc}
\graphicspath{{builtimages/}}
\usepackage{dsfont}
\usepackage{enumerate}
\usepackage{epsfig}
\usepackage[hang, nooneline]{subfigure}
\usepackage{rotating}
\usepackage{stmaryrd}
\usepackage{float}
\usepackage{multirow}

\include{notations_arxiv}

\title{ Oracle inequalities and minimax rates for non-local means \\
  and related adaptive kernel-based methods \thanks{The authors
    started working on the paper at the Institute for Mathematics and
    its Applications and gratefully acknowledge support from DARPA
    grant no.\ FA8650-11-1-7150, AFOSR award no.\ FA9550-10-1-0390,
    and NSF awards no.\ CCF-06-43947.}  }

\author{Ery Arias-Castro\footnotemark[2]\
\and Joseph Salmon\footnotemark[3]\ 
\and Rebecca Willett\footnotemark[3] 
}

\begin{document}

\maketitle

\renewcommand{\thefootnote}{\fnsymbol{footnote}}

\footnotetext[2]{Department of Mathematics; University of California,
  San Diego, CA, USA.
}
\footnotetext[3]{Department of Electrical and Computer Engineering; Duke University,
  Durham, NC, USA.}

\renewcommand{\thefootnote}{\arabic{footnote}}

\newcommand{\slugmaster}{\slugger{MMedia}{xxxx}{xx}{x}{x--x}}

\begin{abstract}
  This paper describes a novel theoretical characterization of the
  performance of non-local means (NLM) for noise removal.  NLM has
  proven effective in a variety of empirical studies, but little is
  understood fundamentally about how it performs relative to classical
  methods based on wavelets or how its parameters 
  should be chosen.  For cartoon images and images which may contain
  thin features and regular textures, the error decay rates of NLM are
  derived and compared with those of linear filtering, oracle
  estimators, 
  Yaroslavsky's filter
  and wavelet thresholding estimators.  The trade-off between global
  and local search for matching patches is examined, and the bias
  reduction associated with the local polynomial regression version of
  NLM is analyzed. The theoretical results are validated via
  simulations for 2D images corrupted by additive white Gaussian
  noise.
\end{abstract}


\pagestyle{myheadings}
\thispagestyle{plain}
\markboth{E. Arias-Castro {\em et al.}}{Minimax bounds for non-local means}

\section{Introduction}

The classical problem of image noise removal has drawn significant
attention during the past few decades from the image processing,
computational harmonic analysis, nonlinear approximation, and
statistics communities.  In recent years there has been a resurgence
of interest in kernel-based methods, including the ubiquitous
non-local means (NLM) algorithm \cite{Buades_Coll_Morel05}, due to
their practical efficacy on broad collections of ``natural'' images.
While there is a wealth of theoretical analysis associated with
nonlinear thresholding estimators based on wavelets and related sparse
multiscale representations of images
\cite{Donoho_Johnstone94,Donoho_Johnstone_Kerkyacharian_Picard95,Mallat09,Starck_Candes_Donoho02}
or on diffusion models
\cite{Szlam_Maggioni_Coifman08,Singer_Shkolniksy_Nadler09} and partial
differential equations \cite{Perona_Malik90,Alvarez_Mazorra94},
performance guarantees for NLM are lacking and this paper aims at
providing some results in this direction.

In this paper, we explore the theoretical underpinnings of adaptive
kernel-based image estimation and derive bounds on the mean squared
error as a function of the number of pixels observed and features of
the underlying image.  The denoising methods we consider are based
on estimating each pixel value with a weighted sum of the surrounding
pixels.  Depending on how the weights in this average are selected,
this corresponds to classical linear filters
\cite{Nadaraya64,Watson64}, Yaroslavsky's filter (YF)
\cite{Yaroslavsky85}, the Sigma filter \cite{Lee83}, or the bilateral
filter \cite{Tomasi_Manduchi98}.  It also includes variable-bandwidth
kernel estimators \cite{Lepski_Mammen_Spokoiny97}, referred to
as Lepski's method by statisticians and as the Intersection of
Confidence Intervals (ICI) rule
\cite{Katkovnik99,Katkovnik_Egiazarian_Astola02} in signal processing. 
Other variants for a local choice of the kernel
include \cite{Spira_Kimmel05, Takeda_Farsiu_Milanfar07}
We refer to \cite{Hastie_Tibshirani_Friedman09,Salmon,Milanfar12} for more
insights on a unifying framework for averaging filters.

As none of these methods have been explicitly designed to deal
with textured regions, many authors, inspired by work on texture
synthesis \cite{Efros_Leung99} and inpainting
\cite{Criminisi_Perez_Toyama04}, have proposed to introduce patches
(small sub-images) to take advantage of natural image redundancy,
especially in textured regions. NLM \cite{Buades_Coll_Morel05} and
UINTA \cite{Awate_Whitaker06} algorithms are typical examples of this
approach, as is their extension using Lepski's method
\cite{Kervrann_Boulanger06}.  Those algorithms rely on averaging
similar pixels, where the similarity is measured through patches
centered on the pixel of interest.  Some more elaborate methods have
tried to remove artifacts appearing in regions with low
redundancy \cite{Salmon_Strozecki12} --- a phenomenon also known as
the \textit{rare patch effect} \cite{Duval_Aujol_Gousseau11} --- for
instance by choosing NLM parameters automatically and locally. A
common tool used for this local adaptivity is the Stein Unbiased Risk
Estimate (SURE)
\cite{Duval_Aujol_Gousseau11,VanDeVille_Kocher09,VanDeVille_Kocher11}.

Most current state-of-the-art methods for denoising take advantage of
the patch framework
\cite{Mairal_Bach_Ponce_Sapiro_Zisserman09,Dabov_Foi_Katkovnik_Egiazarian07,Dabov_Foi_Katkovnik_Egiazarian09}.
The interested reader could get a clear picture of practical performance of those
recent methods,
in the review paper by Katkovnik \textit{et al}. \cite{Katkovnik_Foi_Egiazarian_Astola10}.
Despite the strong empirical performance of these methods, few
performance guarantees exist: bounds with information theory flavor
are derived in \cite{Weissman_Ordentlich_Seroussi_Verdu_Weinberger05}
for a simple version of NLM; a consistency result relying on
beta-mixing assumptions on the image and on the noise (both modeled as
random variables) is obtained in \cite{Buades,Buades_Coll_Morel05};
\cite{Singer_Shkolniksy_Nadler09} proposes a graph-diffusion
interpretation for a simple image model; a bias/variance analysis
aiming at locally choosing NLM parameters is carried out in
\cite{Duval_Aujol_Gousseau11};
\cite{Levin_Nadler11,Chatterjee_Milanfar11} obtain Cramer-Rao type
efficiency results.  While finishing this paper, we became aware of
two related papers by Maleki, Narayan and Baraniuk, addressing optimal performance in the context of non-parametric minimax estimation
\cite{Maleki_Narayan_Baraniuk11,Maleki_Narayan_Baraniuk11b}.
\cite{Maleki_Narayan_Baraniuk11} evaluates the performance of NLM for
the piecewise constant horizon model \cite{Korostelev_Tsybakov93},
while \cite{Maleki_Narayan_Baraniuk11b} considers an anisotropic
variant of NLM for the same image class.  The latter shares several
features with earlier work on anisotropic NLM
\cite{Deledalle_Duval_Salmon11b,Deledalle_Duval_Salmon11}.  Our work
is most closely related to \cite{Maleki_Narayan_Baraniuk11},
addressing the same challenge of quantifying the performance of NLM
and related methods, and at the same time contains several novel
contributions.
While the paper was under review, we learned about an older paper of Tsybakov \cite{Tsybakov89}.  This paper proposes and analyzes a patch-based method that compares the medians over patches.  The paper also derives a minimax lower bound for the cartoon model we consider.  
We comment in more detail on the work of Maleki \cite{Maleki_Narayan_Baraniuk11} et al and the work of Tsybakov \cite{Tsybakov89} in \secref{discussion}.

\subsection{Our contribution}

We derive theoretical performance bounds for the linear filter, oracle 
variable-bandwidth kernel methods, Yaroslavsky's filter and NLM 
--- both the original \cite{Buades_Coll_Morel05} and a fast
patch-mean based variant \cite{Mahmoudi_Sapiro05} --- in the
classical ``cartoon'' model in which an image consists of smooth
surfaces separated by a smooth discontinuity, a popular model in
statistics \cite{Korostelev_Tsybakov93}.  Our results are for the
local polynomial versions of these methods.  (The systematic bias
associated with NLM near discontinuities --- and boundaries
--- is shown to disappear when using a local polynomial regression.)
We also consider nonstandard image classes, one modeling images with
thin features and another one modeling regular textures.  The latter
is particularly significant because it highlights some of the key
advantages of patch-based methods over, say, wavelet thresholding
estimators. Previous insights into the performance of NLM-like methods
on textures are empirical at best; we are not aware of any theory in
this vein.  Our benchmarks are two oracle inequalities, though many of
our theoretical results can be compared directly with similar
classical results in the wavelet literature and known minimax lower
bounds on mean squared error (MSE)
\cite{Donoho_Johnstone_Kerkyacharian_Picard95,Korostelev_Tsybakov93}.

The cartoon model for images has been a benchmark for image denoising methods, 
at least since the work of Korostelev and Tsybakov, condensed in \cite{Korostelev_Tsybakov93}.  
This model is relevant when comparing denoising methods on texture-less images.  
The other models are novel and tailored to situations where the image exhibits some thin features --- 
like the legs of the Cameraman's tripod --- and regular textures --- like the patterns in Barbara's blouse.  
Though these models do not reflect the complexities of real images, we do gain some qualitative insights.  
First, we learn that variable bandwidth kernel methods are fundamentally limited by the bias near discontinuities.  
Yaroslavsky's filter is found to be near-optimal when the noise level is sufficiently 
low that the different regions in the cartoon image do not mix when noise is added; and when this is not the case, 
the method becomes useless.  In non-local means, the patch size should be chosen just sufficiently 
large that nearby patches from different regions look different 
(in the average version of the NLM, this can be made very precise).  
The search window should be chosen like a standard kernel bandwidth.  
We quickly argue that not localizing these methods may lead to very poor performance, in agreement with \cite{Salmon10,Zontak_Irani11}. 
Also, while the NLM average and regular NLM perform similarly on cartoon images, the latter is superior when textures are present.

\subsection{Organization of the paper} 

In \secref{setting} we describe the mathematical framework.  In
\secref{background} we introduce the methods that we analyze in the
sequel.  In \secref{results} we state performance guarantees in the
cartoon model for these methods, and in
\secref{thin-pattern} we do the same in the context of the thin
feature and regular pattern models.  In \secref{experiments} we
perform some numerical experiments carefully illustrating our
theoretical findings.  In \secref{discussion} we contrast our
contribution to that of Maleki {\em et al.}~\cite{Maleki_Narayan_Baraniuk11}
and discuss extensions.  The proofs are gathered in \secref{proofs},
which includes general results on local polynomial regression which
may be of independent interest.

\subsection{Notation}

We use standard notation.  For non-negative sequences $(a_n)$ and
$(b_n)$, $a_n = O(b_n)$ (same as $a_n \preceq b_n$) if the sequence
$|a_n/b_n|$ is bounded from above; $a_n \asymp b_n$ if $a_n=O(b_n)$
and $b_n=O(a_n)$; $a_n = o(b_n)$ if $a_n/b_n \to 0$ as $n \to \infty$.
For real numbers $a$ and $b$, $a\vee b = \max(a,b)$ while $a\wedge b =
\min(a,b)$.  For a Lebesgue-measurable subset $A \subset \bbR^d$,
$\Vol(A)$ denotes its Lebesgue measure.  For any $x\in \R^d$,
we define its Euclidean and sup norm as
\[
\|x\|_2=\left(\sum_{i=1}^d x_i^2 \right)^{1/2}, \qquad \|x\|:=\|x\|_\infty=\max_{i=1}^d |x_i|.
\]
We use the notation $B(0,1)$ (resp. $\overline{B(0,1)}$) to denote the open (resp. closed) unit ball for the supnorm.
For $\eta>0$, we define the $\eta$-neighborhood (for the norm
$\|\cdot\|$) of a set $A \subseteq \R^d$ as
\[
B(A, \eta) = \{x \in \R^d: \dist(x, A) < \eta\}.
\]
For a discrete set $A$, we denote its cardinality by either $|A|$ or
$\# A$.  For a set $A \subset \bbR^d$, $\one{A}$ is the indicator
function of $A$, while for a discrete subset $B \subset \{1, \dots,
m\}$, ${\bf 1}_B$ denotes the vector with entries indexed by $B$ equal
to one, and all others equal to zero.  Additional notation is
introduced in the text as needed.

\section{Function estimation in additive white noise}
\label{sec:setting}

We cast the problem of image denoising as a non-parametric regression
problem in the presence of white noise, a standard model
in statistics~\cite{Korostelev_Tsybakov93}.  We consider the general
$d$-dimensional problem, and use the term ``image'' to denote any
discretized signal on the $d$-dimensional square lattice, with
important cases when $1\leq d\leq 4$.  Though patch-based
methods were designed for 2D images, 
we consider a general dimension,
as the same techniques may apply in color, spectral, 3D and
4D imaging~\cite{Zewail_Thomas09}.

We observe noisy samples $\{y_i \in \R: i \in I_n^d\}$ (where $I_n :=
\{1, \dots, n\}$) of the target function $f : [0,1]^d \to [0,1]$ at
the design points $\{x_i \in \R^d: i \in I_n^d\}$
corrupted
by an additive noise $\{\eps_i \in \R: i \in I_n^d\}$, as follows
\beq \label{model} y_i = f(x_i) + \eps_i, \quad i \in I_n^d.  \eeq

For now, we only assume that the noise $\{\eps_i : i \in I_n^d\}$ are
uncorrelated with mean zero and variance $\sigma^2$, though some
results will require some tail bounds.  Also, for
concreteness, 
we focus on a standard model in image processing where the sample
points are on the square lattice, specifically, $x_i = ((i_1-1/2)/n,
\dots, (i_d-1/2)/n)$ when $i = (i_1, \dots, i_d)$.  Leaving $n$
implicit, define vectors $\by = (y_i: i \in I_n^d)$, $\bbf = (f_i: i
\in I_n^d)$ with $f_i := f(x_i)$ and $\beps = (\eps_i: i \in I_n^d)$.
The vector model can thus be written
\begin{equation}
 \by=\bbf+\beps \, .
\end{equation}

We focus on estimating a function $f$ on the grid, namely our goal is to
estimate the vector $\bbf$ and we measure the performance of an estimator
$\wh{\bbf}$ in terms of (MSE):
\[
\mse_f(\wh{\bbf}) = \frac{\E \|\wh{\bbf} - \bbf\|_2^2}{n^d} =
\frac1{n^d} \sum_{i \in I_n^d} \E (\wh{f}_i - f_i)^2 \, ,
\]
where the expectation $\E$ is with respect to the probability
  measure associated with the noise.

Although our analysis
may be generalized to other norms, mean squared error is handy because
of the point-wise (squared) bias and variance decomposition:
\beq \label{bias-var} \E (\wh{f}_i - f_i)^2 = \underbrace{(\E \wh{f}_i
  - f_i)^2}_{\mbox{\footnotesize{Squared Bias}}} + \underbrace{\E
  \big(\E(\wh{f}_i) -
  \wh{f}_i\big)^2}_{\mbox{\footnotesize{Variance}}} \ , \quad \forall
i \in I_n^d.  \eeq This leads for the vector estimate to the following
decomposition:
\[
\E \|\wh{\bbf} - \bbf\|_2^2 = \|\E(\wh{\bbf}) - \bbf\|_2^2 + \E
\|\E(\wh{\bbf}) - \wh{\bbf}\|_2^2 \, .
\]

To recover the function $f$ only through a finite number of
measurements, it is customary to require that the targeted function
belongs to a class $\cF$ of structured functions such as smooth,
piecewise smooth, or periodic textured images.  In this context, the
minimax risk over the function class $\cF$ is defined as
\[
\cR_n^*(\cF) = \inf_{\wh{\bbf}} \sup_{f \in \cF} \mse_f(\wh{\bbf}),
\]
where the infimum is over all the measurable function with
respect to the observations.  We say that an estimator is
(rate-)optimal for the class $\cF$ if its worst-case MSE over $\cF$ is
comparable to the minimax risk, \ie (assuming implicitly that $n$
becomes large)
\[
\cR_n(\wh{\bbf}, \cF) := \sup_{f \in \cF} \mse_f(\wh{\bbf}) = O(\cR_n^*(\cF)).
\]

\subsection{Cartoon images}
We are particularly interested in situations where the function $f$
has discontinuities: this is typical of images, mainly because of
occlusions occurring in natural scenes. We say that $f$ is a ``cartoon
image" if it is a piecewise smooth image with discontinuities along
smooth hypersurfaces.  This model spurred the greatest part of the
research in image processing and is very common when no texture is
present~\cite{Korostelev_Tsybakov93}.  For simplicity, we consider
that $f$ is made of two pieces with each piece being H\"older
smooth. Note that all our results apply to the more general case where
$f$ is made of more than two pieces.  For a function $g : \R^d \to \R$
and $s = (s_1, \dots, s_d)\in \N^d$, we denote the $s$-derivative of
$g$ at $x \in \R^d$ 
by
\[
g^{(s)}(x) = \frac{\partial^{|s|}}{\partial_{x_1}^{s_1} \cdots \partial_{x_d}^{s_d}} g(x),
\]
where $|s| := s_1 + \cdots + s_d$.
\begin{definition}[H\"older function class] \label{def:holder} For
  $\alpha, \cf > 0$, we define $\cH_d(\alpha, \cf)$ as the H\"older
  class of functions $g : [0,1]^d \to [0,1]$ that are $\afloor$ times
  differentiable ($\afloor$ is the largest integer strictly less than
  $\alpha$) and satisfy
\begin{align}
  \forall x \in [0,1]^d, \ \forall s \in \bbN^d, 1 \leq |s| \leq \afloor: & \quad |g^{(s)}(x)| \leq \cf; \label{eq:deriv_bound} \\
  \forall (x, x') \in [0,1]^d, \ \forall s \in \bbN^d, |s| = \afloor:
  & \quad |g^{(s)}(x) - g^{(s)}(x')| \leq \cf \|x
  -x'\|_{\infty}^{\alpha-\afloor}. \label{eq:deriv_max}
\end{align}
\end{definition}
The main feature of H\"older functions of order $\alpha$ is that they
are well-approximated locally by a polynomial (in fact, their Taylor
expansion) of degree $\afloor$, \lcf \lemref{taylor}.

\begin{definition}[Cartoon function class] \label{def:cartoon} For
  $\alpha, \cf > 0$, let $\cF^{\rm cartoon}(\alpha, \cf)$ denote the
  set of functions of the form 
\beq \label{eq:F} 
f(x) = \one{x \in \Omega} \, \fin(x) + \one{x \in \Omega^c} \, \fout(x), 
\eeq 
where
  $\fin, \fout \in \cH_d(\alpha, \cf)$, with jump (or discontinuity
  gap) \beq \label{diff} \mu(f) := \inf_{x \in \partial \Omega}
  |\fin(x) - \fout(x)| \geq 1/\cf, \eeq and $\Omega \subset (0,1)^d$
  is a bi-Lipschitz image of the (Euclidean) unit ball $B(0,1)$,
  specifically, $\Omega = \phi( B(0,1) )$, where $\phi : \R^d \to
  \R^d$ is injective with $\phi$ and $\phi^{-1}$ both Lipschitz with
  constant $\cf$ (\ie $C_0$-Lipschitz) with respect to the supnorm.
 We refer to $\fin$ as
  the foreground and to $\fout$ as the background. Moreover
    $\partial \Omega$ represents the (topological) boundary of
    $\Omega$.
\end{definition}

The condition \eqref{diff} is a lower bound on the minimum ``jump''t
along the discontinuity $\partial \Omega$.  We require that $\phi$ is
bi-Lipschitz to ensure that the set $\Omega$ is sufficiently smooth
and does not have a serious bottleneck, which could potentially
mislead the methods discussed here.

We define the jump-to-noise ratio (JNR) for a target function $f$
with jump $\mu(f)$, and noise standard deviation $\sigma$, as being
the quantity
\begin{equation} {\rm JNR} = \frac{\mu(f)}{\sigma} \, .
\end{equation}
We assume throughout that $\mu \asymp 1$, so that our bounds (which
scale with $\sigma$) reflect performance also as a function of JNR.
In the cartoon model, we focus on the case where the noiseless image
is at least piecewise Lipschitz, that is, $\alpha \geq 1$.  Note that
our results apply to the case where $\alpha > 1/2$, and that simple
linear filtering is essentially optimal when $\alpha \leq 1/2$. The
setting is illustrated in \figref{original_noisy_images}(a).

\begin{figure}[hbt]
\centering
\subfigure{\includegraphics[width=0.118\linewidth]{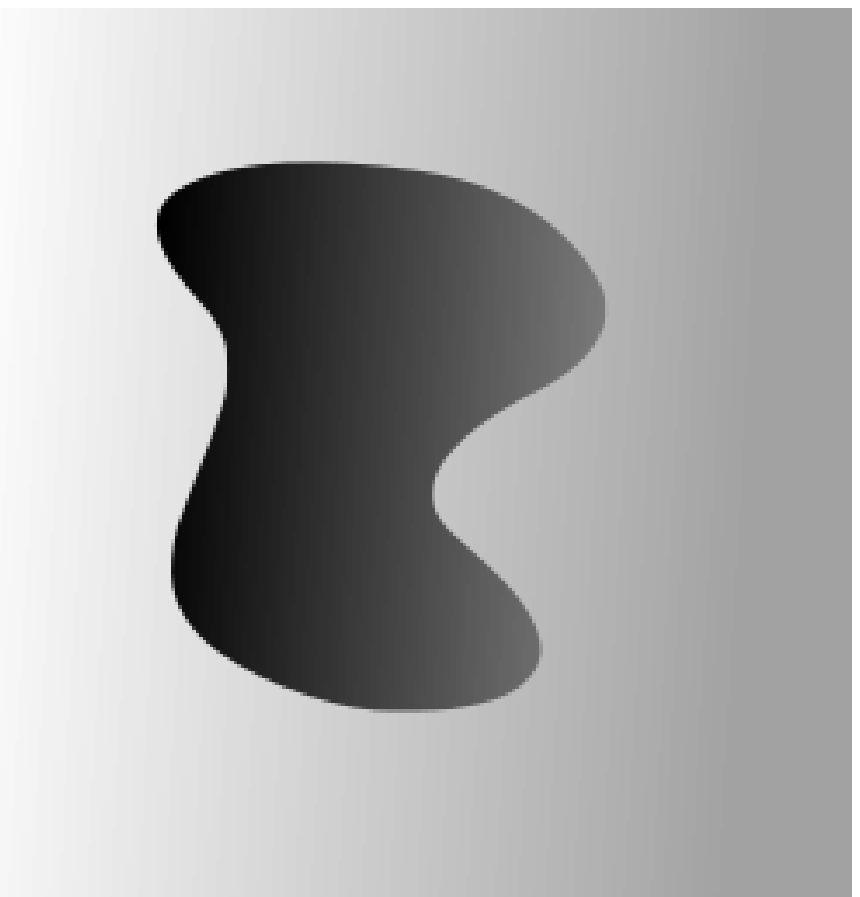}} 
\subfigure{\includegraphics[width=0.118\linewidth]{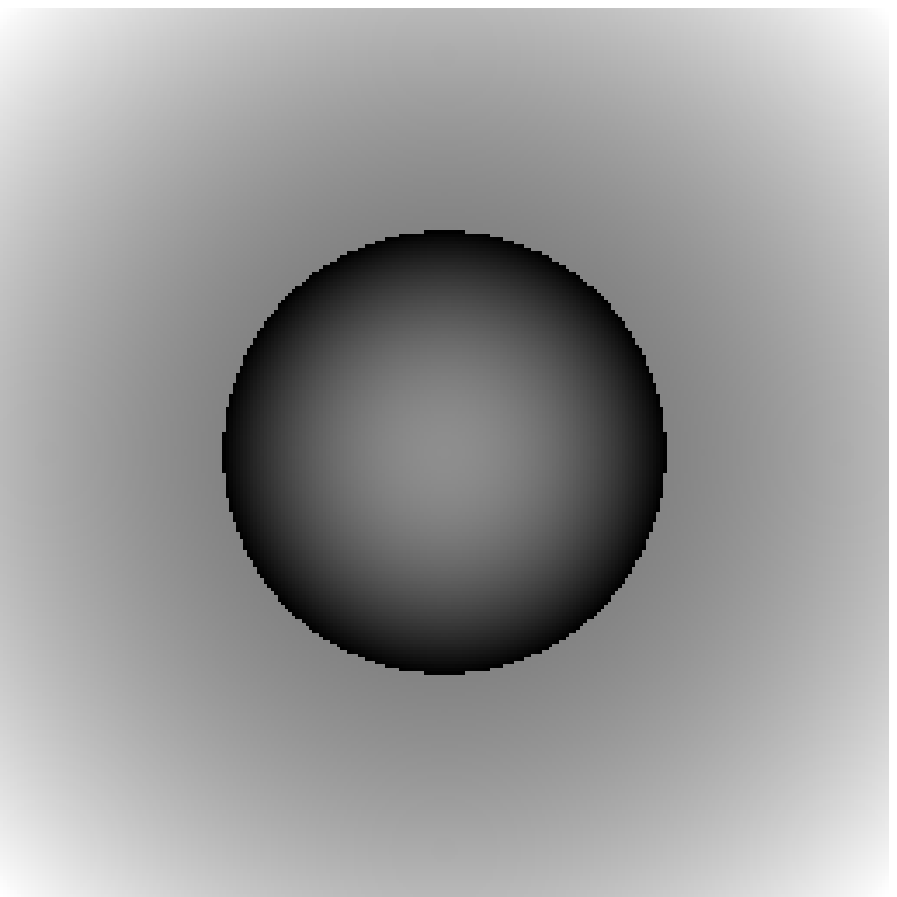}} 
\subfigure{\includegraphics[width=0.118\linewidth]{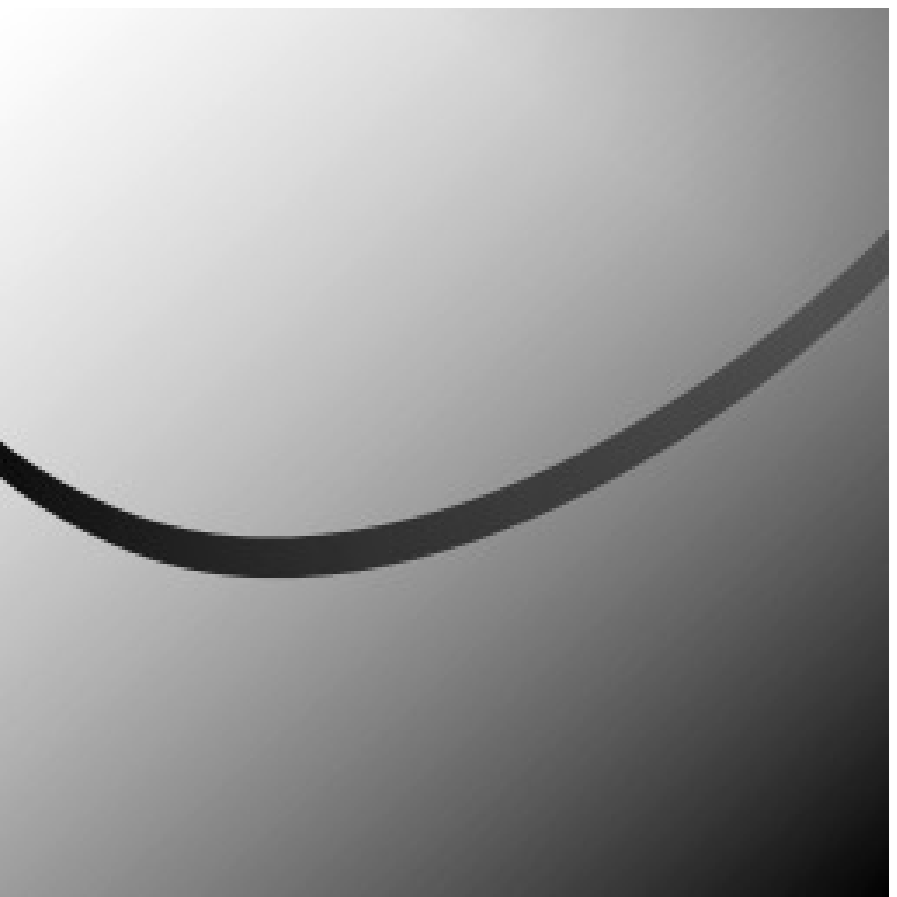}} 
\subfigure{\includegraphics[width=0.118\linewidth]{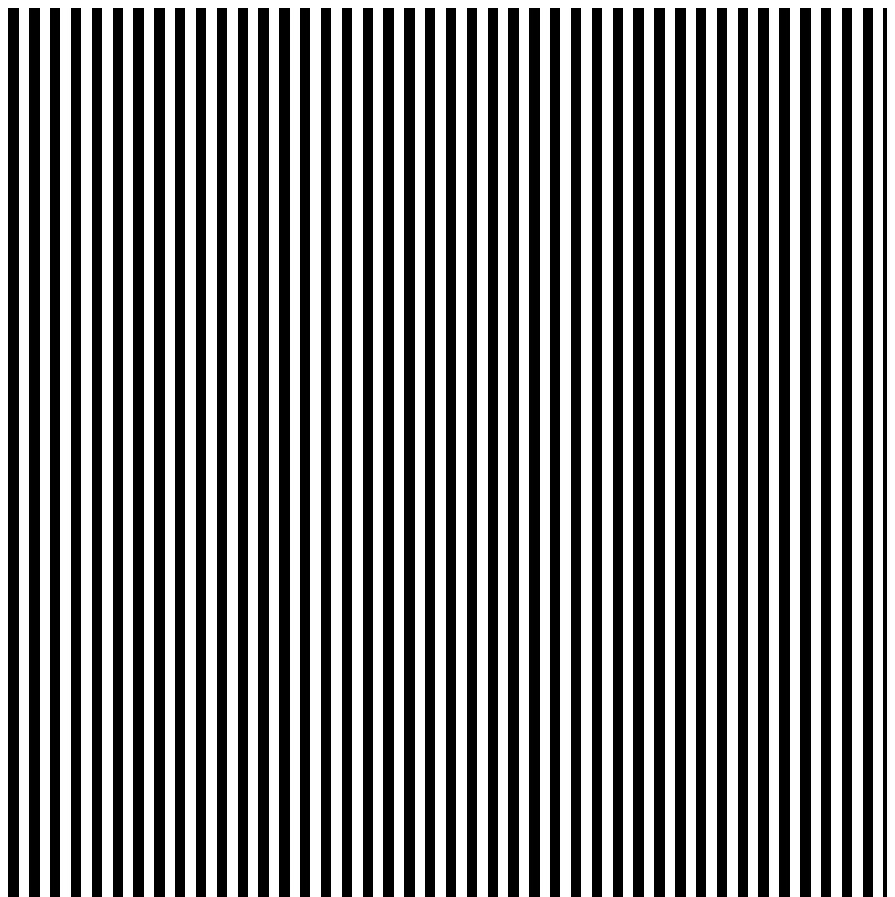}}
\subfigure{\includegraphics[width=0.118\linewidth]{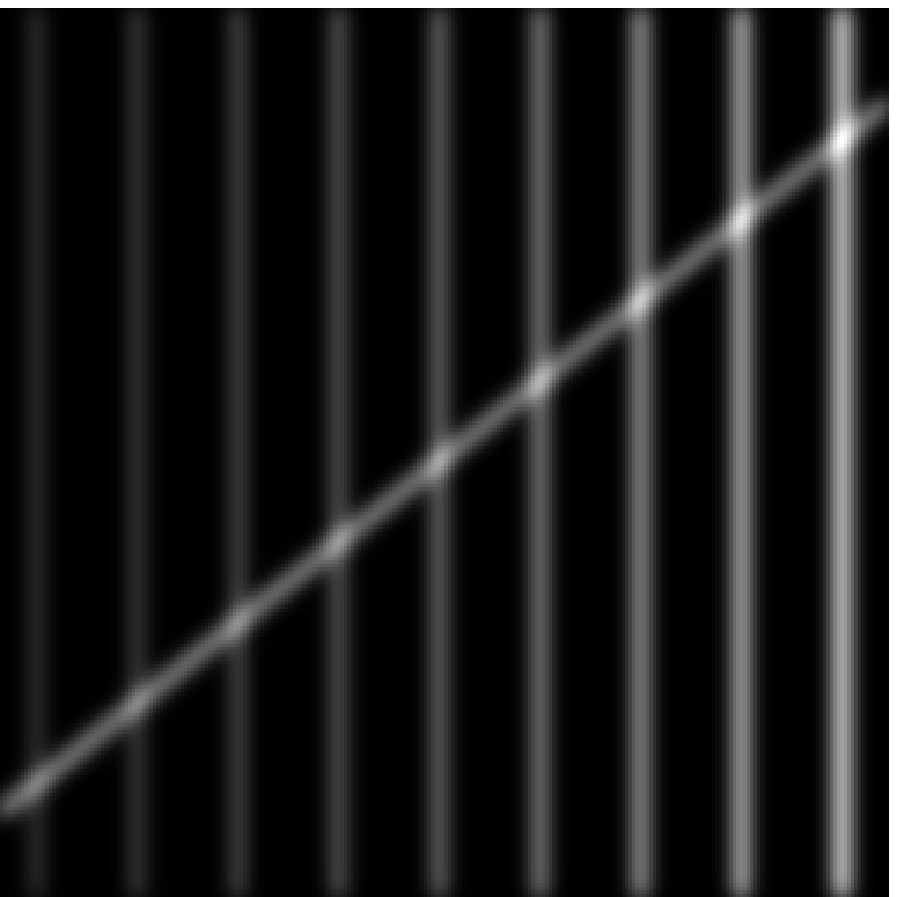}} 
\subfigure{\includegraphics[width=0.118\linewidth]{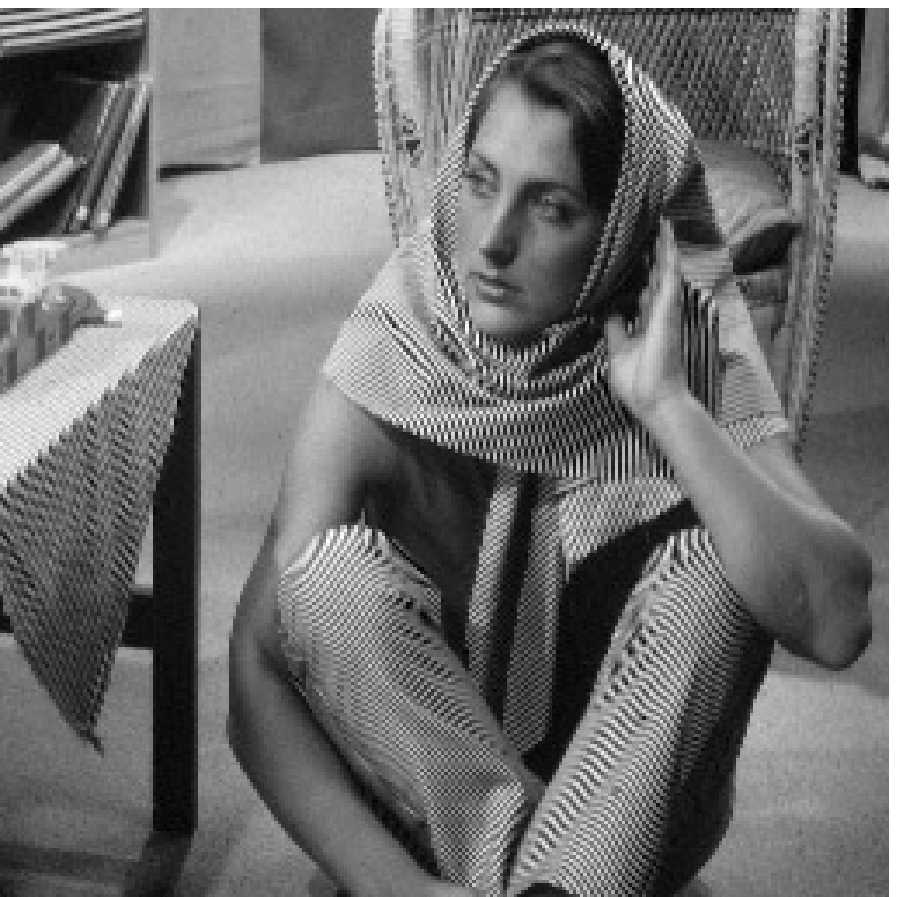}}
\subfigure{\includegraphics[width=0.118\linewidth]{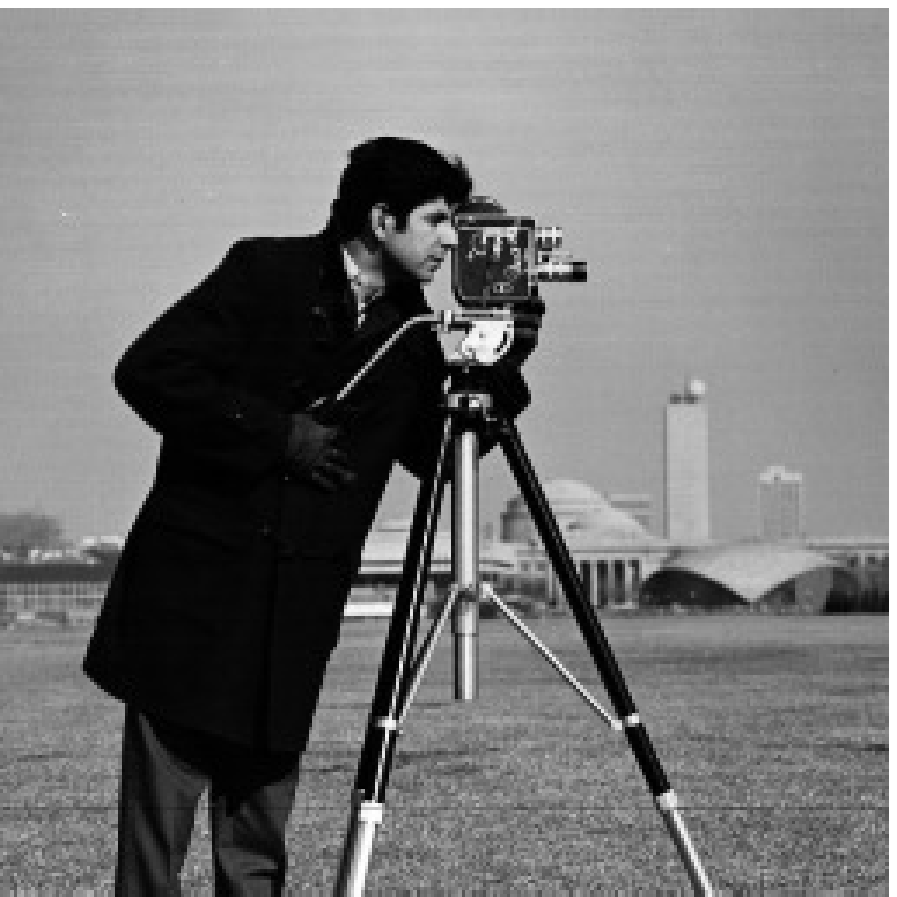}}\vspace{-0.1cm}
\setcounter{subfigure}{0}
\subfigure[Blob]{\includegraphics[width=0.118\linewidth]{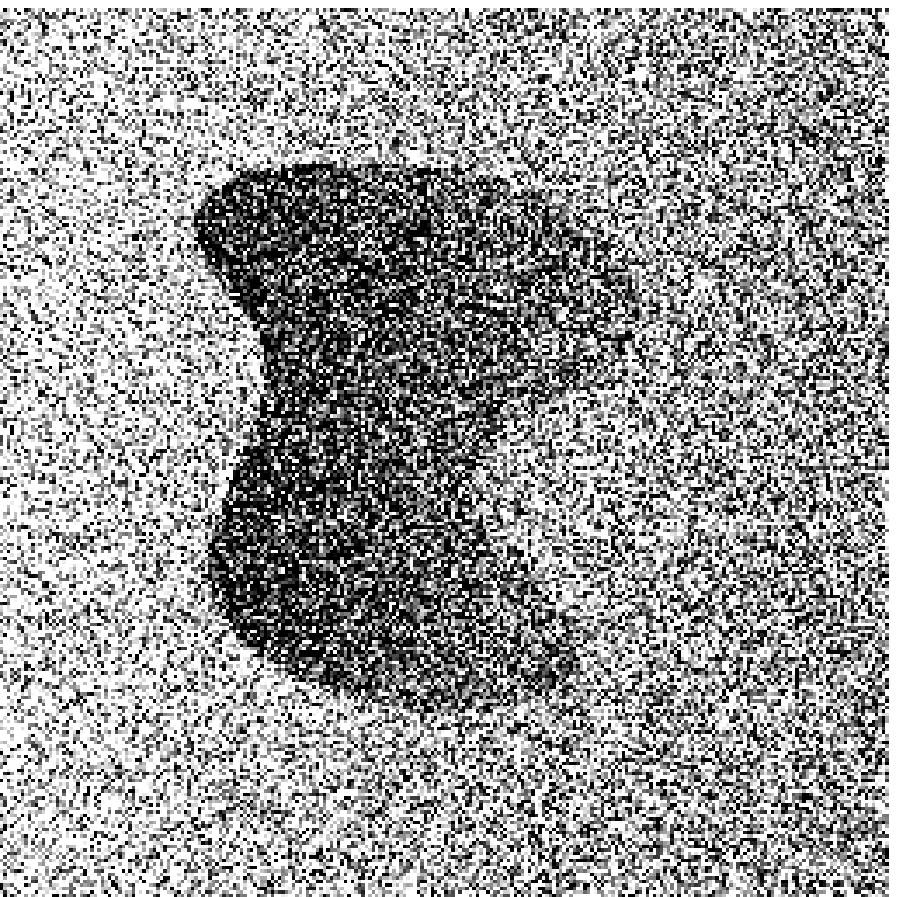}} 
\subfigure[Bowl]{\includegraphics[width=0.118\linewidth]{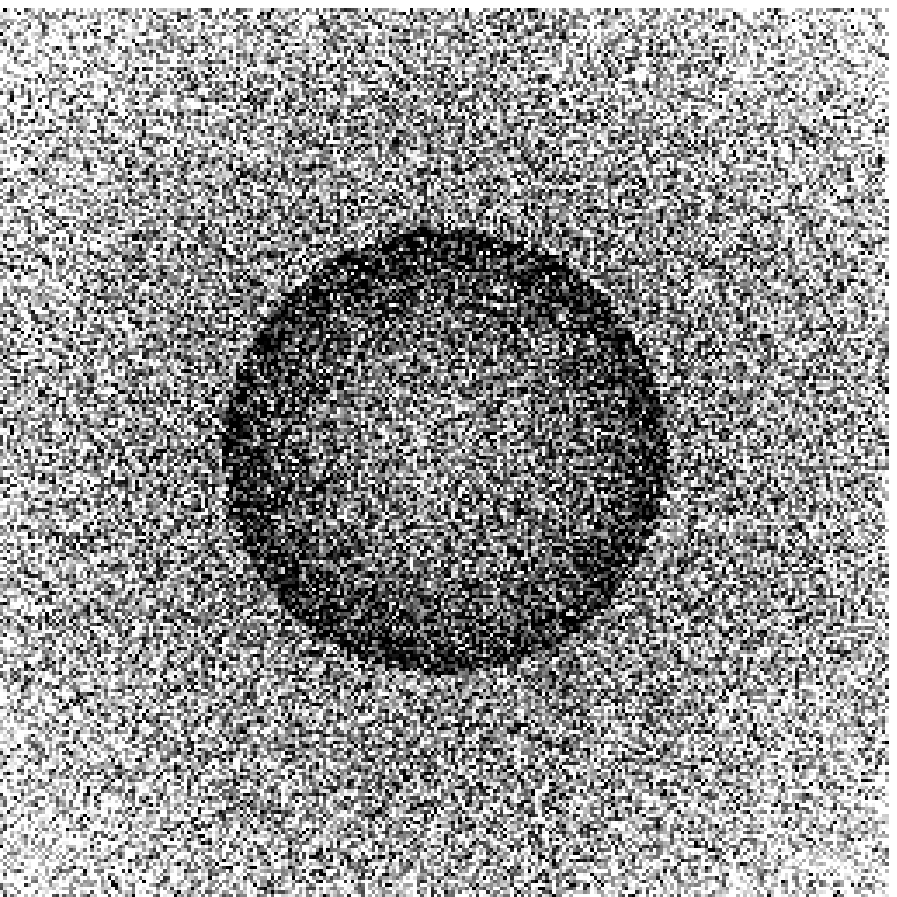}} 
\subfigure[Swoosh]{\includegraphics[width=0.118\linewidth]{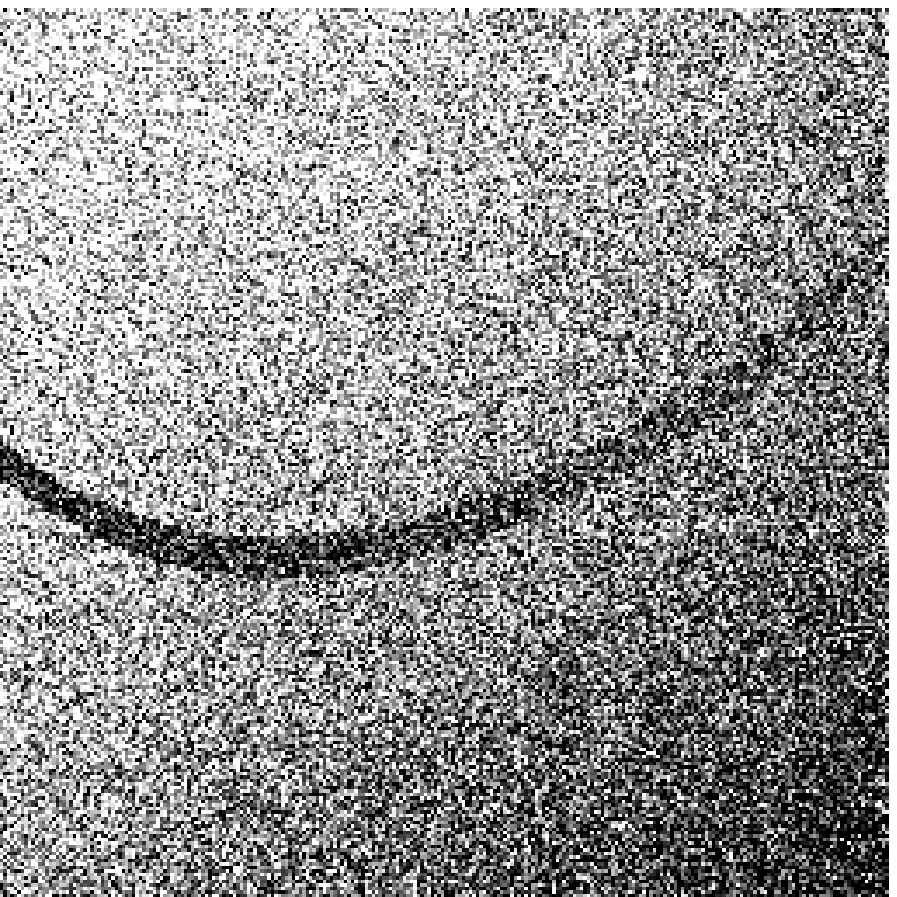}} 
\subfigure[Stripes]{\includegraphics[width=0.118\linewidth]{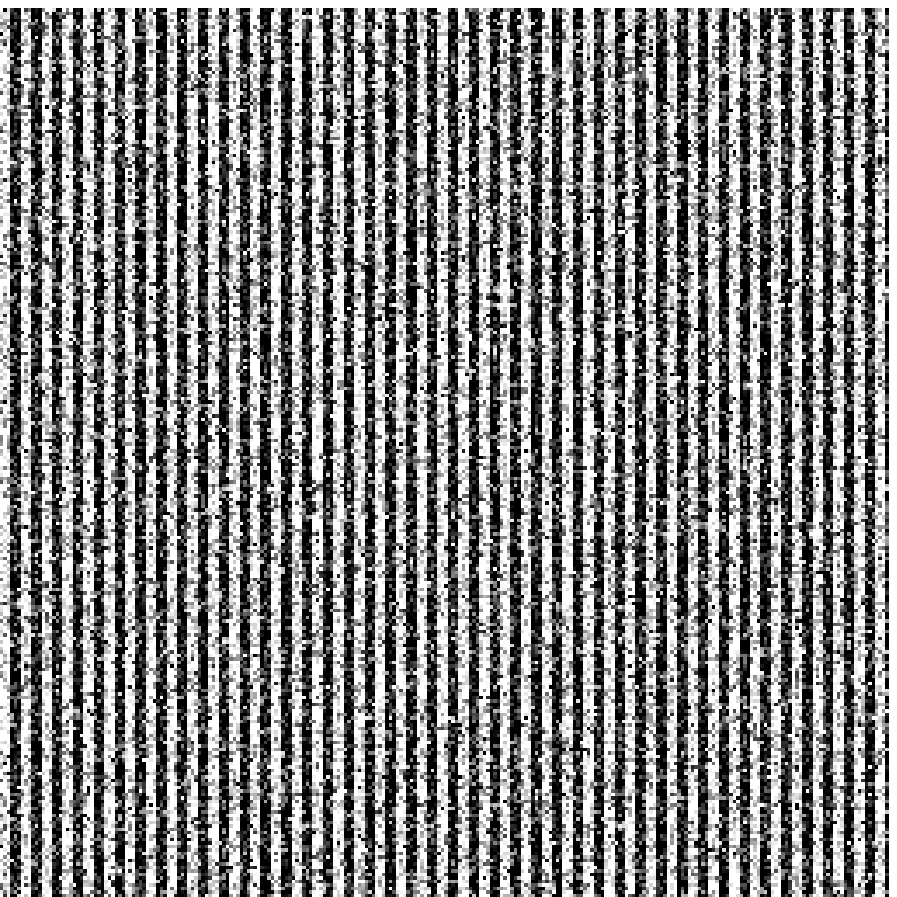}}
\subfigure[Ridges]{\includegraphics[width=0.118\linewidth]{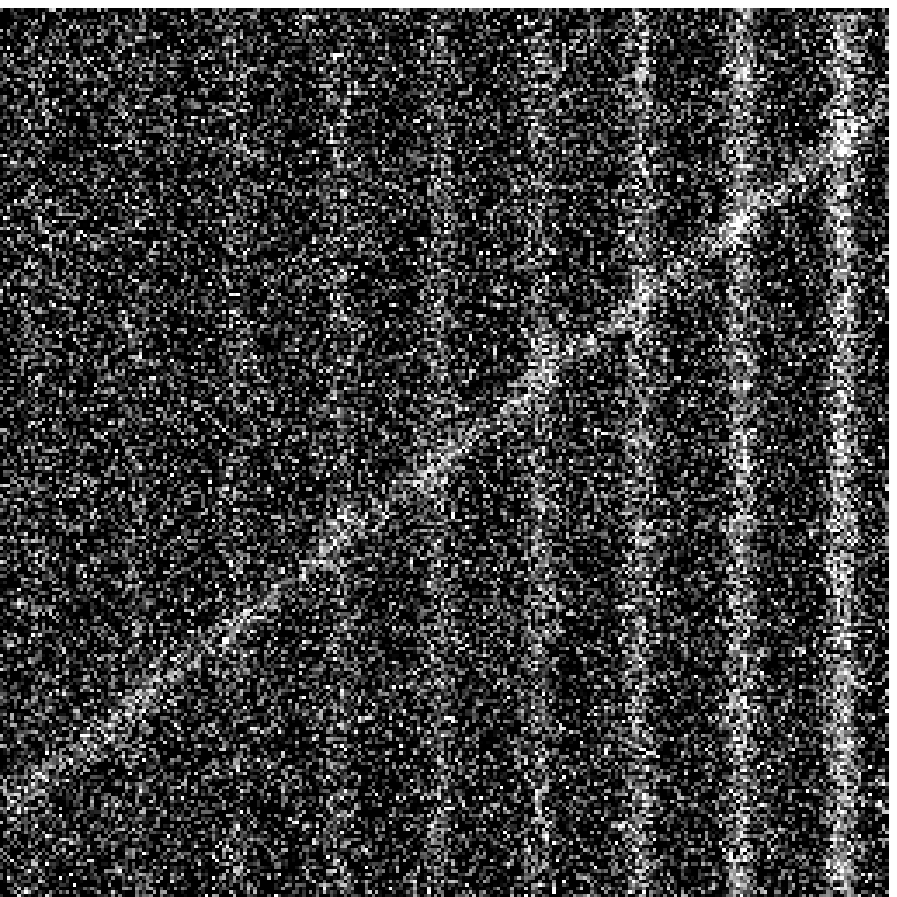}}  
\subfigure[Barbara]{\includegraphics[width=0.118\linewidth]{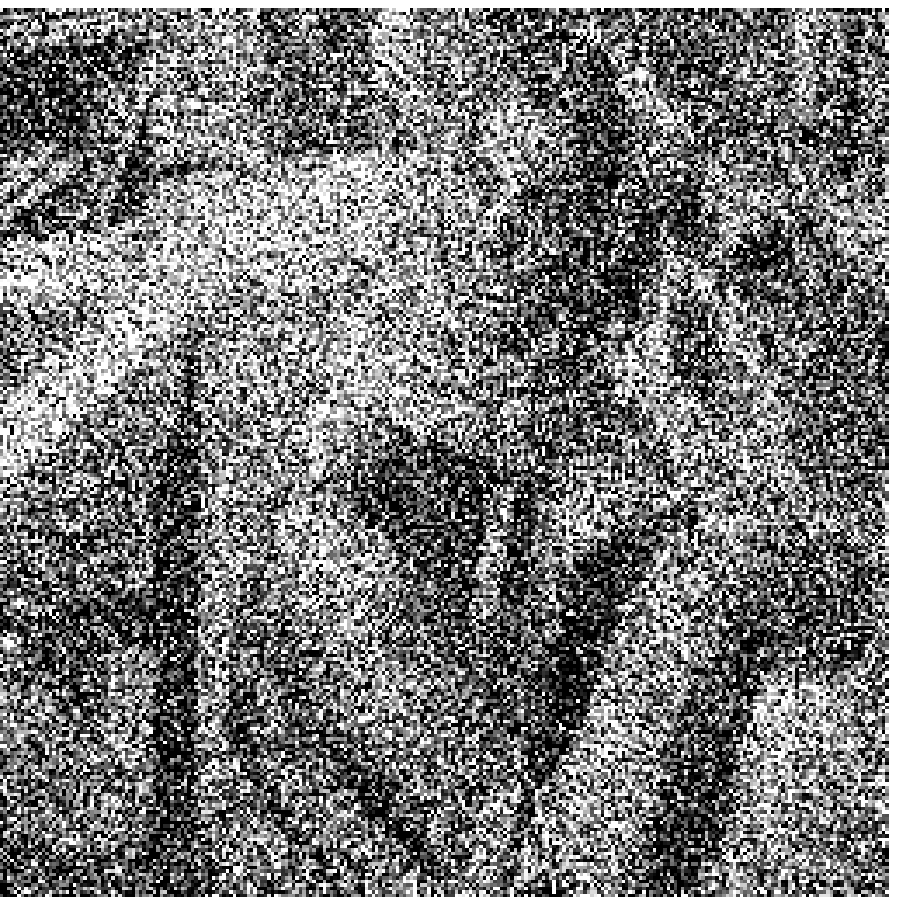}}
\subfigure[Cameraman]{\includegraphics[width=0.118\linewidth]{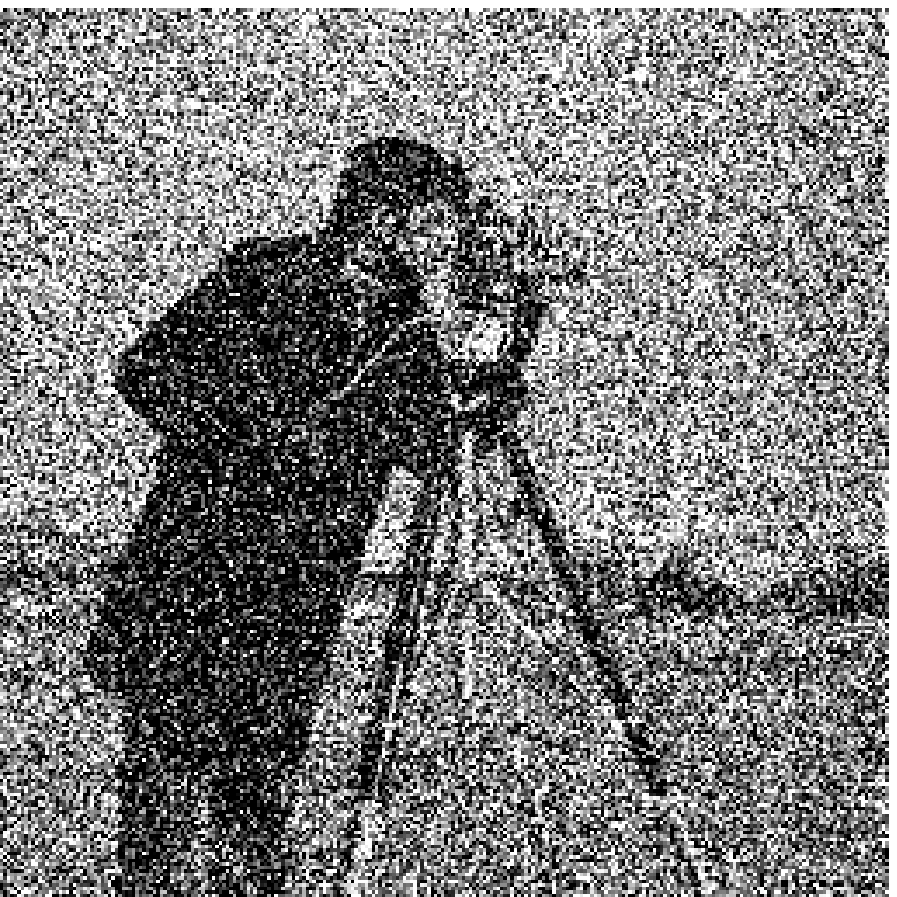}}   
\caption{Original and noisy images: cartoon (Blob, Bowl), thin features (Swoosh),  texture (Stripes) 
and natural images (Ridges, Barbara, Cameraman). 
}
\label{fig:original_noisy_images}
\end{figure}

\subsection{Thin features and textures}
In addition to considering cartoon images as defined above, we will
consider images which contain other features common in natural images,
such as thin regions a few pixel wide and regular textures.  We
consider simple models for these and show that YF and, more generally,
the NLM perform much better than linear filtering.  These models are
instances of the cartoon model where the forefront $\Omega$ varies
with $n$.  Let $\cF(\alpha, \cf)$ be defined as $\cF^{\rm
  cartoon}(\alpha, \cf)$ but without constraints on $\Omega$.
 
As a simple model of thin feature, consider an image $f$ in the
cartoon family, but where $\Omega$ is a thin $d_0$-dimensional surface
of thickness $a$ --- which will vary with $n$. A classical example of
this kind of structure is the support bar of the Cameraman's tripod,
see \figref{original_noisy_images}(g).  An example of function
from this class is illustrated by the Swoosh image, see
\figref{original_noisy_images}(c).

\begin{definition}[Thin feature function class]\label{def:thin}
$$\cF^{\rm thin}(\alpha, \cf, d_0, a) := \left\{ f \in \cF(\alpha, \cf): \Omega = \{x = (x', z): \dist(z, \phi(x')) < a\}\right\},$$
where $\phi : (0,1)^{d_0} \to (0,1)^{d-d_0}$ is $\cf$-Lipschitz.
\end{definition}

We may similarly define a class of regular pattern functions which
themselves may not be smooth, but which occur repeatedly across the
image domain.  This structure would be difficult to exploit with, say,
wavelet-based methods that fail to take advantage of image redundancy.
However, empirical evidence suggests that non-local adaptive kernels
can perform quite well on these images. A classical example of this
type of image structure is the striped scarf in the Barbara image.
The following is a class where $\Omega$ is made of the disjoint union
of translates of a smaller region $\Omega_0$ of diameter of order $a$
--- which will vary with $n$.

\begin{definition}[Regular pattern function class]\label{def:pattern}
  $$\cF^{\rm pattern}(\alpha, \cf, a) := \left\{ f \in \cF(\alpha,
    \cf): \Omega = (0,1)^d \cap \bigcup_{v \in a \Z^d} (\Xi + v)
  \right\},$$ where $\Xi \subset (0,a)^d$ is any set.  Note that the
  union above is
  disjoint. 
\end{definition}

An example of function from this class is illustrated by the Stripes
image in \figref{original_noisy_images}(d).

\section{Background on kernel methods for denoising}
\label{sec:background}
We now describe NLM and other related methods.  The story starts
with kernel smoothing (\ie linear filtering).  Though this age-old
method (with a proper choice of kernel) is essentially optimal when
the image does not have discontinuities, its performance suffers
dramatically in the presence of edges, which it tends to blur.  YF,
and more generally NLM, attempt to choose the kernel adaptively so
as to avoid averaging over the discontinuity.  

The estimates we consider are weighted averages of the pixel values
of the form
\begin{equation}\label{eq:neighbor_filter}
  \wh{f}_i = \displaystyle \frac{ \sum_{ j \in I_n^d} \, \wnlm_{i,j} \, y_j}{\sum_{j \in I_n^d} \, \wnlm_{i,j} } \, .
\end{equation}
The various methods that we study in this paper differ only in the
choice of weights $\wnlm_{i,j}$.  Adaptation to higher order of
smoothness is often accomplished by a local polynomial regression
(LPR)~\cite{Fan_Gijbels96,Hastie_Tibshirani_Friedman09}.  The local
polynomial estimator of degree $r$ and weights $(\wnlm_{i,j})$ is
\beq \label{local_poly}
\begin{cases}
&\wh{f}_i = \widehat{a}^{(i)}_0  \\
&\widehat{\ba}^{(i)}=\displaystyle\argmin_{\ba} \sum_{j \in I_n^d} \wnlm_{i,j} \left(y_j - \sum_{0 \leq |s| \leq r} a_s \, (x_j - x_i)^s\right)^2,
\end{cases}
\eeq where $ x^s := x_1^{s_1} \cdots x_d^{s_d}, $ for $x = (x_1,
\dots, x_d) \in \R^d$ and $s = (s_1, \dots, s_d) \in \R^d$, and the
minimization in \eqref{local_poly} is over $\ba = (a_s : 0 \leq |s|
\leq r)\in \R^q $ where $q = {r+d \choose d}$.  Note that, in fact,
\eqref{local_poly} leads to an estimator of the form
\eqref{eq:neighbor_filter} with different weights (\eg a smoother
kernel) \cite[p.\
34]{Tsybakov09}. We assume throughout that the polynomial degree $r$
is sufficiently large to take full advantage of the smoothness of $f$.
Specifically, if $f \in \cF^{\rm cartoon}(\alpha, \cf)$, we assume
that $r \geq \afloor$.  When the number of nonzero weights in
\eqref{local_poly} is not enough to determine $\wh{f}_i$ uniquely, we
define $\wh{f}_i$ as $y_i$, namely, we do not apply any smoothing.
Alternatively, one could decrease the degree of the polynomial
regression until the fit is well-defined, but this is not important in
our setting.

Since we know that $f$ takes values in $[0,1]$, we clip $\wh{f}$ so
that it also takes values in $[0,1]$.  This clipping does not increase
the MSE.

\subsection{Linear filtering (LF)}
\label{sec:linear}
This method can be traced back in the statistics literature to the
work of Nadaraya \cite{Nadaraya64} and Watson \cite{Watson64} 
(\lcf \cite{Hastie_Tibshirani_Friedman09} for details on kernel
methods).  In
this context the choice of the similarity between two pixels is only
controlled by spatial proximity:
\begin{equation}\label{eq:kernel_weights}
\wnlm_{i,j} = K_h(x_i,x_j) \, ,
\end{equation}
where $K_h(x,x')= K(\frac{x}{h},\frac{x'}{h})$ for a kernel function
$K$ and a bandwidth $h > 0$, which is independent of the location in
the nonadaptive (classical) version. Common choices include the
Gaussian kernel, but we focus on the box kernel
\begin{align}
K_h(x,x') &= \1{\{\|x - x'\|_{\infty} \leq h\}}\label{box-kernel}.
\end{align}

\subsection{Yaroslavsky's filter (YF)}
\label{sec:yaros}
YF was introduced by Yaroslavsky \cite{Yaroslavsky85} and
independently by Lee \cite{Lee83}, and more modern variants such as
SUSAN \cite{Smith_Brady97} and Bilateral filtering
\cite{Tomasi_Manduchi98}.  Here, similarity between
pixels is based on their spatial distance and on the relative
proximity of image intensity at these pixels.  This translates into
choosing weights in \eqref{eq:neighbor_filter} of the form
\begin{equation}\label{eq:yaroslavsky_weights}
\wnlm_{i,j} = K_{h}(x_j,x_j) \ L_{h_y}(y_i,y_j) \,,
\end{equation}
where $K,L$ are kernels and $h,h_y$ the associated bandwidths.  $(K,
h)$ control the spatial proximity while $(L, h_y)$ control the
photometric proximity.  As in classical kernel smoothing, $h$ plays
the role of spatial bandwidth, while $h_{y}$ is a photometric
bandwidth. 
In this work we only consider the simple version using the box kernel:
\begin{align}
K_{h_{y}}(y,y') &= \1{\{|y - y'|\leq h_y\}}.
\end{align}

\subsection{Non-Local Means (NLM) and patch-based methods}
\label{sec:nlm}
NLM and other patch-based methods generalize the idea of including the
photometric proximity in the kernel.  In \cite{Buades_Coll_Morel05},
the distance between two pixels is solely measured in terms of the
discrepancy between patches surrounding the pixels considered. Though
spatial proximity was already introduced in
\cite{Buades_Coll_Morel05}, it was only mentioned as a numerical
parameter to solve a computational issue. However, later works (\lcf
\cite{Salmon10,Zontak_Irani11}) have shown that spatial proximity can improve NLM
performance.  We consider NLM with spatial proximity, which includes
the non-local version, the two being identical when $h$ is
sufficiently large.

A generic description is the following.  Let $h_{\patch} > 0$ and let
$\patch_i$ (leaving $h_{\patch}$ implicit) be the hypercube of width
$h_{\patch}$ centered at $x_i$, \ie
\begin{equation}
\patch_i =  x_i  + \left[-\frac{h_{\patch}}{2},\frac{h_{\patch}}{2}\right]^d = \left\{x: \|x - x_i\|_{\infty} \leq \frac{h_{\patch}}{2} \right\}.   
\end{equation}
Such a patch corresponds to a pixel patch of width $[h_{\patch} n] + 1$ in the
digital image (where $[a]$ denotes the largest integer not exceeding
$a \in \R$).  Let $\bypi = (y_j : x_j \in \patch_i)$ be the
vector of pixel values over the patch centered at $x_i$.  With this
notation, the weights used in NLM are:
\begin{equation}\label{eq:nlm_weights}
\wnlm_{i,j}=  K_{h}(x_i,x_j) \ L_{h_y}\left( \bypi, \bypj \right) \, ,
\end{equation}
where $K,L$ are kernel functions and $h,h_{y}$ are bandwidths, as before. One classical choice of $L_{h_y}$ (which we consider in our theoretical results) is
\begin{equation}\label{eq:nlmL}
L_{h_y}\left( \bypi, \bypj \right) = \1\{\|\bypi - \bypj\|_2 \le h_y\}.
\end{equation}
The photometric similarity is based on the Euclidean distance between
the patches (as vectors) around the pixels.  We refer to this as
``classical'' or Euclidean NLM (or just NLM).

Computing $L_{h_y}$ can be computationally intensive for large $h_{\patch}$. To
address this, some authors have considered projecting
$\bypi$ onto a low-dimensional subspace and using this
projection to compute an approximation of
$L_{h_y}(\bypi,\bypj)$.  This introduces an
interesting trade-off between computational complexity and accuracy
which is examined in \cite{Azzabou_Paragios_Guichard07,Tasdizen09}.
In this paper, we consider a $1$-dimensional projection introduced
in~\cite{Mahmoudi_Sapiro05} where patches are simply compared via their
means alone, resulting in a photometric kernel of the form
\beq \label{mean_patch} 
L_{h_y}\left( \bypi, \bypj \right) = \overline{L}_{h_y}\left( \ypbari, \ypbarj
\right), \quad \ypbari := \ave(\bypi). 
\eeq
We refer to this method as
NLM-average.   For our theoretical results, we consider the kernel
\begin{equation} \label{eq:nlmML}
\overline{L}_{h_y}\left( \ypbari, \ypbarj \right) = \1\{|\ypbari - \ypbarj| \le h_y\}.
\end{equation}


In our analysis, Euclidean NLM \eqref{eq:nlmL} and NLM-average
\eqref{eq:nlmML} behave similarly, except for the regular pattern
model, where the former is generally superior.  In practice, however,
we note a difference.  In smooth regions, the average in
\eqref{eq:nlmML} has little bias and little variance, making it
significantly more robust to noise than the Euclidean distance
\eqref{eq:nlmL}. Near edges or patterns, however, the bias of the
average in \eqref{eq:nlmML} can outweigh the variance, making 
Euclidean NLM \eqref{eq:nlmL} superior.  This insight is supported by our
experimental results in \secref{experiments}.

The spatial bandwidth $h$ is typically larger than the patch width
$h_{\patch}$.  Common sizes used in practice are $21 \times 21$ kernel windows
(also referred to as the searching zone) and $7 \times 7$ patches (in
pixel units). Common kernels are the box-kernel for $K$ and the
Gaussian kernel for $L$.  Though we assume box kernels for both, our results extend readily to other kernel functions.

\section{Oracle inequalities and minimax results for cartoon images}
\label{sec:results}

We analyze the performance of the kernel-based methods described in
\secref{background} within the mathematical framework detailed in
\secref{setting}.  Qualitatively speaking, our theoretical results are
congruent with what is observed in practice; see our experiments in
\secref{experiments}.

Indeed, we show that LF blurs edges, which is
in fact well-known both in theory and practice.  YF performs well when
the JNR is large, and poorly otherwise.  This filter relies on a clear
gap between the pixel values on either side of the discontinuity: when
the JNR is large, there is indeed a gap, which ceases to exist when
the JNR is of order 1 (\lcf \figref{disks}).  The
latter situation is where NLM shines.  Indeed, patches of size larger
than one pixel gather more information about the area surrounding
the pixel, which NLM (implicitly) uses to assess whether two pixels
are on the same size of the discontinuity.  For example, comparing
patches in \figref{disk_patch}, we see that the means of sufficiently
large patches allow us to estimate reliably whether each center pixel
is in $\Omega$ or not, even with an JNR of order 1.

\begin{figure}[htb]
\centering
\subfigure{\includegraphics[width=0.19\linewidth]{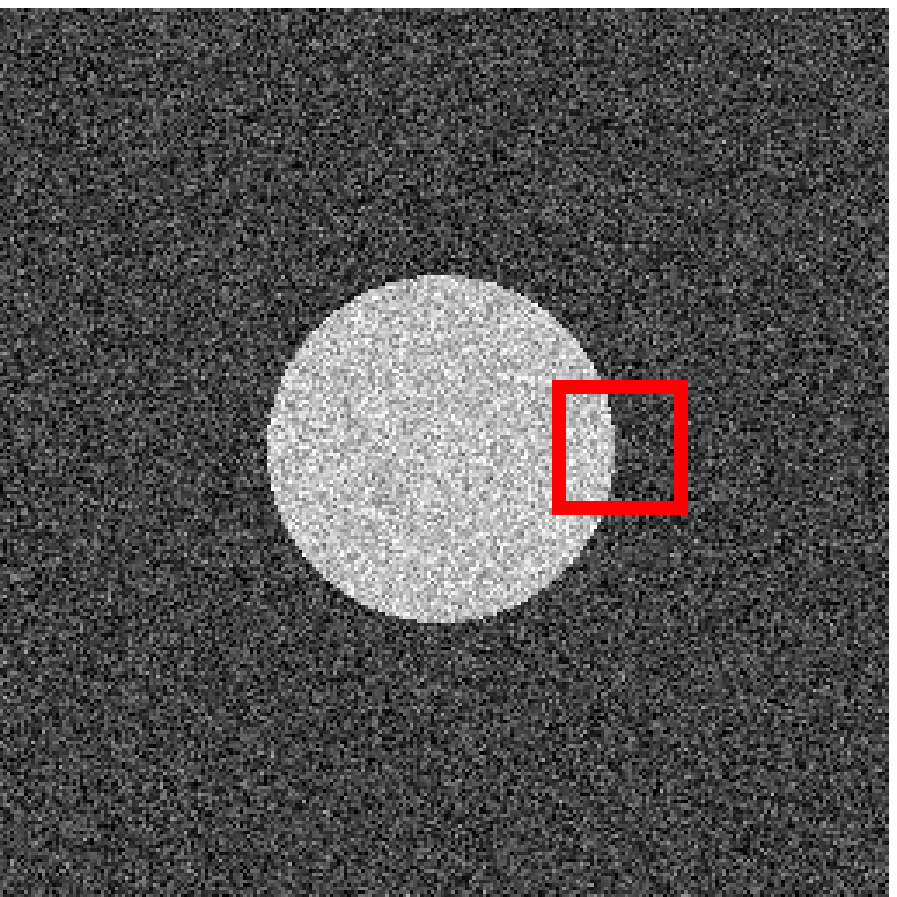}}
\subfigure{\includegraphics[width=0.19\linewidth]{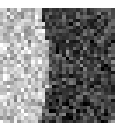}}
\subfigure{\includegraphics[width=0.19\linewidth,height=0.186\linewidth]{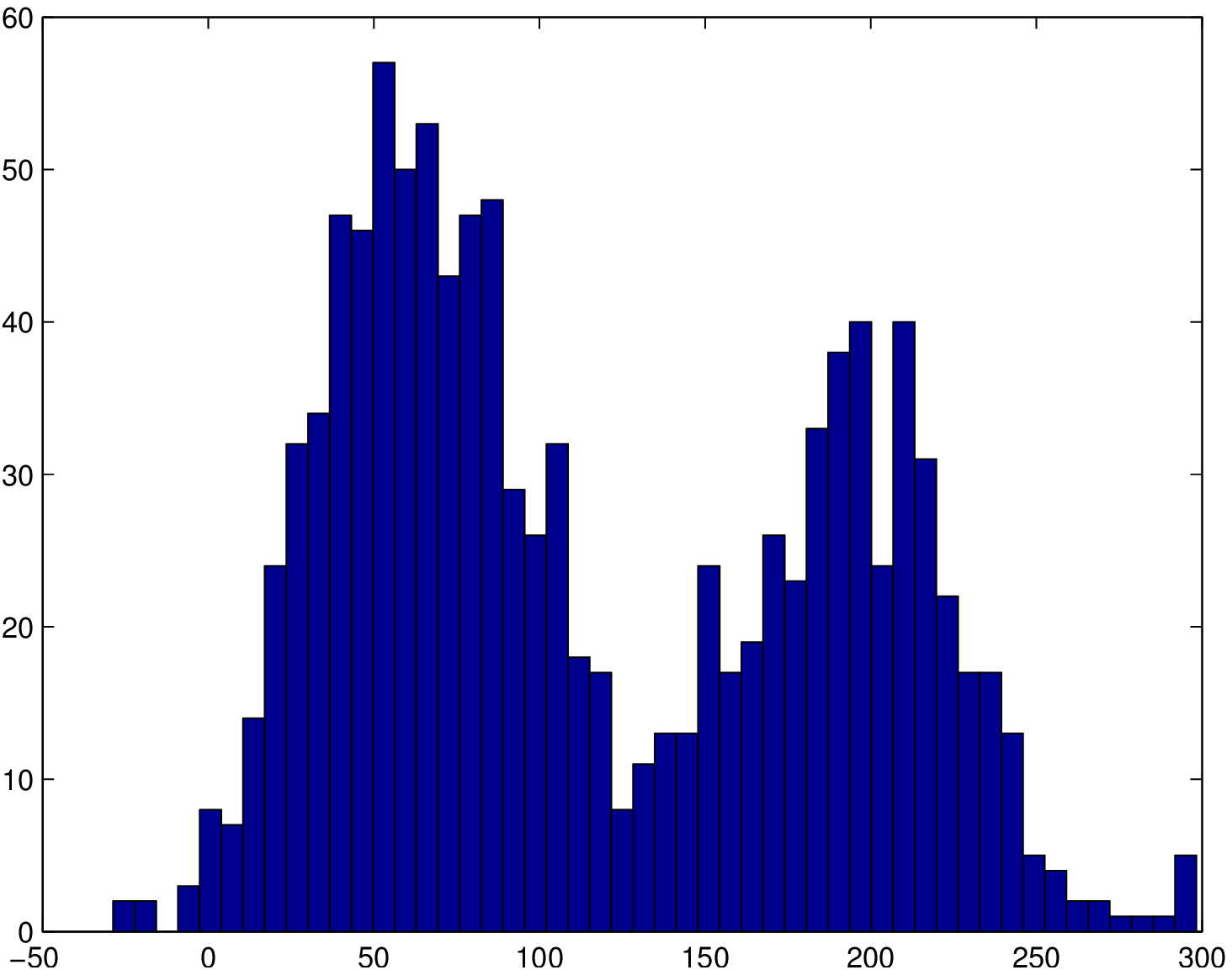}}
\\[-0.070cm]
\subfigure{\includegraphics[width=0.19\linewidth]{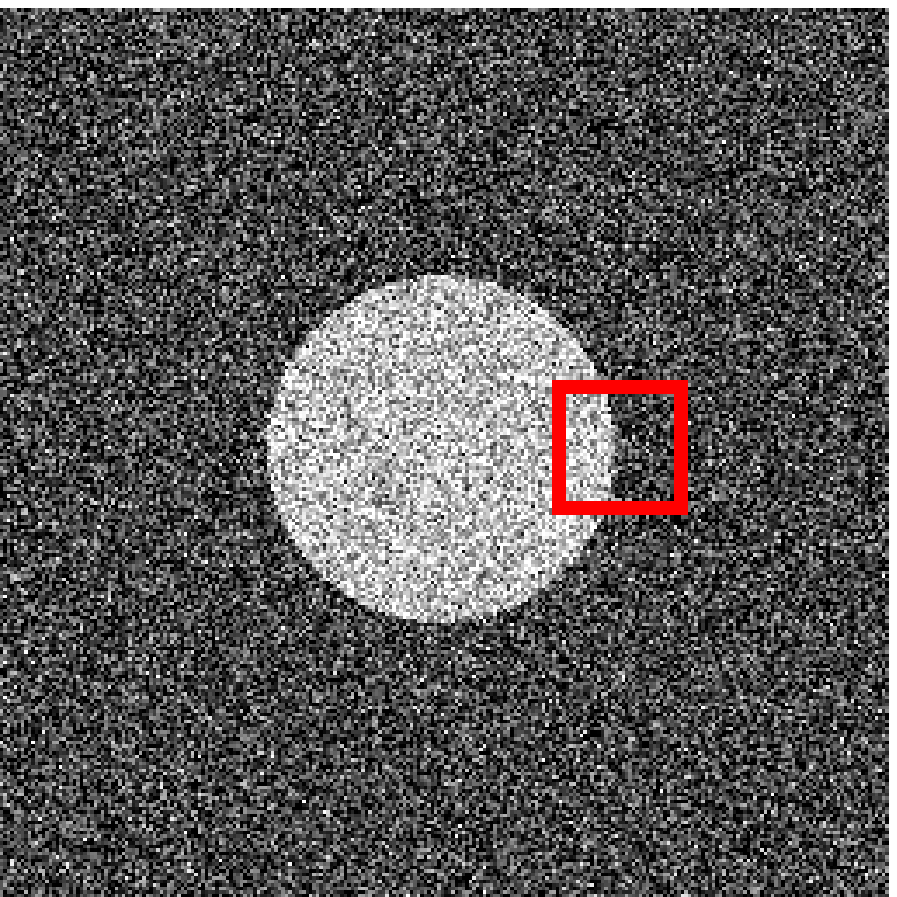}}
\subfigure{\includegraphics[width=0.19\linewidth]{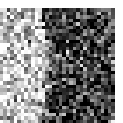}}
\subfigure{\includegraphics[width=0.19\linewidth,height=0.186\linewidth]{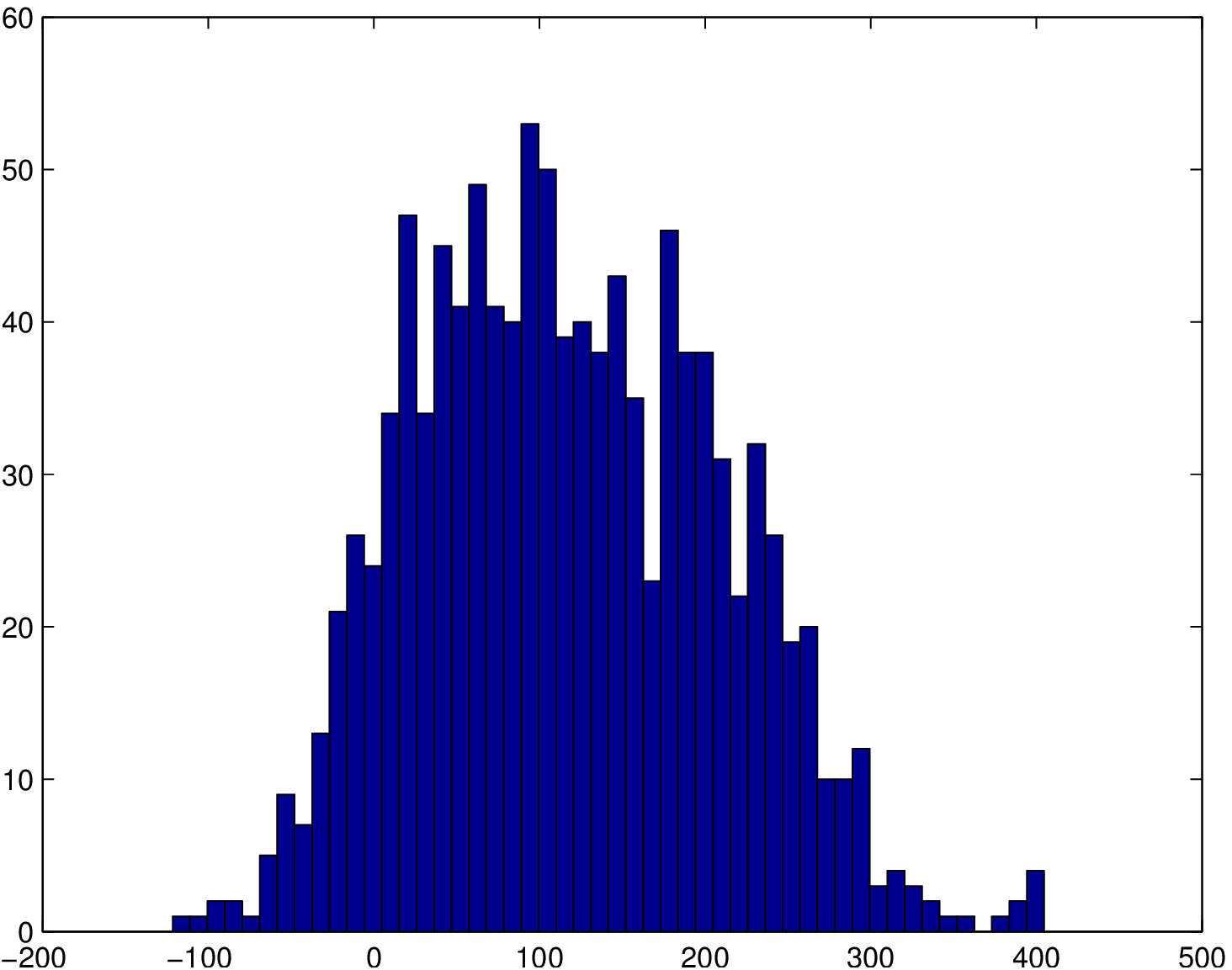}}
\\[-0.070cm]
\subfigure{\includegraphics[width=0.19\linewidth]{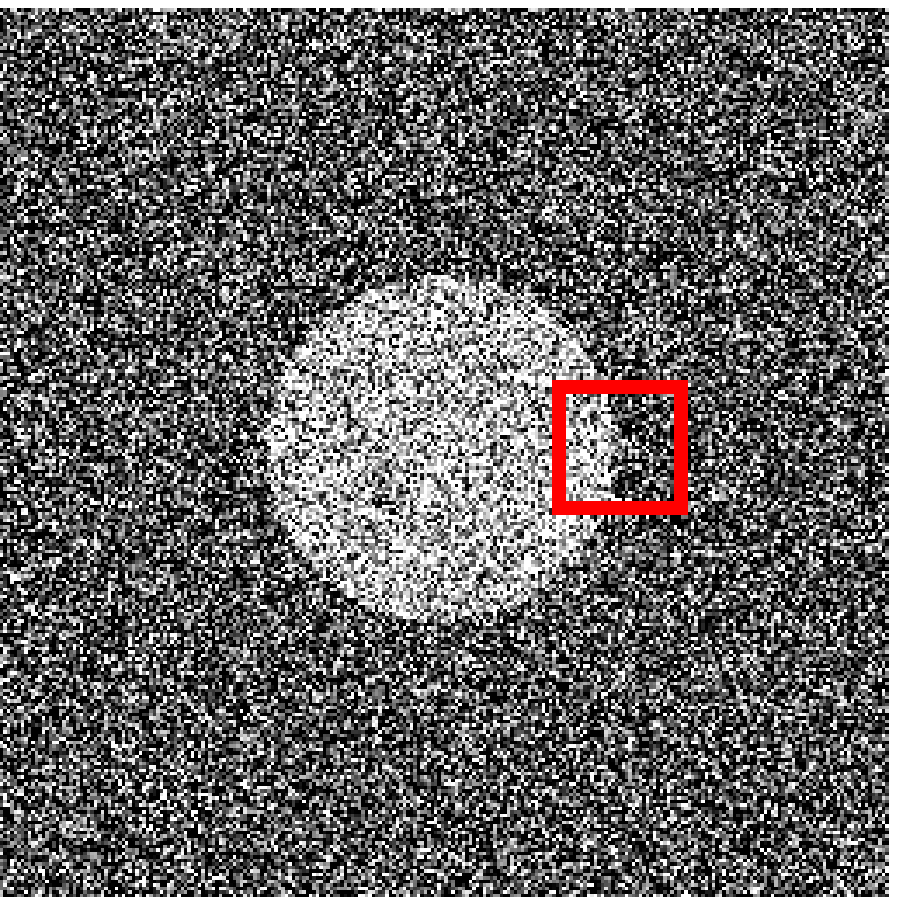}}
\subfigure{\includegraphics[width=0.19\linewidth]{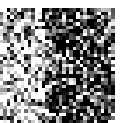}}
\subfigure{\includegraphics[width=0.19\linewidth,height=0.186\linewidth]{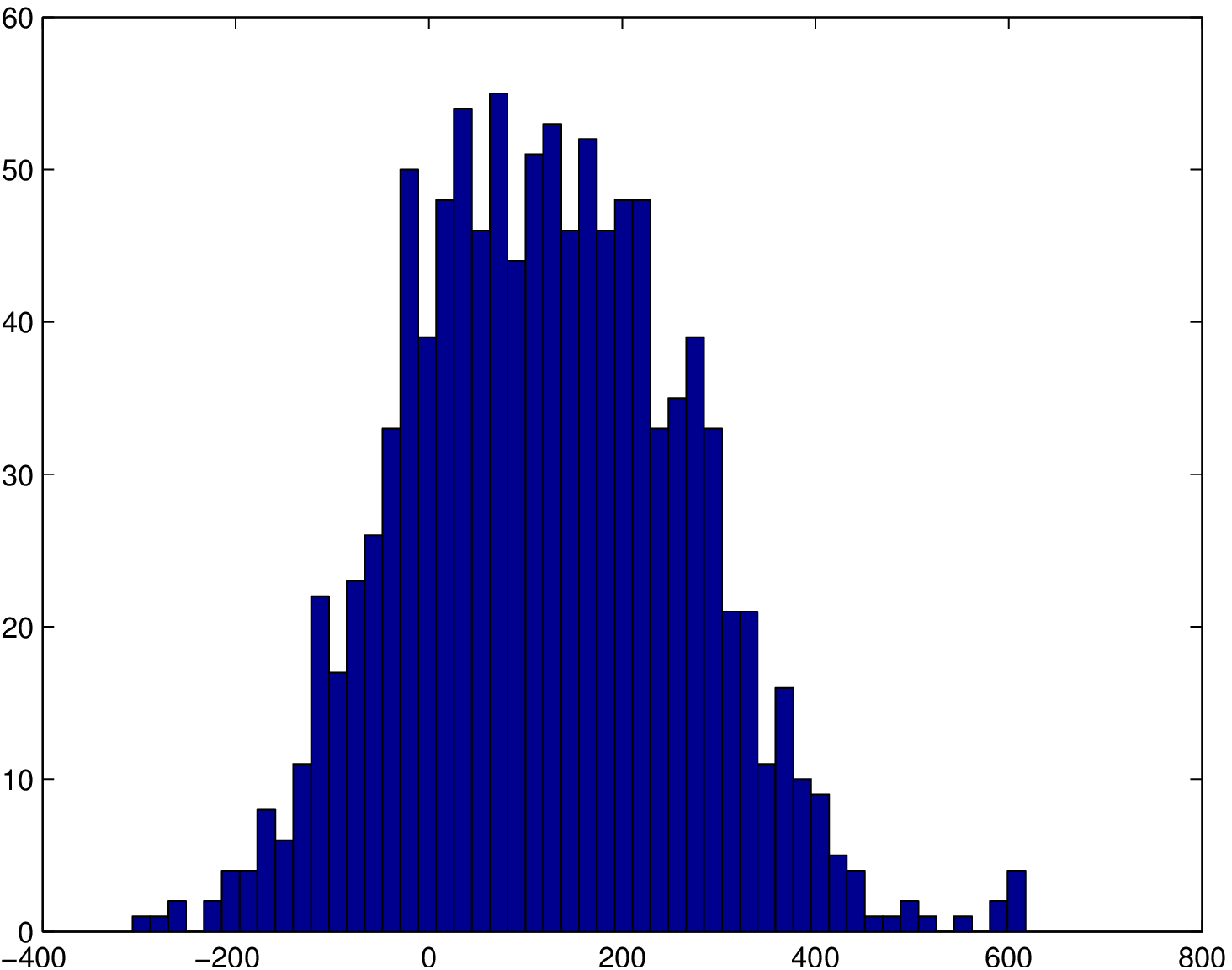}}
\caption{On the left column are cartoon images with
  increasing levels of noise (rows are with $JNR= 4, 2,
  1$ from top to bottom).  A searching zone is displayed in red, for a
  pixel near the discontinuity.  The middle column is a close-up
  of the searching zone, while the right one provides histograms
  of pixel values within
  it.}
\label{fig:disks}
\end{figure}

\begin{figure}[hbt!]
\centering
\subfigure{\includegraphics[width=0.19\linewidth]{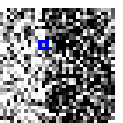}}
\subfigure{\includegraphics[width=0.19\linewidth]{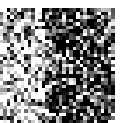}}
\subfigure{\includegraphics[width=0.19\linewidth,height=0.186\linewidth]{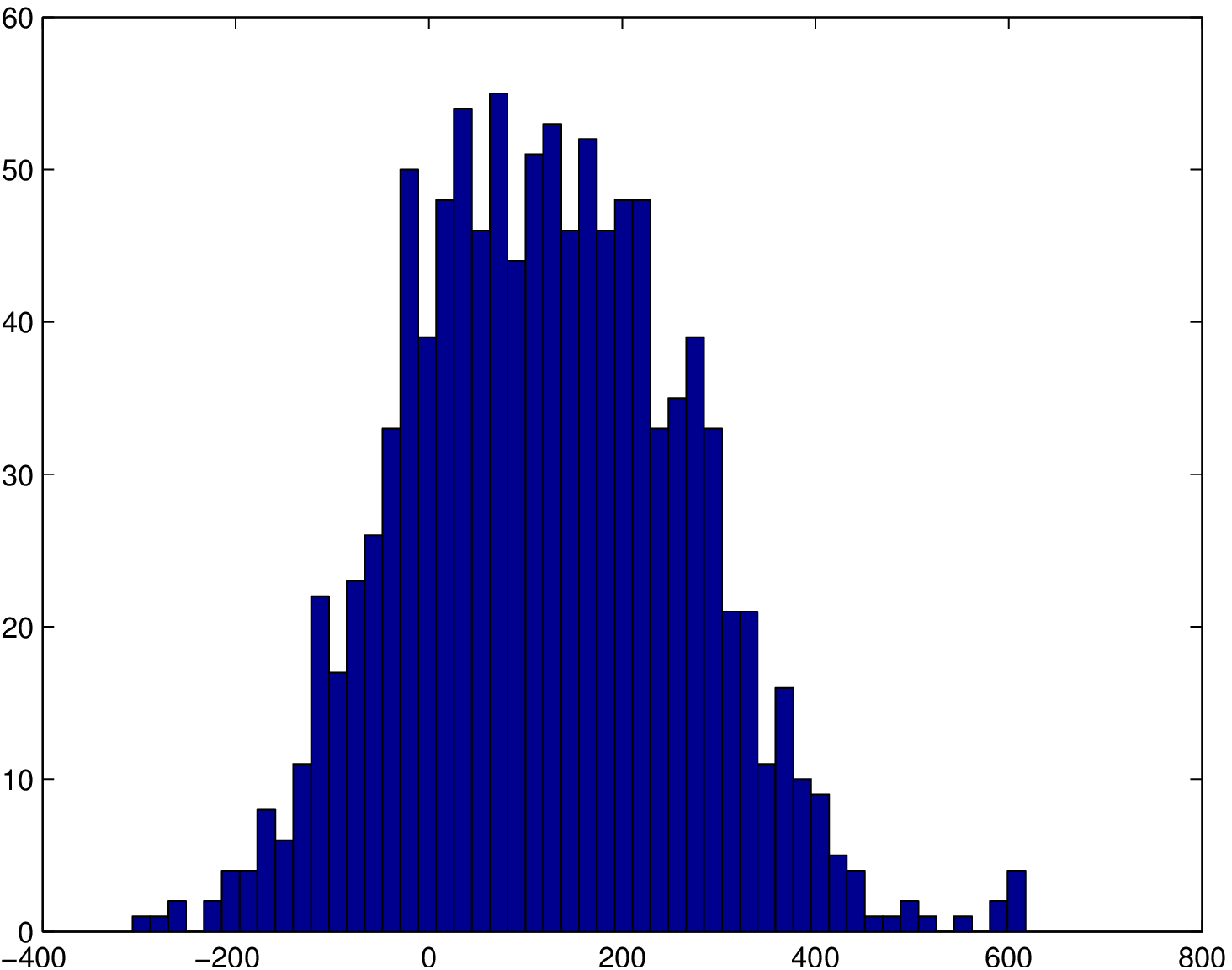}}
\\[-0.070cm]
\subfigure{\includegraphics[width=0.19\linewidth]{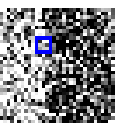}}
\subfigure{\includegraphics[width=0.19\linewidth]{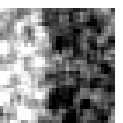}}
\subfigure{\includegraphics[width=0.19\linewidth,height=0.186\linewidth]{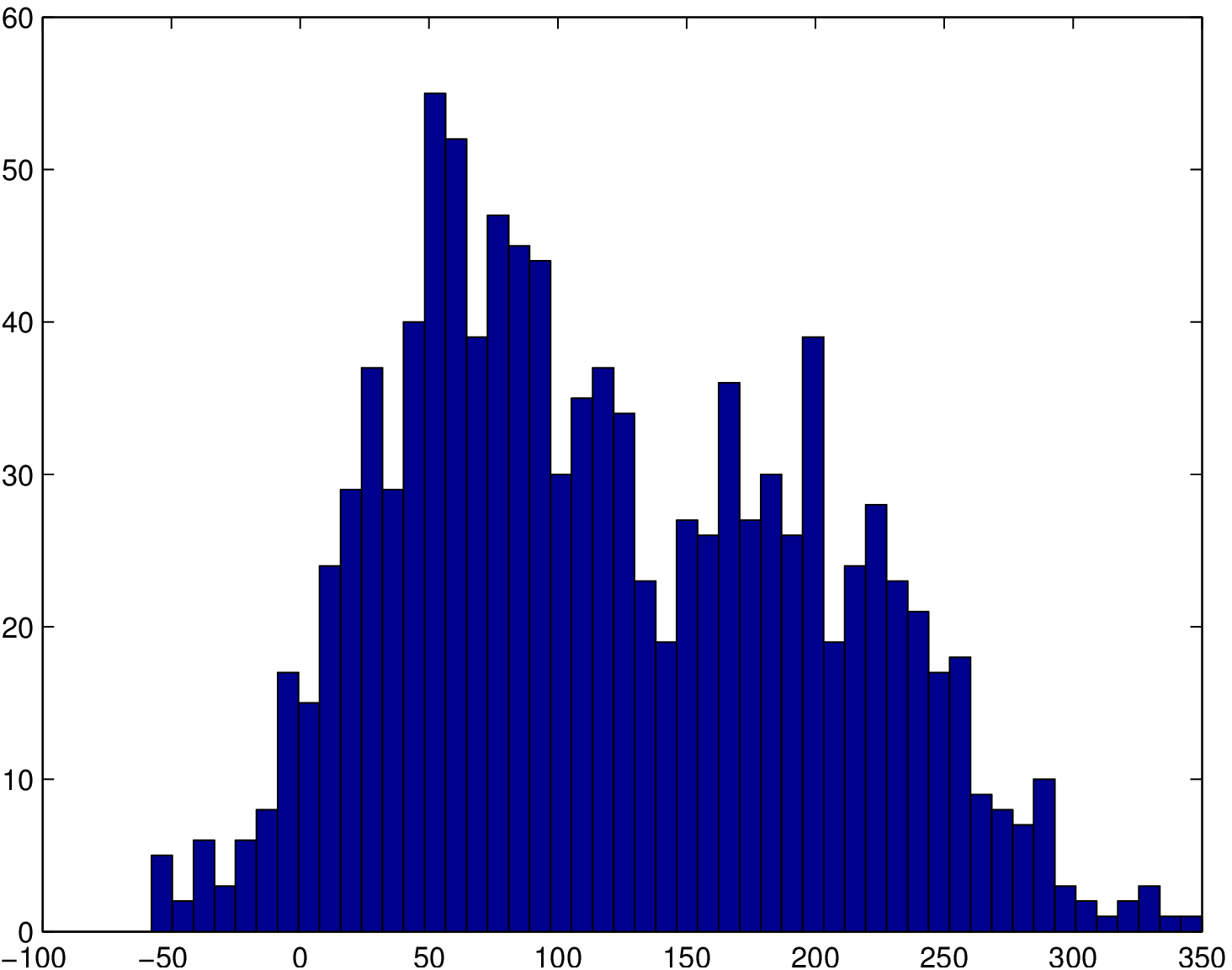}}
\\[-0.070cm]
\subfigure{\includegraphics[width=0.19\linewidth]{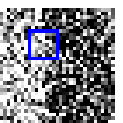}}
\subfigure{\includegraphics[width=0.19\linewidth]{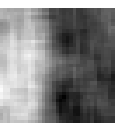}}
\subfigure{\includegraphics[width=0.19\linewidth,height=0.186\linewidth]{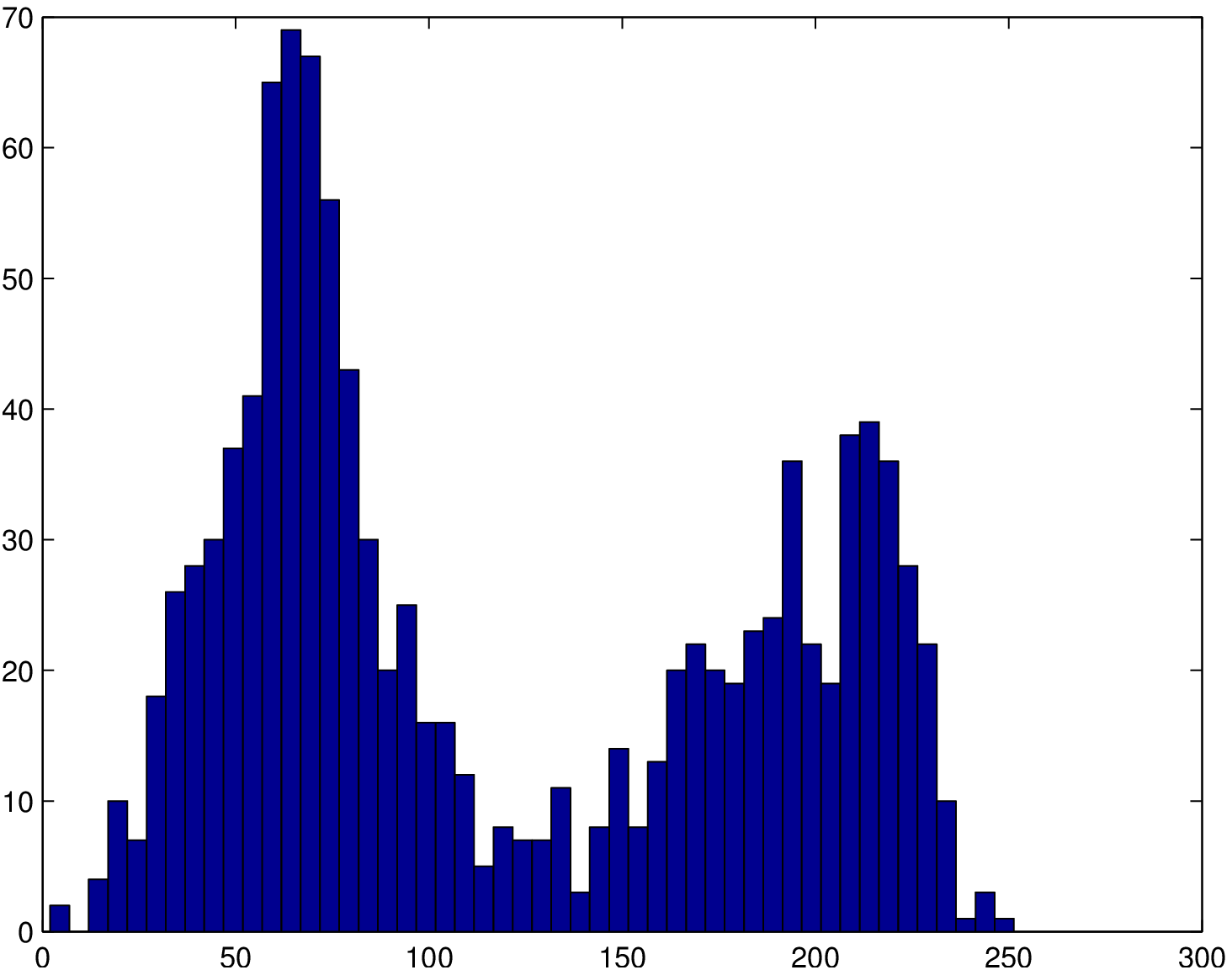}}
\caption{We use again the image from \figref{disks}, last row
  (JNR=1).  The noisy image is displayed in the first column, with
  kernel supports. The second column is the result of the
  local (box) kernel averaging using support of width 1, 3 and 7 from
  top to bottom. The last column provides histograms of the filtered
  pixels }
\label{fig:disk_patch}
\end{figure}

In what follows, we focus on the LPR variants described
in \eqref{local_poly} to avoid a systematic bias that conventional
weighted average variants suffer from. It is well-known that this bias
appears near the boundary of the image, though this can be corrected
with a proper extension of the image. More importantly, this bias
arises also near the discontinuity.  Note that enforcing the spatial
windows to have the same (symmetric) shape puts a real constraint on
the resulting performance of the algorithm, as discussed in
\secref{variable}.  The choice of kernel $K$ for LPR variants
\eqref{local_poly} is unimportant for standard kernel regression as long as it satisfies some basic
properties. (For example, in \cite[Th.~3.1]{Fan_Gijbels96}, the kernel does not impact 
the error rate except for a multiplicative constant.) Less is known
about the impact of the choice of $L$. In this paper, we consider box kernels for both 
spatial and photometric components, namely \eqref{box-kernel},
\eqref{eq:nlmL} or \eqref{eq:nlmML}.

To obtain our bounds, we minimize the error with respect to the
  bandwidth parameter $h$, effectively striking a good balance between
  the bias and variance in \eqref{bias-var}.  Indeed, the larger the
bandwidth, the larger the bias and the smaller the variance.  The
issue with kernel smoothing --- whether in the form of weighted
average \eqref{eq:neighbor_filter} or LPR \eqref{local_poly} --- is
that it suffers from a substantial bias when the smoothing window
(those points where the weights are equal to one) includes points from
the ``other side'' of the discontinuity.  At the same time, the window
cannot be too small, for otherwise the variance will be overwhelming.

\subsection{Linear kernel smoothing blurs edges}
It is well-known that LF blurs discontinuities.
This comes from the fact that the window size is fixed --- the same at
all pixels --- so points near and points far from the discontinuity
are treated in the same way.  This lack of adaptivity leads to a
substantial MSE.  How does that statement translate into a
mathematical result within our framework?  The following is proved
in~\cite{AriasCastro_Donoho09} for $d \leq 2$, though the result (at
least the upper bound) is probably older; see also \cite{Korostelev_Tsybakov93}.  
We provide a proof for LPR in \secref{proofs}.

\medskip
\begin{thm} \label{thm:linear} Let $\wh{\bbf}_h^{\rm LF}$ denote the
linear estimator, in the form of either local average
\eqref{eq:neighbor_filter} or LPR \eqref{local_poly}, with weights
as in \eqref{eq:kernel_weights}.  We have
\[
\inf_h \cR_n(\wh{\bbf}_h^{\rm LF}, \cF^{\rm cartoon}(\alpha, \cf)) \asymp \cR^{\rm LF} := (\sigma^{2}/n^{d})^{1/(d+1)},
\]
and the optimal choice of bandwidth is $h \asymp h^{\rm LF} :=
(\sigma^{2}/n^{d})^{1/(d+1)}$.
\end{thm}
\medskip

Note that the bound does not depend on the regularity $\alpha \geq 1$
of the function $f$.  As apparent in the proof, this is because LF blurs edges: to
strike a good bias-variance trade-off, the smoothing window cannot be
too small, transforming sharp edges into ramps.  The resulting bias is
then larger than the bias over the smooth regions, which is where
$\alpha$ appears.

\subsection{Oracle kernel}

What can we hope to achieve with adaptive kernel methods?  Statisticians have
used the notion of an oracle to answer this
question~\cite{Donoho_Johnstone94,Johnstone98,Tsybakov09}.  We saw
that what limits linear filtering is a large bias
near the discontinuity, due to the mixing of pixels from both sides.
What if we had access to an oracle that would identify for us the
foreground and the background?

The \textit{membership oracle} tells us which sample points belong to
$\Omega$ or to $\Omega^c$.  With access to this oracle, we simply
process the smooth pieces, $\fin$ and $\fout$, separately.  By doing
so we achieve the minimax rate for the class $\cH_d(\alpha, \cf)$: the
information this oracle provides is sufficient to do as well as if
there were no discontinuity. This is illustrated in
  \figref{windows_oracles} (best viewed in color).

\begin{figure}[hbt!]
\centering
\subfigure[Kernel smoothing]{\includegraphics[width=0.24\linewidth]{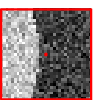}}\hspace{1cm}
\subfigure[\textit{Bandwidth oracle}]{\includegraphics[width=0.24\linewidth]{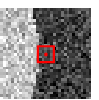}}\hspace{1cm}
\subfigure[\textit{Membership oracle}]{\includegraphics[width=0.24\linewidth]{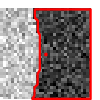}}
\caption{The kernel supports for the linear filter, the bandwidth oracle and  the membership oracle} 
\label{fig:windows_oracles}
\end{figure}

\medskip
\begin{thm} \label{thm:MO} Let $\wh{\bbf}_h^{\rm MO}$ denote LPR
  estimator \eqref{local_poly} with weights as in
  \eqref{eq:kernel_weights} when $x_i$ and $x_j$ belong to the same
  side of the discontinuity, and set to zero otherwise.  We have
\[
\inf_h \cR_n(\wh{\bbf}_h^{\rm MO}, \cF^{\rm cartoon}(\alpha, \cf))
\asymp \cR^{\rm MO} := (\sigma^{2}/n^{d})^{2\alpha/(d+2\alpha)},
\]
and the optimal choice of bandwidth is $h \asymp h^{\rm MO} :=
(\sigma^{2}/n^{d})^{1/(d+2\alpha)}$.
\end{thm}
\medskip

The lower bound is a well-known minimax bound~\cite[Theorem 5.1.2, p.\
133]{Korostelev_Tsybakov93}.  If we consider a class of piecewise
polynomial functions, then this oracle estimator, without spatial
proximity (\ie $h = \infty$), achieves the parametric rate of
$\sigma^2/n^d$.  It is worth noting that LPR plays a crucial role
here.  Indeed, the window around a point near the discontinuity ---
comprised of all points belonging to the same side of the
discontinuity --- will be irregularly shaped. For instance, imagine a
point on a linear surface adjacent to the discontinuity. For a
symmetric window (sufficiently small not to include the
discontinuity), the linear variations around the pixel of interest
will average out and we can accurately estimate the pixel value.  For
an {\em asymmetric} window caused by the discontinuity, the linear
variations will not average out, inducing a small bias and leading to
a higher risk of order $(\sigma^{2}/n^{d})^{3/(d+3)}$ when $\alpha \ge
3/2$.  This phenomenon can be observed in practice and is illustrated
in \figref{bias}.

\begin{figure}[t]
\centering
\includegraphics[width=0.71\linewidth]{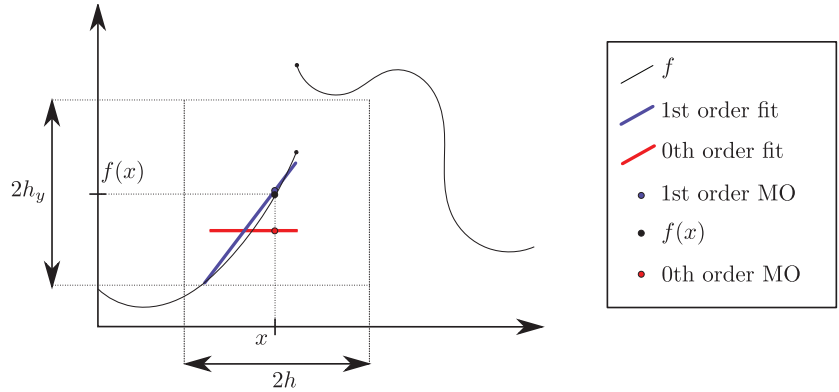} 
\caption{\textit{Membership oracles} of order 0 and 1 on a non-noisy 1D signal. Note how the bias is reduced by
going to the order 1.}
\label{fig:bias}
\end{figure}

Note that the oracle only has to provide the membership information
locally, within the searching window.  The insight we get from this is
that we only need to know over which pixel values to average to attain
the same error rate as we would without discontinuities.  This is
exactly what adaptive kernel methods~\cite{Lepski_Mammen_Spokoiny97},
including patch-based methods, PDE
methods~\cite{Perona_Malik90,Alvarez_Mazorra94,Gilboa_Osher07} and
graph diffusion methods
\cite{Szlam_Maggioni_Coifman08,Singer_Shkolniksy_Nadler09} aim at
doing.

\subsection{Variable bandwidth kernel methods}
\label{sec:variable}

\sloppypar These
methods~\cite{Muller_Stadtmuller87,Spokoiny98,Katkovnik99,Katkovnik_Egiazarian_Astola02,Katkovnik_Foi_Egiazarian_Astola04},
including Lepski's method~\cite{Lepski_Mammen_Spokoiny97} and
variants~\cite{Polzehl_Spokoiny00,Polzehl_Spokoiny03}, choose the
bandwidth adaptively at every location, the goal being to avoid
smoothing over discontinuities and to adapt to the regularity of the
signal when unknown.  Wavelet shrinkage methods are often thought to
perform some sort of variable-bandwidth kernel
smoothing~\cite{Donoho_Johnstone94}.  
Clearly, we cannot do better than if we knew the
discontinuity, meaning if we had access to the \textit{membership
  oracle}.  In that case, at each point we would choose the bandwidth
equal to its distance to the discontinuity (BO below stands for
\textit{bandwidth oracle}). See \figref{windows_oracles} for a
comparison of the MO and BO spatial supports.

\medskip
\begin{thm} \label{thm:BO} Let $\wh{\bbf}_h^{\rm BO}$ denote LPR
  estimator \eqref{local_poly} with weights chosen as in
  \eqref{eq:kernel_weights} when $\|x_i - x_j\|_\infty <
  \dist(x, \partial \Omega):=\displaystyle\inf_{y \in \partial \Omega}
  \|x - y\|_{\infty}$, and set to zero otherwise.  We have
\[
\inf_h \cR_n(\wh{\bbf}_h^{\rm BO}, \cF^{\rm cartoon}(\alpha, \cf)) \asymp \cR^{\rm BO} := (A_n \sigma^2/n) \vee (\sigma^{2}/n^{d})^{2\alpha/(d+2\alpha)},
\]
where $A_n = \log n$ when $d=1$ and $A_n = 1$ when $d \geq 2$, for an
optimal choice of maximal bandwidth $h \asymp h^{\rm MO}$.
\end{thm}
\medskip

Note that BO achieves the error rate of MO only when $d=1$, when $d=2$
and $\alpha = 1$, or when $d \geq 2$ and $\sigma^2 = O(n^{- 2\alpha
  (1-1/d) + 1})$, which is polynomially small when $d = 2$ and $\alpha
> 1$ or when $d \geq 3$.  Thus in general,  BO is substantially weaker than MO.
That said, BO achieves the minimax rate established in \cite{Tsybakov89} when $\sigma$ is fixed.

\subsection{Yaroslavsky's filter is oracle-optimal under low noise}

As the practitioner knows, YF can be quite good on natural images.  In
fact, it can dramatically outperforms the linear filter and compares
favorably with methods such as wavelet thresholding, particularly when
the noise level is small.  We substantiate this empirical evidence
with a theoretical study of its performance, showing it achieves
 MO bound in such situations (\ie when $\sigma$ is small).

Assume that for a fixed cumulative distribution function $F$, the
noise satisfies the following \beq\label{F} \P(|\eps_i| \le t) \ge
F(t/\sigma), \ \forall t, \ \forall i \in I_n^d, \eeq
  
The following result states that YF achieves a performance comparable
to that of MO if $\sigma$ is small.  We only require
that the noise distribution in \eqref{F} has quickly decaying tails.

\medskip
\begin{thm} \label{thm:Y} Let $\wh{\bbf}_{h, h_y}^{\rm YF}$ denote the
  LPR estimator \eqref{local_poly} with Yaroslavsky's
  weights~\eqref{eq:yaroslavsky_weights}. Suppose that, for some
  constants $C,b>0$, \eqref{F} holds with $1-F(t) \leq C
  \exp(-(t/C)^b)$ for $t$ large enough.  Then there is another
  constant $C' > 0$ such that, if $\sigma \leq (C' \log n)^{-1/b}$,
\[
\inf_{h, h_y} \cR_n(\wh{\bbf}_{h, h_y}^{\rm YF}, \cF^{\rm cartoon}(\alpha, \cf)) \leq (1 + o(1)) \cR^{\rm MO},
\]
where an optimal choice of bandwidths is $h \asymp h^{\rm MO}$ and
$h_y \asymp 1$.  
\end{thm}
\medskip

Gaussian noise satisfies the requirements of \thmref{Y} with $C = \sqrt{2}$
and $b = 2$, resulting in the constraint $\sigma = O(1/\sqrt{\log
  n})$, which is quite mild. This explains why YF tends to
perform well in practice, at least for low noise level.

This excellent performance hinges on the fact that the photometric
kernel is able to mimic the \textit{membership oracle} when the noise
level is small.  When the noise level is of order 1 or larger, this is
no longer true, as illustrated in \figref{disks}.  There, we clearly
see that in a window containing points from both $\Omega$ and its
complement, the pixel values are mixed in the histogram if the noise
level is too large, making a clear separation impossible.  We formally
argue this point after the proof of \thmref{Y} in
\secref{proof-Y}.  

It is worth noting that the proof helps clarify exactly the artifacts
encountered in practice by the YF for strong noise (\lcf
\figref{Sinusoid_sigma=100}). Indeed, the output often looks like the
original scene contaminated by something like ``salt and pepper''
noise. As mentioned in the proof, this is because the YF does not
alter pixels with extreme values.

\subsection{Performance analysis for Non-Local Means}

In the previous section we established that YF performs as well as MO
when the noise level is small, while it is useless otherwise.  A
natural strategy consists of, first, reducing the noise level by
averaging and, then, applying YF.  This is almost exactly what
NLM-average does.  We precisely quantify the MSE performance of both
NLM-average and Euclidean NLM in this section.  Note that we state our
results for i.i.d.~Gaussian noise for simplicity, though they are
valid for many other distribution families such as uniform and
double-exponential.

\medskip
\begin{thm} \label{thm:NLM} Let $\wh{\bbf}_{h, h_y}^{\rm NLM}$ denote
  LPR estimator \eqref{local_poly} with NLM
  weights~\eqref{eq:nlm_weights} and photometric kernel either
  Euclidean \eqref{eq:nlmL} or Average \eqref{eq:nlmML}.  
If the noise conditions of
  Theorem~\ref{thm:Y} hold,
then
\[
\inf_{h, h_y} \cR_n(\wh{\bbf}_{h, h_y}^{\rm NLM}, \cF^{\rm
  cartoon}(\alpha, \cf)) \leq
(1 + o(1)) \cR^{\rm MO},
\]
where $h_{\patch} = 1/n$. Otherwise, assuming
  $\sigma$ is bounded away from 0, we have
\[
\inf_{h, h_y} \cR_n(\wh{\bbf}_{h, h_y}^{\rm NLM}, \cF^{\rm
  cartoon}(\alpha, \cf)) \preceq 
\cR^{\rm NLM} := (B_n/n) \vee
(\sigma^{2}/n^{d})^{2\alpha/(d+2\alpha)},
\]
where $B_n := (\sigma^{4} \log n)^{1/d}$ (Euclidean) or $:=
(\sigma^{2} \log n)^{1/d}$ (Average), and an optimal choice of
bandwidths is $h \asymp h^{\rm MO}$, $h_y \asymp h_y^{\rm NLM} :=
\sigma^3 \sqrt{\log n}$ (Euclidean) or $:= (2\cf)^{-d} \mu/2$
(Average), and $h_{\patch} \asymp h_{\patch}^{\rm NLM} := B_n/n$.
\end{thm}
\medskip
In other words, if the low-noise conditions of
  Theorem~\ref{thm:Y} hold, then the optimal patch size is a single
  pixel, and the NLM is exactly YF and we achieve the YF performance
  bound. There is an elbow in the performance bound, since once the
  optimal patch size exceeds a single pixel, estimation errors within
  a patch sidelength of the boundary impact the performance.

There is a strong correspondence between this bound and the
BO bound in \thmref{BO}.  If $\sigma$ is fixed, $\cR^{\rm NLM} \asymp \cR^{\rm BO}$ for
$d=1$ and $\cR^{\rm NLM} \asymp (\log n)^{1/d} \cR^{\rm BO}$ for
$d\geq 2$, therefore $\cR^{\rm BO}$ is the minimax rate \cite{Tsybakov89} and NLM is minimax optimal up to a logarithmic factor. 

Note that our bandwidth $h$ is not infinite as in
\cite{Maleki_Narayan_Baraniuk11}. There, the authors use an infinite
window for searching for matching patches: this is optimal in their
setting because they consider piecewise {\em constant} images.  In our
setting, images are piecewise smooth, and a smaller bandwidth can not
only help us reduce the risk of our estimate, but also lead to more
computationally efficient estimation algorithms.


\section{Performance analysis for thin features and textures}
\label{sec:thin-pattern}

In the cartoon model of \secref{setting} with JNR of order 1, the
performance of NLM is comparable to that of variable bandwidth kernel
smoothing, and actually that of wavelets as well
\cite{Kervrann_Boulanger06,Portilla_Strela_Wainwright_Simoncelli03}.
In natural images, however, NLM can perform substantially better.  We
explain this by the fact that the cartoon model we considered so far,
though useful as a benchmark, does not account for features common in
natural images, particularly, very thin regions a
few pixels wide and regular textures.

Below, we do as if the image contained regions of cartoon type and
regions with thin features and/or texture, and keep the same
bandwidths that we found to be optimal in the cartoon model in the
previous results.

\subsection{Thin features}

Both YF and NLM achieve a good performance on thin features.  We focus
on sample points within the feature and focus on the interesting case where the
thickness is of smaller order of magnitude than the bandwidth $h$.  \medskip
\begin{thm} \label{thm:thin} Consider $f \in \cF^{\rm thin}(\alpha,
  \cf, d_0, a)$ with band $\Omega$; assume all parameters are fixed
  except $a \ge 4/n$ and $a \rightarrow 0$ as $n \rightarrow
  \infty$ . In
  terms of {\em point-wise risk} \eqref{bias-var} at $x_i \in \Omega$, we have: \benum
\item The linear filter with bandwidth $h^{\rm LF}$ has risk of order
  1 if $a = o(h^{\rm LF})$.
\item BO with maximal bandwidth $h^{\rm MO}$ has a point-wise risk of
  order $a^{2\alpha} \vee \sigma^2 (na)^{-d}$, if $\dist(x_i,
  \Omega^c) \ge a/C$ for some $C>3$, if $n a \to \infty$.
\item MO with bandwidth $h^{\rm MO}$ has risk of order $(h^{\rm MO}/a)^{d-d_0} \cR^{\rm MO}$ if $a = o(h^{\rm MO})$. 
\item The latter is true of YF with bandwidths $h^{\rm MO}, h_y \asymp 1$, if the noise satisfies the conditions of \thmref{Y}.
\item This is also the case of NLM (Euclidean or Average) with
  bandwidths $h = h^{\rm MO}, h_y = h_y^{\rm NLM}$, and patch size
  $h_{\patch} = h_{\patch}^{\rm NLM}$, if $\dist(x_i, \Omega^c) \ge
  h_{\patch}^{\rm NLM}$.  \eenum
\end{thm}  
\medskip

In view of this result, we can say that linear filtering essentially
erases the feature. In contrast, YF still performs very well (relative
to the MO) under low
noise, and NLM performs well in this case and for higher noise
settings. Note that when $h_{\patch}^{\rm NLM} = o(a)$, the bound
above holds for most points within the thin features.  Though not
stated here, we found that NLM is able to handle such bands under
special circumstances --- when $d \ge 3$ and the band is straight.

\subsection{Regular patterns and textures}
We consider very general patterns where YF will do as well as in the
cartoon model, situations where most other methods are essentially
useless.  Euclidean NLM performs well too, under additional
assumptions on the pattern.  \medskip
\begin{prp} \label{prp:pattern} Consider $f \in \cF^{\rm
    pattern}(\alpha, \cf, a)$ with all parameters are fixed except for
  $a$, which satisfies $a = o(h^{\rm MO})$. Let $N_\Omega :=
 \# \{i: x_i \in \Omega\}$ and $N_{\Omega^c}$ is defined similarly.
  If $N_\Omega \vee N_{\Omega^c}
  \le C (N_\Omega \wedge N_{\Omega^c})$ and $na \ge (r+1)(2C+2)$, with $C > 1$ fixed, we have the
  following: \benum
\item MO with $h = h^{\rm MO}$ achieves an MSE of order
  $\cR^{\rm MO}$.
\item The latter is true of YF with bandwidths $h^{\rm MO}, h_y \asymp
  1$, if the noise satisfies the conditions of \thmref{Y}.
\item Suppose in addition that for every $x_i \in \Omega$ and $x_j \in
  \Omega^c$, 
\beq \label{pattern-NLM} 
\|\1 (\patch_i \cap \Omega) - \1 (\patch_j \cap \Omega)\|_2^2 \ge (\sigma^2 \log
  n)/C', 
\eeq 
for some $C' > 1$ fixed.  Then  (Euclidean) NLM  with
  bandwidths $h = h^{\rm MO}, h_y = h_y^{\rm NLM}$ and patch size
  $h_{\patch}^{\rm NLM}$, achieves an MSE of order $(na)^d \, \cR^{\rm MO}$.
  \eenum
\end{prp}  
\medskip

The condition \eqref{pattern-NLM} essentially means that any
  two patches, where on is centered in the foreground and the other is
  centered in the background, must be sufficiently 
  distinct -- and the necessary degree of distinction increases with
  the noise level.
For instance, \eqref{pattern-NLM} is satisfied by such patterns as a
chessboard or stripes.  A regular pattern in a real image (\eg
Barbara's blouse) is often referred to as texture, and NLM is able to
effectively denoise such patterns under some regularity conditions.
For random models of textures, such as Markov random fields, we do not
expect NLM to do well unless the pattern is not very random.  The
reason is that few patches are close in Euclidean distance to a given
patch.


\section{Experiments}
\label{sec:experiments}

In this section we provide numerical results for images with
$d=2$, whose pixel intensities are between $0$ and $255$. In our
experiments, the noise is Gaussian with standard deviation $\sigma \in
\{5, 20, 50,100\}$ (Note that this corresponds, for normalized images in
$[0,1]$ to noise with $\sigma\in \{5/255,10/255,50/255,100/255\}$).

On both toy and classical images, we have compared the behavior of the
following methods: linear filtering (LF), Yaroslavsky's filter (YF),
Euclidean NLM (NLM), average non-local means (\NLMm) and
the membership oracle (MO).  In all cases we have implemented LPR
version of the methods for the orders $r \in\{0, 1, 2\}$. Note that,
as expected, for linear filtering LPR of order 0 and 1 are exactly
identical because the support of the kernel is symmetric.  However,
for other methods the symmetry of the support is no longer guaranteed
and the estimators differ.  The higher order LPR versions are computed
by solving the linear system in \eqref{eq:LPR_system}. A small
numerical constant ($10^{-8}$) is added to the diagonal elements of
$\bX^T \bX$ so that inverting this matrix is always a well conditioned
problem.

For fair comparisons we have used the same box kernel with every method.
The patch size is $7 \times 7$ (\ie $h_{\patch}=7$). It
is kept fixed for all the methods.  For the
spatial bandwidth $h$ we have chosen to use the values obtained by
considering the best $h$ (in term of MSE) for the MO, on the Bowl image.
Thus, we use for each noise level $\sigma\in\{5, 20, 50, 100\}$ and polynomial
order $r\in\{0,1,2\}$ an $h$ optimized on Bowl.  The values are provided by the
MSE optimization in \figref{searching_zone_evolution} and summarized
in Tab.~\ref{tab:optimal_bandwidth}.
The photometric bandwidth $h_y$ is chosen by hand, and differs from
method to method: $\sqrt{10}\sigma$ (YF), $0.29\sigma$
(\NLMm), $13.1 \sigma$ (NLM), $30$ (MO). It
is to be noted that the parameters are given for comparison in between
methods, we do not claim those are the best parameters for all applications.

\begin{figure}[tbp]
\centering
\subfigure[$\sigma=5$]{\includegraphics[width=.40\linewidth]{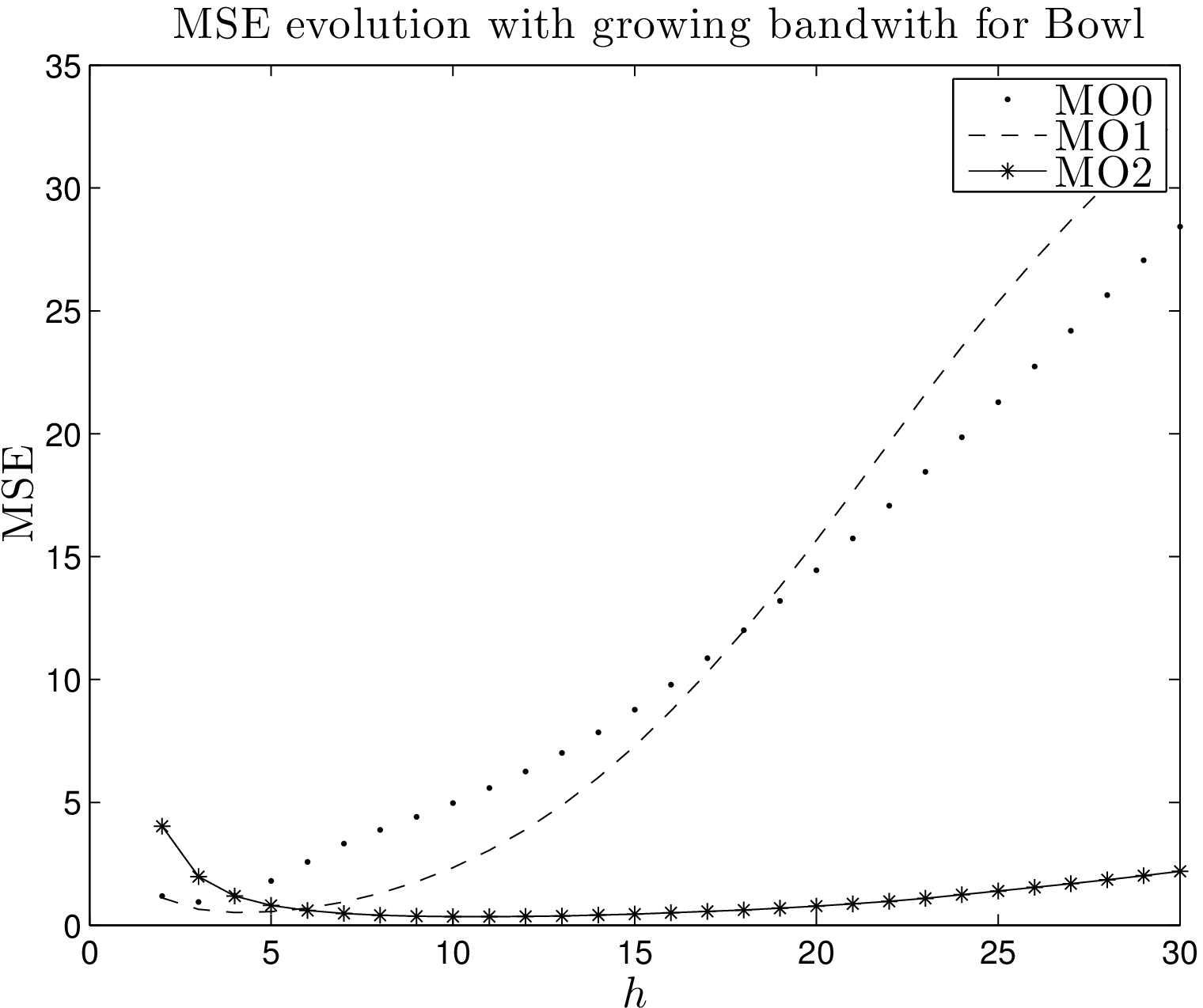}}
\subfigure[$\sigma=20$]{\includegraphics[width=.40\linewidth]{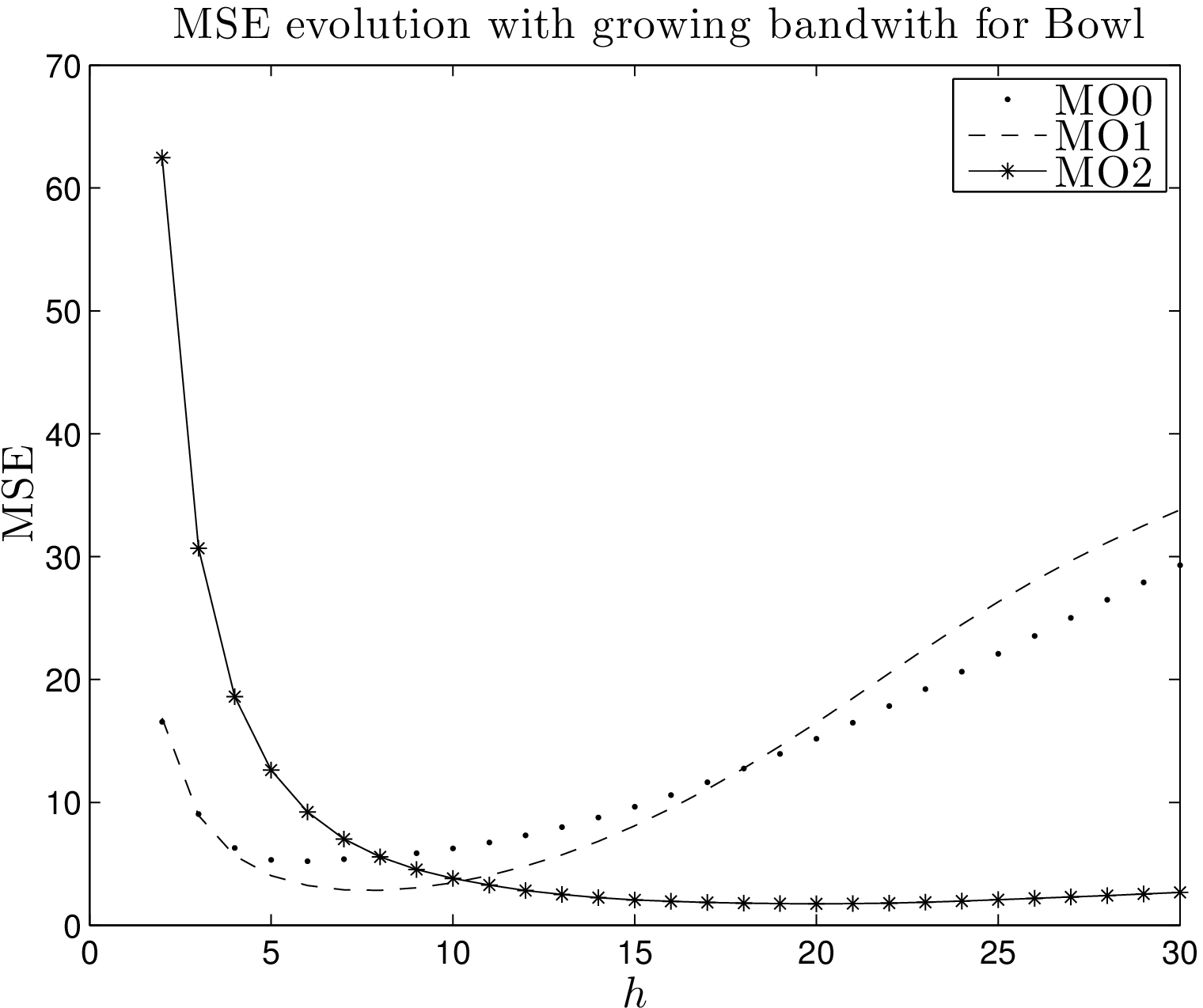}}
\subfigure[$\sigma=20$]{\includegraphics[width=.40\linewidth]{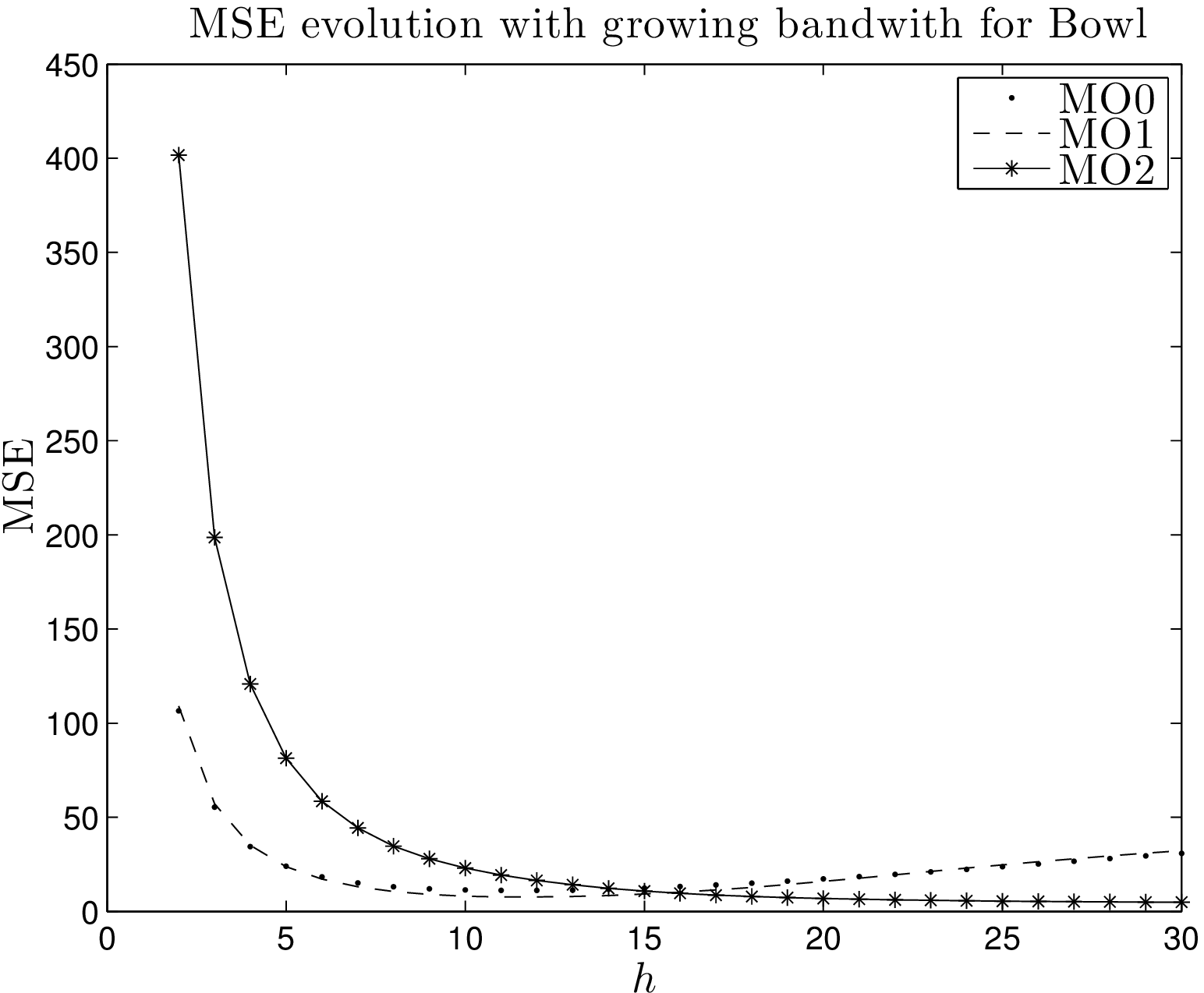}}
\subfigure[$\sigma=100$]{\includegraphics[width=.40\linewidth]{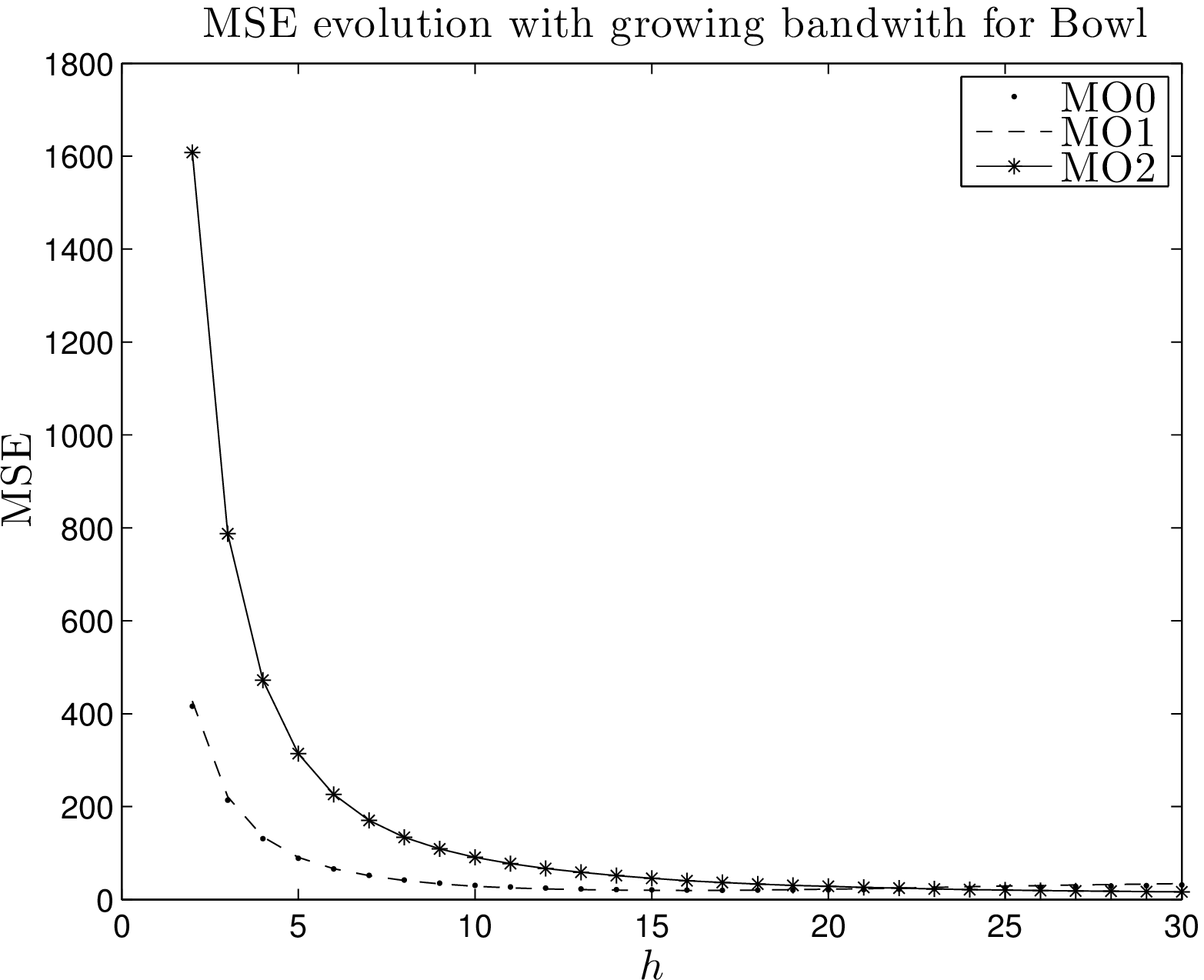}}
\caption{ MSE with respect to $h$ for
  the image Bowl, with noise level $\sigma \in\{5, 20, 50,
  100\}$.}\label{fig:searching_zone_evolution}
\end{figure}

\begin{table}[tbp]
\begin{center}
{\scriptsize
\begin{tabular}{|c|c@{\quad}c@{\quad}c@{\quad}c@{\quad}|}
\hline
\input{mse_oracle_Bowl.tex}
\hline
\end{tabular}
}
\caption{Spatial bandwidth $h$ used in practice obtained by minimizing the MSE of
the MO on Bowl (\lcf \figref{searching_zone_evolution}).\label{tab:optimal_bandwidth} 
} 
\end{center}
\end{table}

In practice, $h$ needs to increase with $r$ 
to ensure that the LPR is stable.  Note that
there are $q = {r+d \choose d}$ polynomial coefficients in each search
window.  If we apply the rule of thumb of 10 observations per unknown
parameter, the search window needs to include about $10 q$ pixels.
This is illustrated in \figref{searching_zone_evolution}, where we see
the best $h$ increasing with $r$.

Since for natural (not cartoon-like) images MO is irrelevant, we have used a modified YF oracle
instead. This oracle has access to the original image to compute the
weights as in \eqref{eq:kernel_weights}, and then performs LPR on the
noisy pixel values with these weights.  For
piecewise constant images, this coincides exactly with MO  
as soon as the bandwidth is large enough.

The experiments conducted show that LF is always outperformed in
practice by YF, NLM and \NLMm.  For low noise level ($\sigma=5$), the
YF with $r=2$ outperforms the other methods (\lcf
\figref{Bowl_sigma=5} and
Table~\ref{tab:general_comparison}). However, in the presence of
strong noise the NLM and the \NLMm~are the clear winners on most
images. Interestingly, and may be surprisingly the \NLMm~ can even
improve on the NLM for very strong noise, even for natural images.  On
the other hand, one can see that the \NLMm~mimics the behavior of LF
for textured images with strong noise (\lcf
\figref{Stripes_sigma=100}), due to the fact that different sides of
periodic features are averaged together. This limitation is
particularly obvious for the Stripe image \ref{fig:Stripes_sigma=100}.

The influence of the degree $r$ of the LPR depends on the nature of the image being denoised (\eg natural vs. cartoon images). In practice
it remains unclear how this parameter should be tuned. 
\figref{Sinusoid_sigma=5}, \figref{Sinusoid_sigma=20},
\figref{Sinusoid_sigma=50}, and \figref{Sinusoid_sigma=100}
demonstrate the importance of the jump parameter $\mu$ in practice. On
the left end of the Swoosh, the jump is larger and we reconstruct it
accurately across all noise levels. On the right end, the jump is
smaller and the performance degrades with $\sigma$, exactly as
predicted by our theory.

The MATLAB codes are available on the authors' webpages to reproduce those results.


\begin{table}[tbp]
\begin{center}
{\tiny
\begin{tabular}{|c|r@{\quad}r@{\quad}r@{\quad}r@{\quad}r@{\quad}r@{\quad}r@{\quad}|}
\hline
\input{mse_replica5.tex}
\hline
\end{tabular}}
\caption{MSE comparisons of the denoising methods considered for LPR
  of order 0, 1 and 2.  The compared methods are the Linear Filter
  (LF), the Yaroslavsky Filter (YF), the NLM-average (NLM-Av.), the
  classical NLM and the \textit{Membership Oracle} (MO). Results are
  averaged over 5 Gaussian noise replicas.  }
\label{tab:general_comparison}
\end{center}
\vspace{2in}
\end{table}

\begin{figure}[hbt!]
\centering
\input{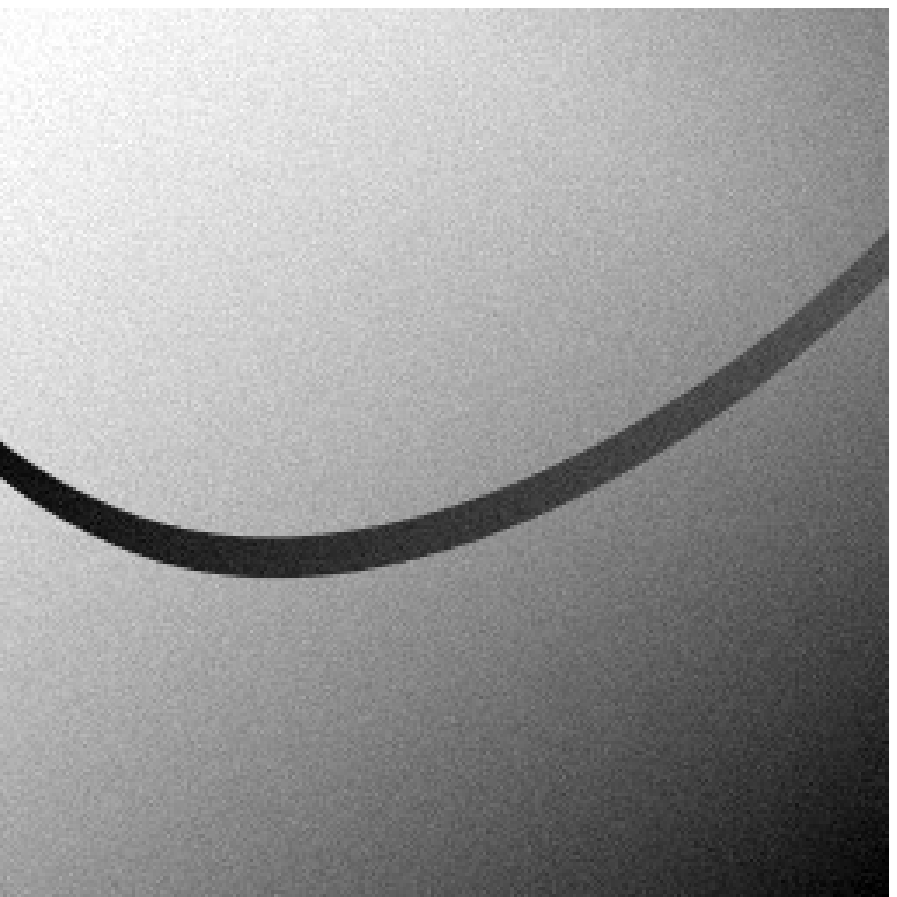}  \subfigure[\currentcaption]{\includegraphics[width=0.233\linewidth]{\currentname}}\hfill
\input{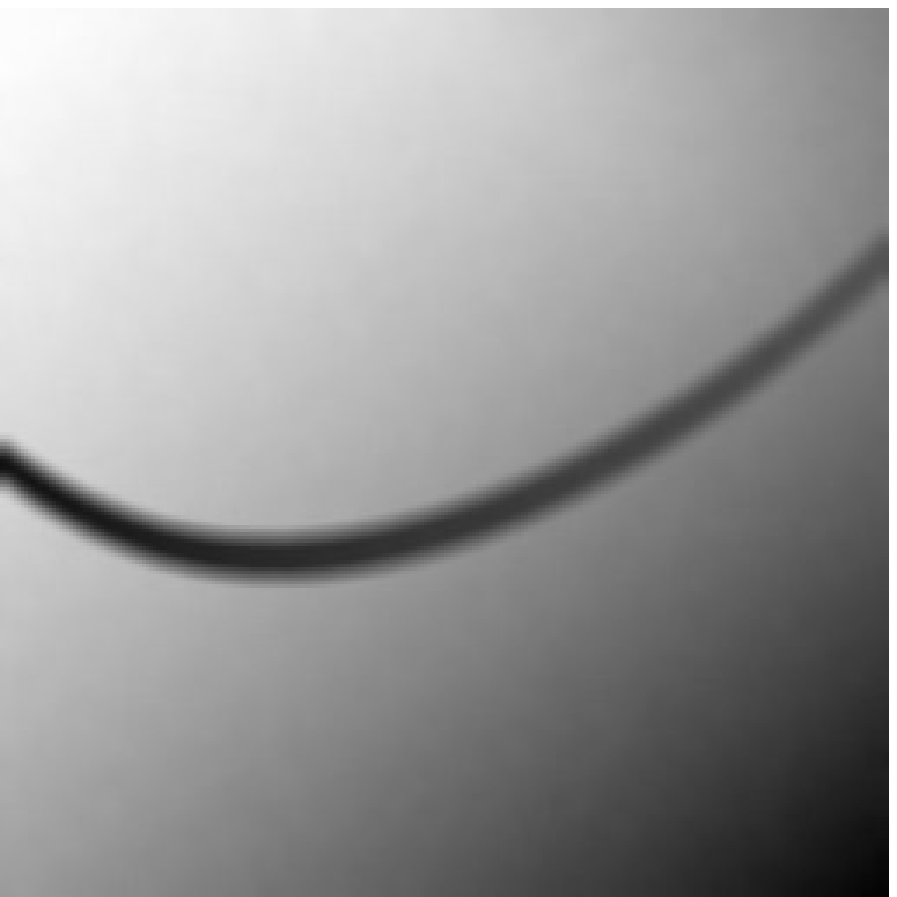}  \subfigure[\currentcaption]{\includegraphics[width=0.233\linewidth]{\currentname}}\hfill
\input{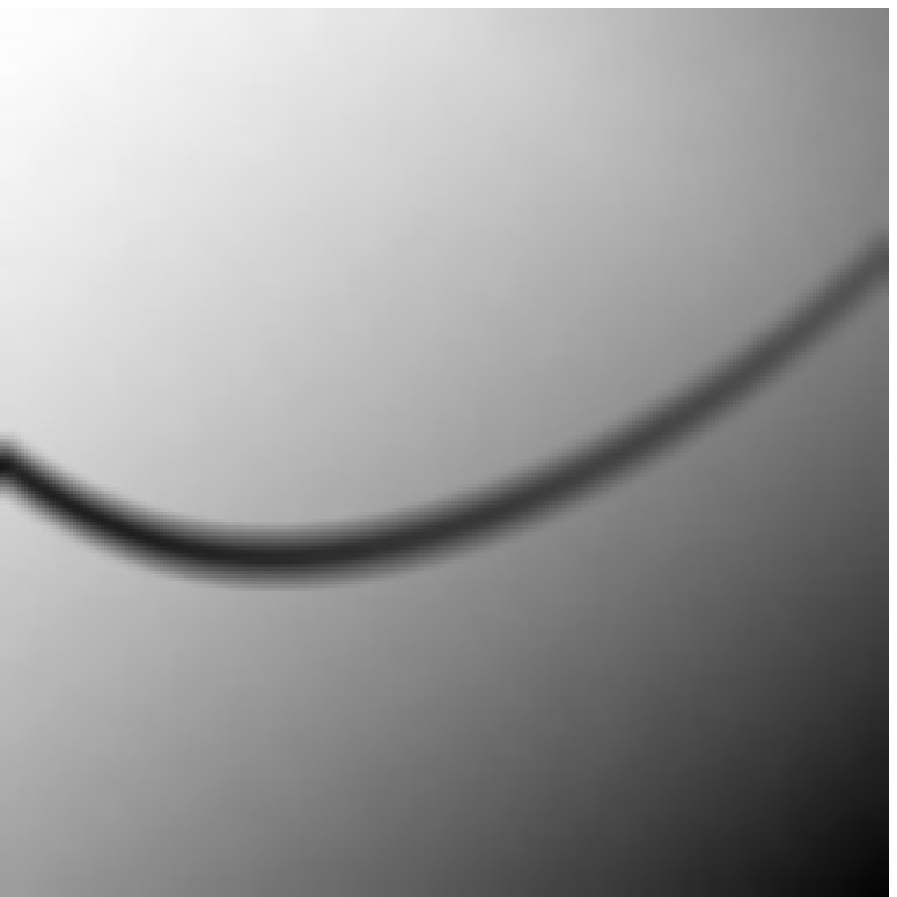}  \subfigure[\currentcaption]{\includegraphics[width=0.233\linewidth]{\currentname}}\hfill
\input{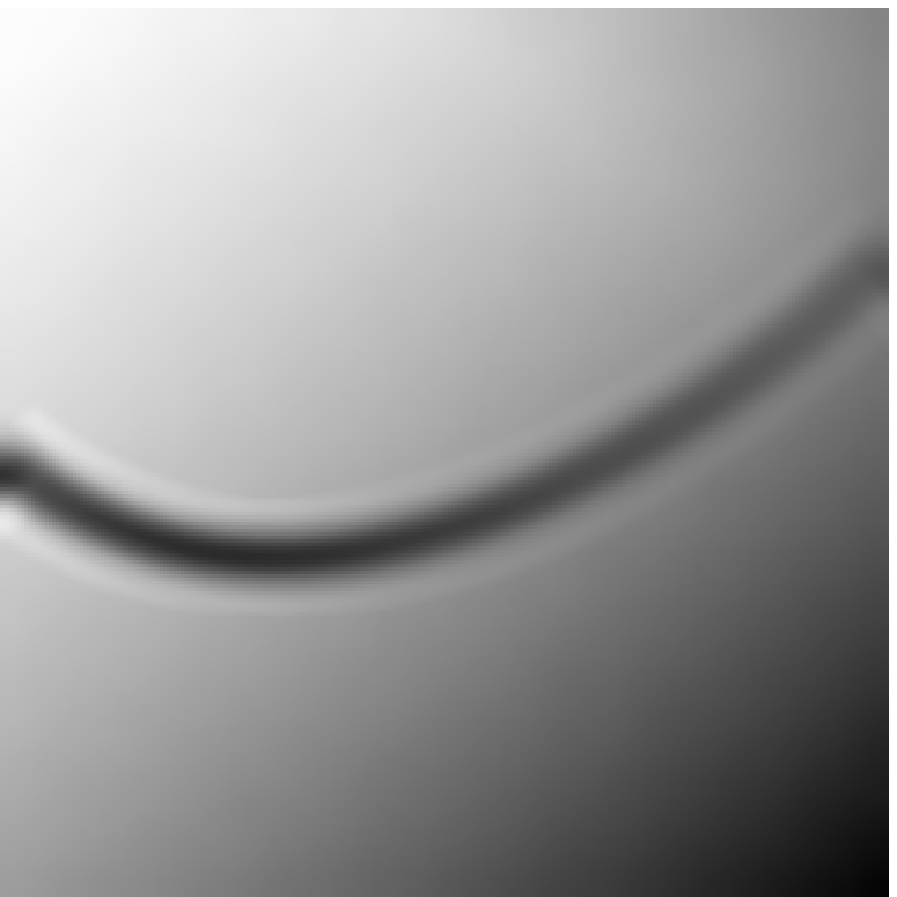}  \subfigure[\currentcaption]{\includegraphics[width=0.233\linewidth]{\currentname}}\hfill
\input{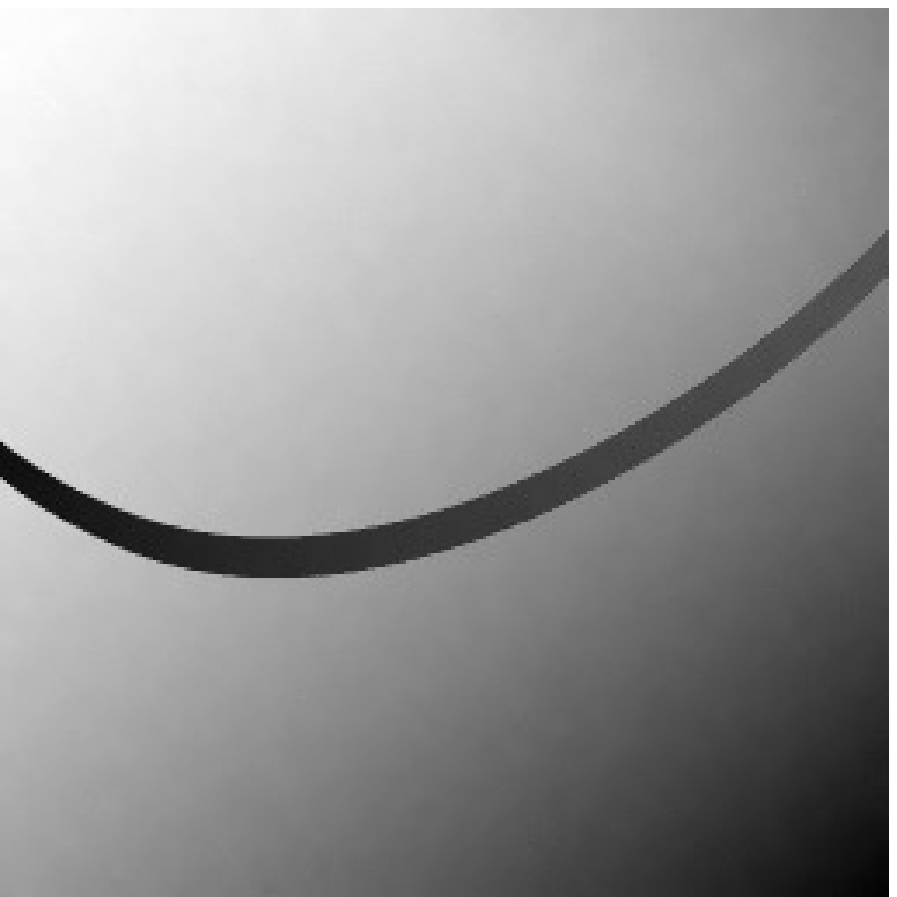}  \subfigure[\currentcaption]{\includegraphics[width=0.233\linewidth]{\currentname}}\hfill
\input{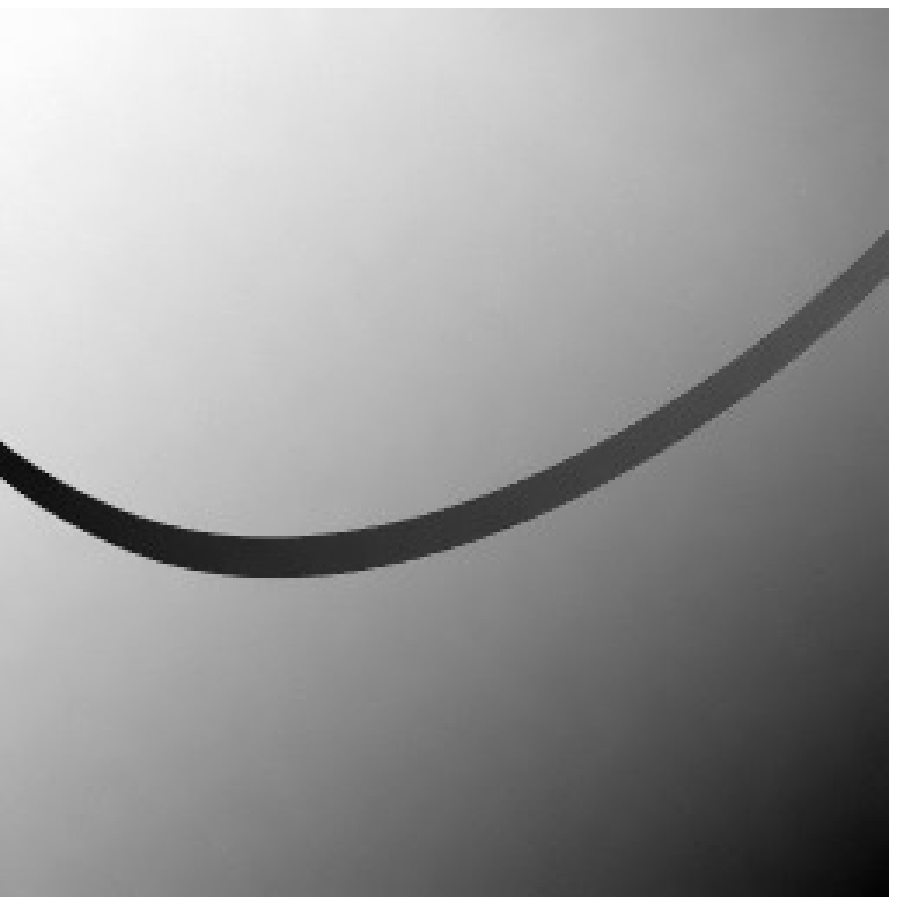}  \subfigure[\currentcaption]{\includegraphics[width=0.233\linewidth]{\currentname}}\hfill
\input{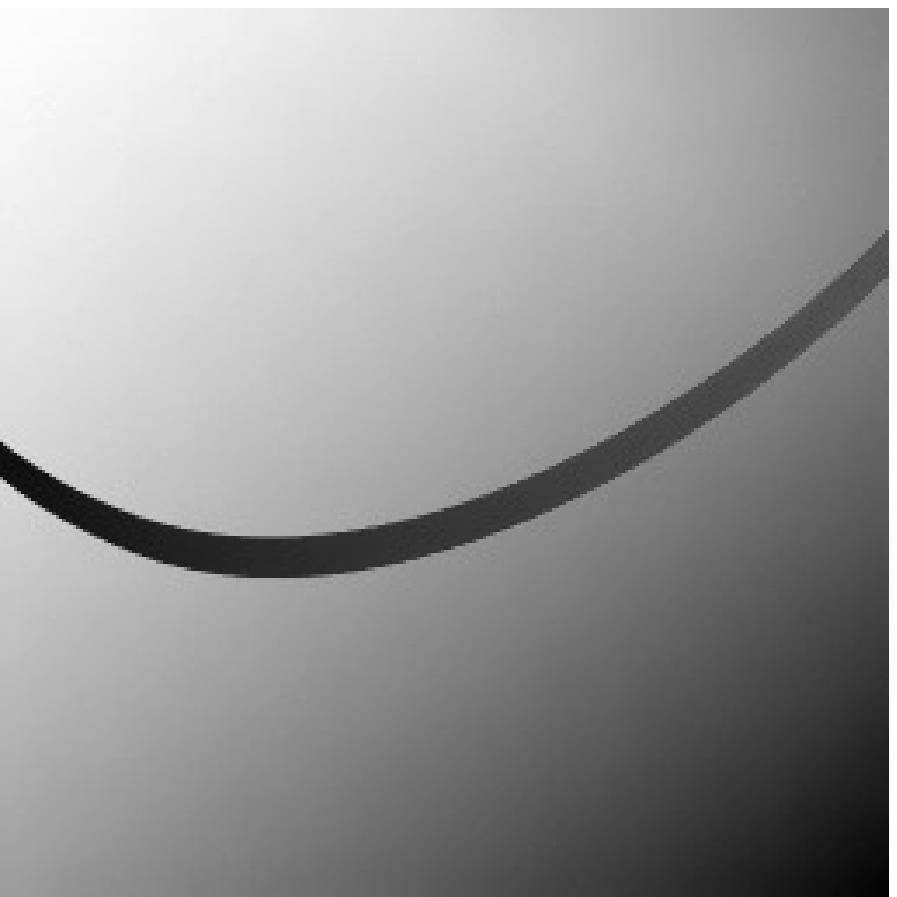}  \subfigure[\currentcaption]{\includegraphics[width=0.233\linewidth]{\currentname}}\hfill
\input{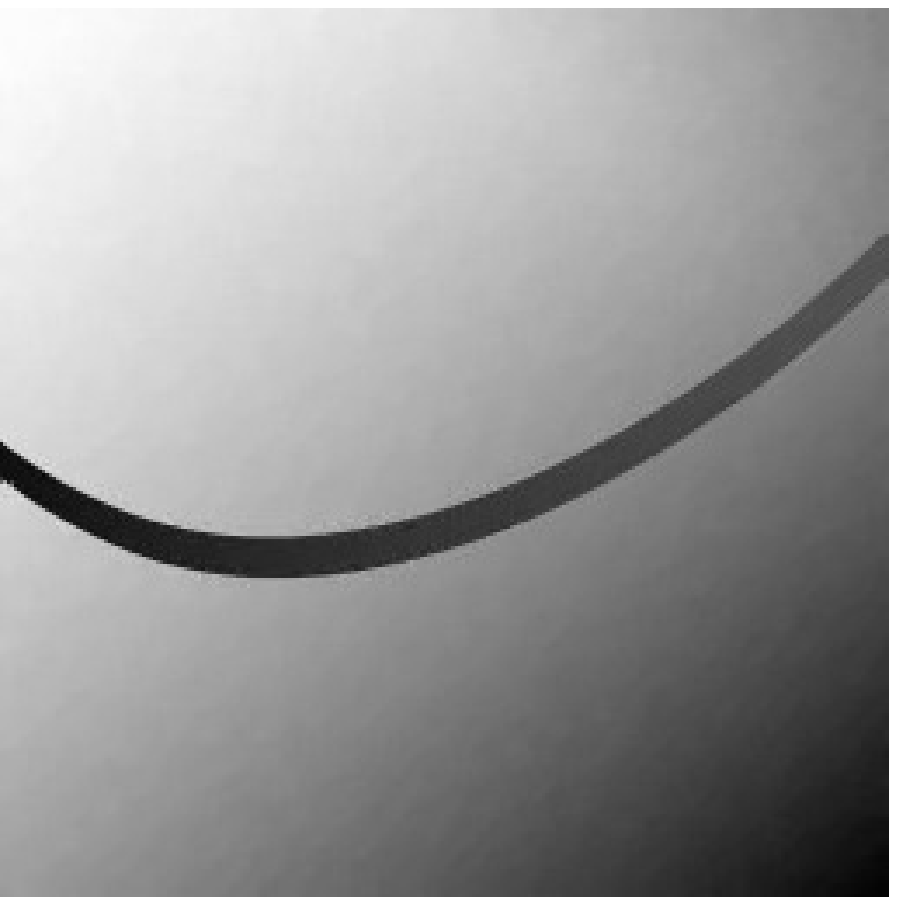}  \subfigure[\currentcaption]{\includegraphics[width=0.233\linewidth]{\currentname}}\hfill
\input{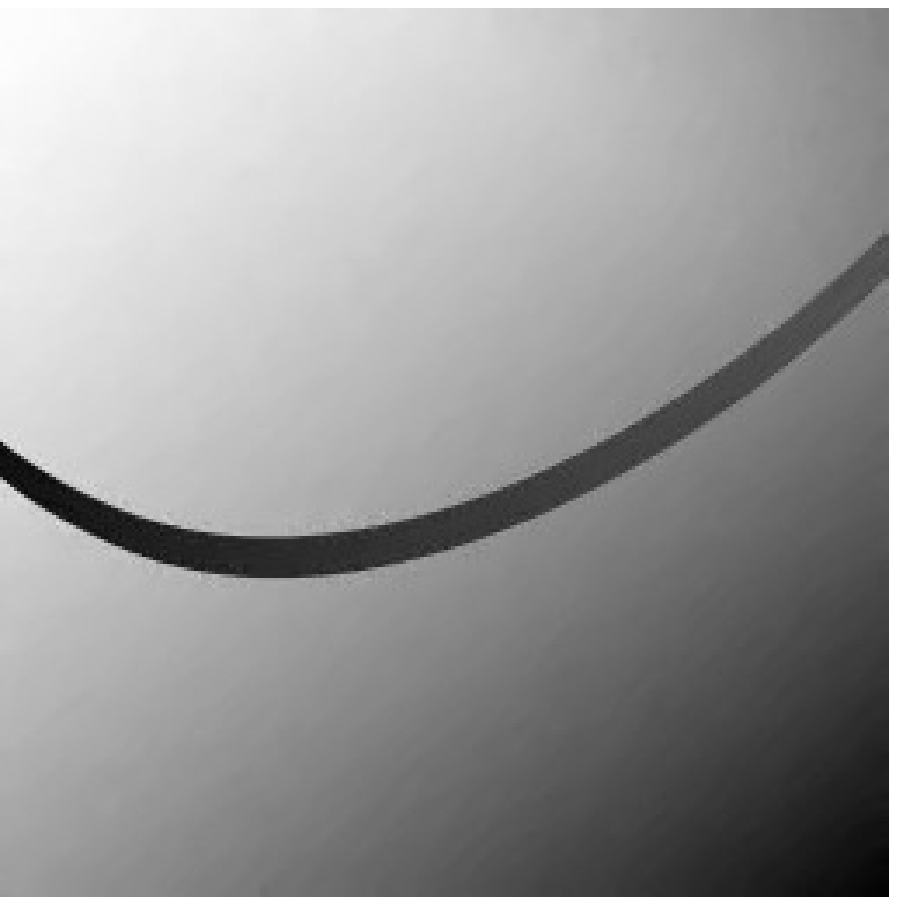}  \subfigure[\currentcaption]{\includegraphics[width=0.233\linewidth]{\currentname}}\hfill
\input{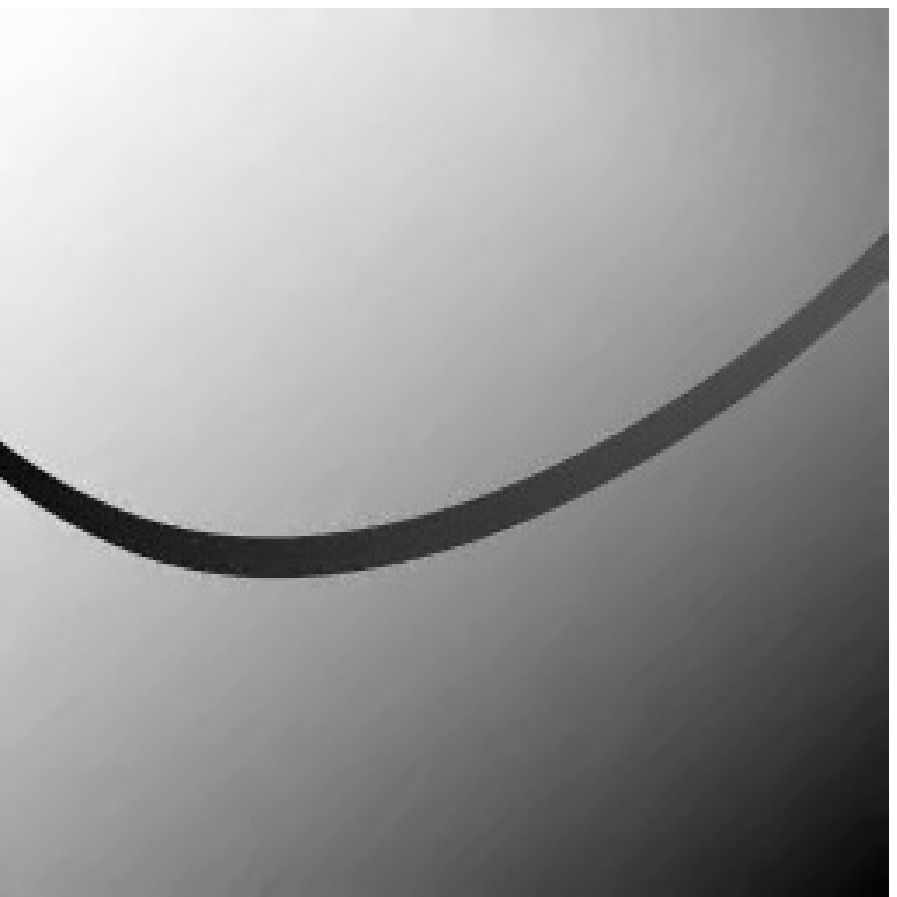}  \subfigure[\currentcaption]{\includegraphics[width=0.233\linewidth]{\currentname}}\hfill
\input{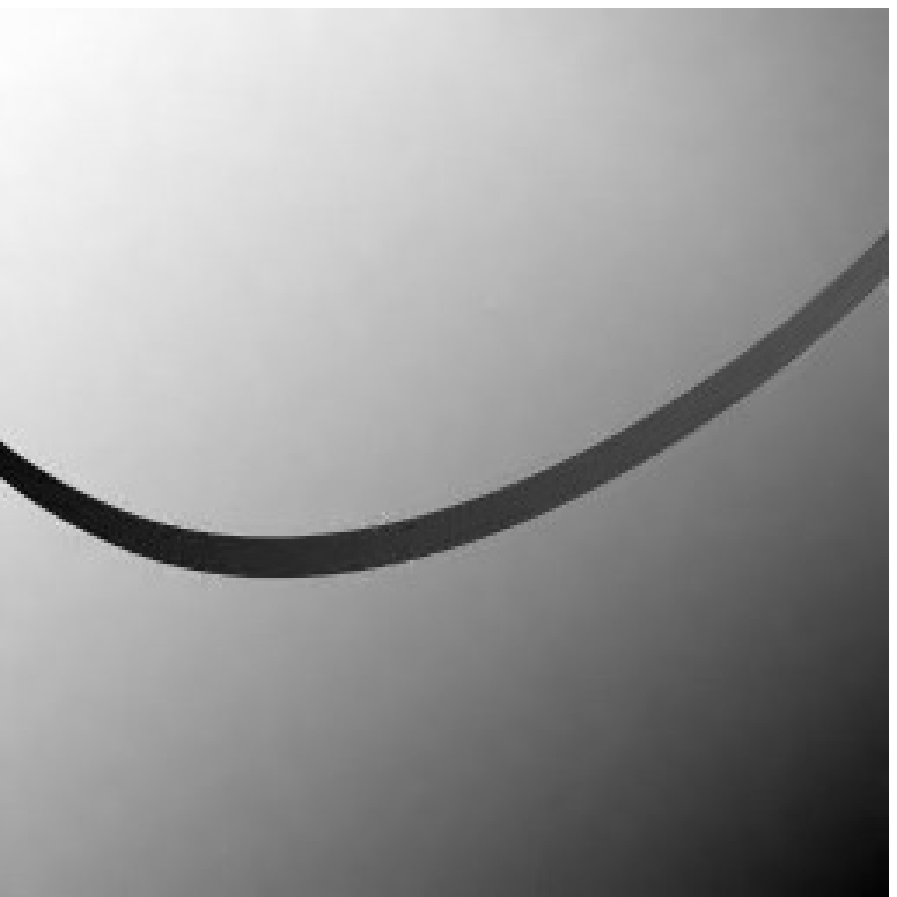}  \subfigure[\currentcaption]{\includegraphics[width=0.233\linewidth]{\currentname}}\hfill
\input{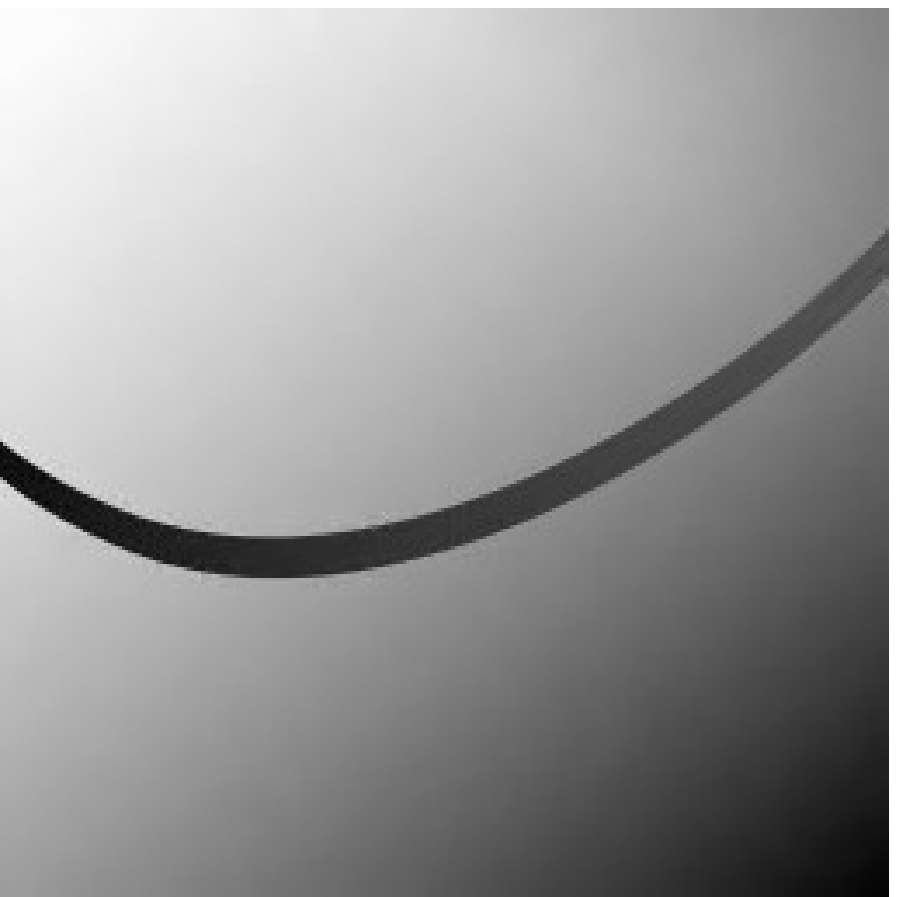}  \subfigure[\currentcaption]{\includegraphics[width=0.233\linewidth]{\currentname}}\hfill
\input{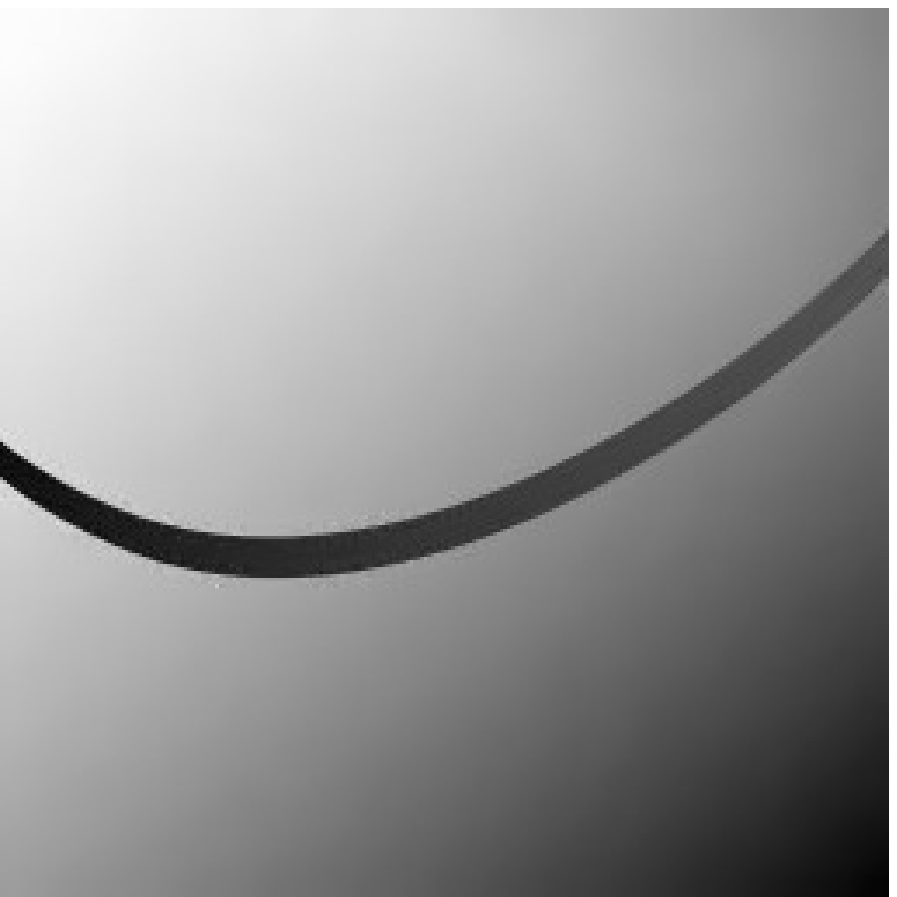}  \subfigure[\currentcaption]{\includegraphics[width=0.233\linewidth]{\currentname}}\hfill
\input{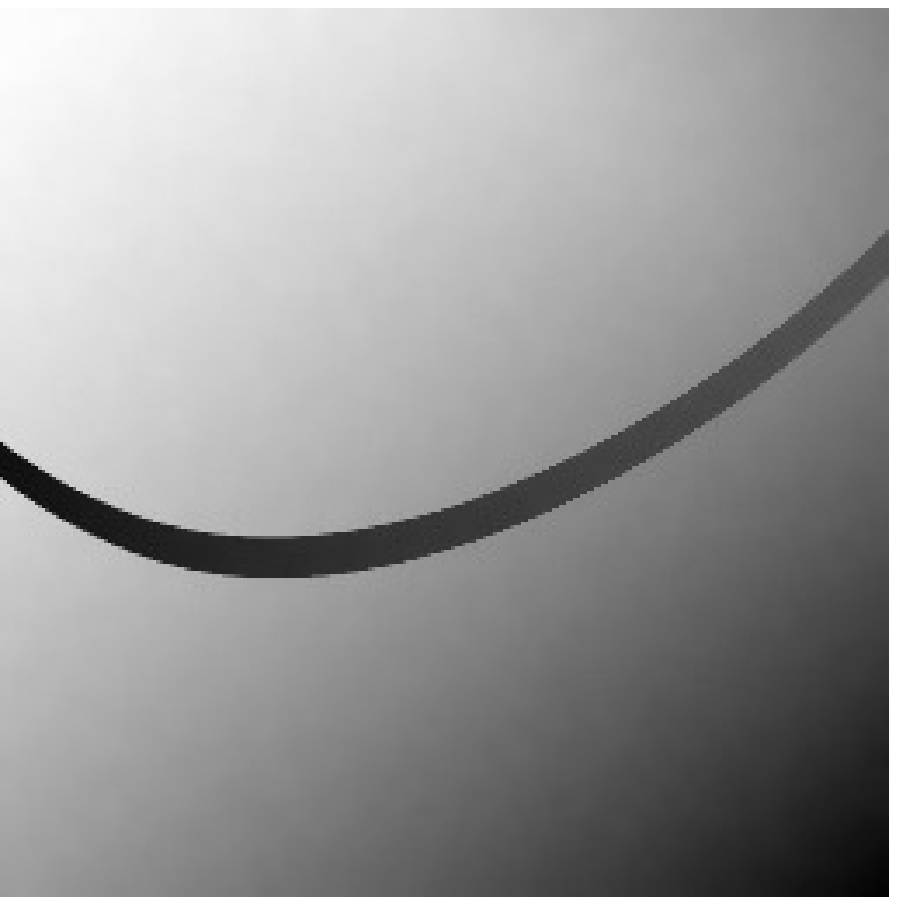}  \subfigure[\currentcaption]{\includegraphics[width=0.233\linewidth]{\currentname}}\hfill
\input{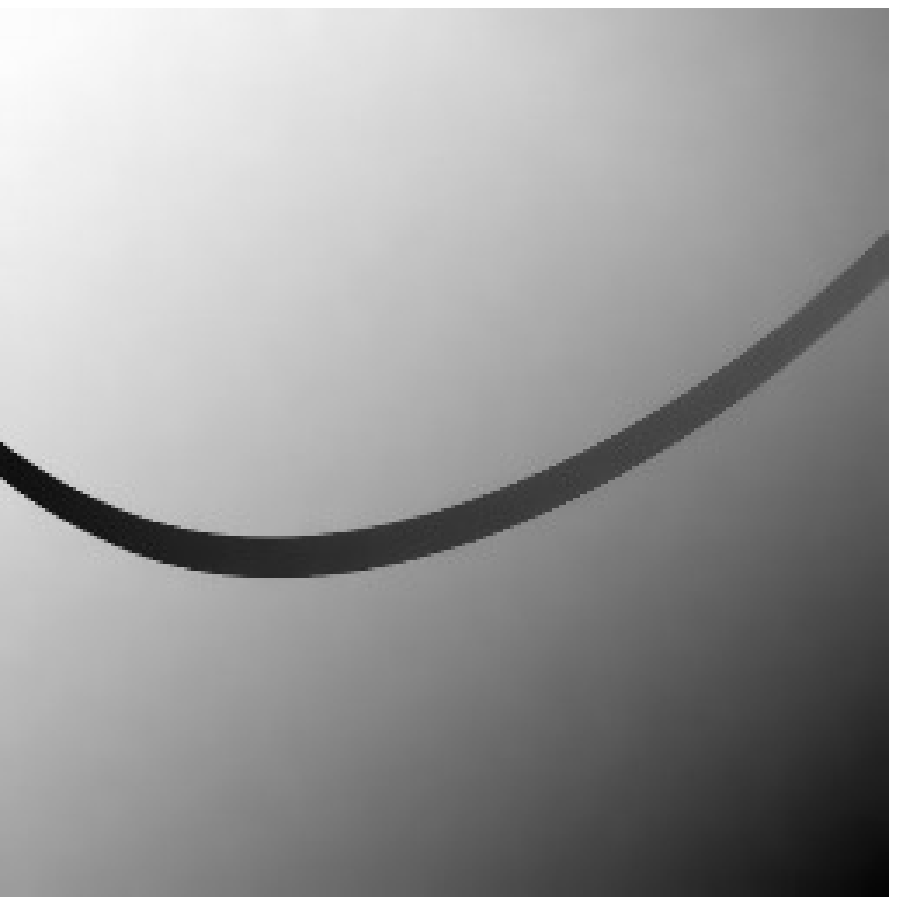}  \subfigure[\currentcaption]{\includegraphics[width=0.233\linewidth]{\currentname}}\hfill
\input{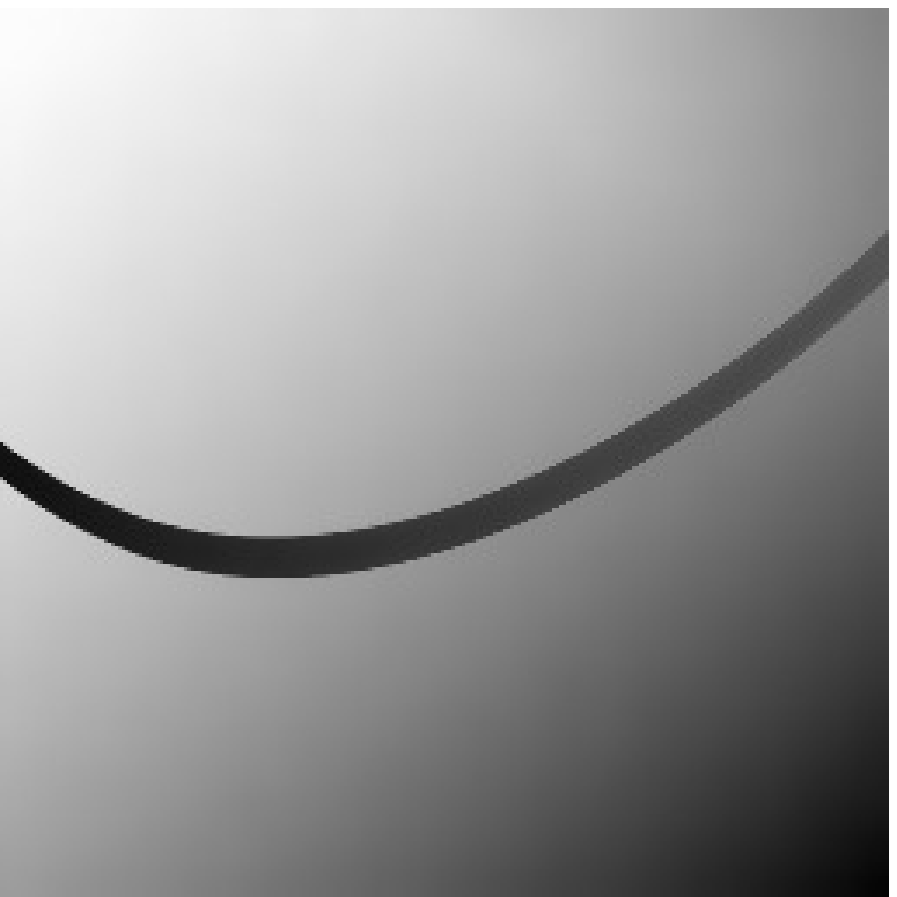}  \subfigure[\currentcaption]{\includegraphics[width=0.233\linewidth]{\currentname}}
\caption{Toy thin feature  image  (Swoosh) corrupted Gaussian
  noise with $\sigma=5$.}
\label{fig:Sinusoid_sigma=5}
\end{figure}

\begin{figure}[hbt!]
\centering
\input{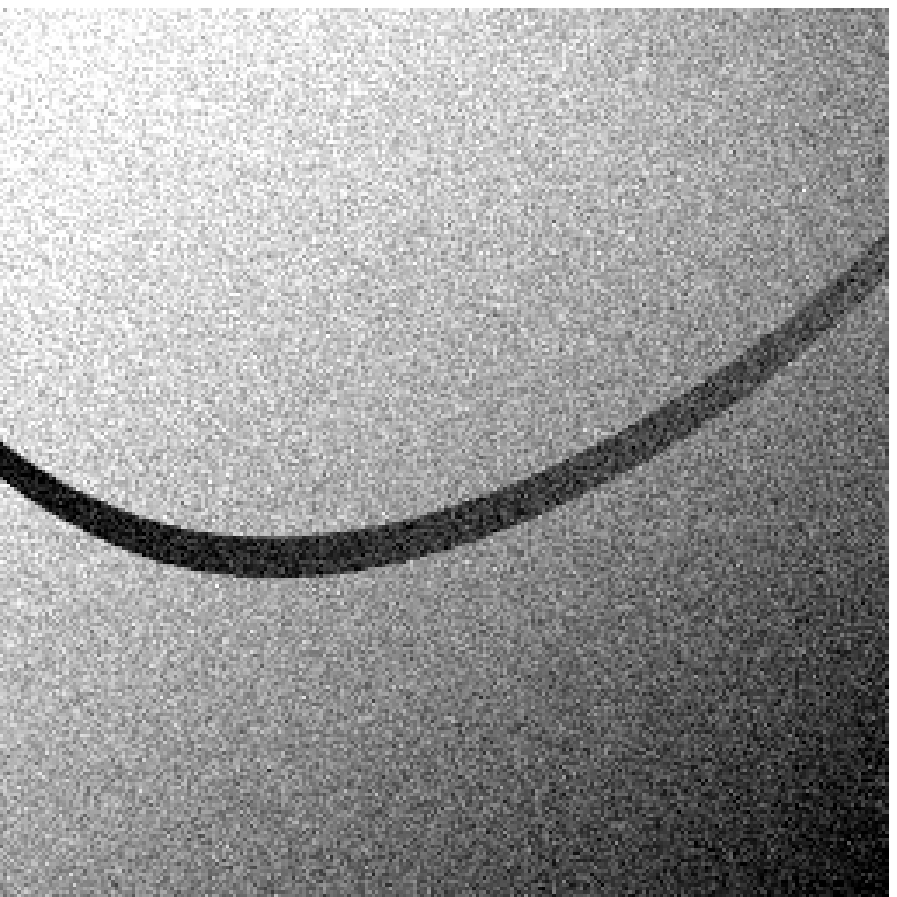}  \subfigure[\currentcaption]{\includegraphics[width=0.233\linewidth]{\currentname}}\hfill
\input{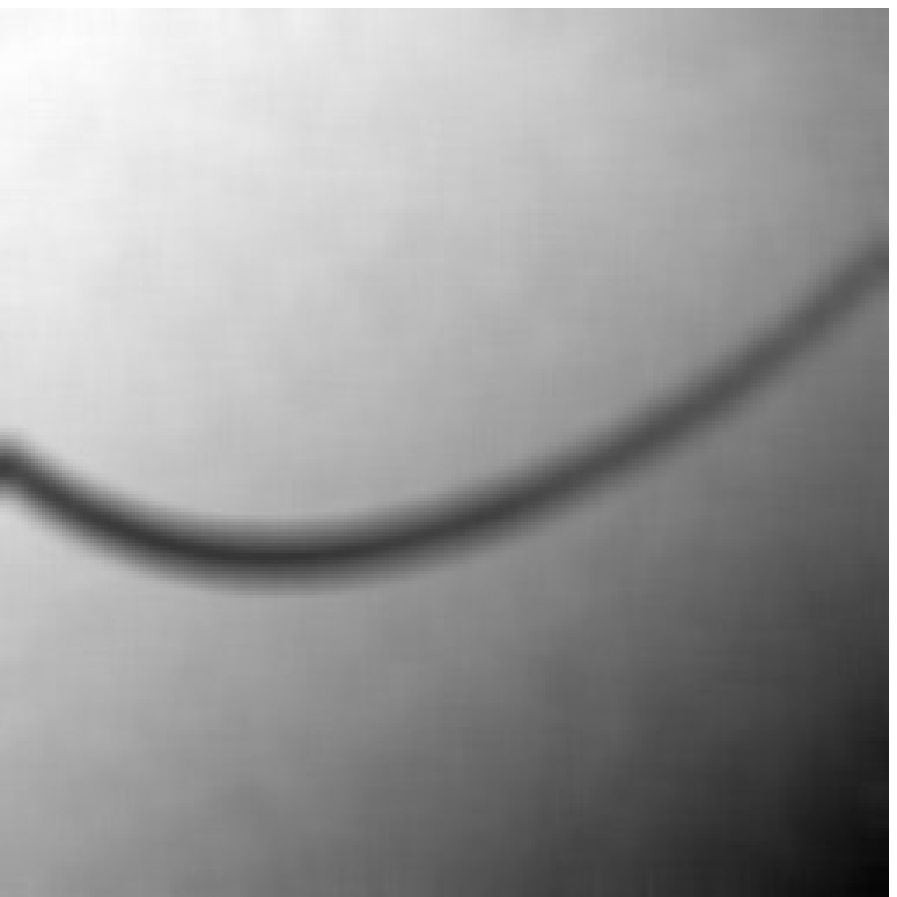}  \subfigure[\currentcaption]{\includegraphics[width=0.233\linewidth]{\currentname}}\hfill
\input{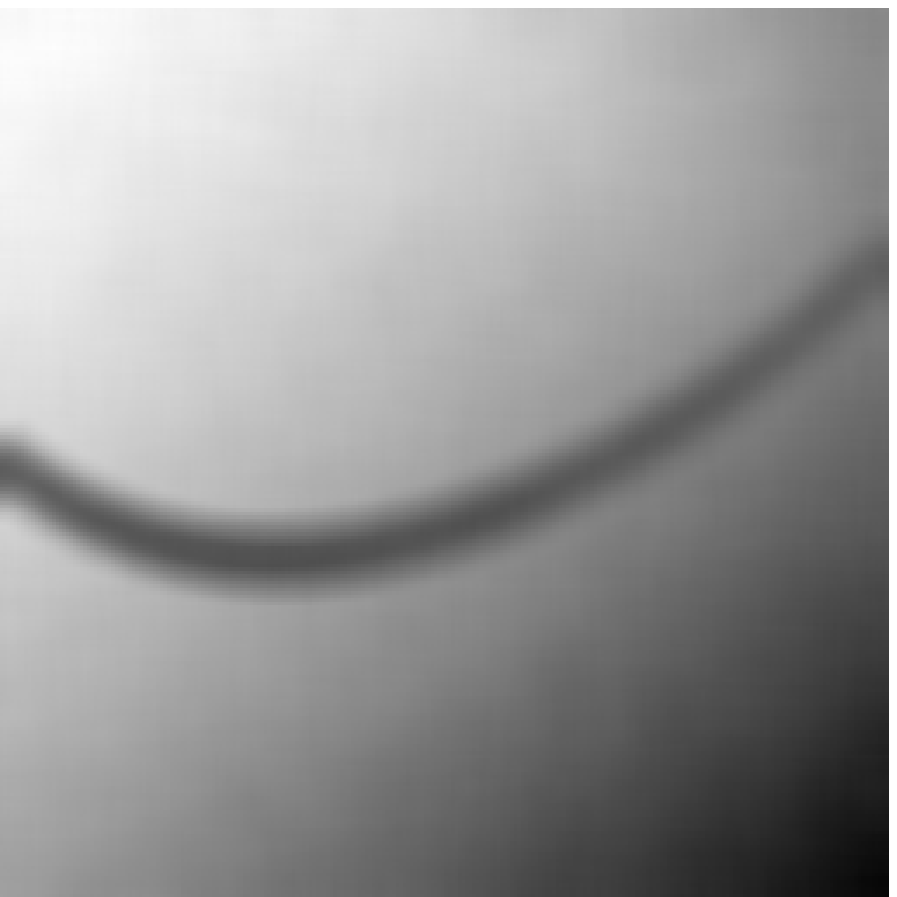}  \subfigure[\currentcaption]{\includegraphics[width=0.233\linewidth]{\currentname}}\hfill
\input{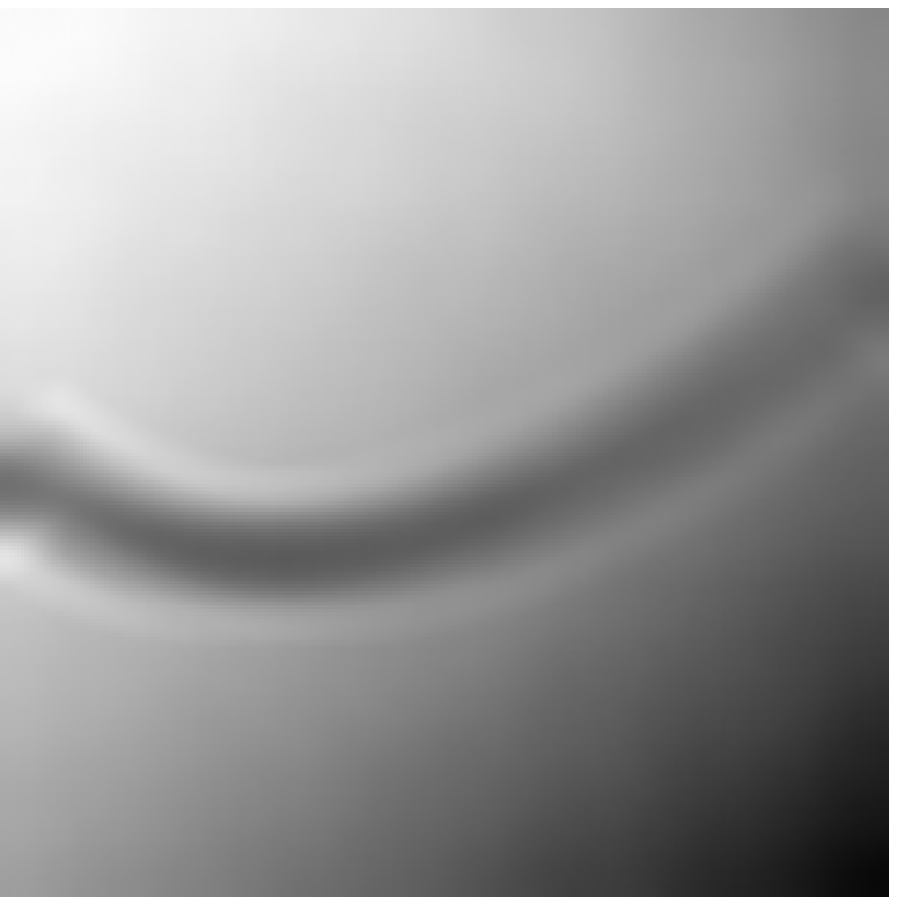}  \subfigure[\currentcaption]{\includegraphics[width=0.233\linewidth]{\currentname}}\hfill
\input{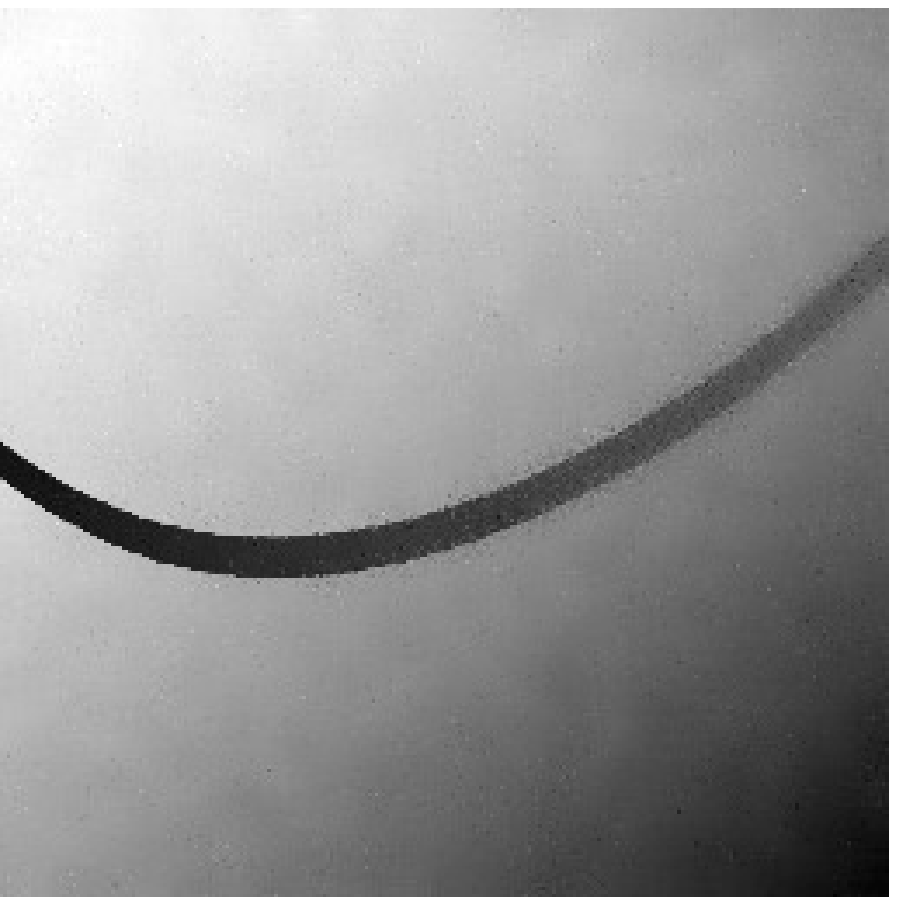}  \subfigure[\currentcaption]{\includegraphics[width=0.233\linewidth]{\currentname}}\hfill
\input{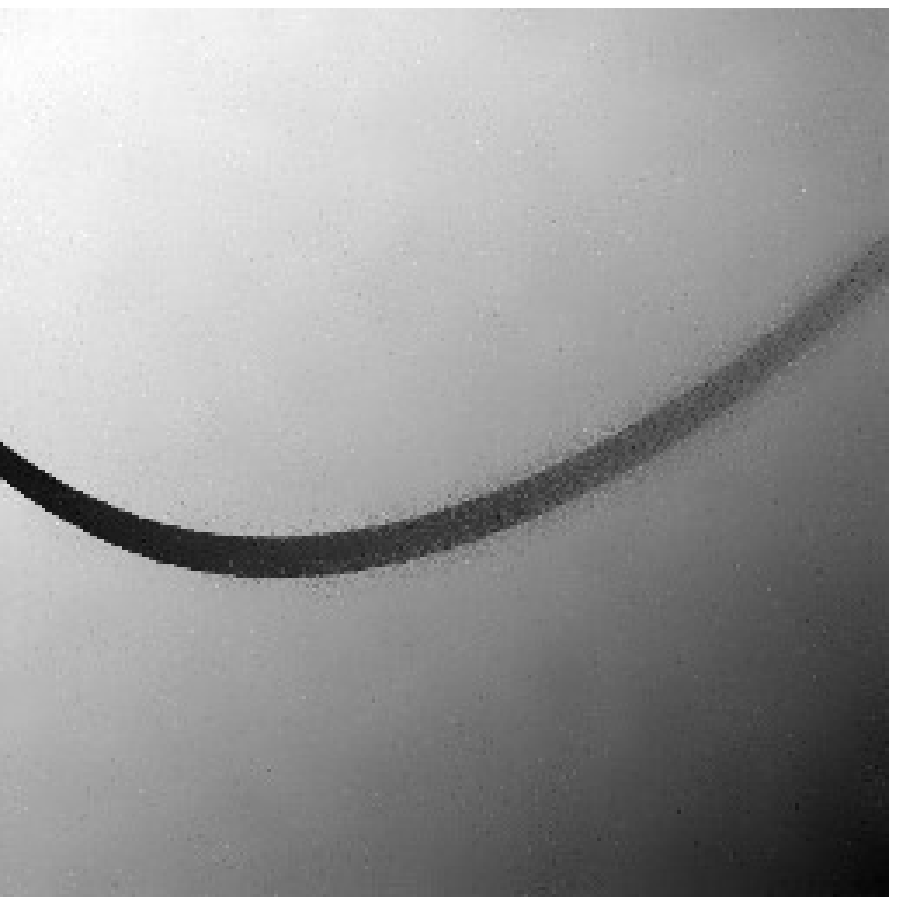}  \subfigure[\currentcaption]{\includegraphics[width=0.233\linewidth]{\currentname}}\hfill
\input{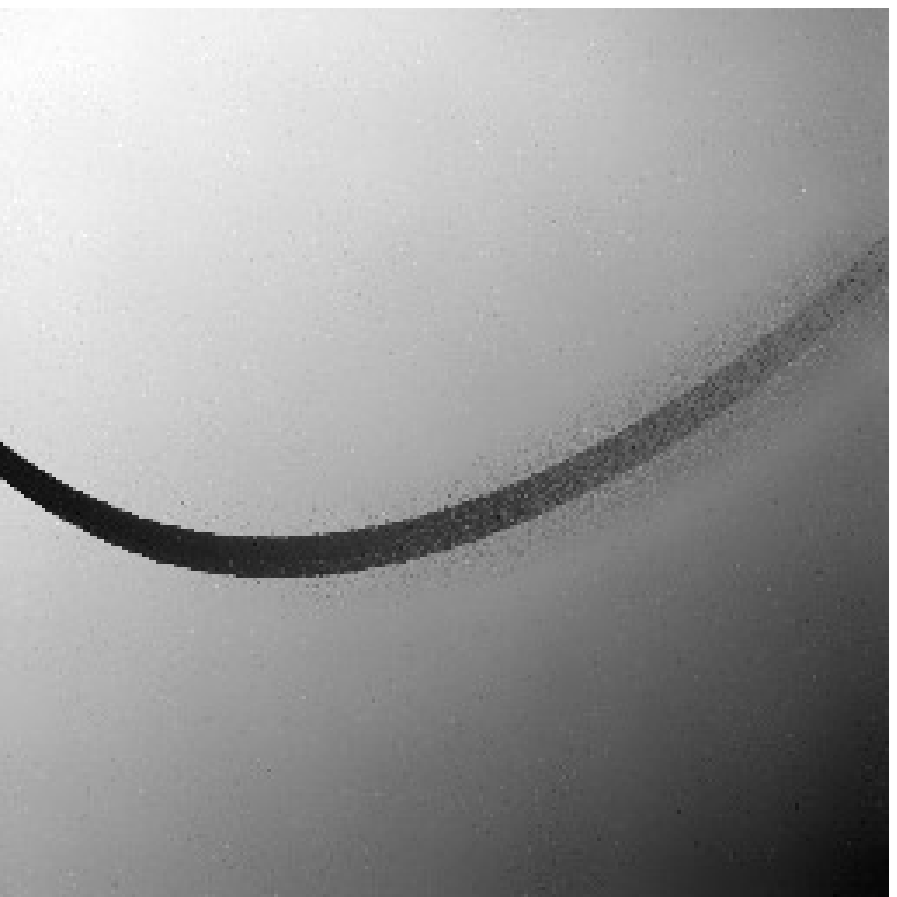}  \subfigure[\currentcaption]{\includegraphics[width=0.233\linewidth]{\currentname}}\hfill
\input{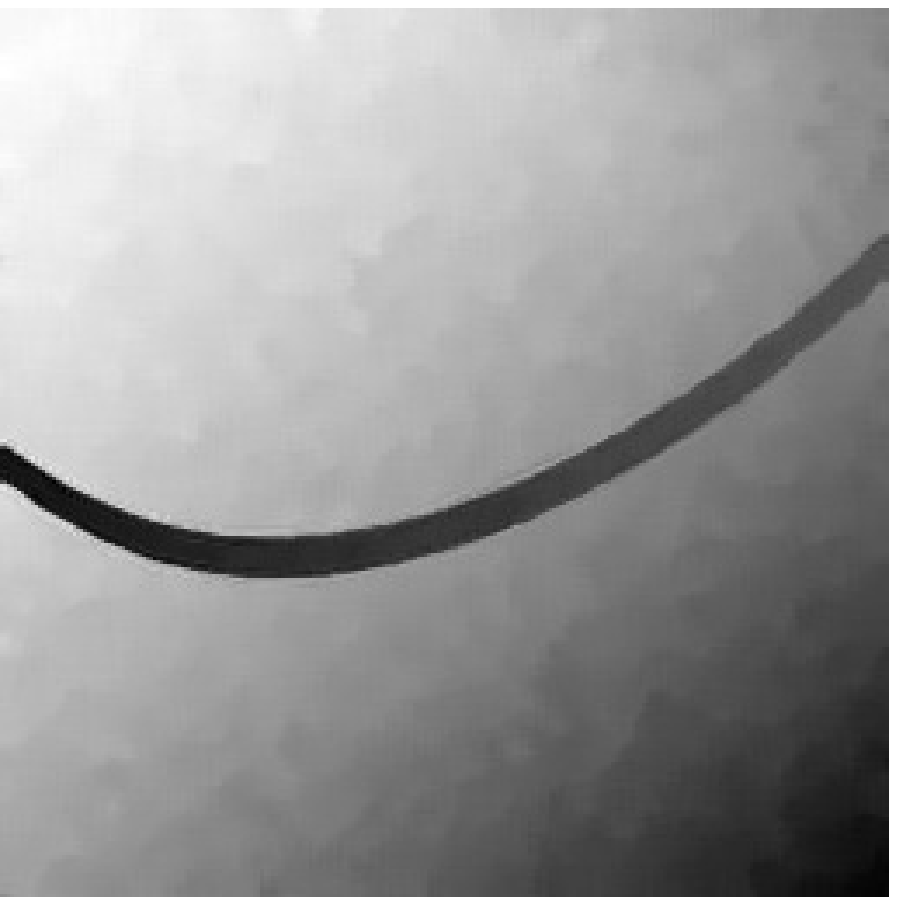}  \subfigure[\currentcaption]{\includegraphics[width=0.233\linewidth]{\currentname}}\hfill
\input{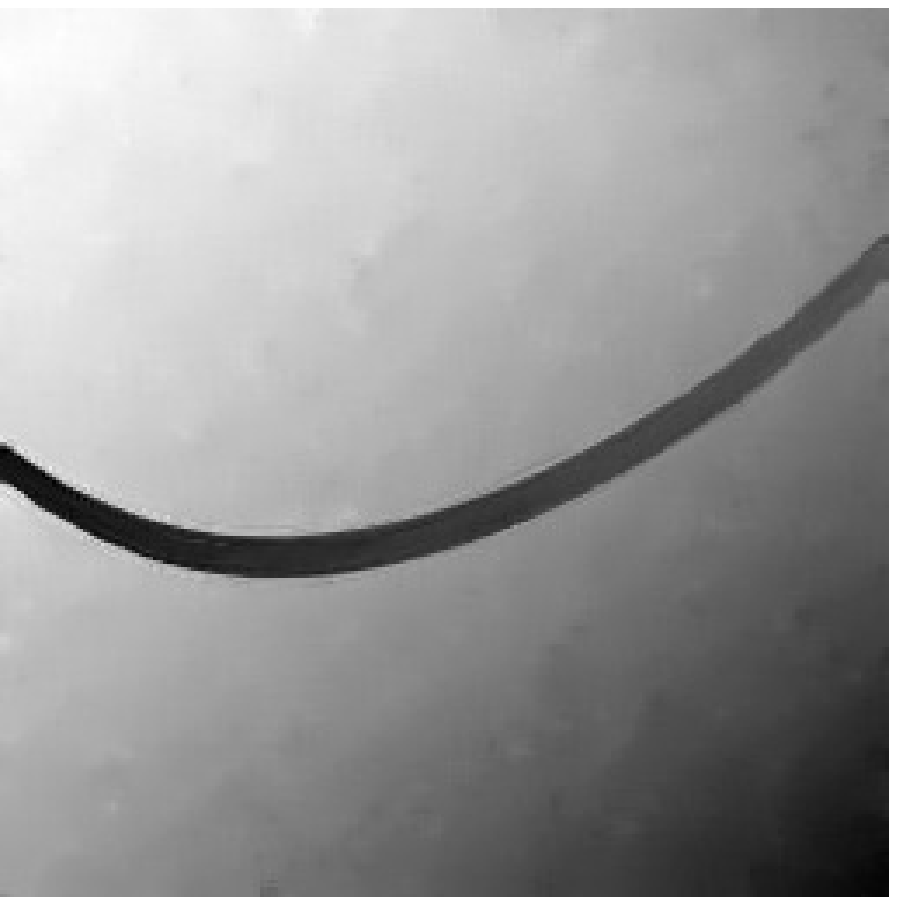}  \subfigure[\currentcaption]{\includegraphics[width=0.233\linewidth]{\currentname}}\hfill
\input{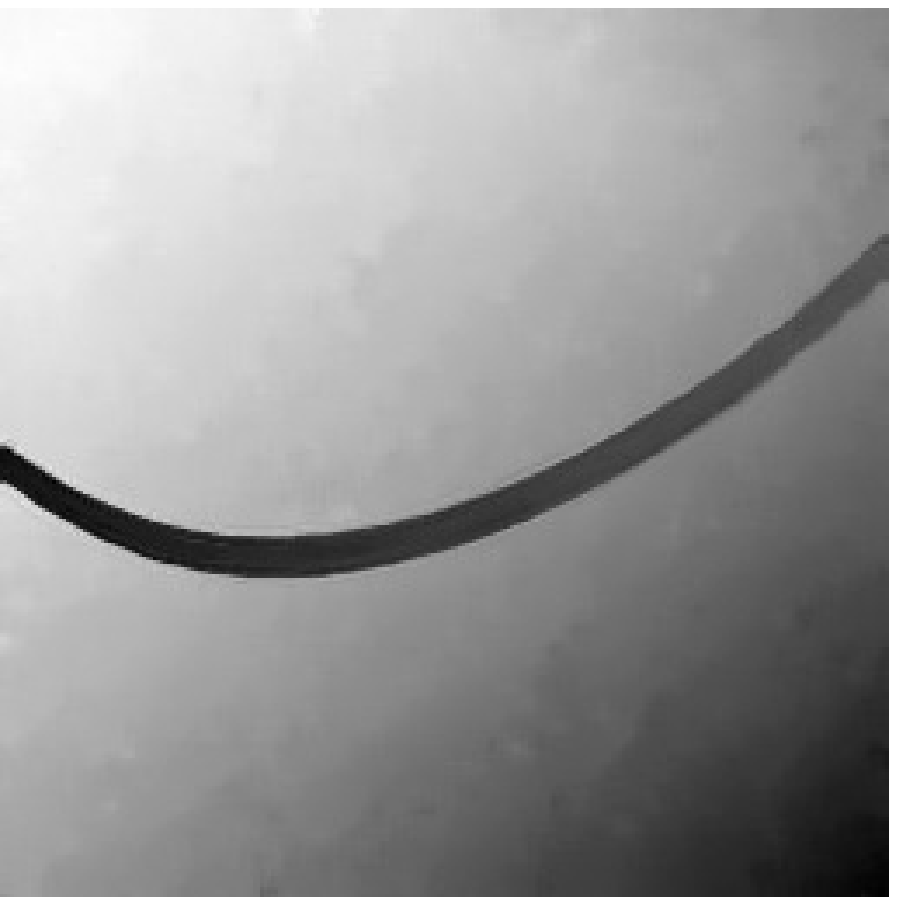}  \subfigure[\currentcaption]{\includegraphics[width=0.233\linewidth]{\currentname}}\hfill
\input{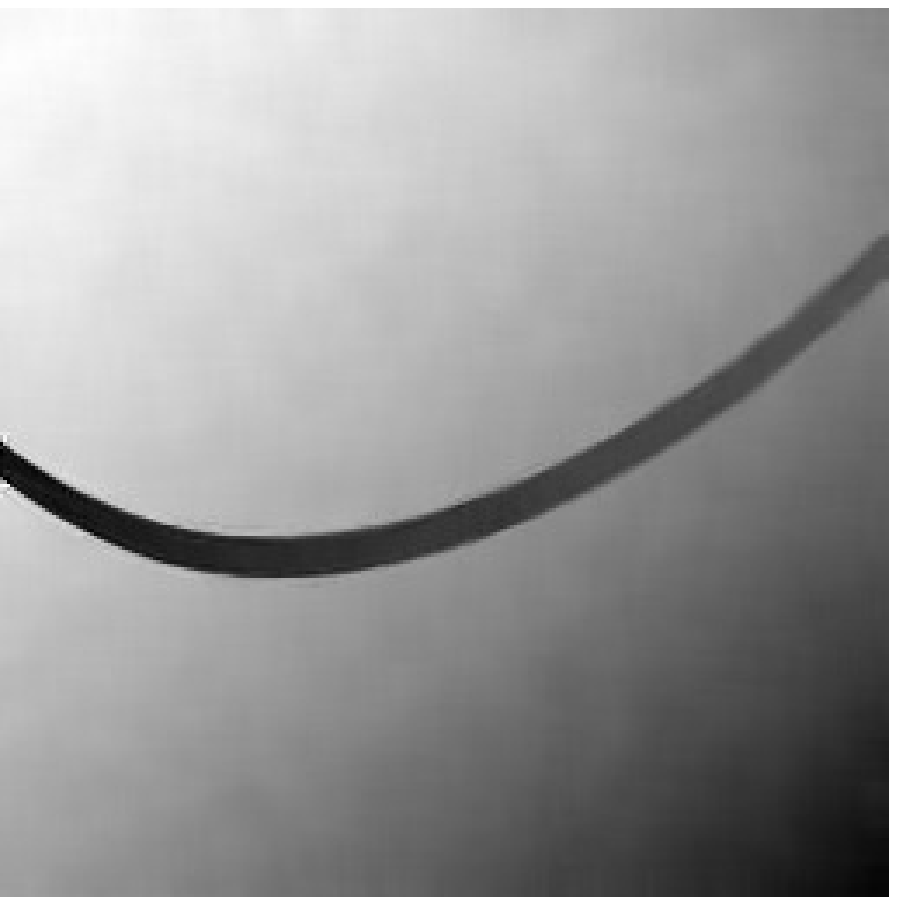}  \subfigure[\currentcaption]{\includegraphics[width=0.233\linewidth]{\currentname}}\hfill
\input{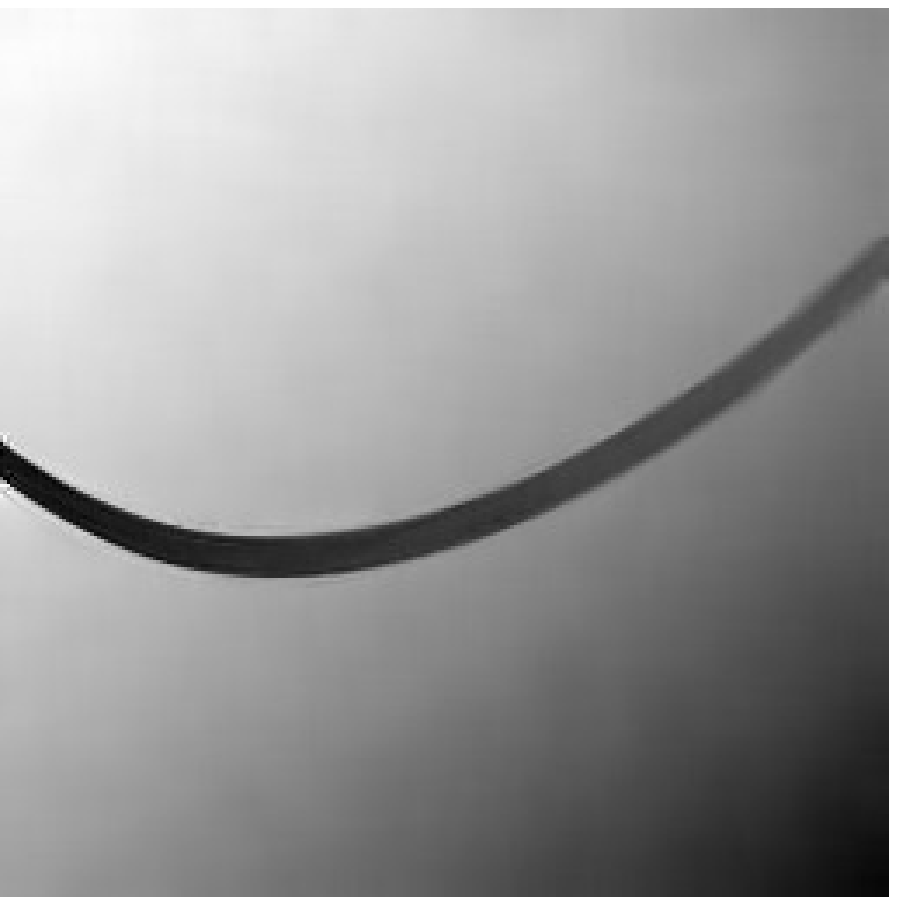}  \subfigure[\currentcaption]{\includegraphics[width=0.233\linewidth]{\currentname}}\hfill
\input{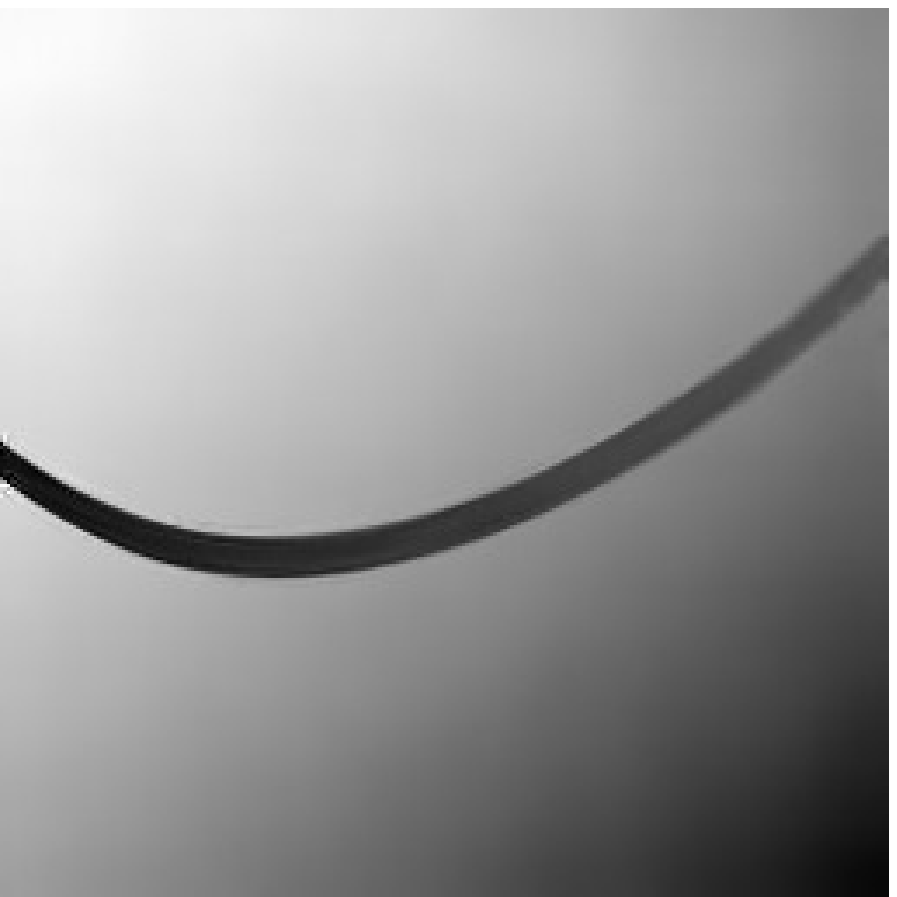}  \subfigure[\currentcaption]{\includegraphics[width=0.233\linewidth]{\currentname}}\hfill
\input{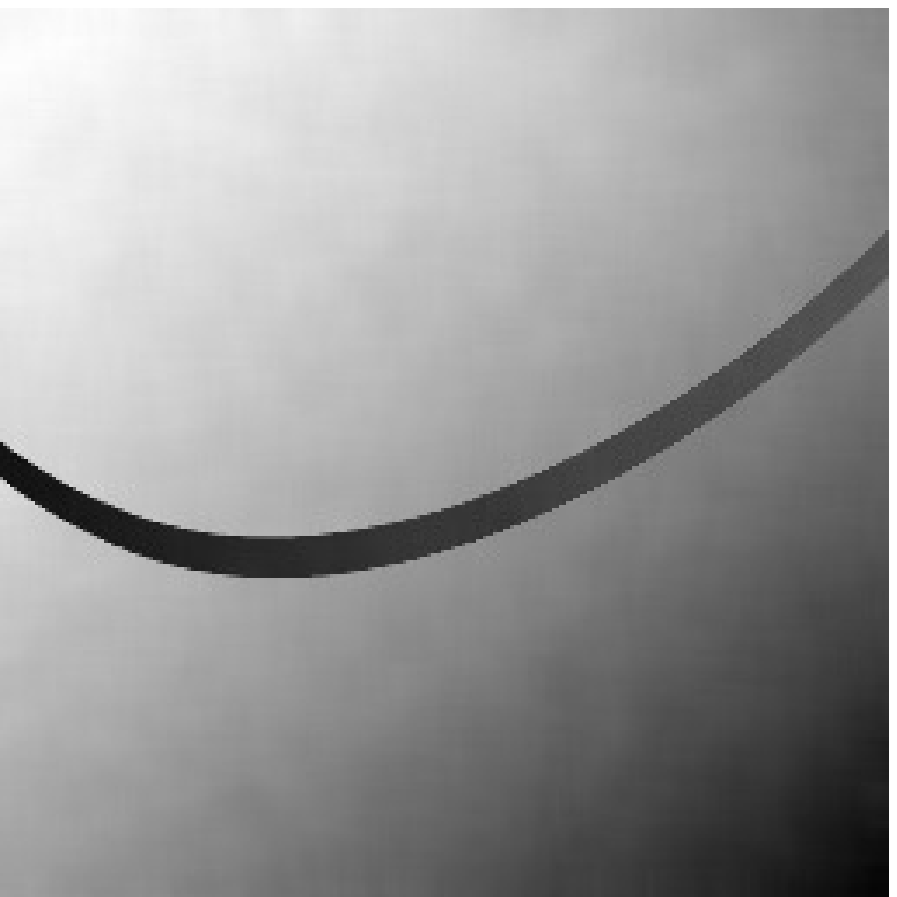}  \subfigure[\currentcaption]{\includegraphics[width=0.233\linewidth]{\currentname}}\hfill
\input{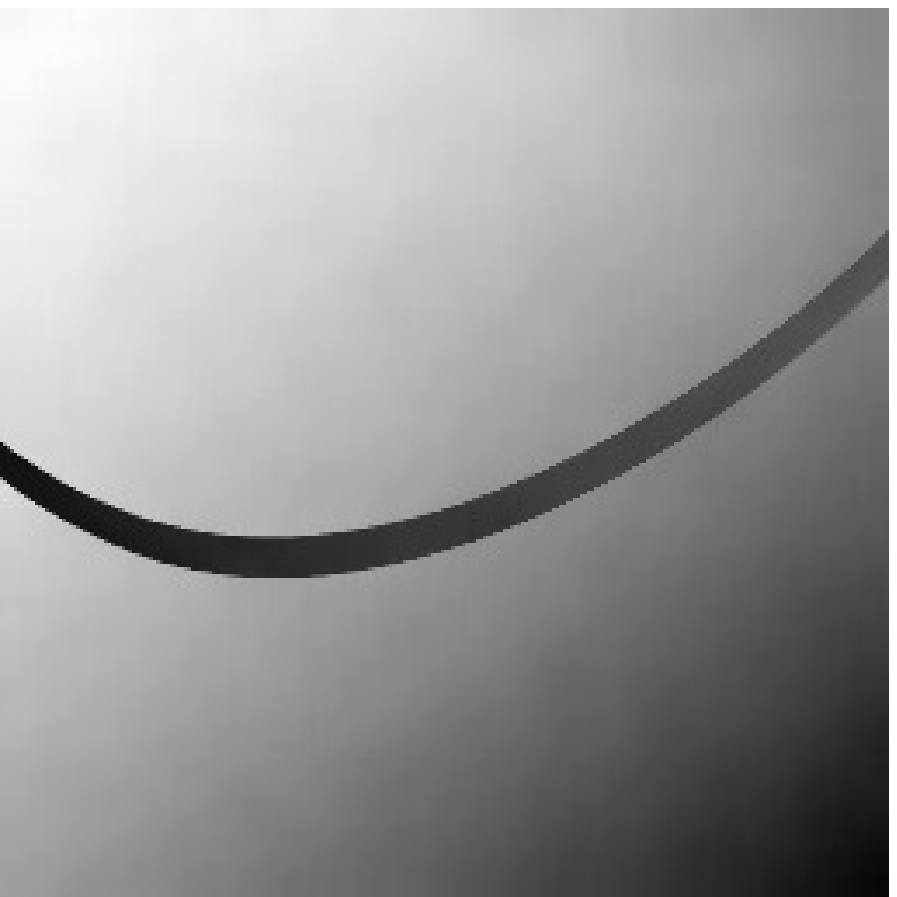}  \subfigure[\currentcaption]{\includegraphics[width=0.233\linewidth]{\currentname}}\hfill
\input{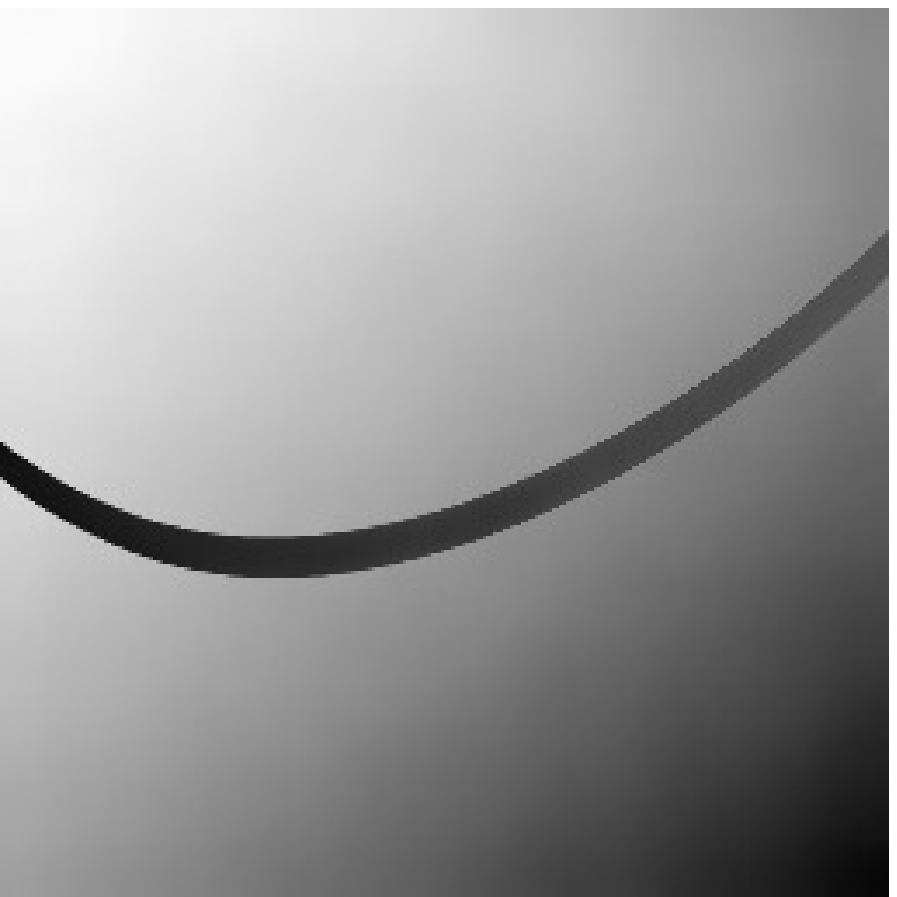}  \subfigure[\currentcaption]{\includegraphics[width=0.233\linewidth]{\currentname}}
\caption{Toy thin feature image (Swoosh)  corrupted Gaussian noise with $\sigma=20$. }
\label{fig:Sinusoid_sigma=20}
\end{figure}

\begin{figure}[hbt!]
\centering
\input{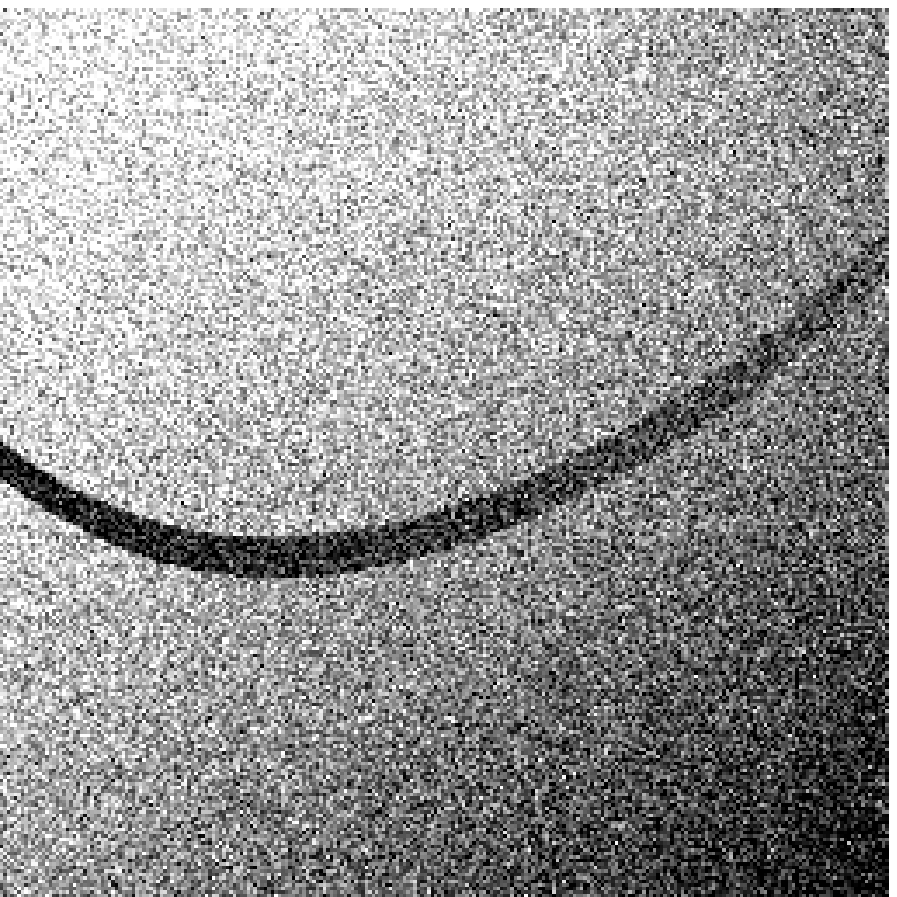}  \subfigure[\currentcaption]{\includegraphics[width=0.233\linewidth]{\currentname}}\hfill
\input{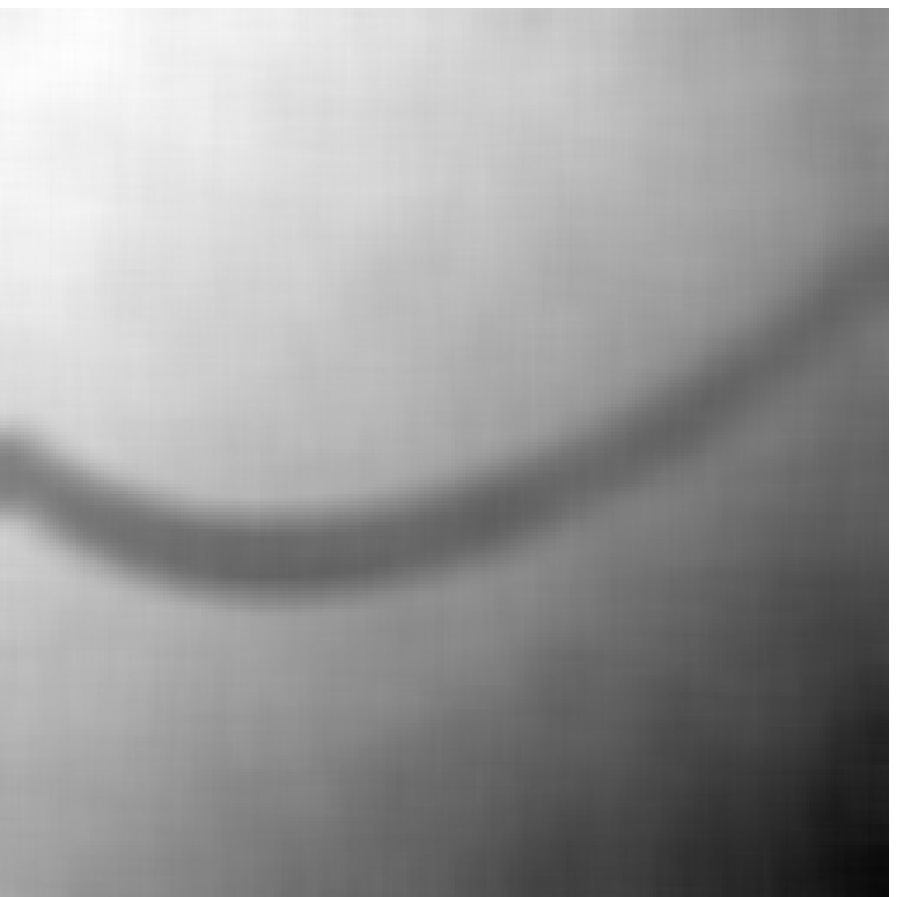}  \subfigure[\currentcaption]{\includegraphics[width=0.233\linewidth]{\currentname}}\hfill
\input{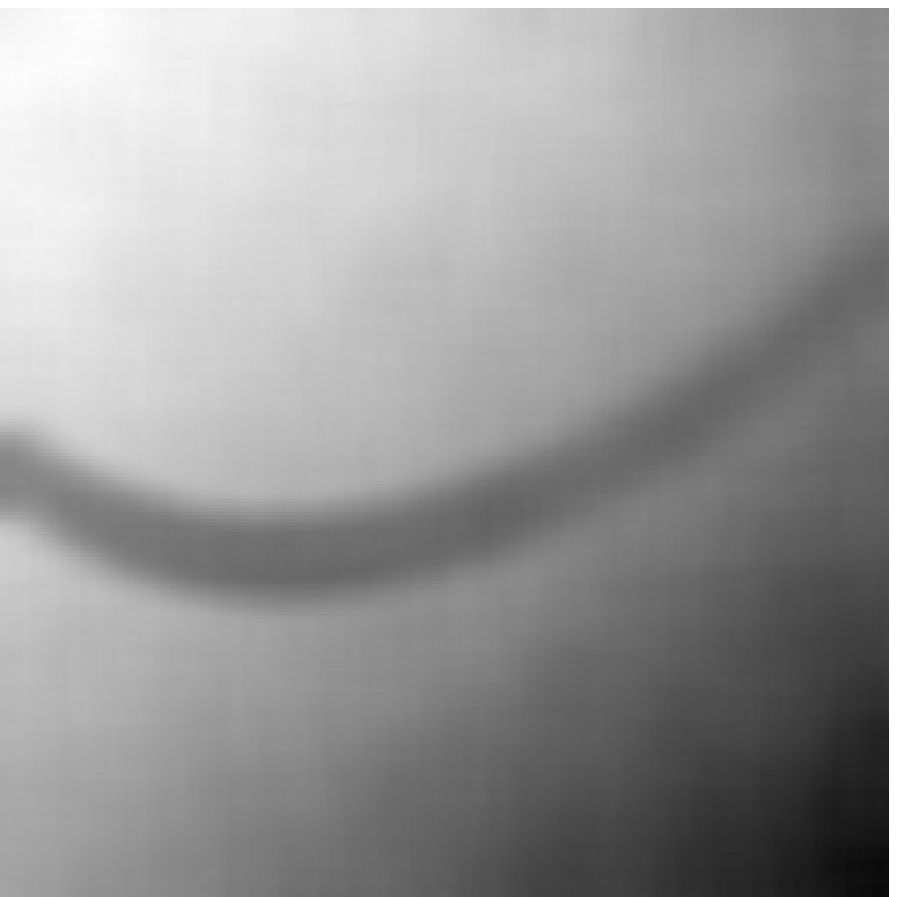}  \subfigure[\currentcaption]{\includegraphics[width=0.233\linewidth]{\currentname}}\hfill
\input{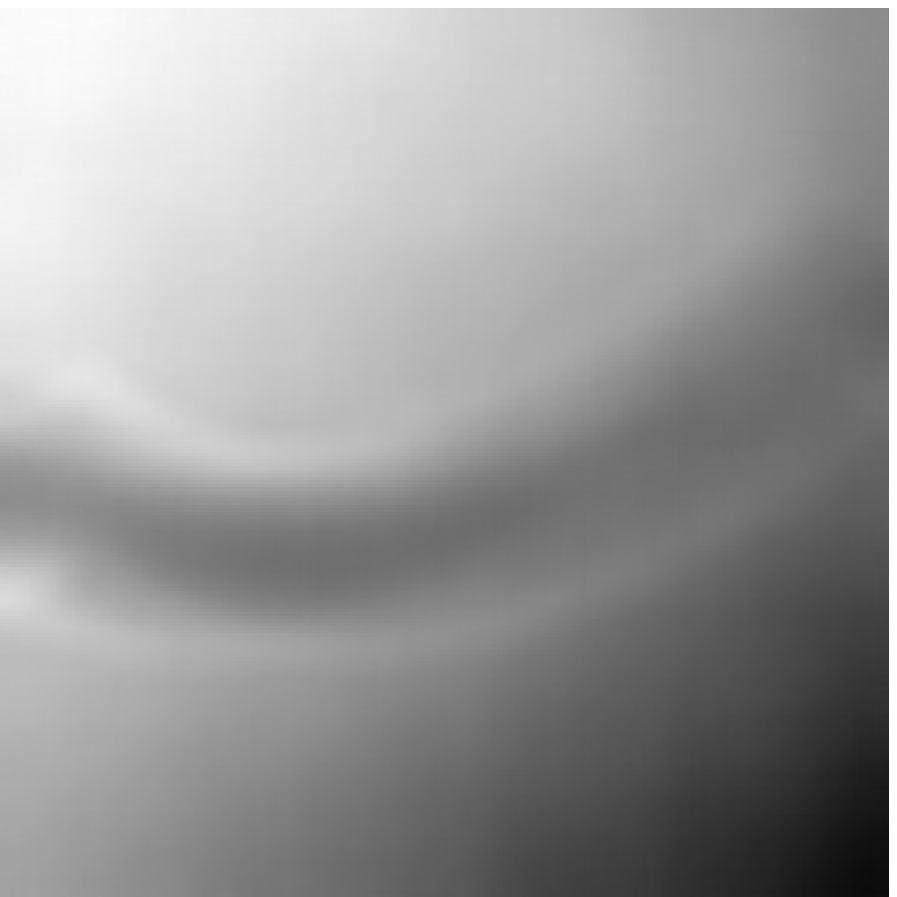}  \subfigure[\currentcaption]{\includegraphics[width=0.233\linewidth]{\currentname}}\hfill
\input{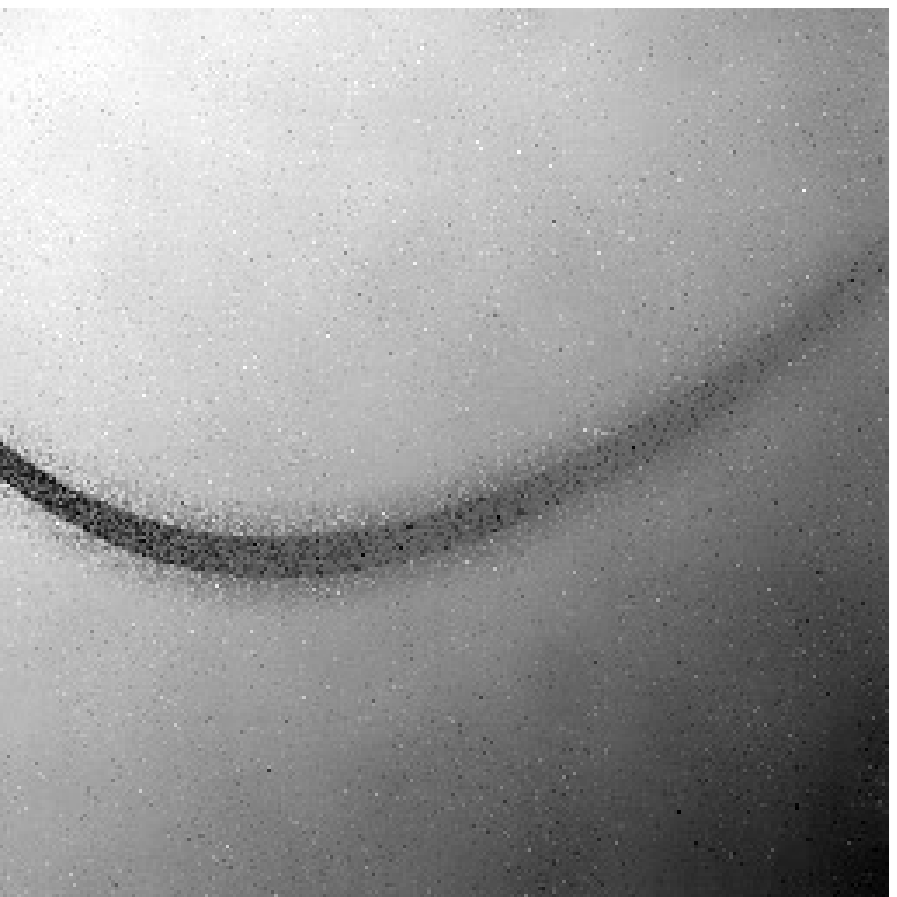}  \subfigure[\currentcaption]{\includegraphics[width=0.233\linewidth]{\currentname}}\hfill
\input{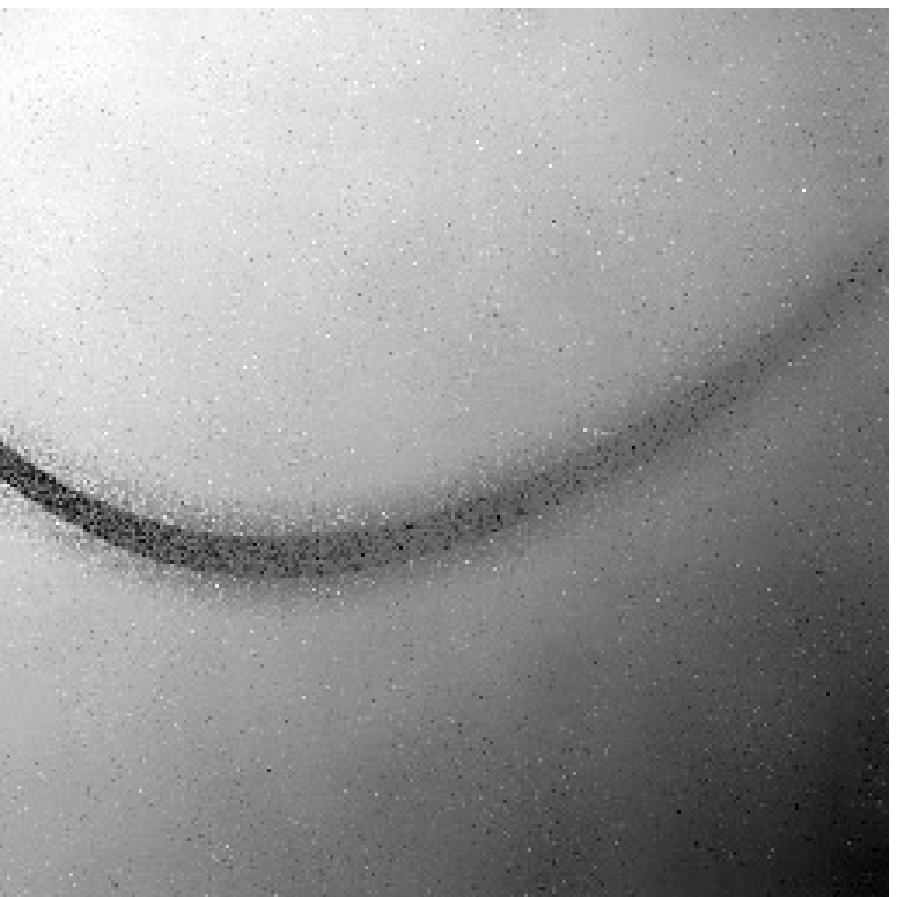}  \subfigure[\currentcaption]{\includegraphics[width=0.233\linewidth]{\currentname}}\hfill
\input{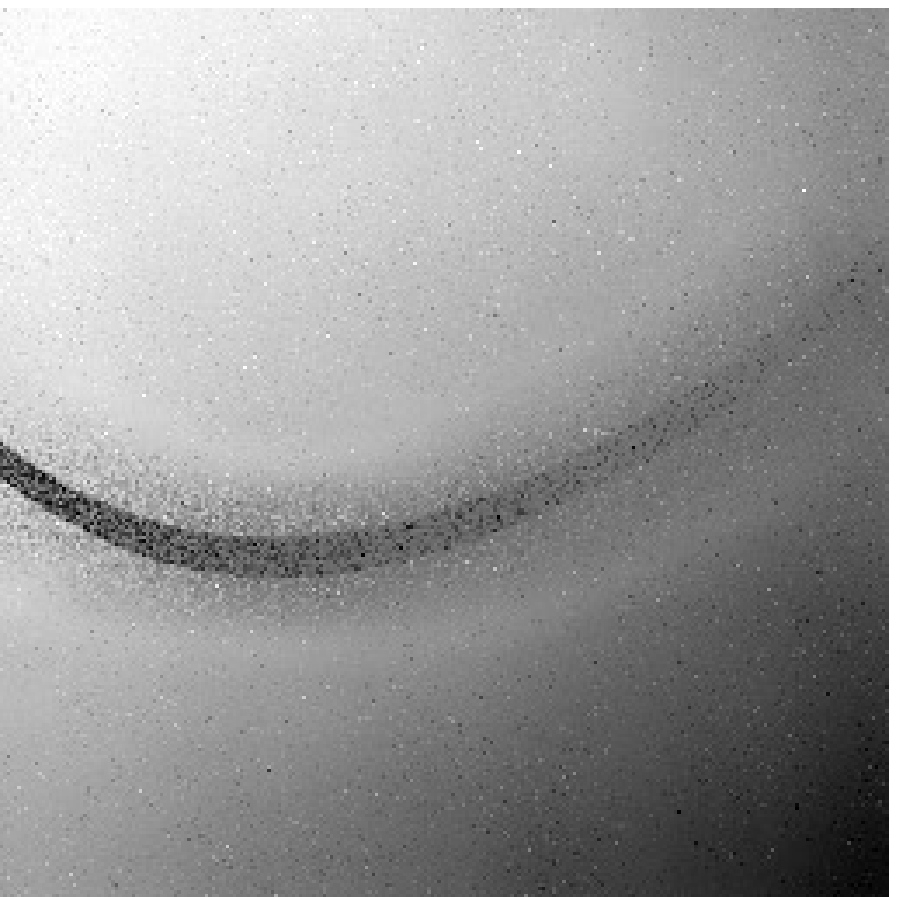}  \subfigure[\currentcaption]{\includegraphics[width=0.233\linewidth]{\currentname}}\hfill
\input{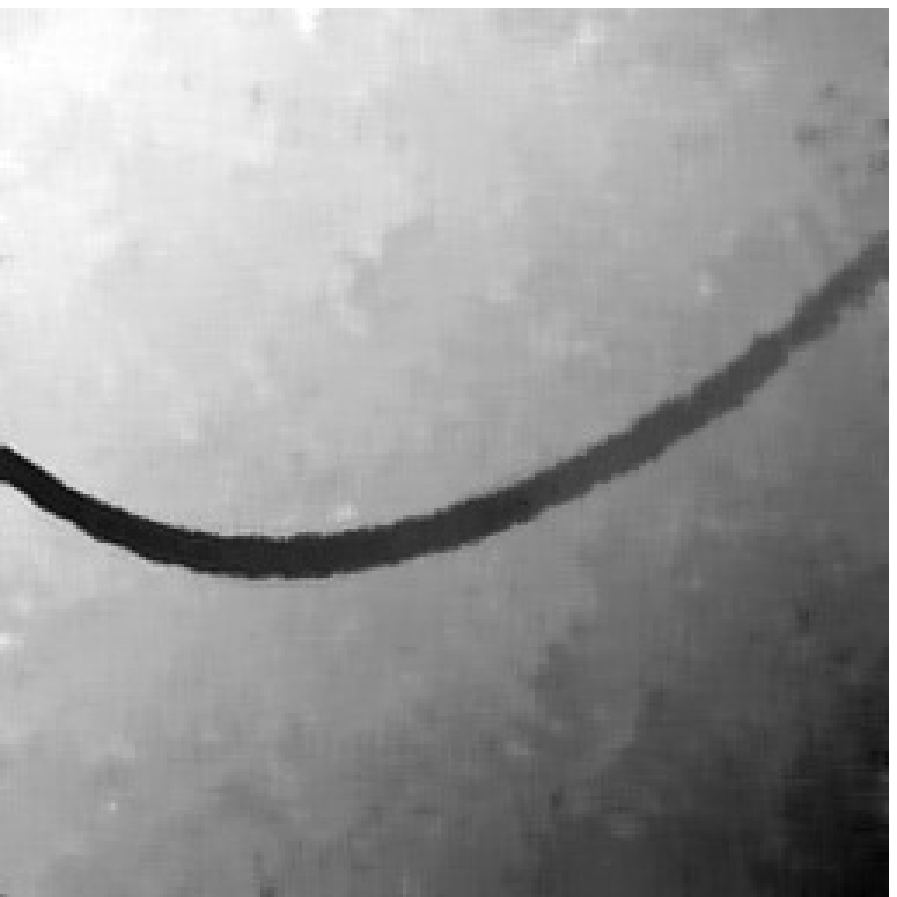}  \subfigure[\currentcaption]{\includegraphics[width=0.233\linewidth]{\currentname}}\hfill
\input{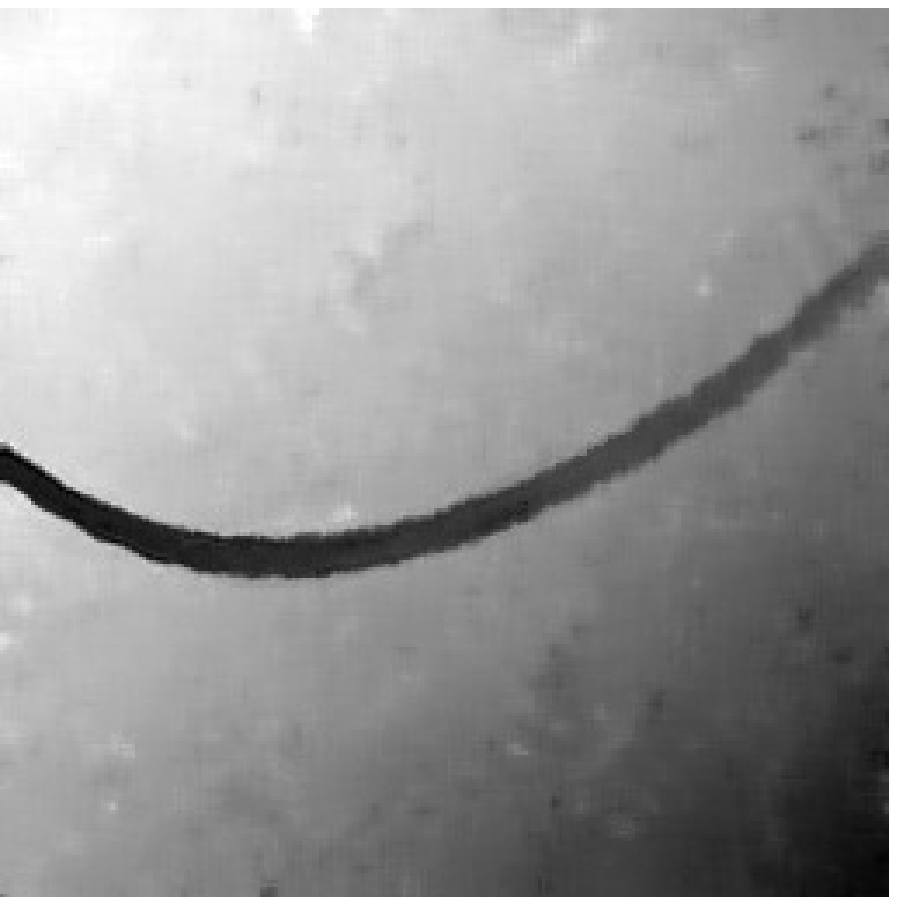}  \subfigure[\currentcaption]{\includegraphics[width=0.233\linewidth]{\currentname}}\hfill
\input{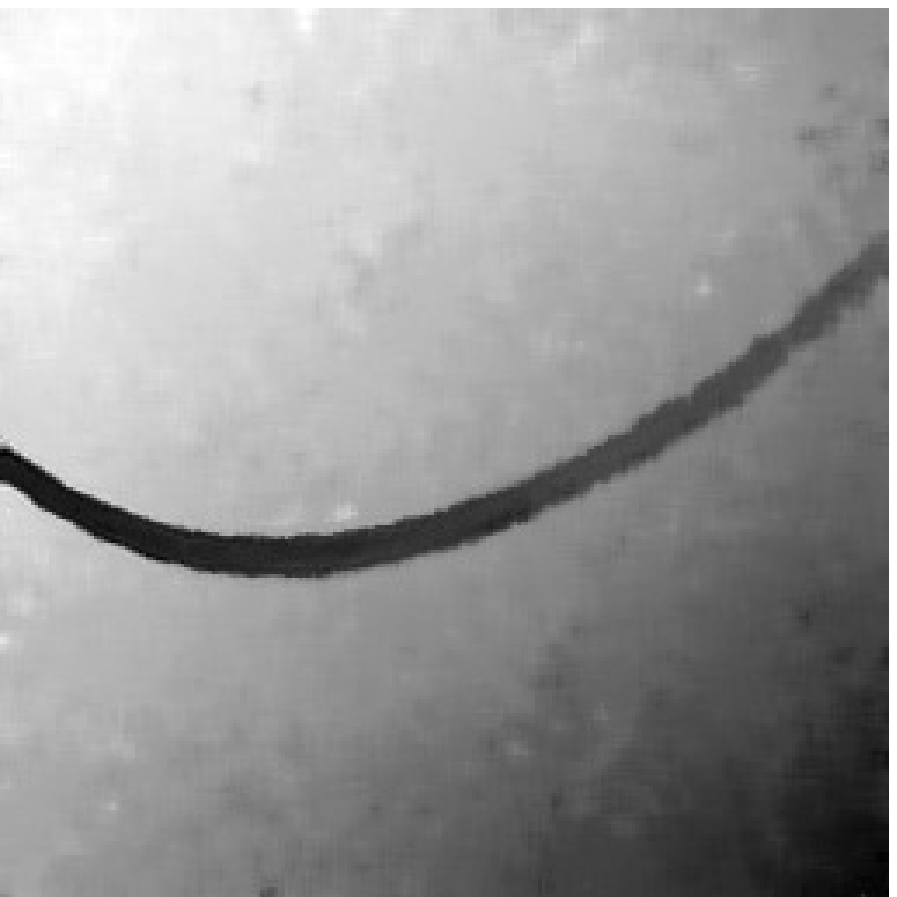}  \subfigure[\currentcaption]{\includegraphics[width=0.233\linewidth]{\currentname}}\hfill
\input{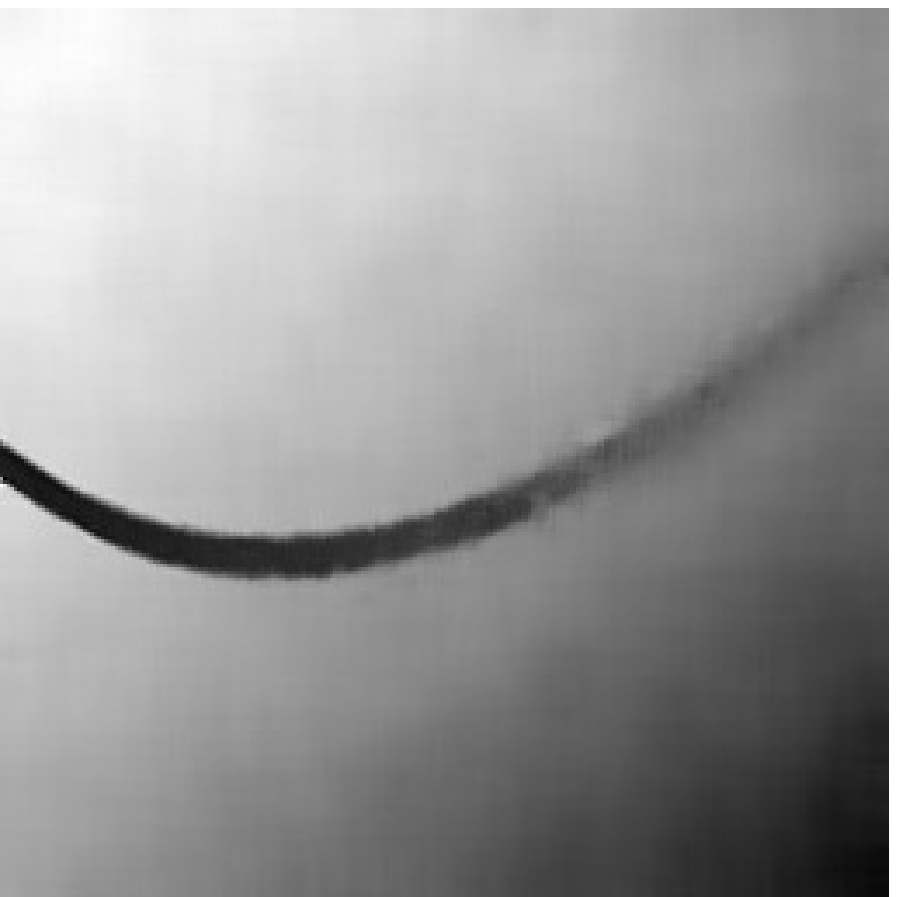}  \subfigure[\currentcaption]{\includegraphics[width=0.233\linewidth]{\currentname}}\hfill
\input{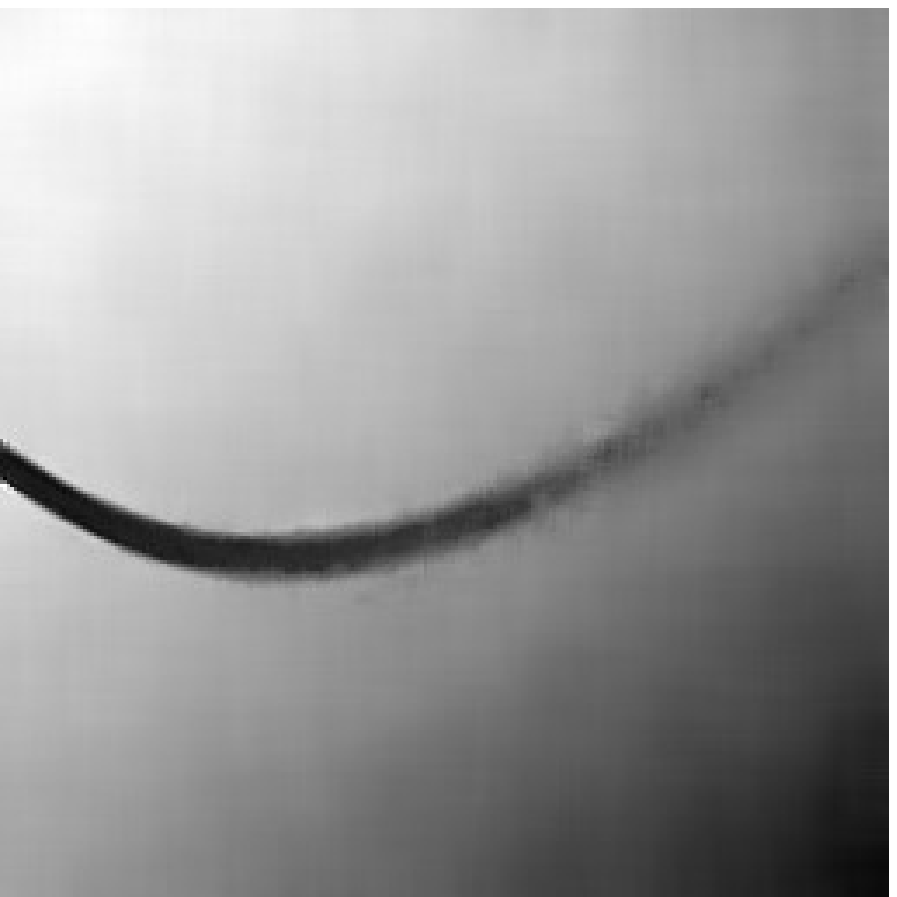}  \subfigure[\currentcaption]{\includegraphics[width=0.233\linewidth]{\currentname}}\hfill
\input{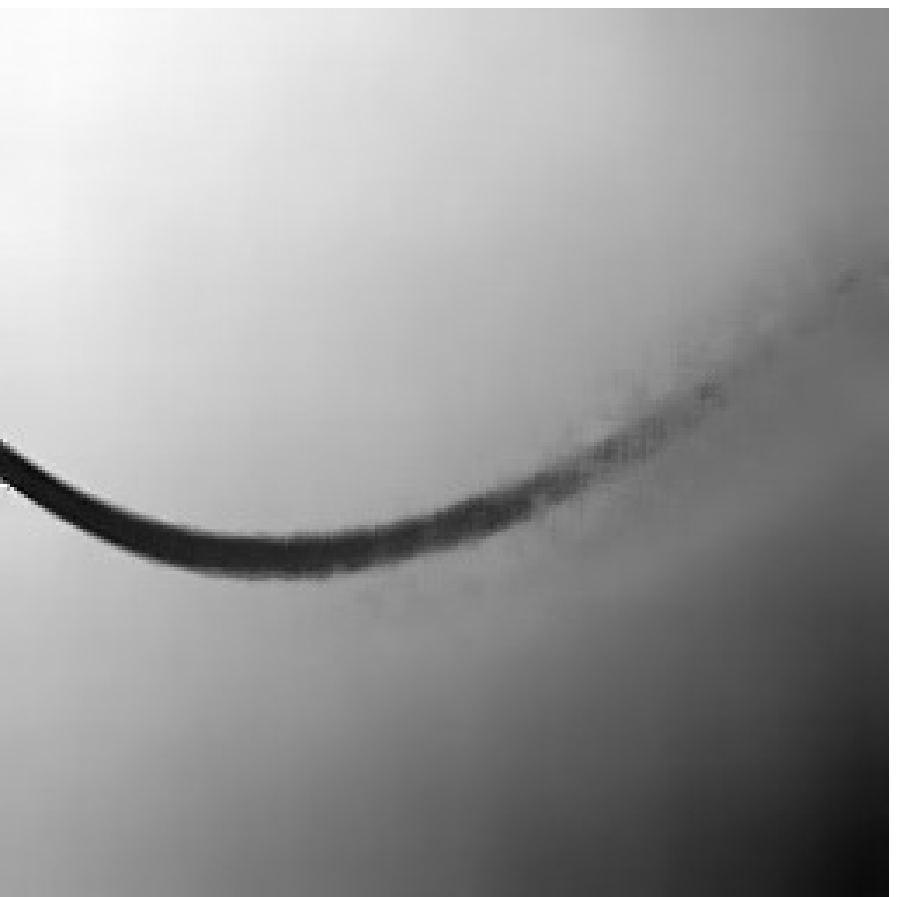}  \subfigure[\currentcaption]{\includegraphics[width=0.233\linewidth]{\currentname}}\hfill
\input{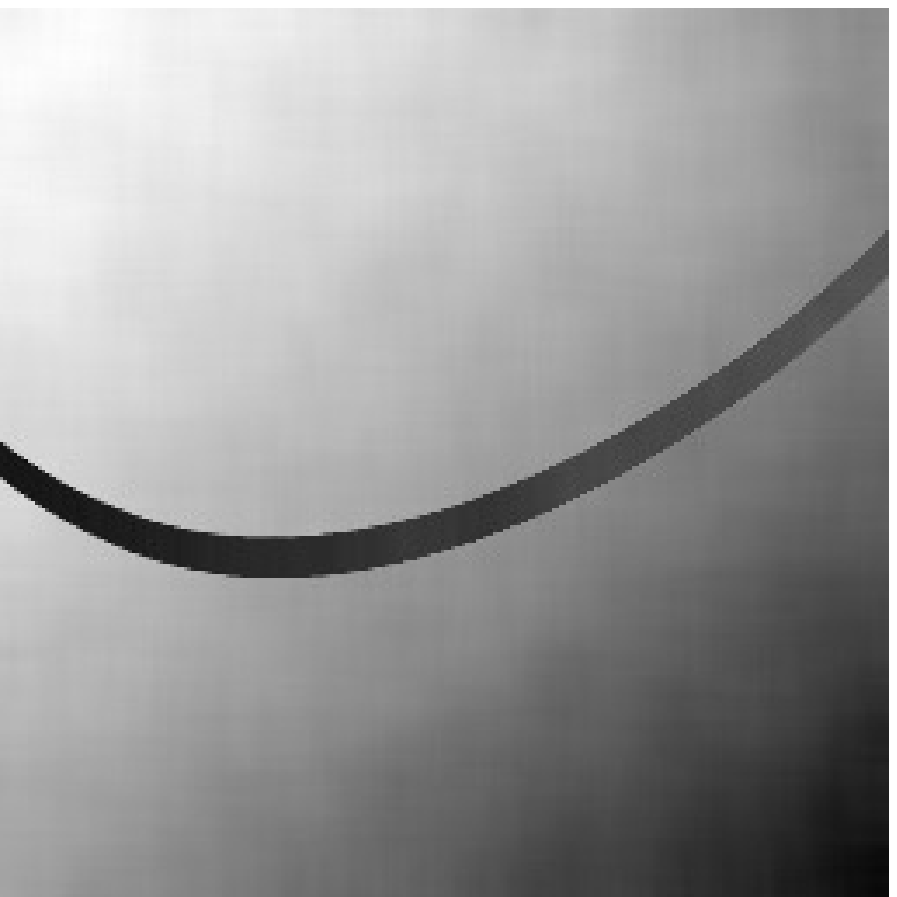}  \subfigure[\currentcaption]{\includegraphics[width=0.233\linewidth]{\currentname}}\hfill
\input{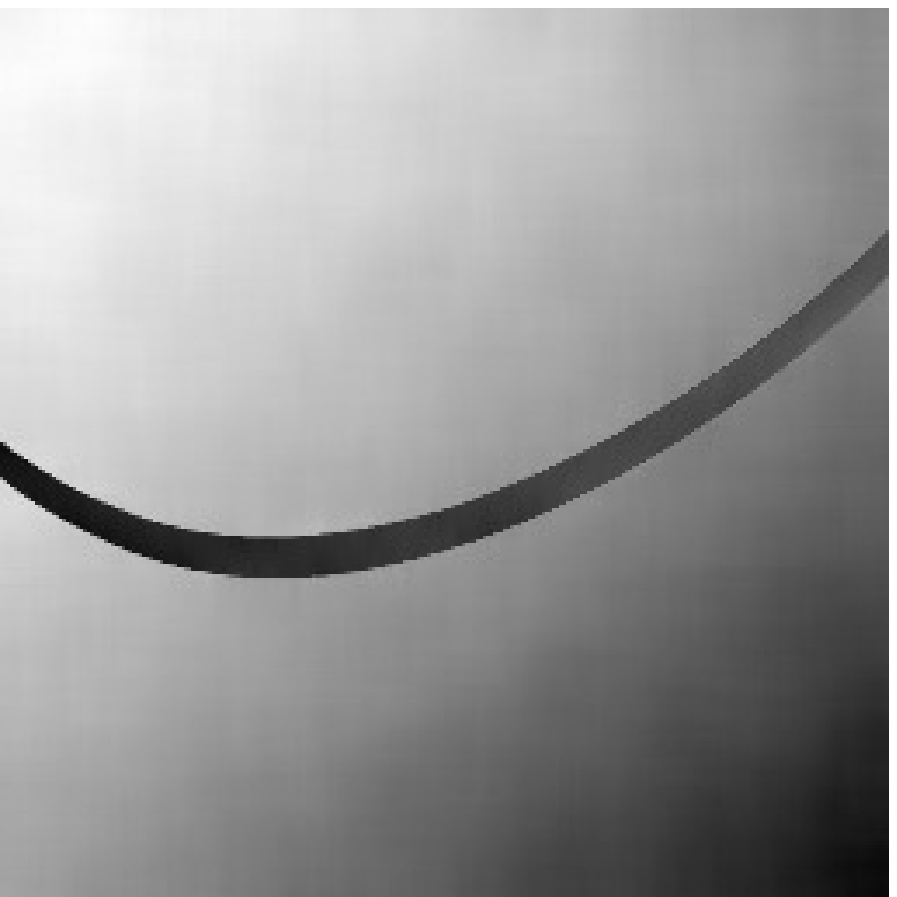}  \subfigure[\currentcaption]{\includegraphics[width=0.233\linewidth]{\currentname}}\hfill
\input{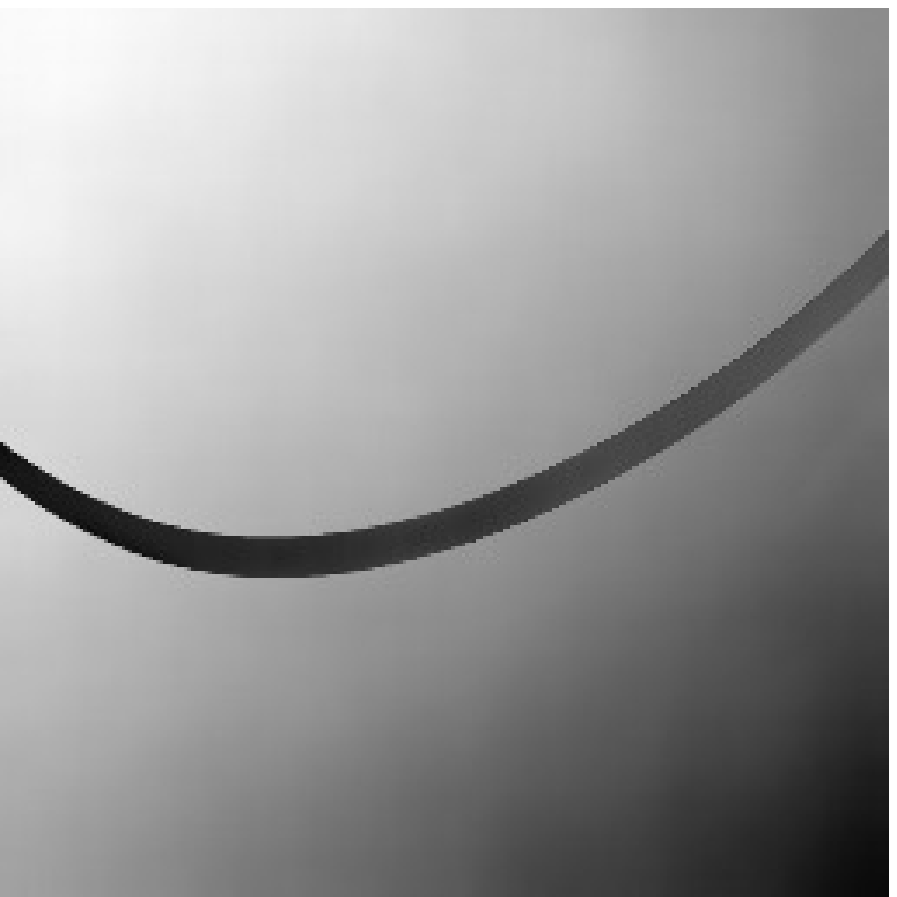}  \subfigure[\currentcaption]{\includegraphics[width=0.233\linewidth]{\currentname}}
\caption{Toy thin feature image (Swoosh) corrupted Gaussian noise with $\sigma=50$. }
\label{fig:Sinusoid_sigma=50}
\end{figure}

\begin{figure}[hbt!]
\centering
\input{./latex_figures/Im_Noisy_100_Sinusoid}  \subfigure[\currentcaption]{\includegraphics[width=0.233\linewidth]{\currentname}}\hfill
\input{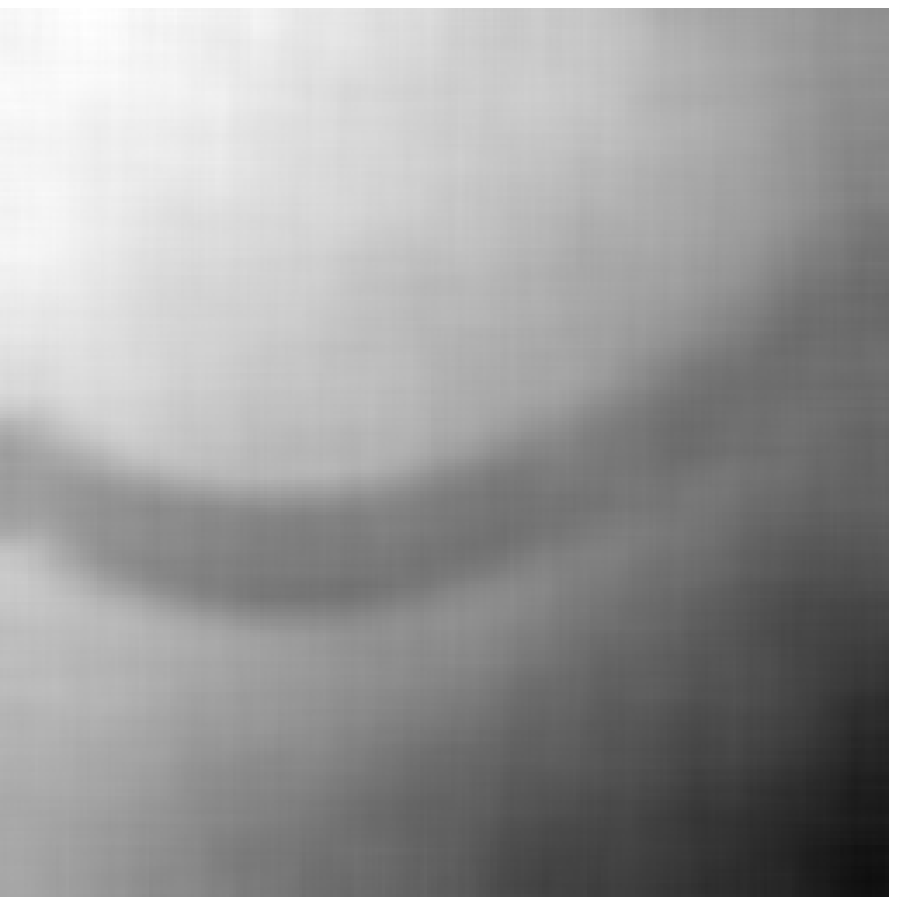}  \subfigure[\currentcaption]{\includegraphics[width=0.233\linewidth]{\currentname}}\hfill
\input{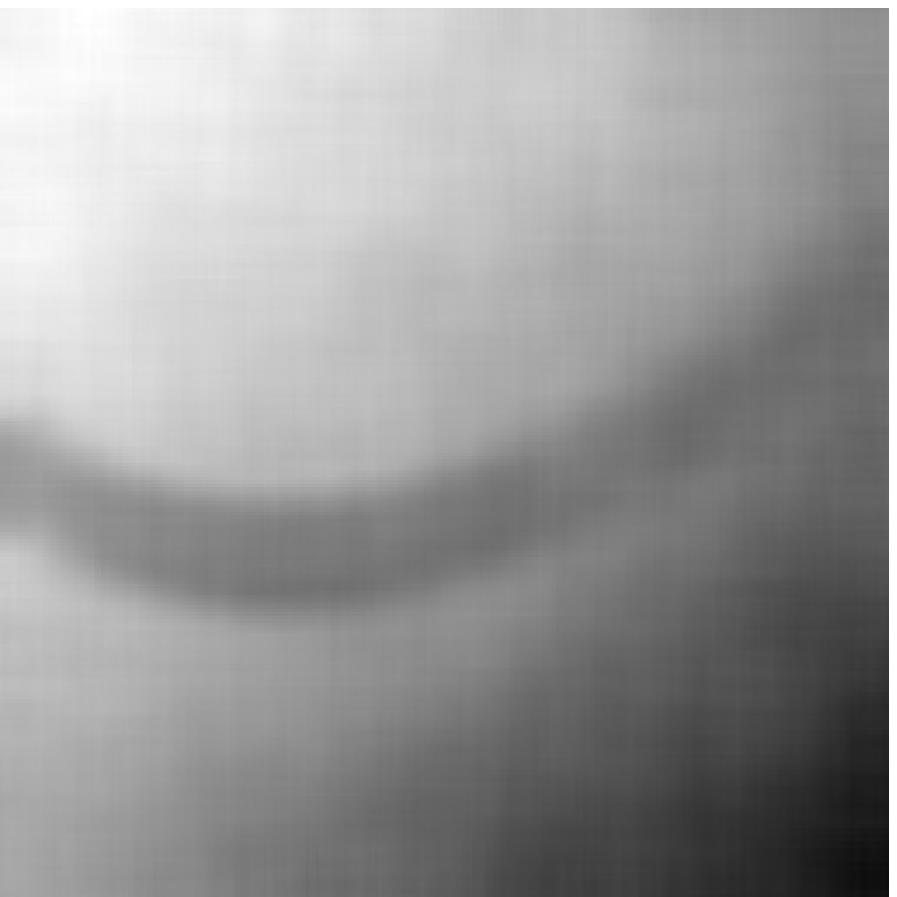}  \subfigure[\currentcaption]{\includegraphics[width=0.233\linewidth]{\currentname}}\hfill
\input{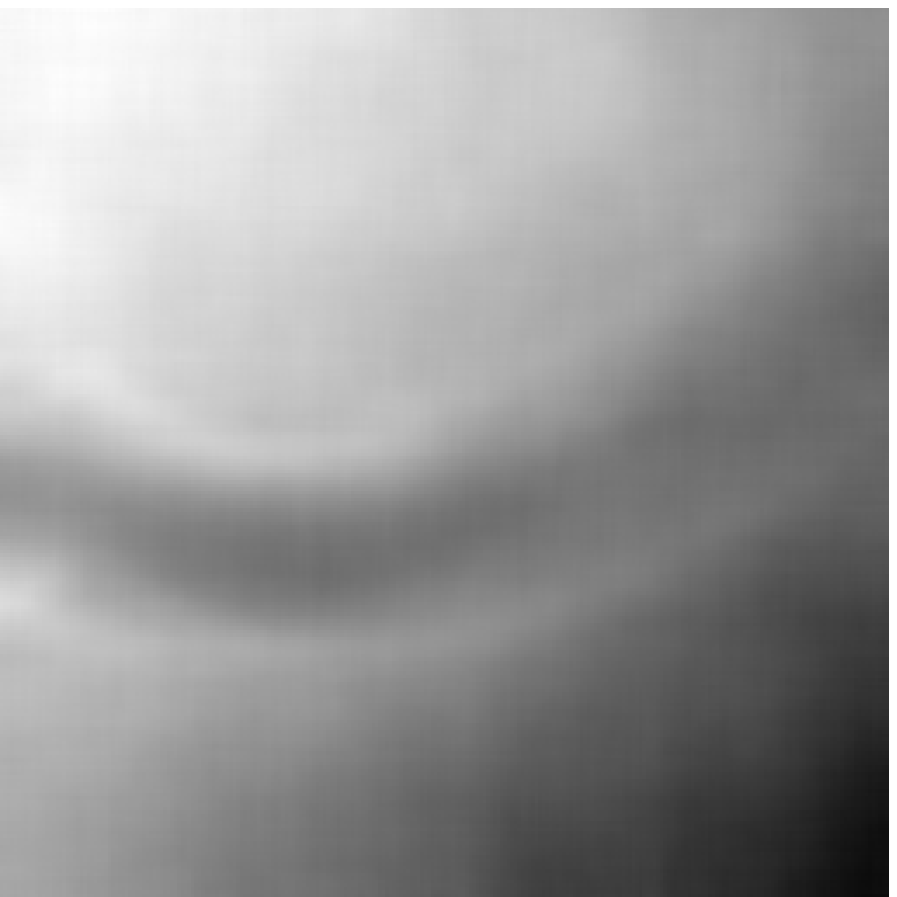}  \subfigure[\currentcaption]{\includegraphics[width=0.233\linewidth]{\currentname}}\hfill
\input{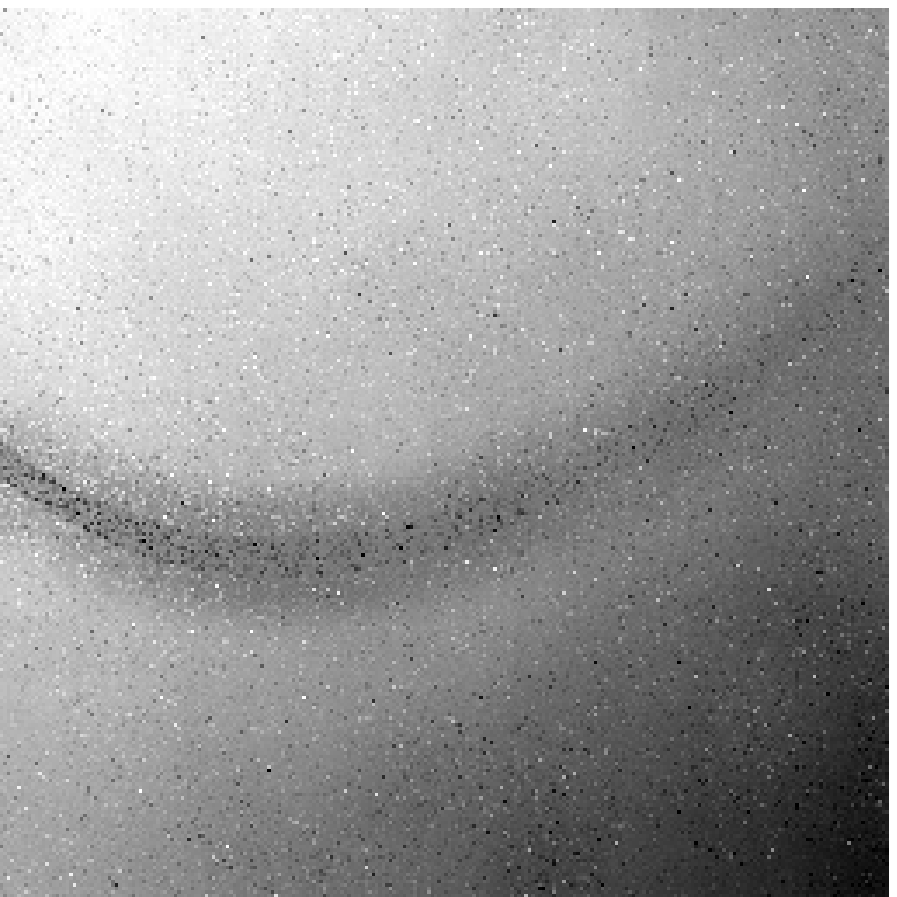}  \subfigure[\currentcaption]{\includegraphics[width=0.233\linewidth]{\currentname}}\hfill
\input{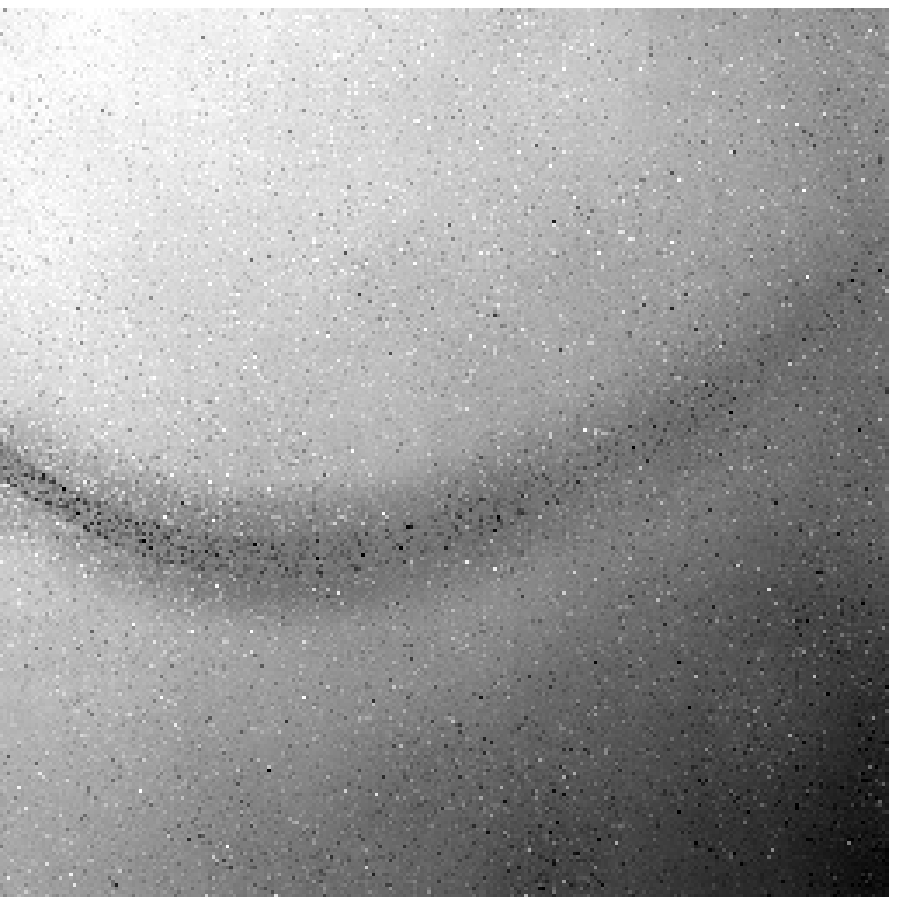}  \subfigure[\currentcaption]{\includegraphics[width=0.233\linewidth]{\currentname}}\hfill
\input{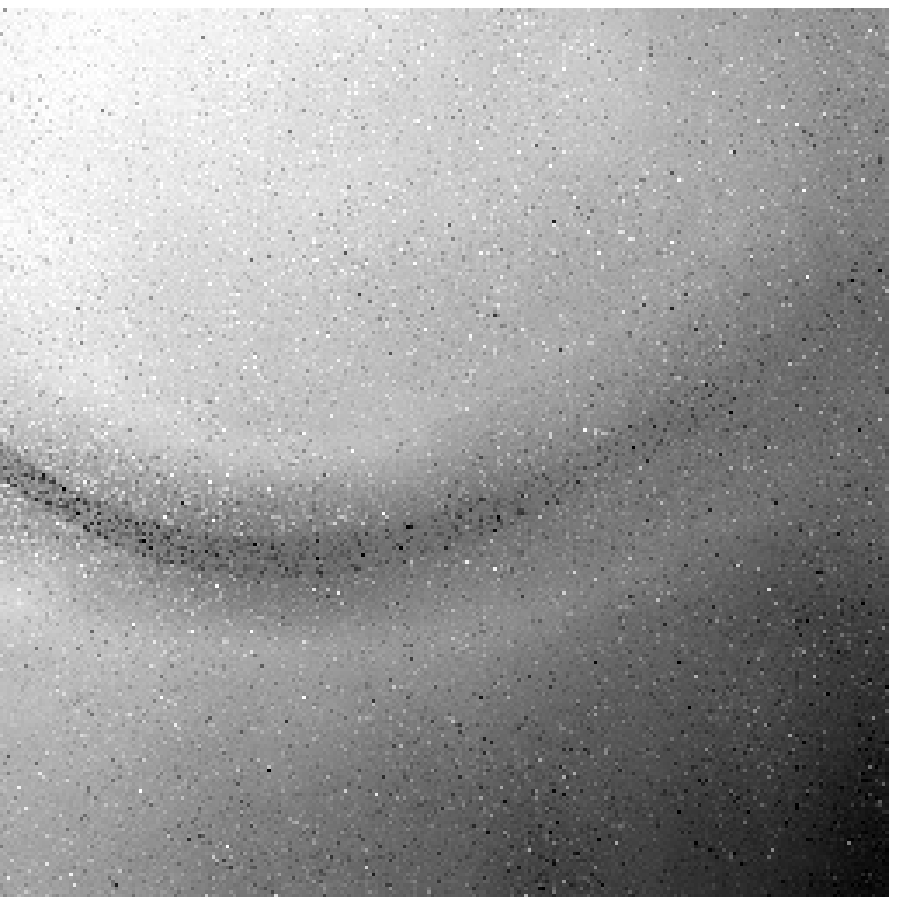}  \subfigure[\currentcaption]{\includegraphics[width=0.233\linewidth]{\currentname}}\hfill
\input{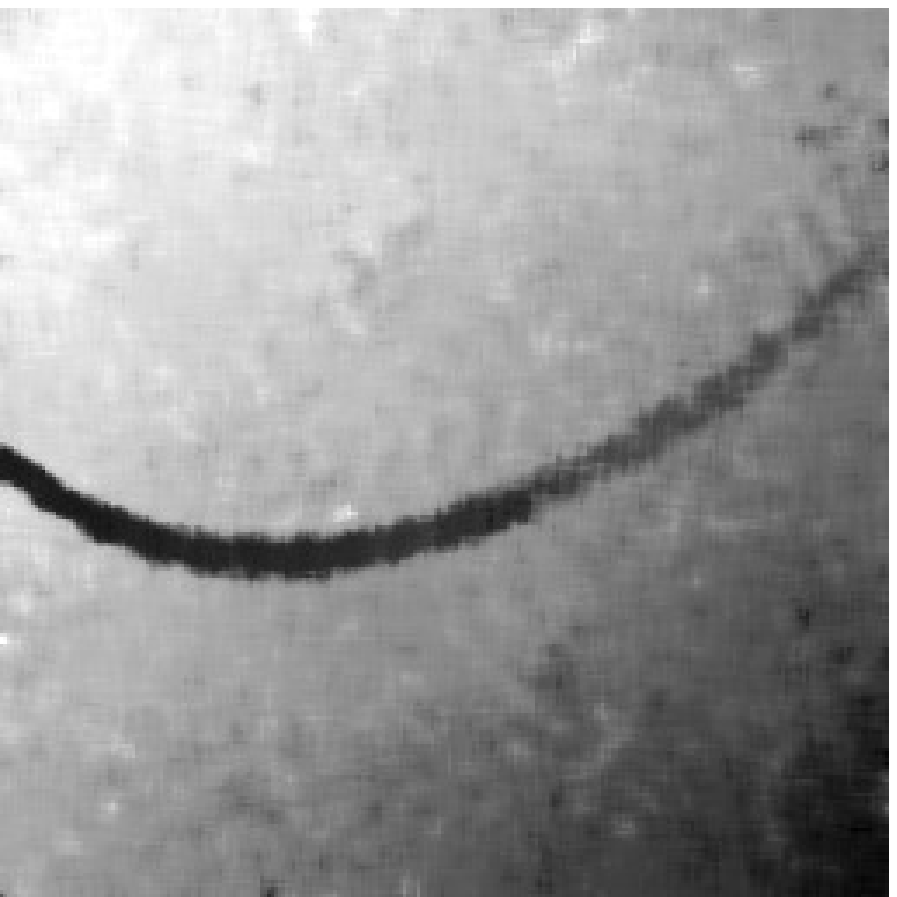}  \subfigure[\currentcaption]{\includegraphics[width=0.233\linewidth]{\currentname}}\hfill
\input{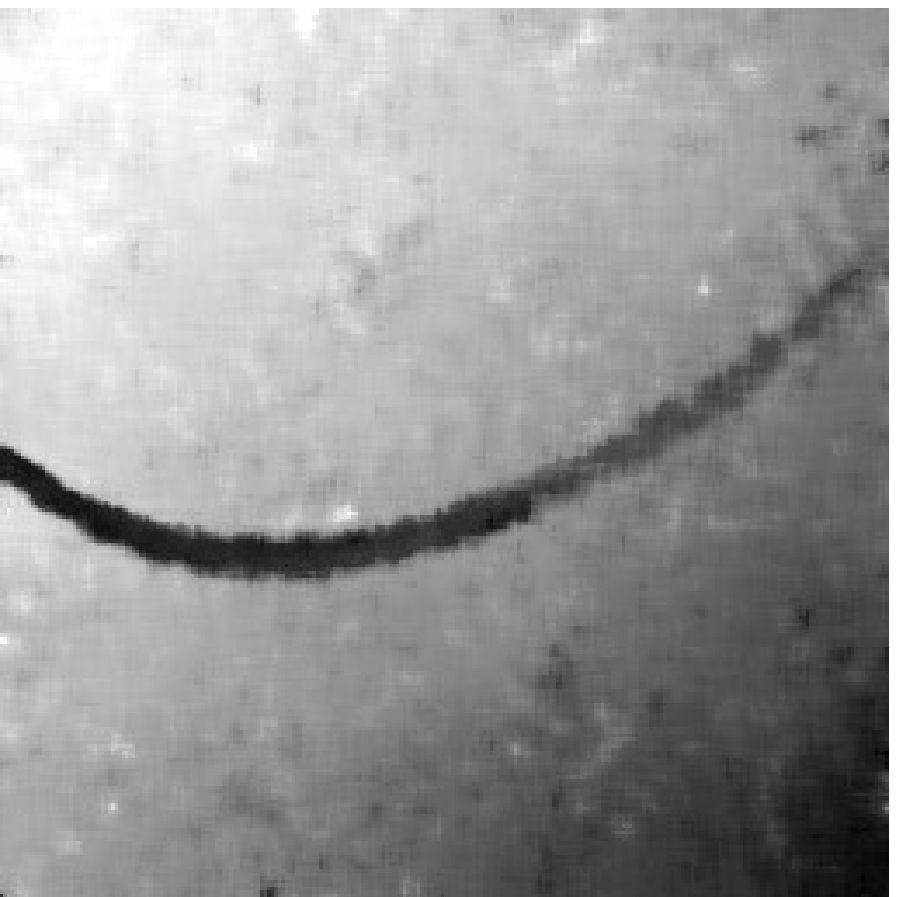}  \subfigure[\currentcaption]{\includegraphics[width=0.233\linewidth]{\currentname}}\hfill
\input{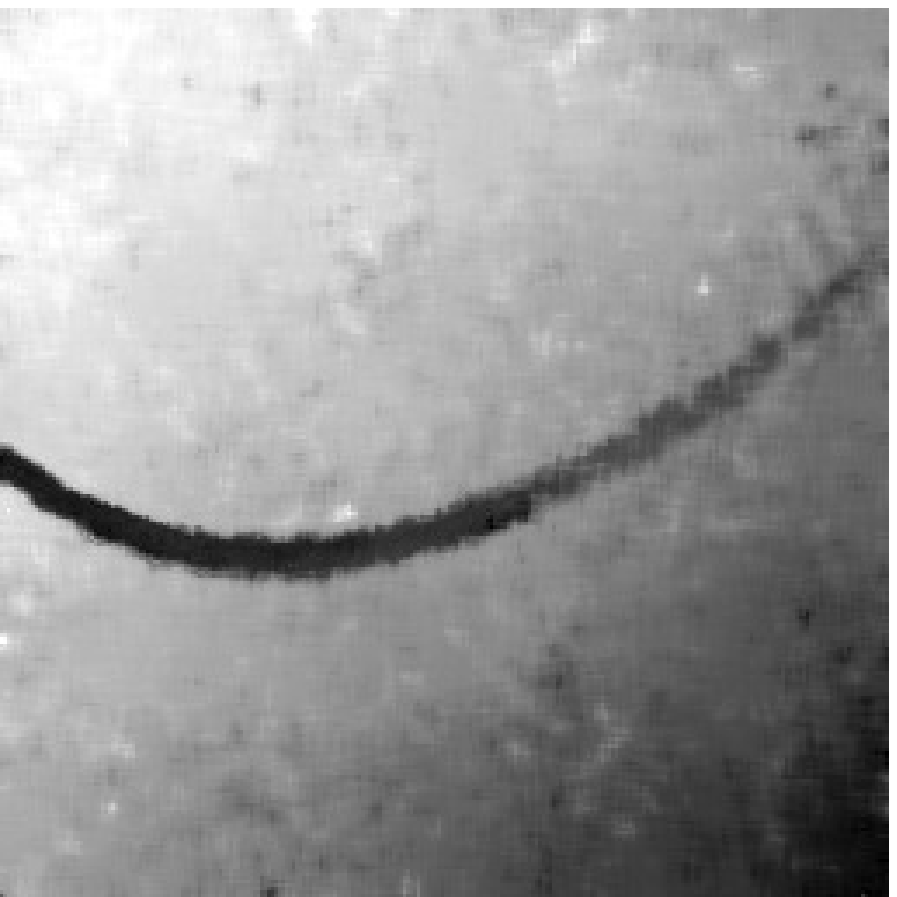}  \subfigure[\currentcaption]{\includegraphics[width=0.233\linewidth]{\currentname}}\hfill
\input{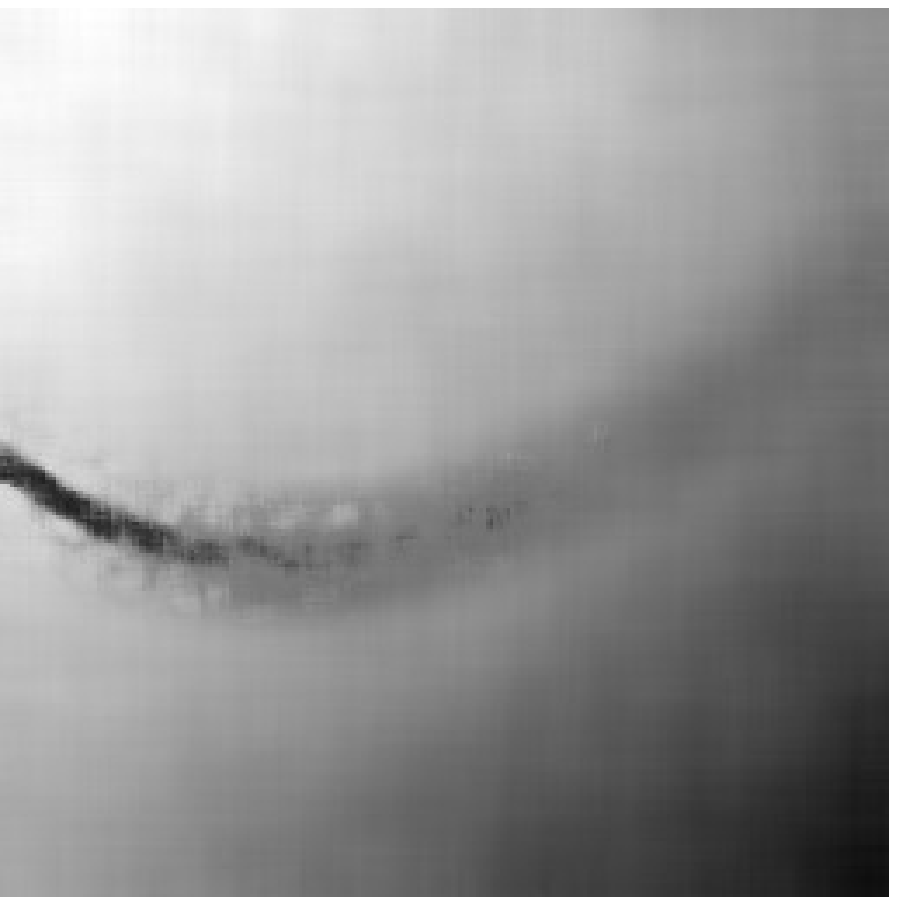}  \subfigure[\currentcaption]{\includegraphics[width=0.233\linewidth]{\currentname}}\hfill
\input{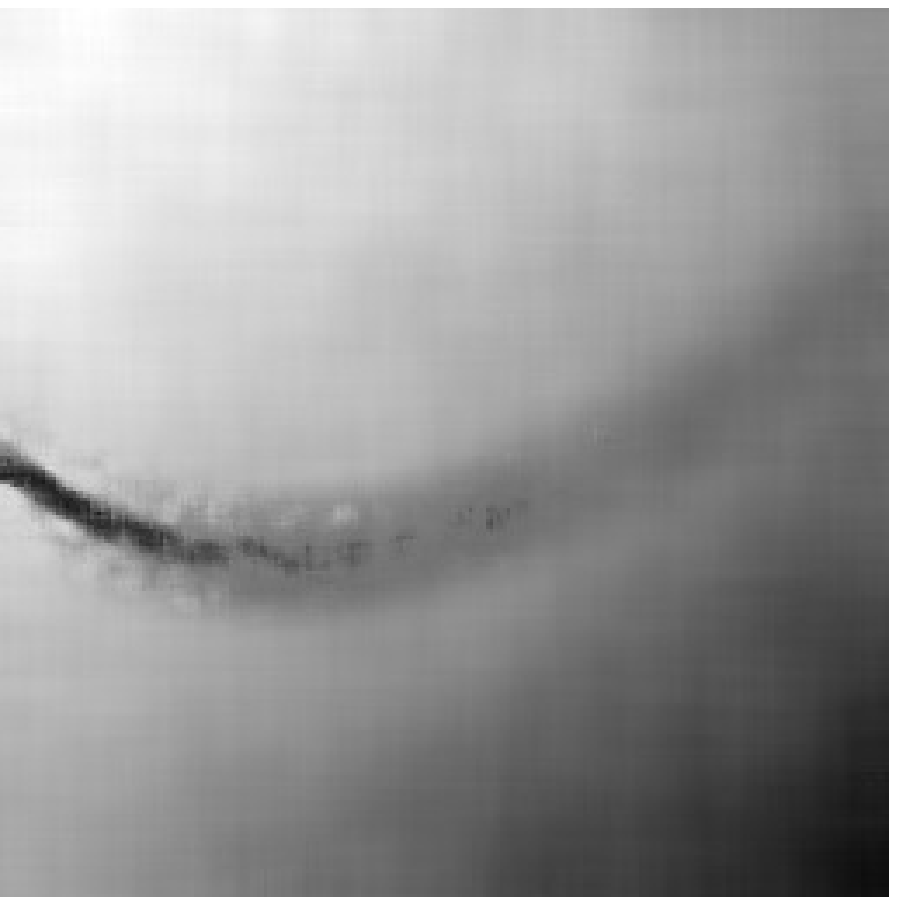}  \subfigure[\currentcaption]{\includegraphics[width=0.233\linewidth]{\currentname}}\hfill
\input{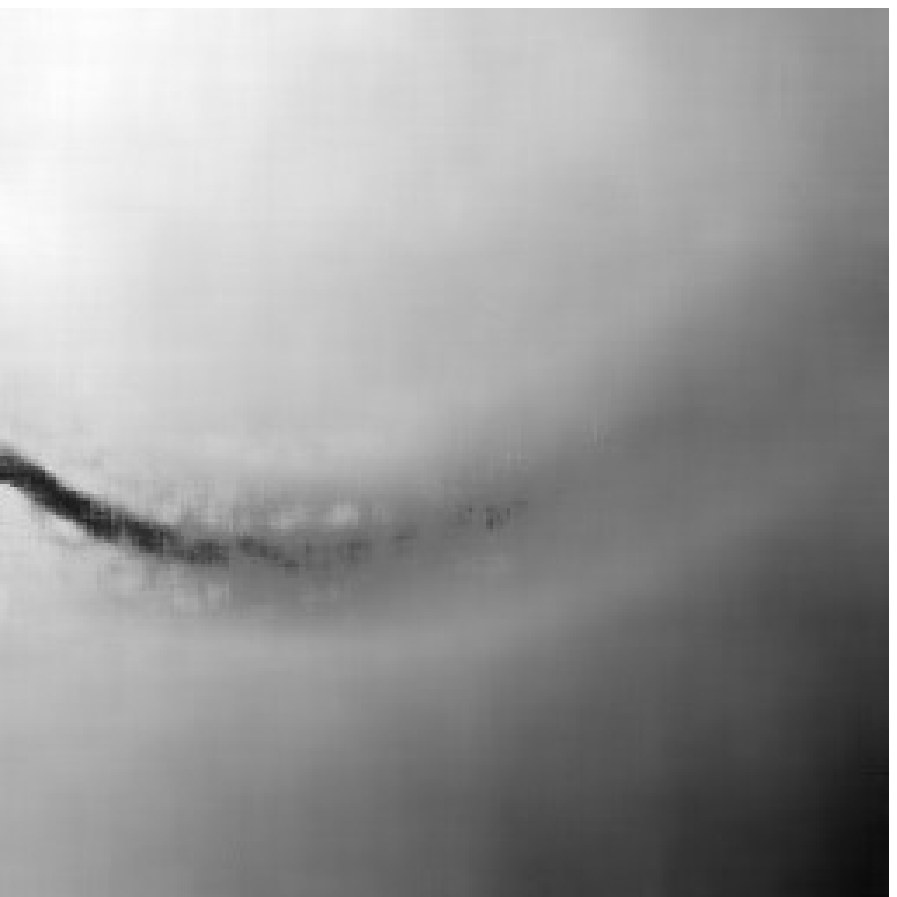}  \subfigure[\currentcaption]{\includegraphics[width=0.233\linewidth]{\currentname}}\hfill
\input{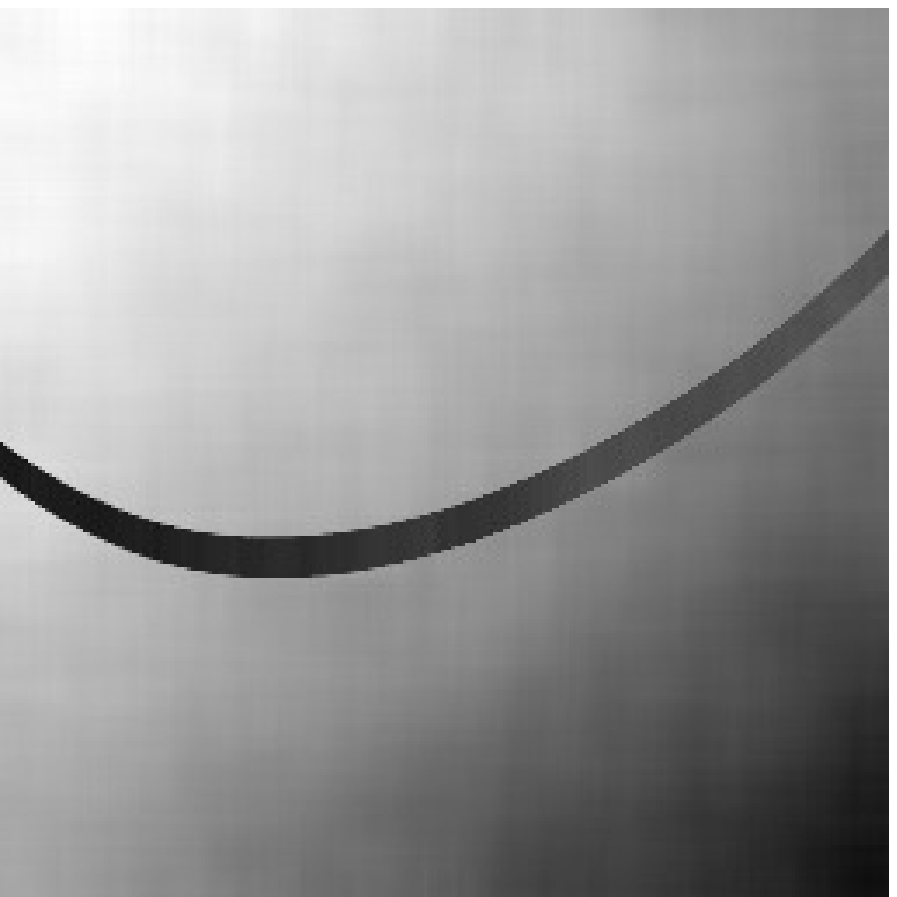}  \subfigure[\currentcaption]{\includegraphics[width=0.233\linewidth]{\currentname}}\hfill
\input{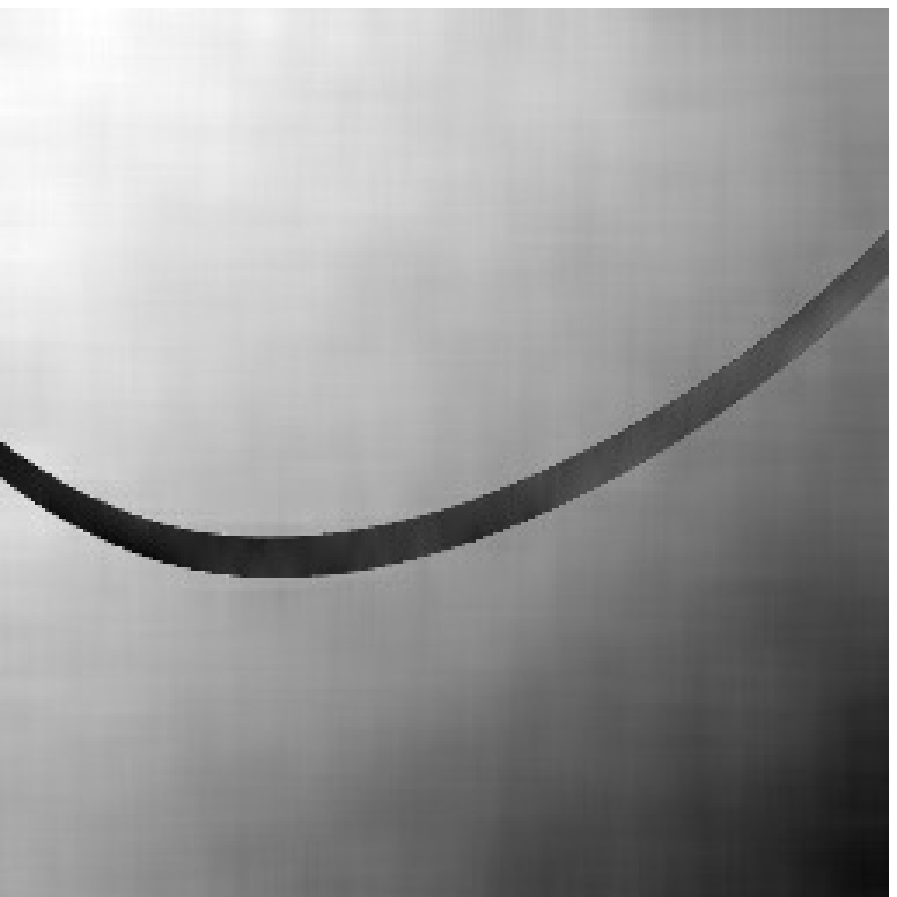}  \subfigure[\currentcaption]{\includegraphics[width=0.233\linewidth]{\currentname}}\hfill
\input{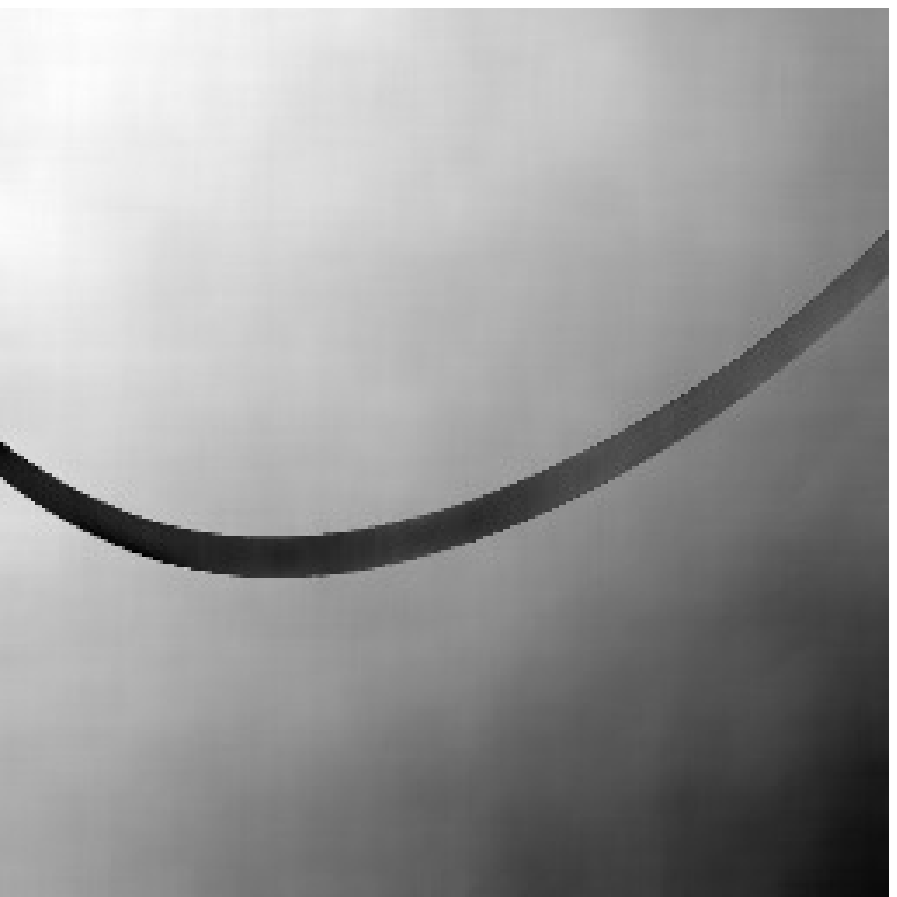}  \subfigure[\currentcaption]{\includegraphics[width=0.233\linewidth]{\currentname}}
\caption{Toy thin feature image (Swoosh) corrupted Gaussian noise with $\sigma=100$.}
\label{fig:Sinusoid_sigma=100}
\end{figure}

\begin{figure}[hbt!]
\centering
\input{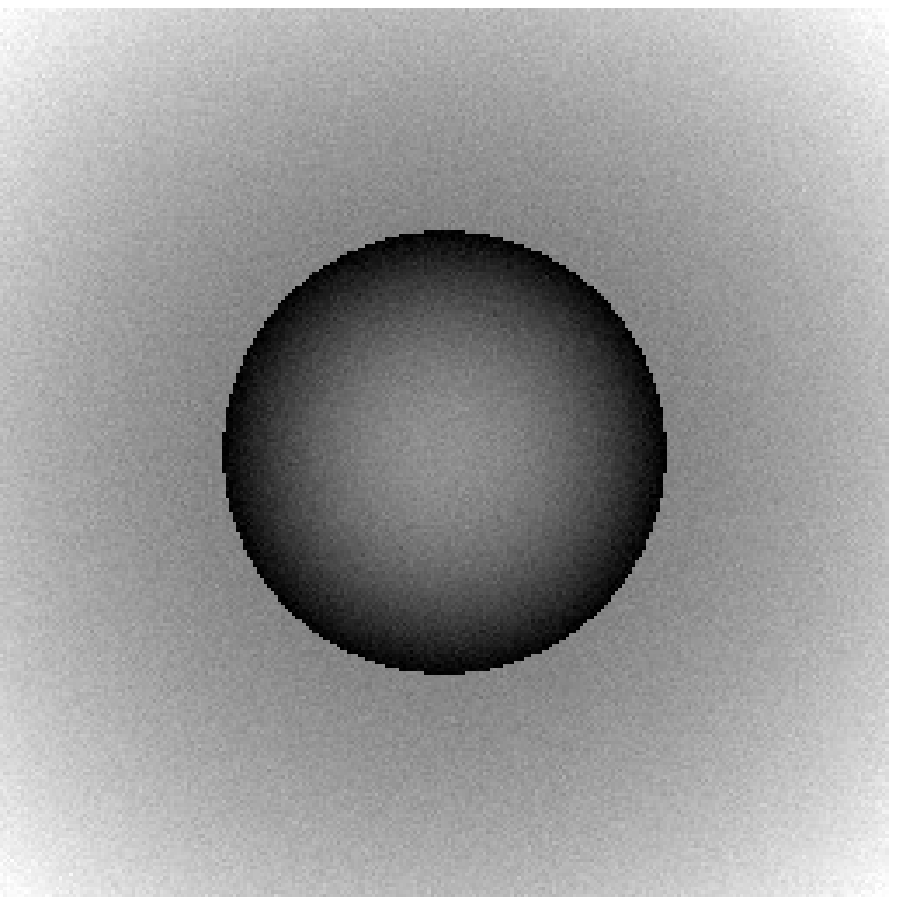}  \subfigure[\currentcaption]{\includegraphics[width=0.233\linewidth]{\currentname}}\hfill
\input{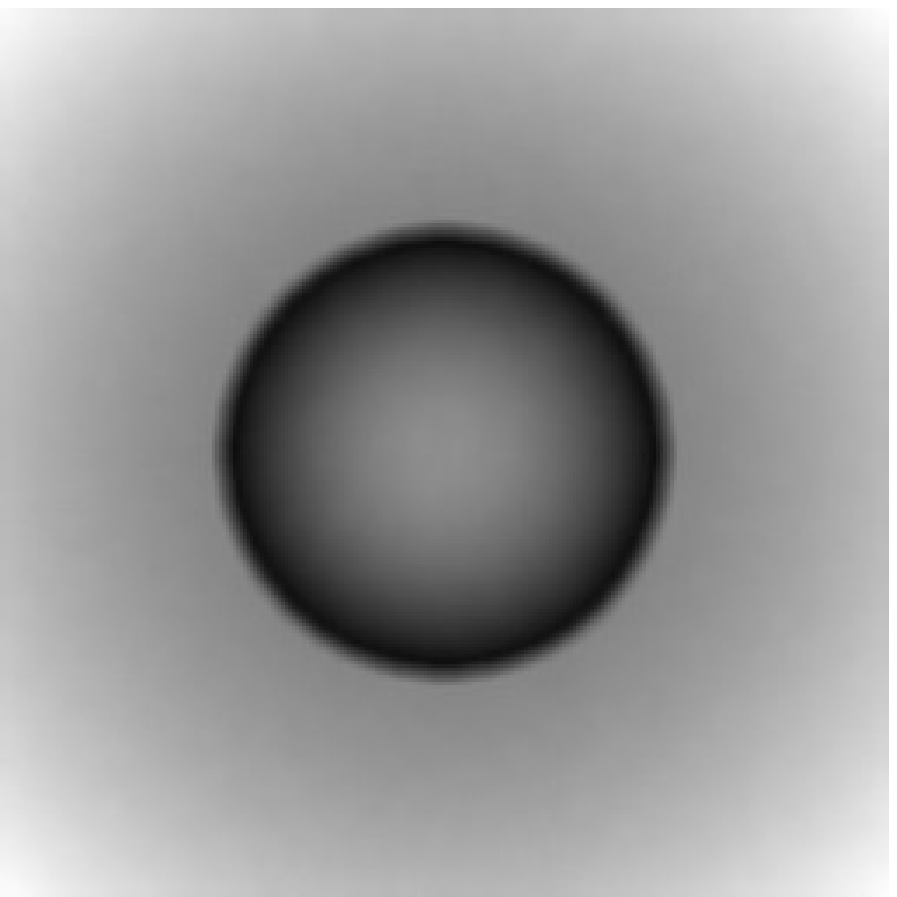}  \subfigure[\currentcaption]{\includegraphics[width=0.233\linewidth]{\currentname}}\hfill
\input{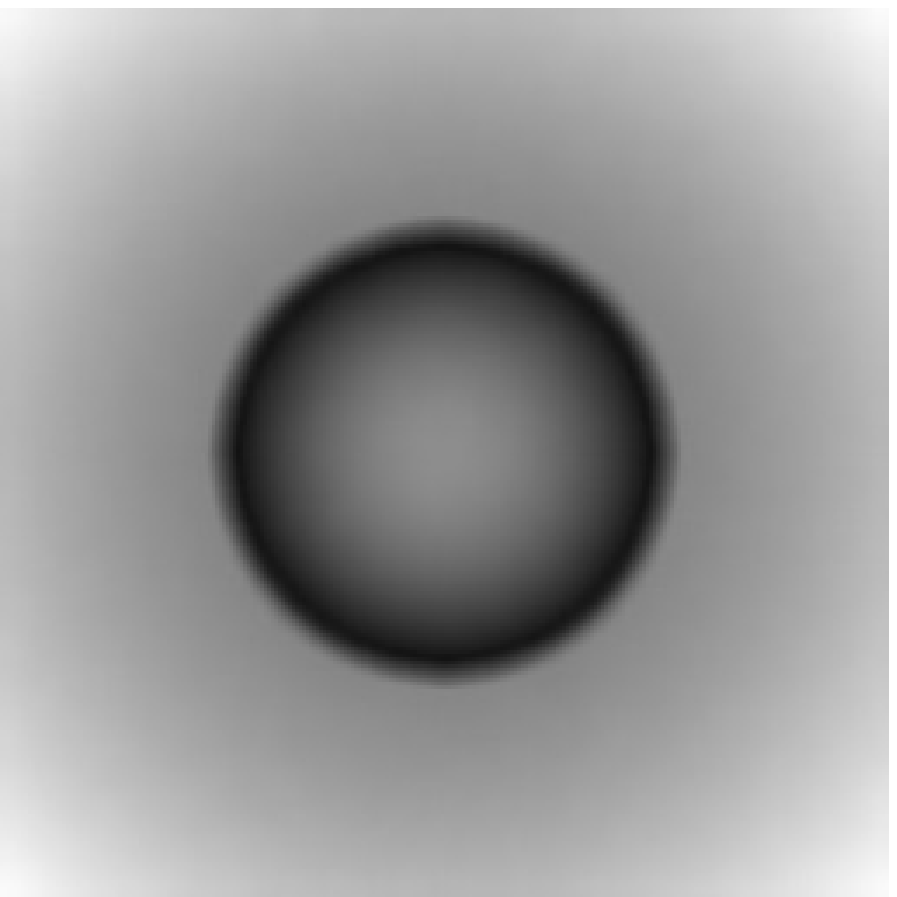}  \subfigure[\currentcaption]{\includegraphics[width=0.233\linewidth]{\currentname}}\hfill
\input{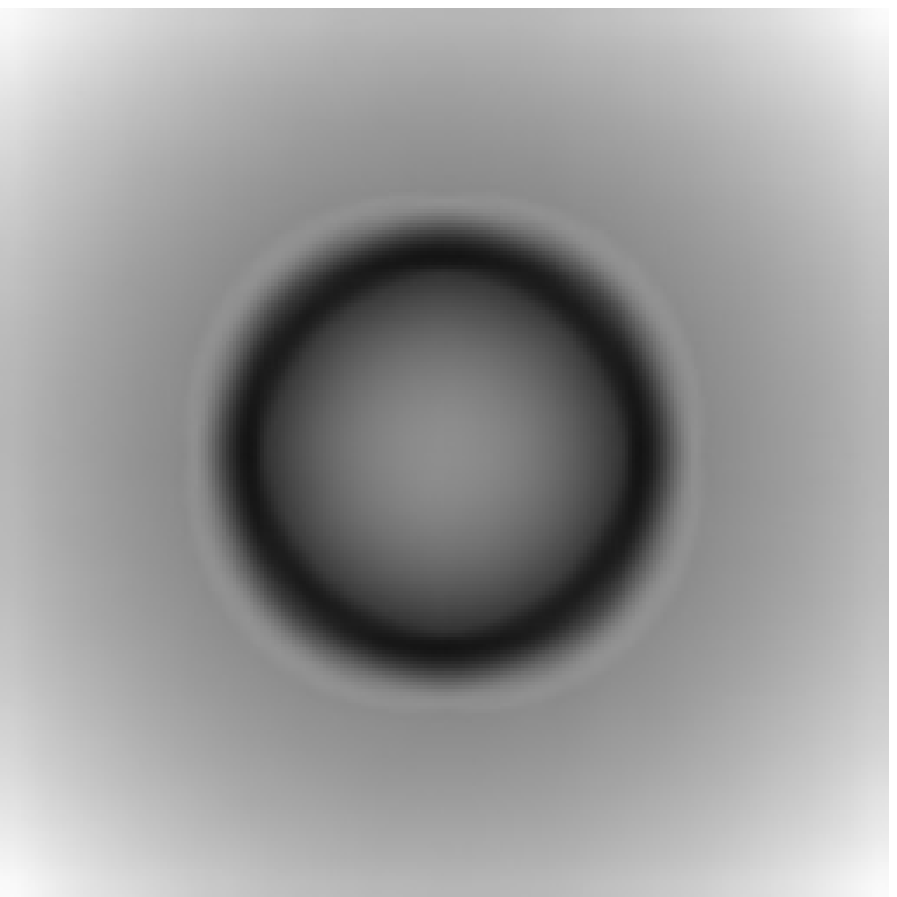}  \subfigure[\currentcaption]{\includegraphics[width=0.233\linewidth]{\currentname}}\hfill
\input{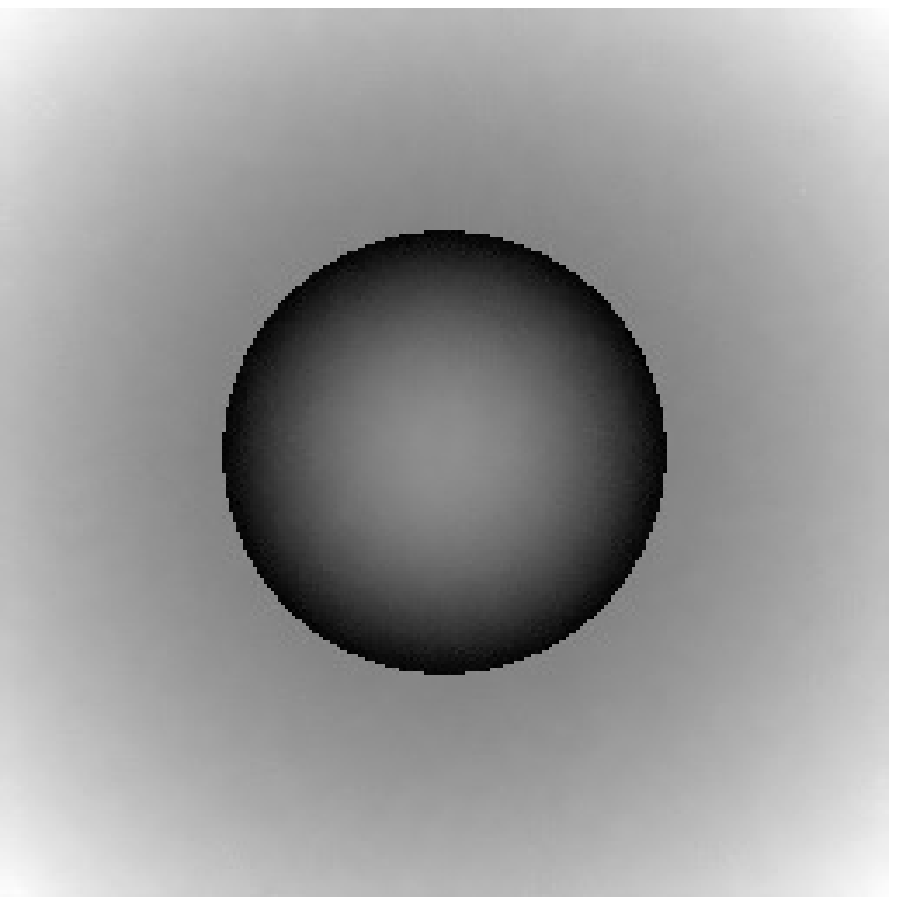}  \subfigure[\currentcaption]{\includegraphics[width=0.233\linewidth]{\currentname}}\hfill
\input{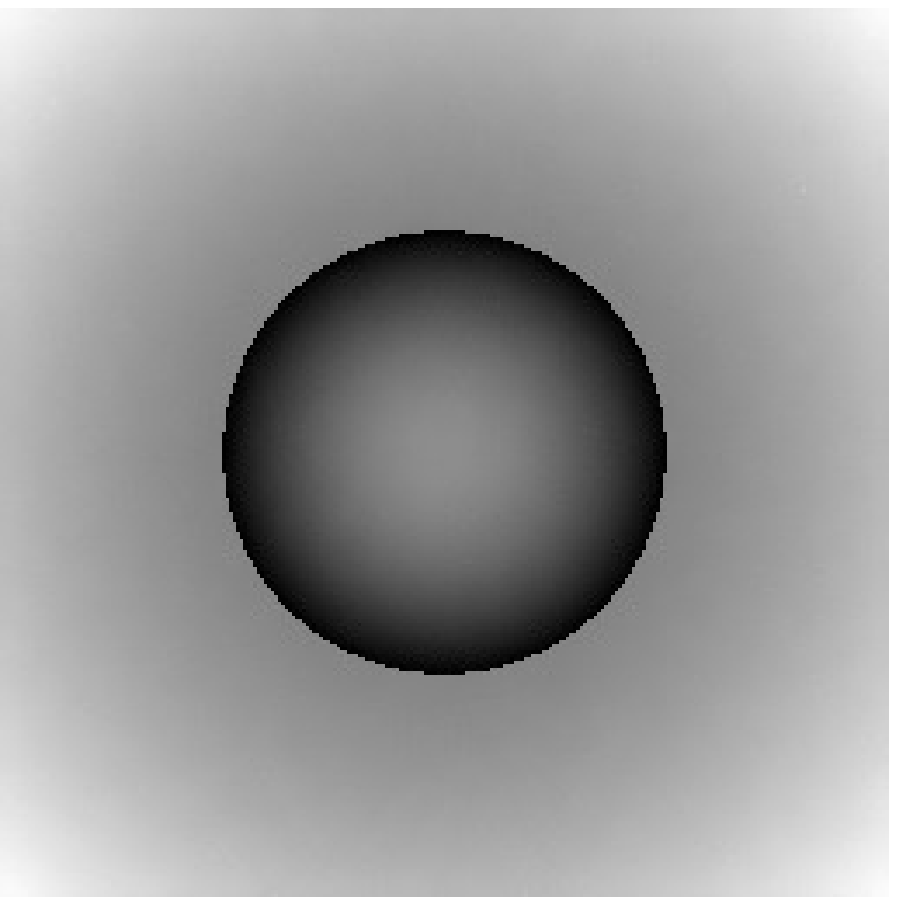}  \subfigure[\currentcaption]{\includegraphics[width=0.233\linewidth]{\currentname}}\hfill
\input{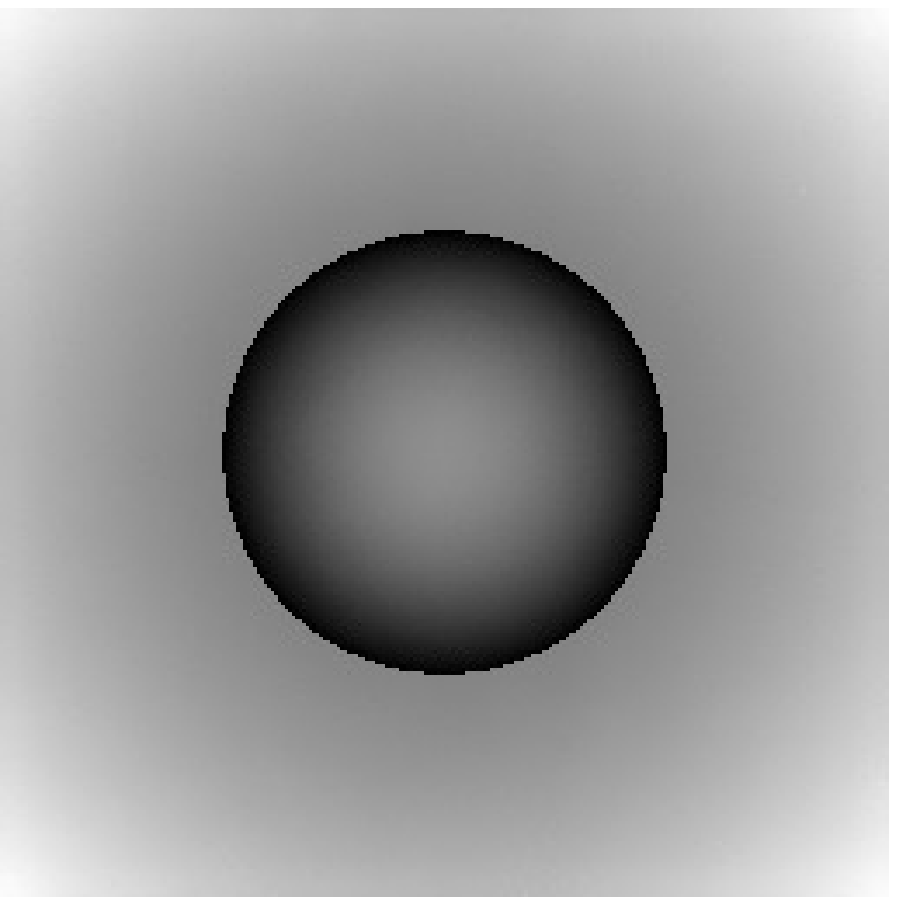}  \subfigure[\currentcaption]{\includegraphics[width=0.233\linewidth]{\currentname}}\hfill
\input{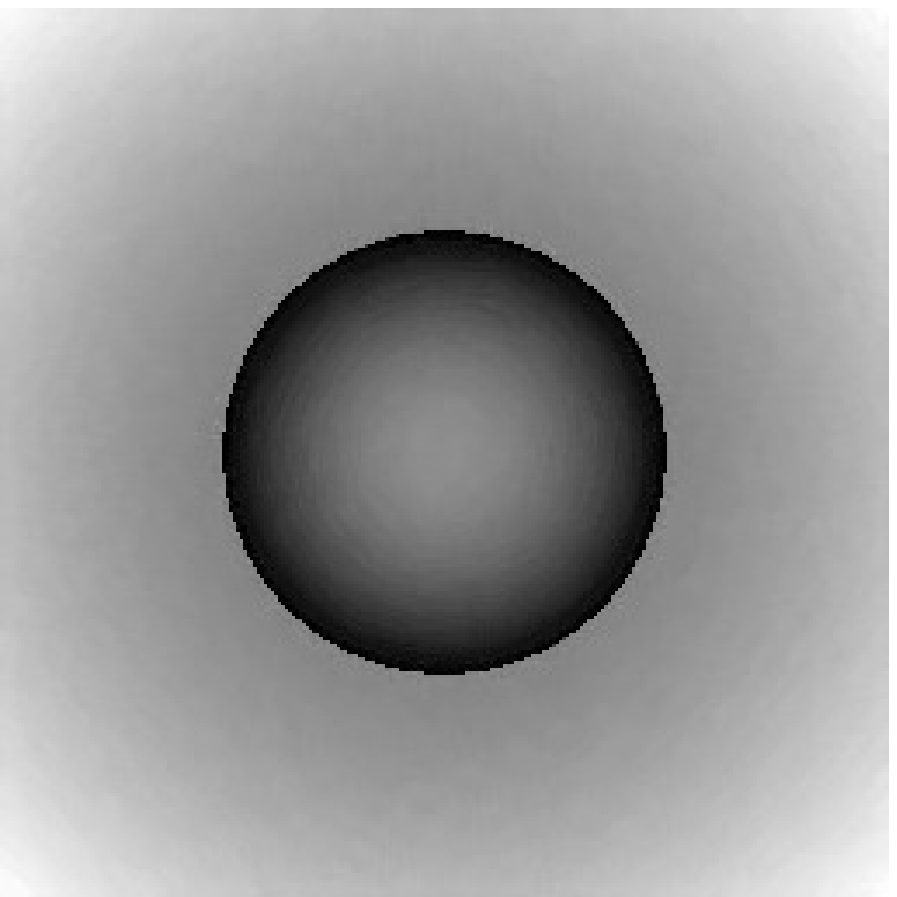}  \subfigure[\currentcaption]{\includegraphics[width=0.233\linewidth]{\currentname}}\hfill
\input{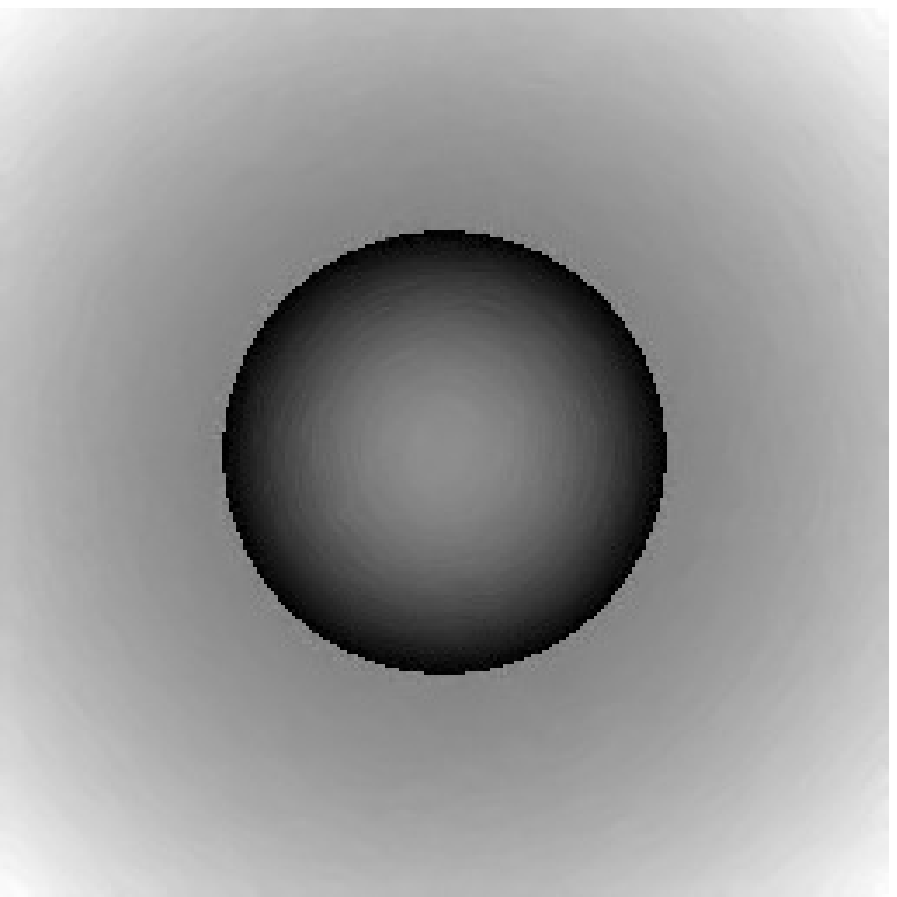}  \subfigure[\currentcaption]{\includegraphics[width=0.233\linewidth]{\currentname}}\hfill
\input{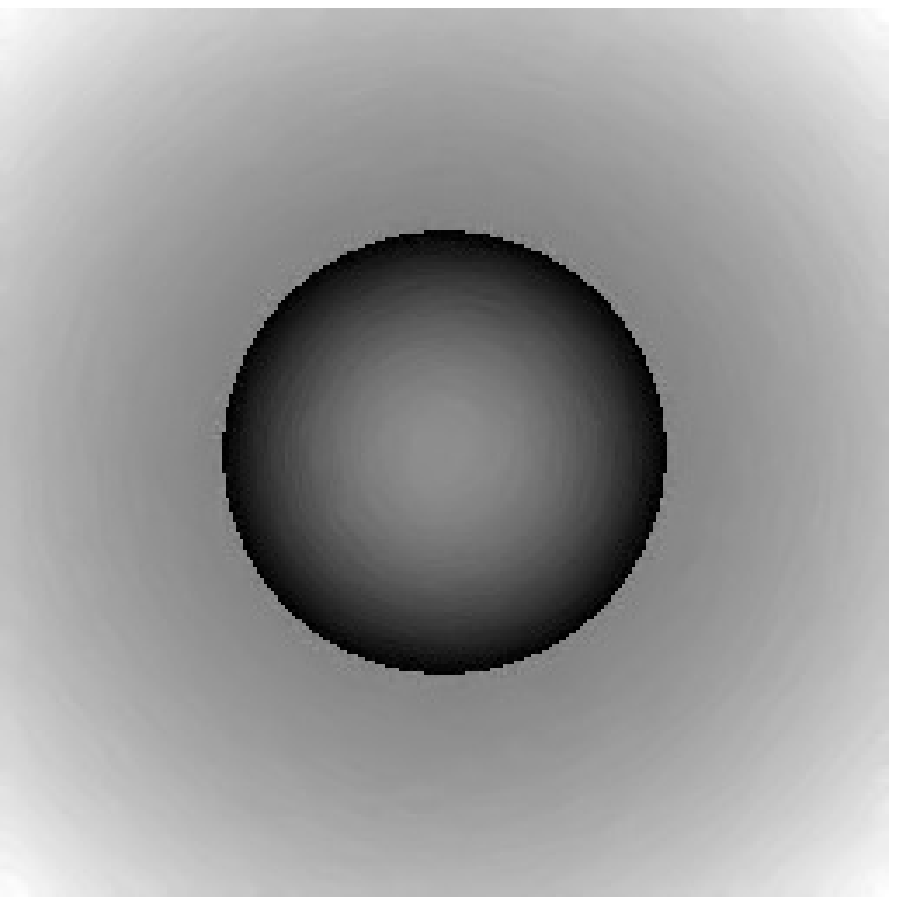}  \subfigure[\currentcaption]{\includegraphics[width=0.233\linewidth]{\currentname}}\hfill
\input{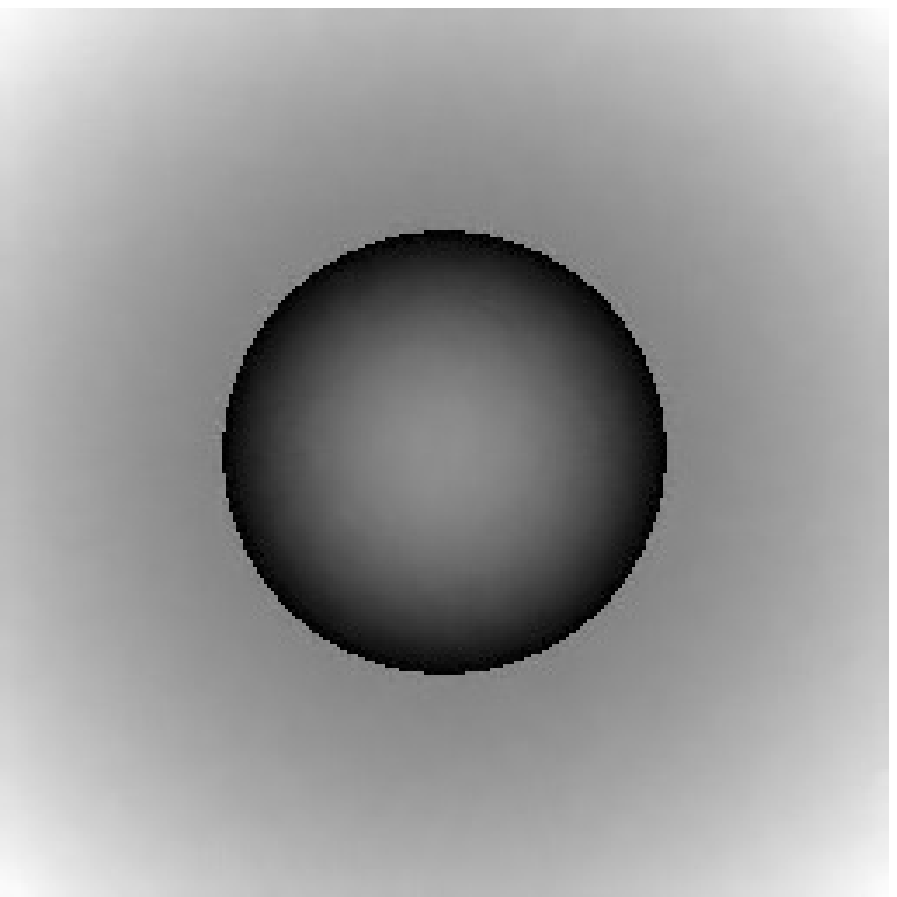}  \subfigure[\currentcaption]{\includegraphics[width=0.233\linewidth]{\currentname}}\hfill
\input{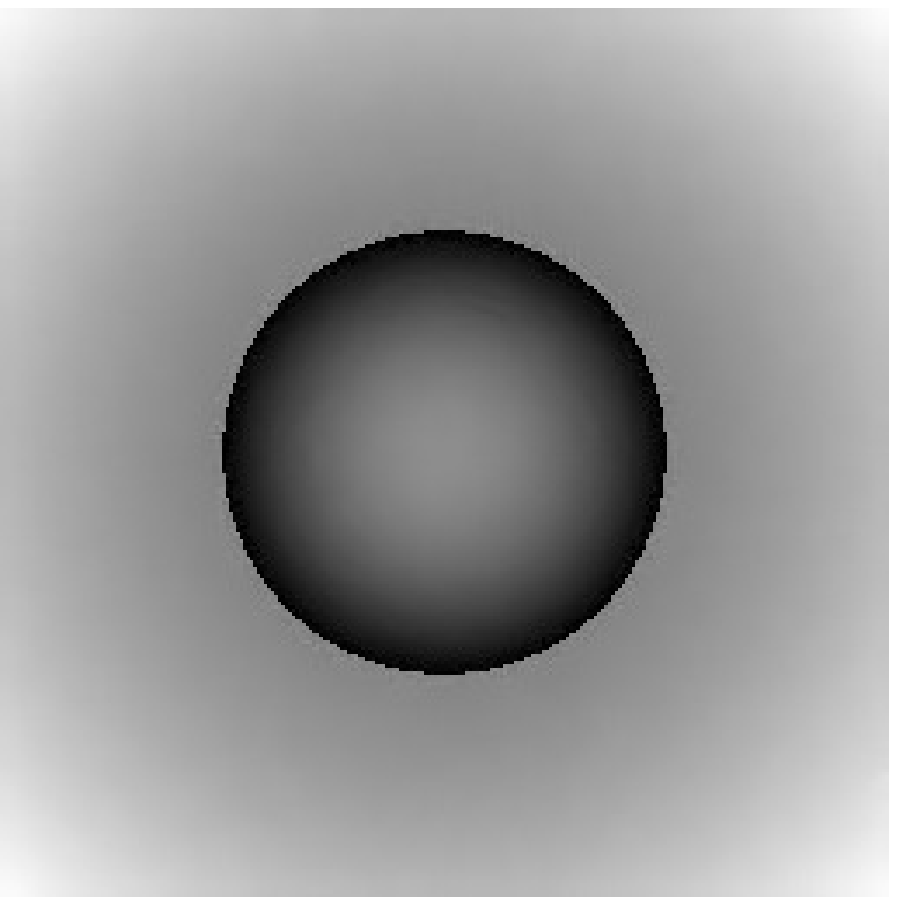}  \subfigure[\currentcaption]{\includegraphics[width=0.233\linewidth]{\currentname}}\hfill
\input{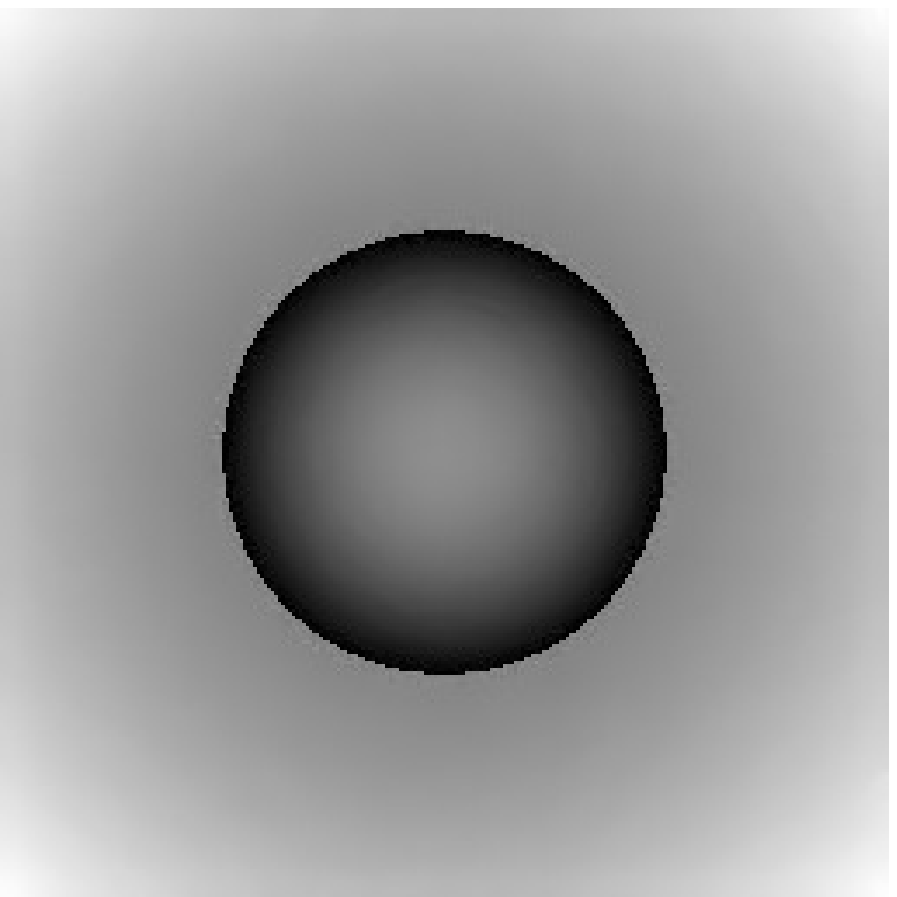}  \subfigure[\currentcaption]{\includegraphics[width=0.233\linewidth]{\currentname}}\hfill
\input{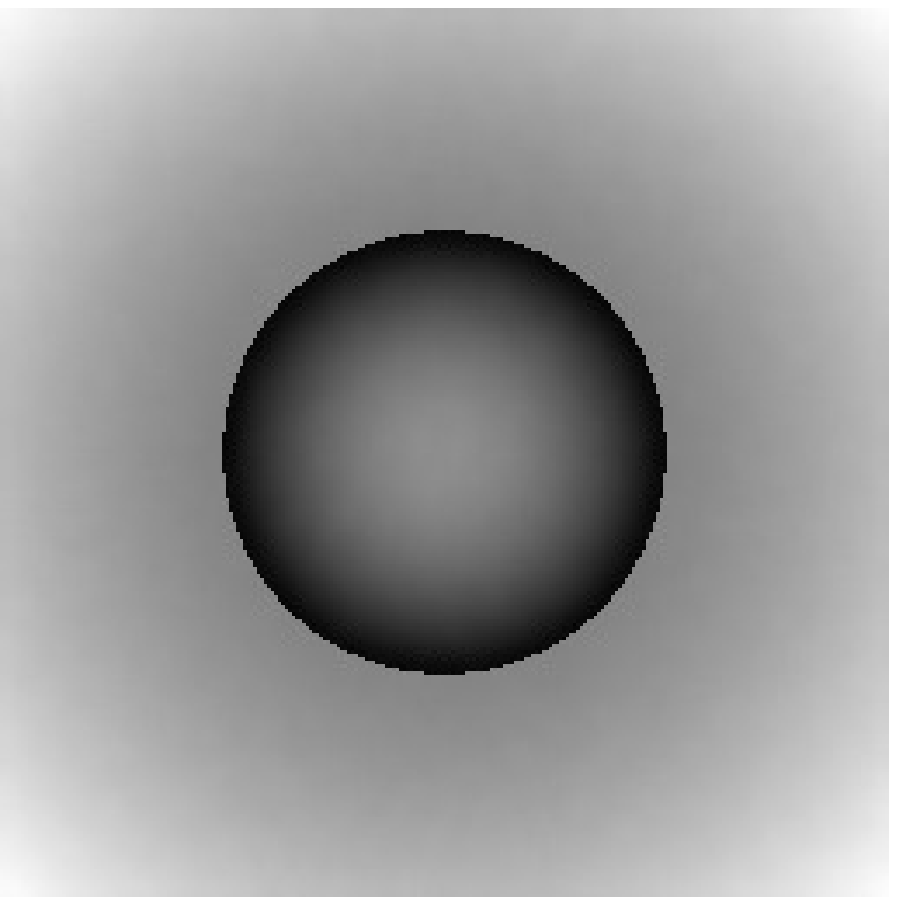}  \subfigure[\currentcaption]{\includegraphics[width=0.233\linewidth]{\currentname}}\hfill
\input{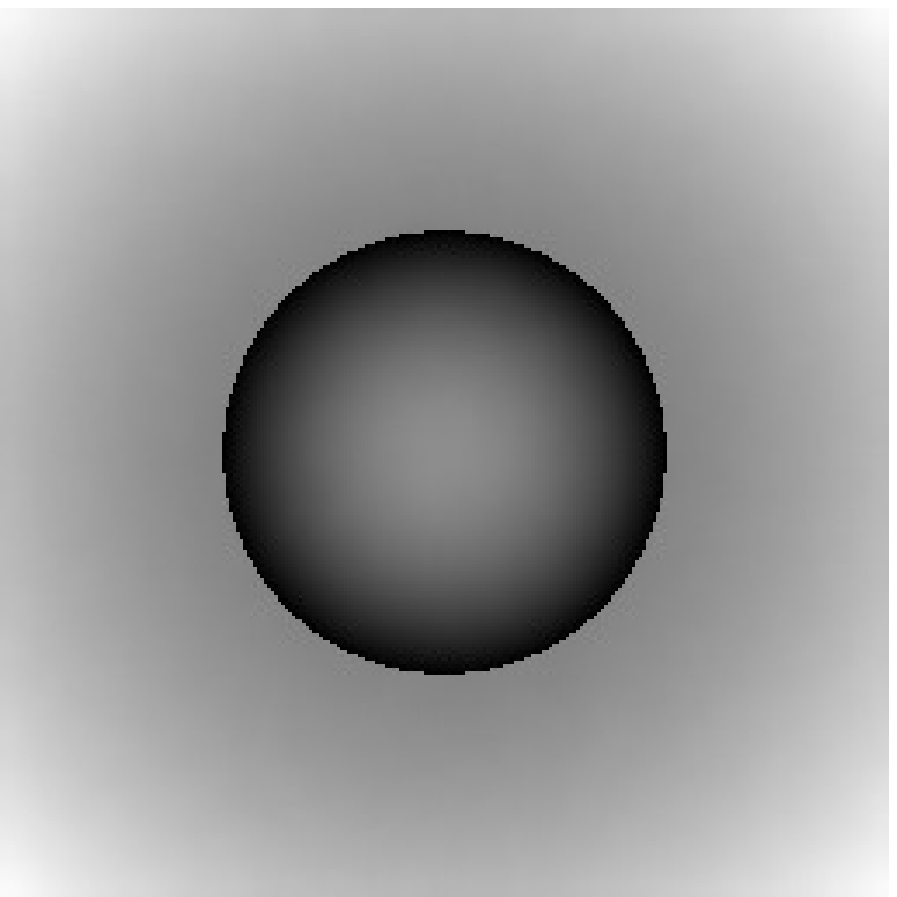}  \subfigure[\currentcaption]{\includegraphics[width=0.233\linewidth]{\currentname}}\hfill 
\input{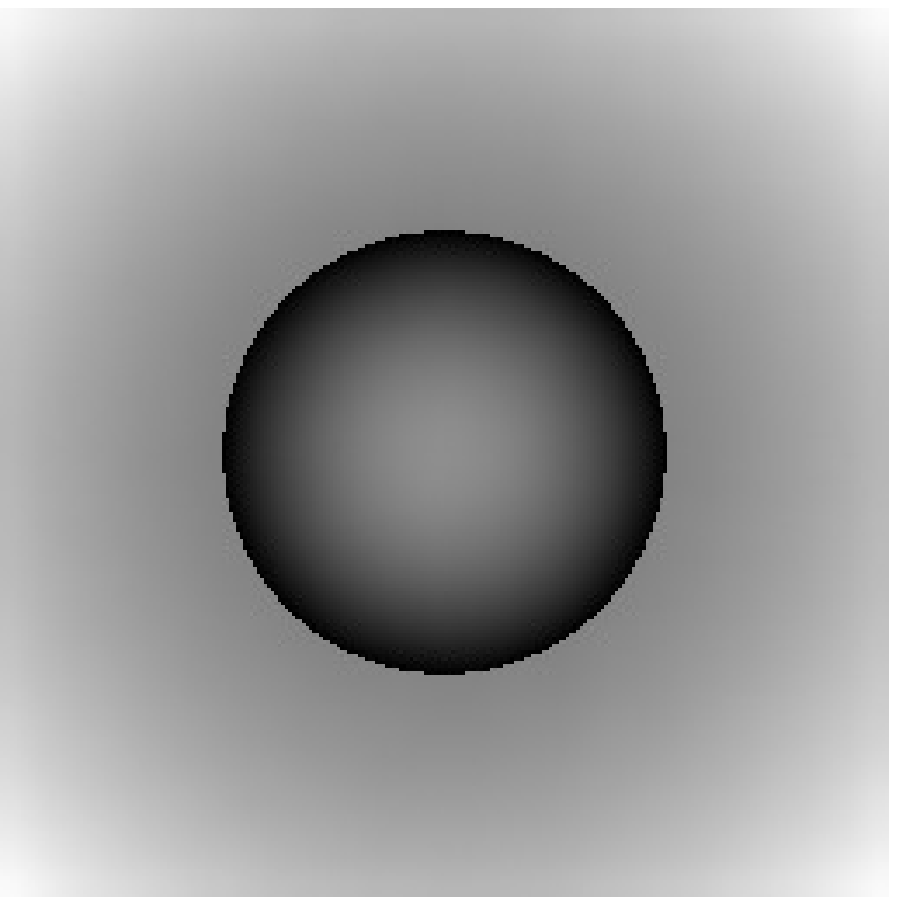}  \subfigure[\currentcaption]{\includegraphics[width=0.233\linewidth]{\currentname}}
\caption{Toy cartoon image (Bowl)  corrupted Gaussian noise with $\sigma=5$.  }
\label{fig:Bowl_sigma=5}
\end{figure}

\begin{figure}[hbt!]
\centering
\input{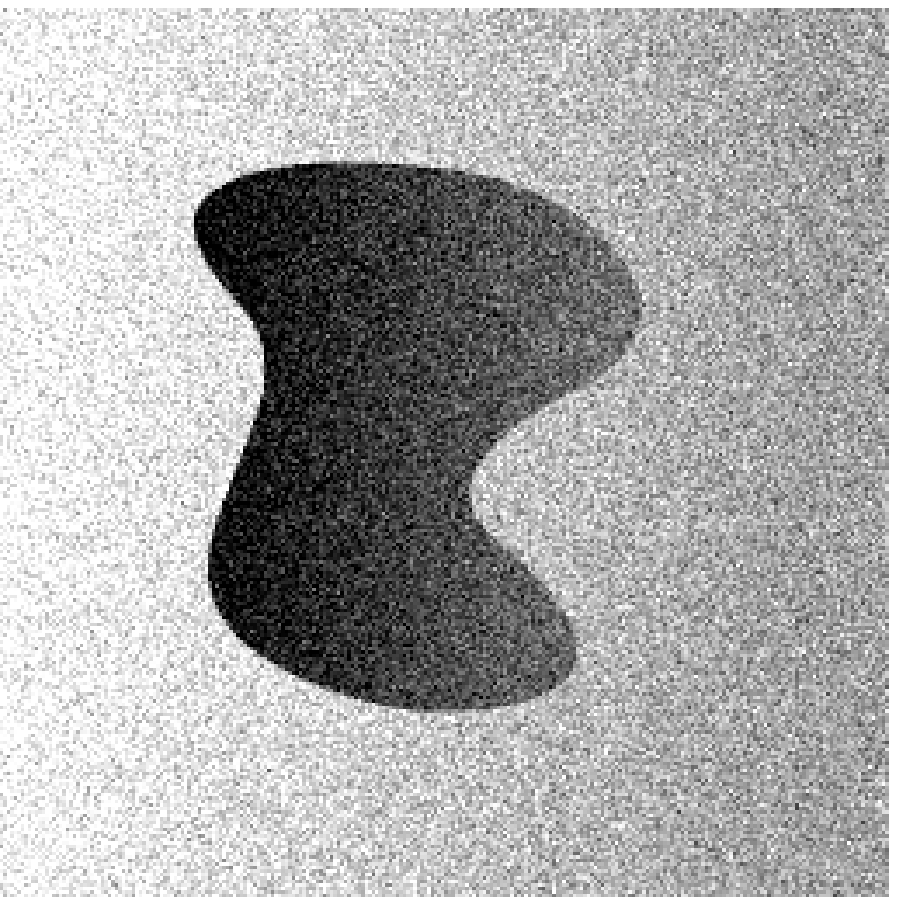}  \subfigure[\currentcaption]{\includegraphics[width=0.233\linewidth]{\currentname}}\hfill
\input{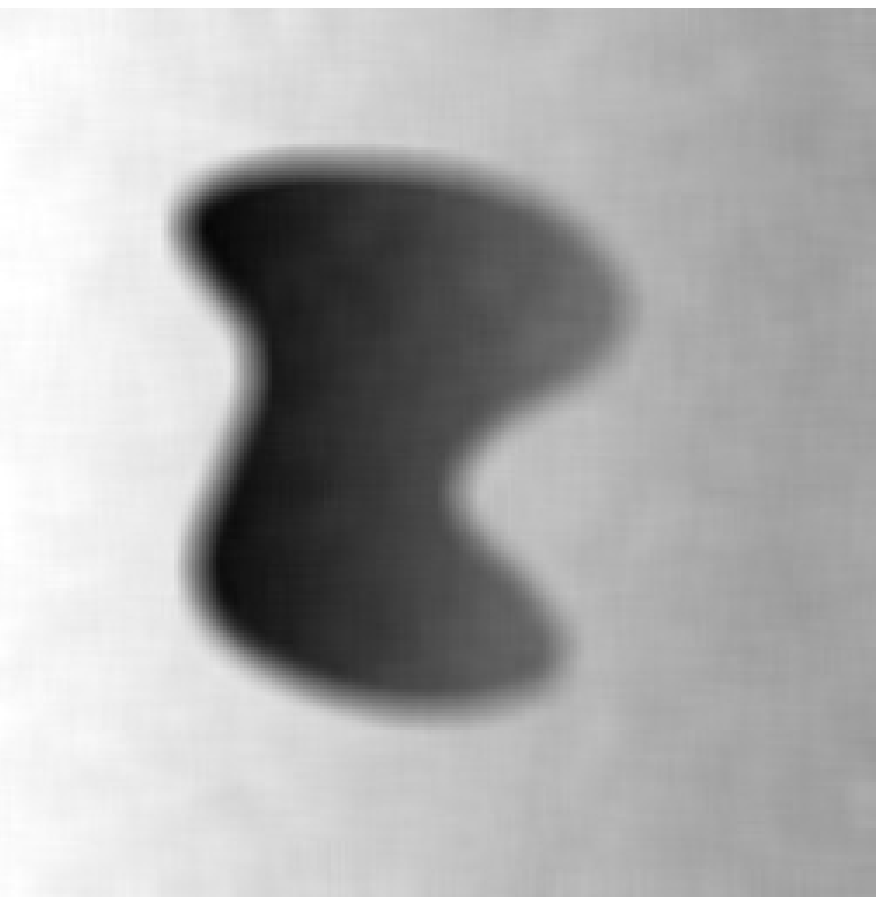}  \subfigure[\currentcaption]{\includegraphics[width=0.233\linewidth]{\currentname}}\hfill
\input{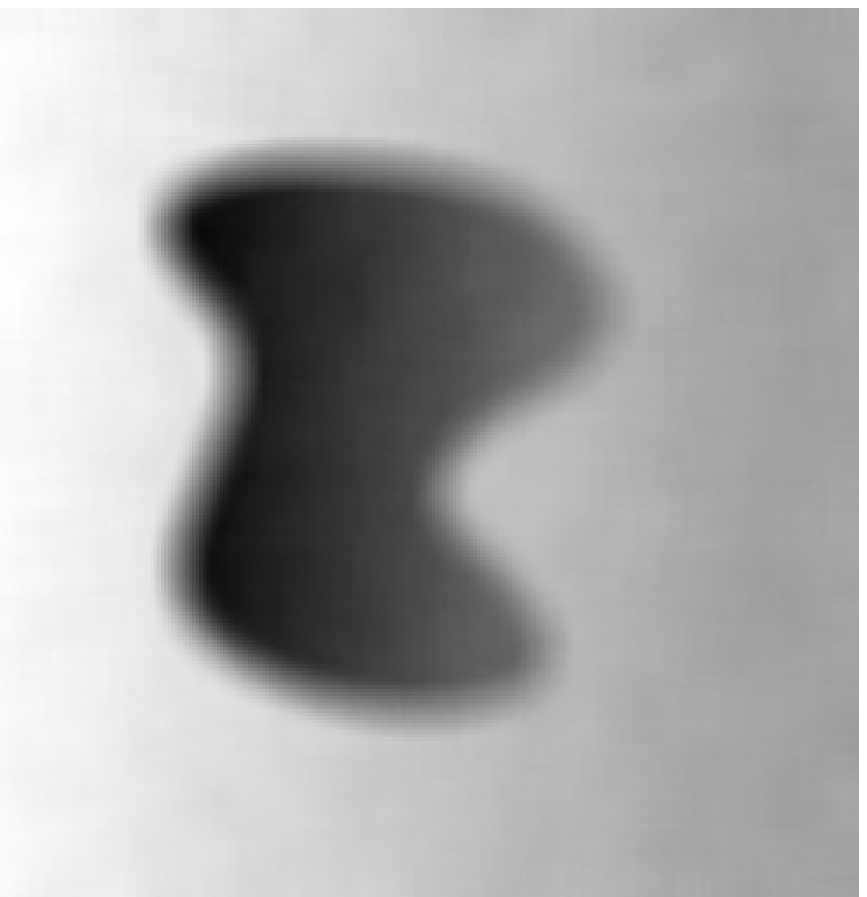}  \subfigure[\currentcaption]{\includegraphics[width=0.233\linewidth]{\currentname}}\hfill
\input{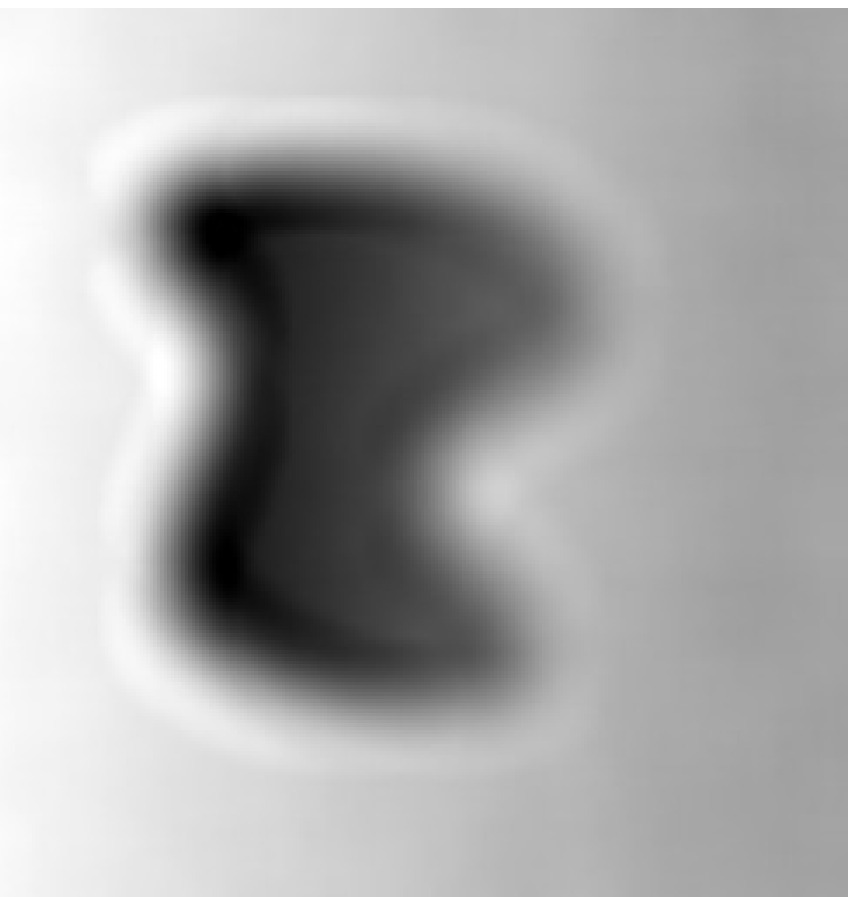}  \subfigure[\currentcaption]{\includegraphics[width=0.233\linewidth]{\currentname}}\hfill
\input{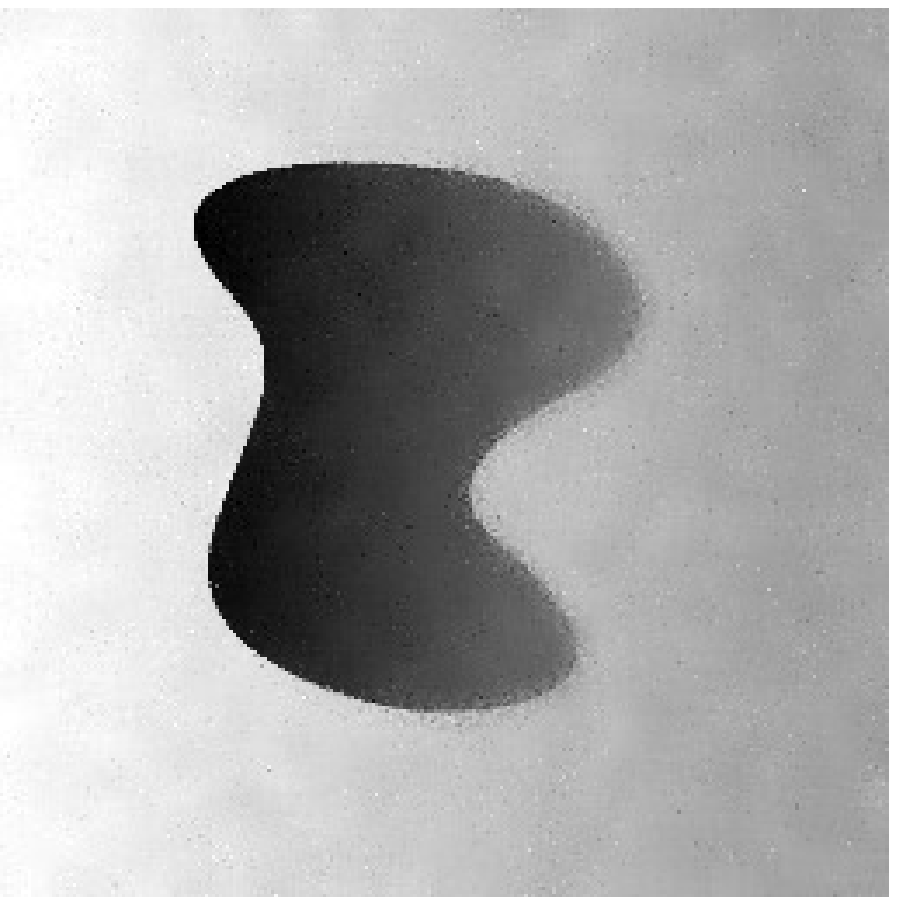}  \subfigure[\currentcaption]{\includegraphics[width=0.233\linewidth]{\currentname}}\hfill
\input{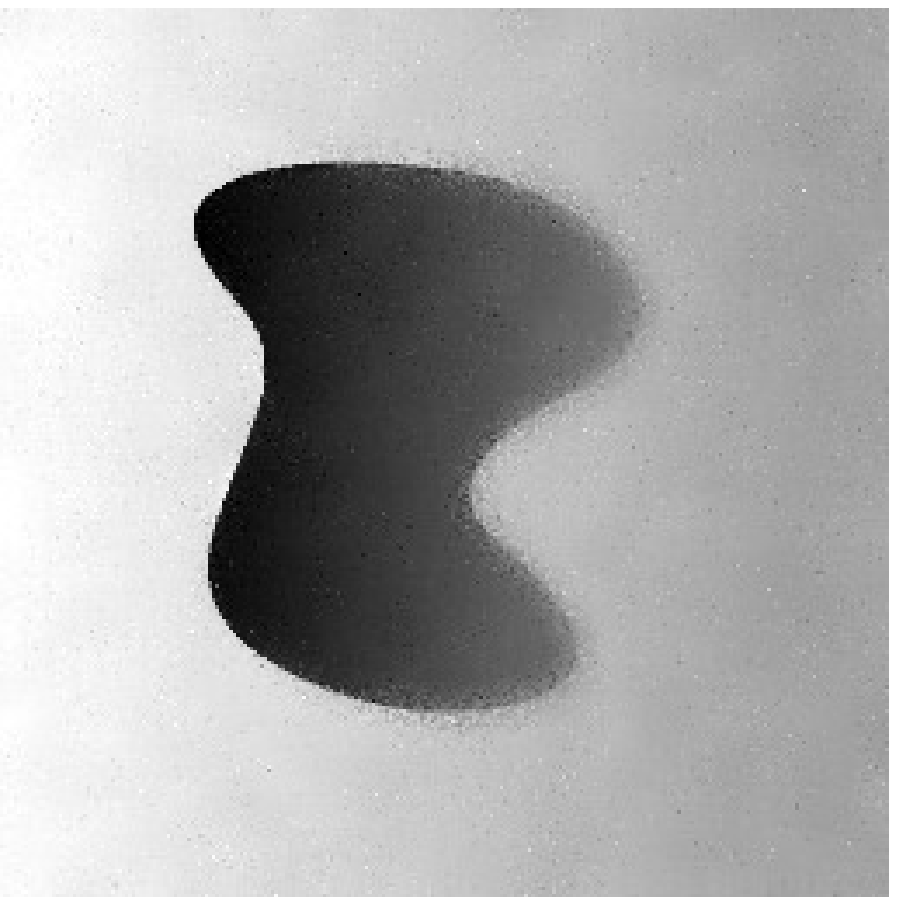}  \subfigure[\currentcaption]{\includegraphics[width=0.233\linewidth]{\currentname}}\hfill
\input{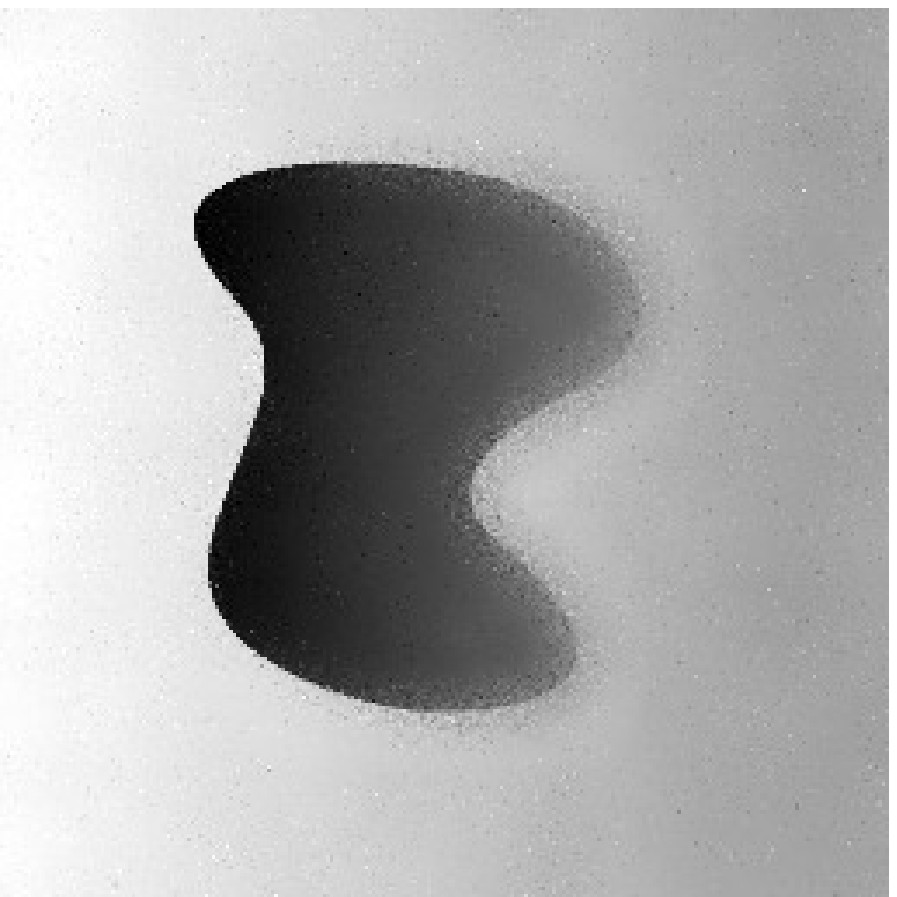}  \subfigure[\currentcaption]{\includegraphics[width=0.233\linewidth]{\currentname}}\hfill
\input{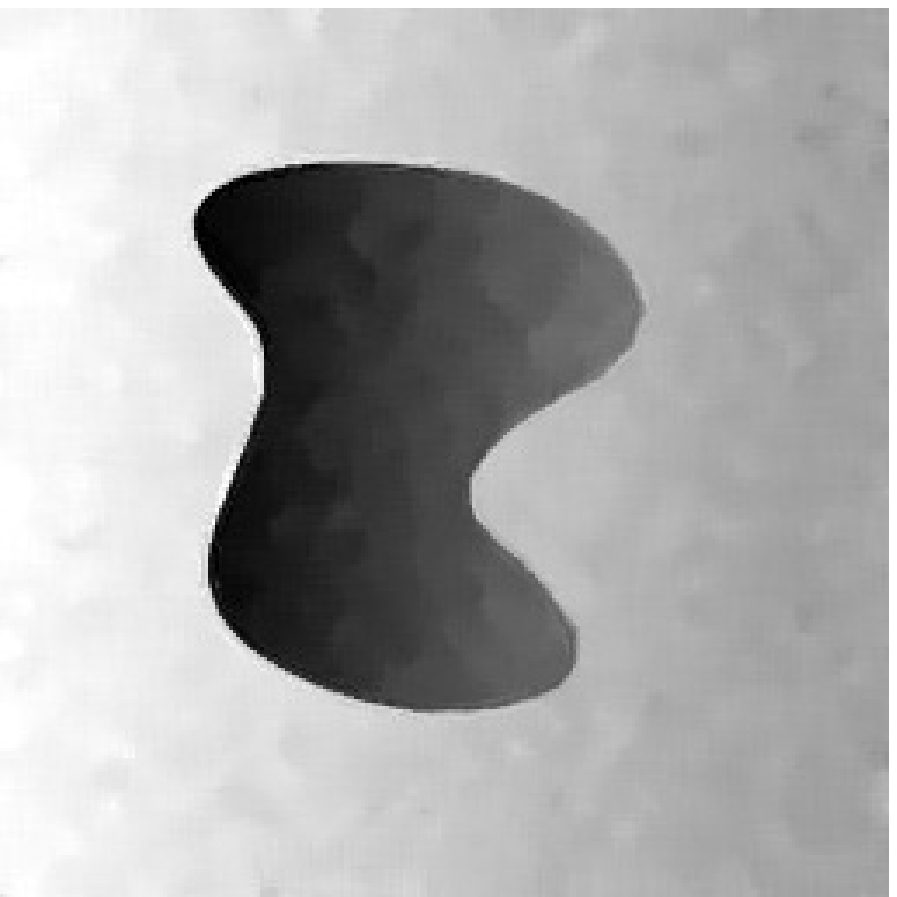}  \subfigure[\currentcaption]{\includegraphics[width=0.233\linewidth]{\currentname}}\hfill
\input{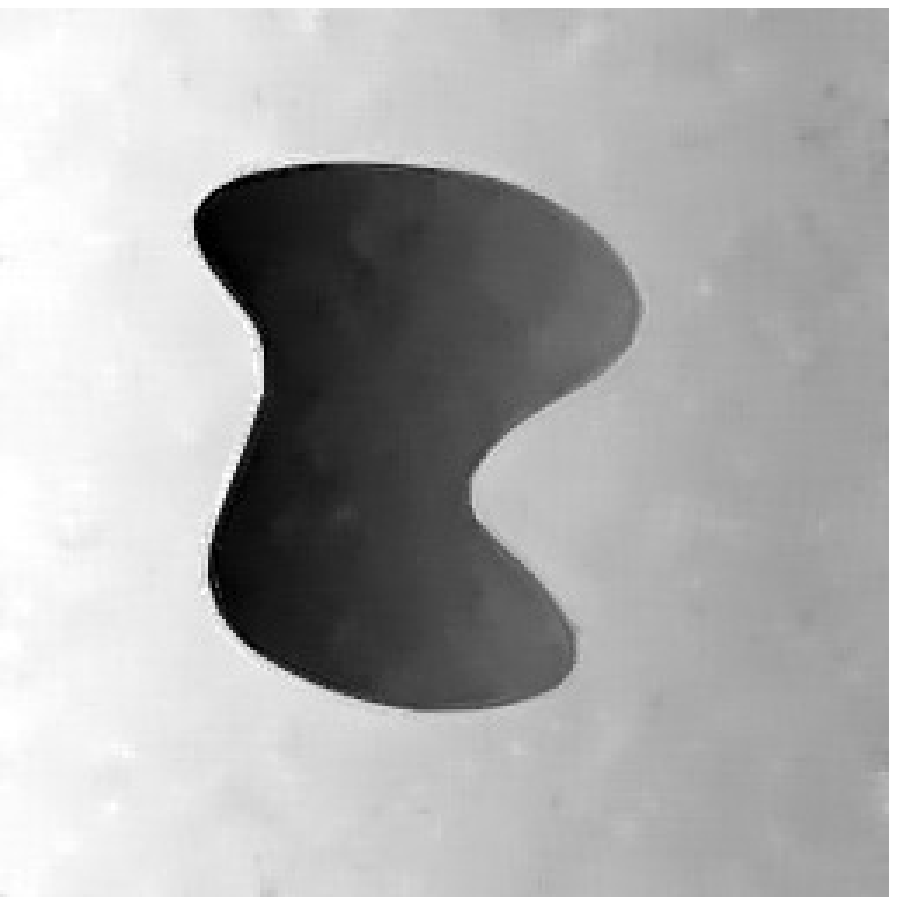}  \subfigure[\currentcaption]{\includegraphics[width=0.233\linewidth]{\currentname}}\hfill
\input{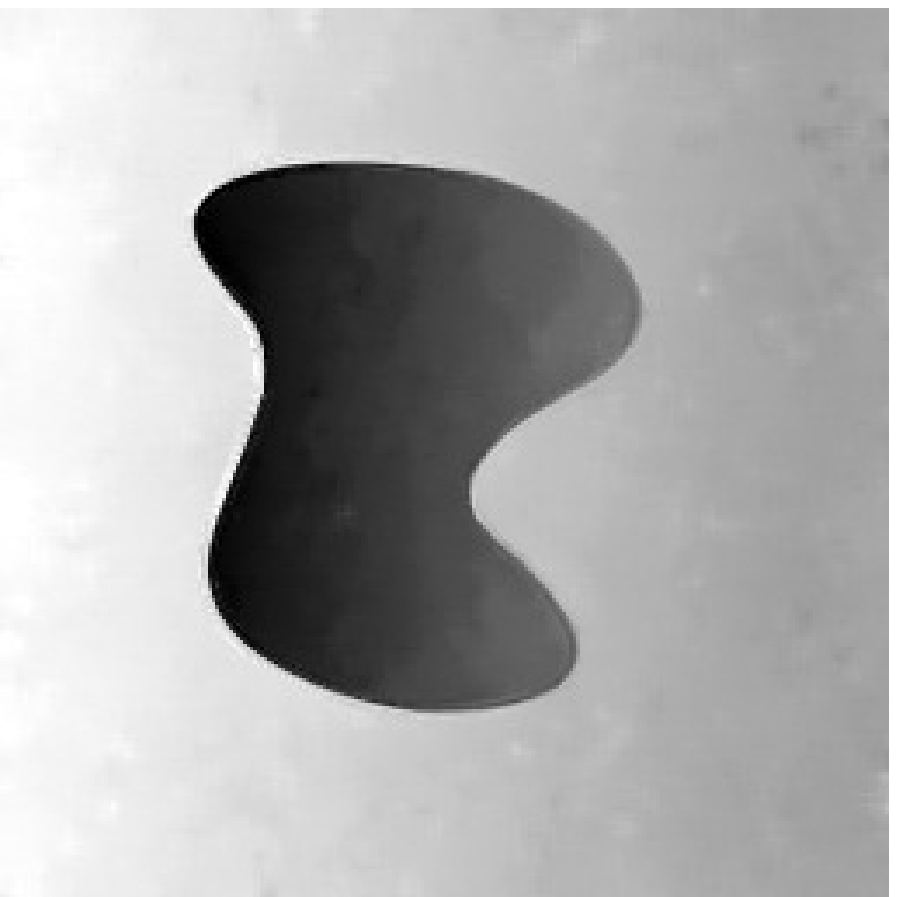}  \subfigure[\currentcaption]{\includegraphics[width=0.233\linewidth]{\currentname}}\hfill
\input{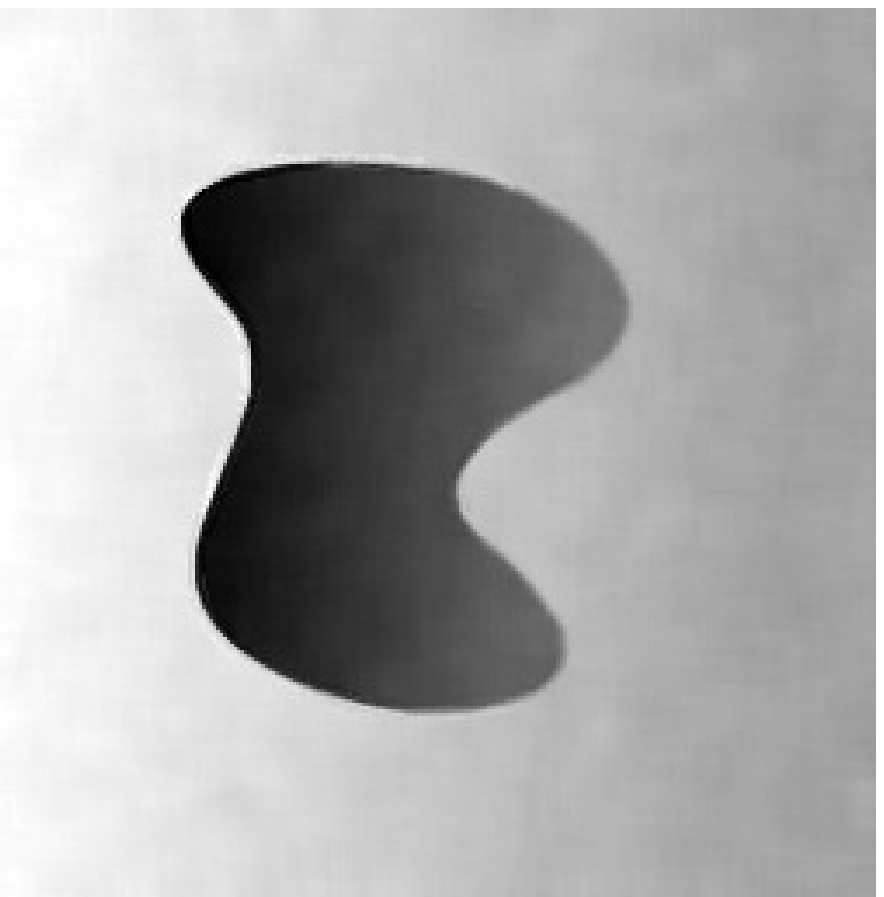}  \subfigure[\currentcaption]{\includegraphics[width=0.233\linewidth]{\currentname}}\hfill
\input{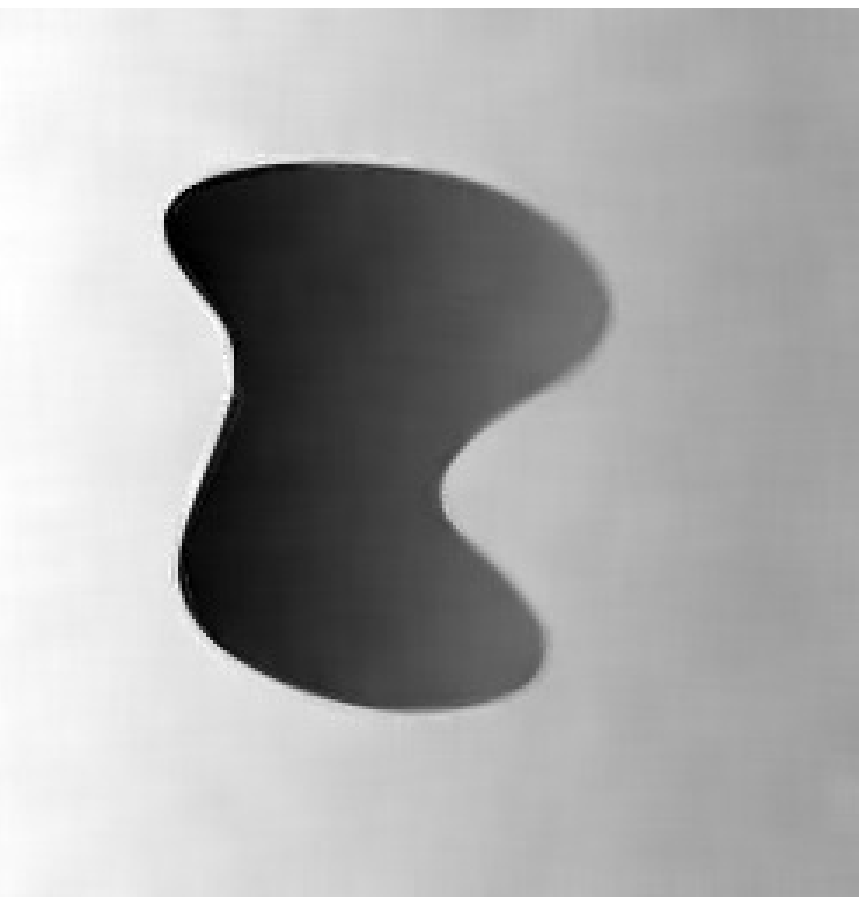}  \subfigure[\currentcaption]{\includegraphics[width=0.233\linewidth]{\currentname}}\hfill
\input{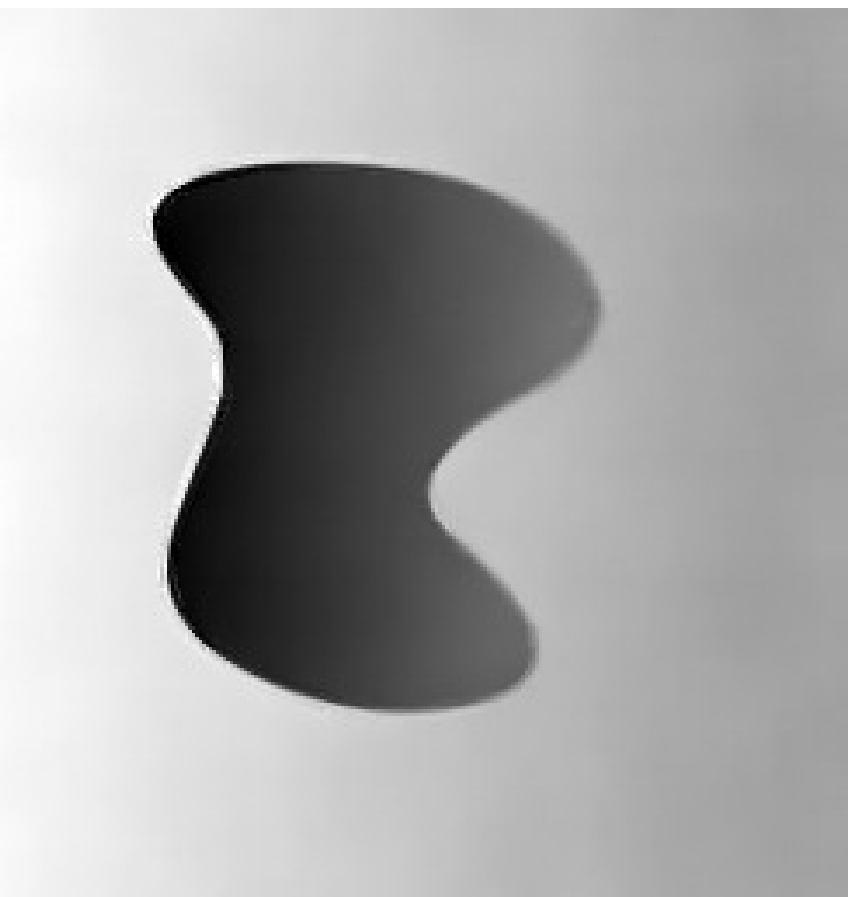}  \subfigure[\currentcaption]{\includegraphics[width=0.233\linewidth]{\currentname}}\hfill
\input{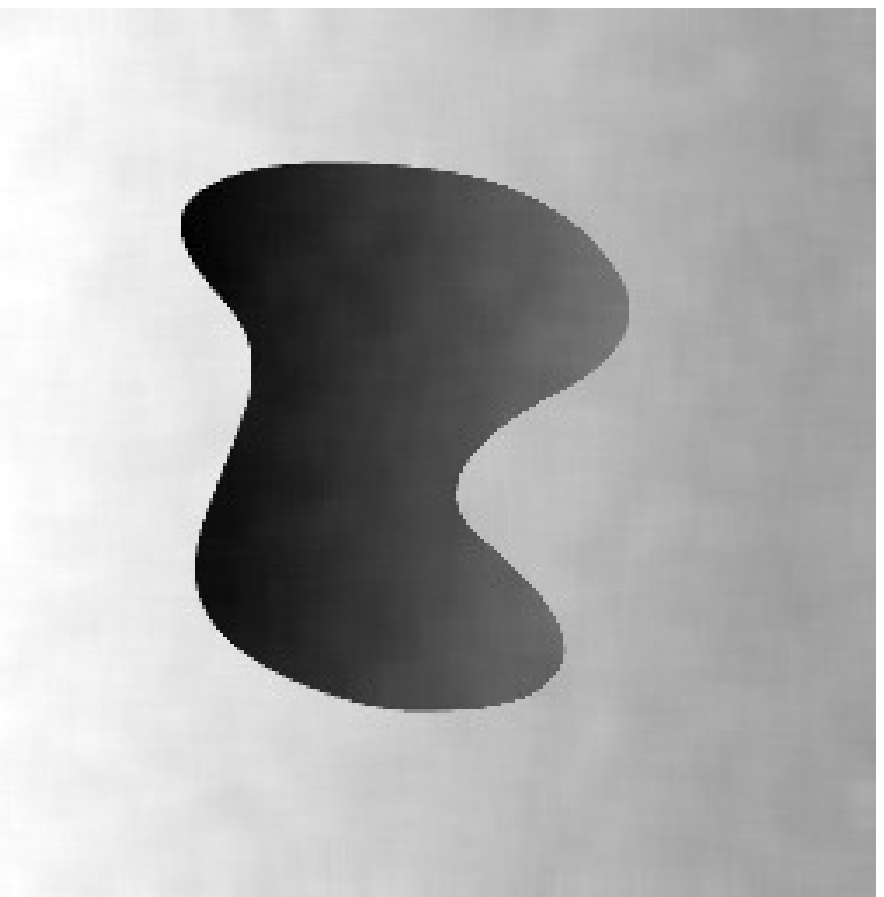}  \subfigure[\currentcaption]{\includegraphics[width=0.233\linewidth]{\currentname}}\hfill
\input{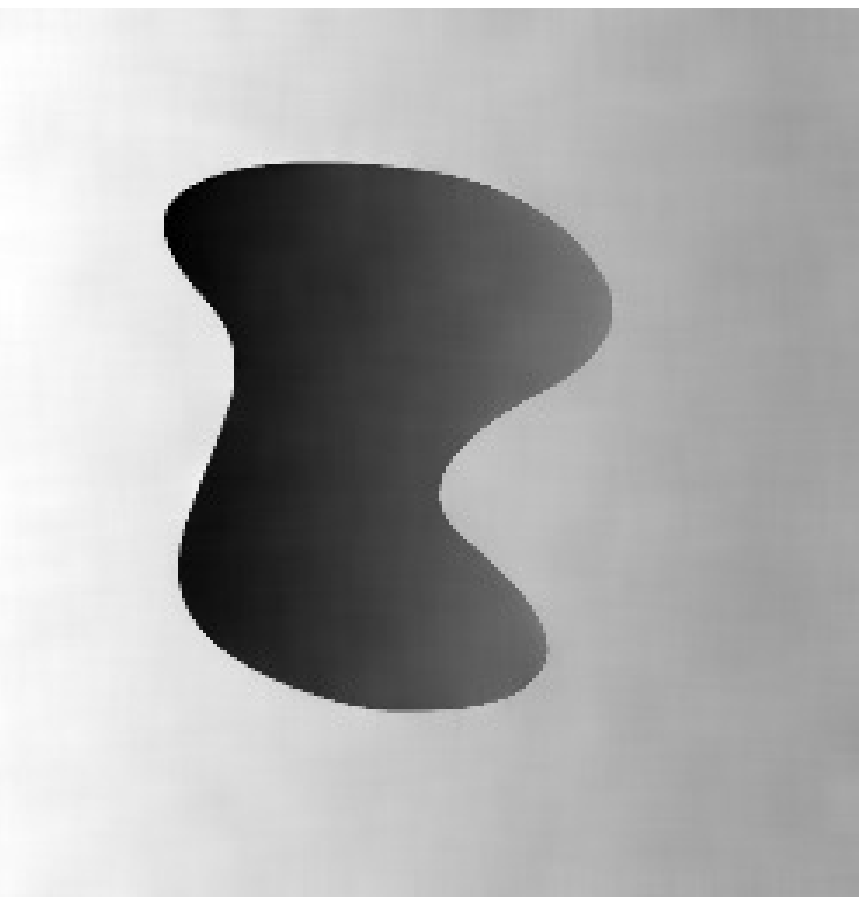}  \subfigure[\currentcaption]{\includegraphics[width=0.233\linewidth]{\currentname}}\hfill
\input{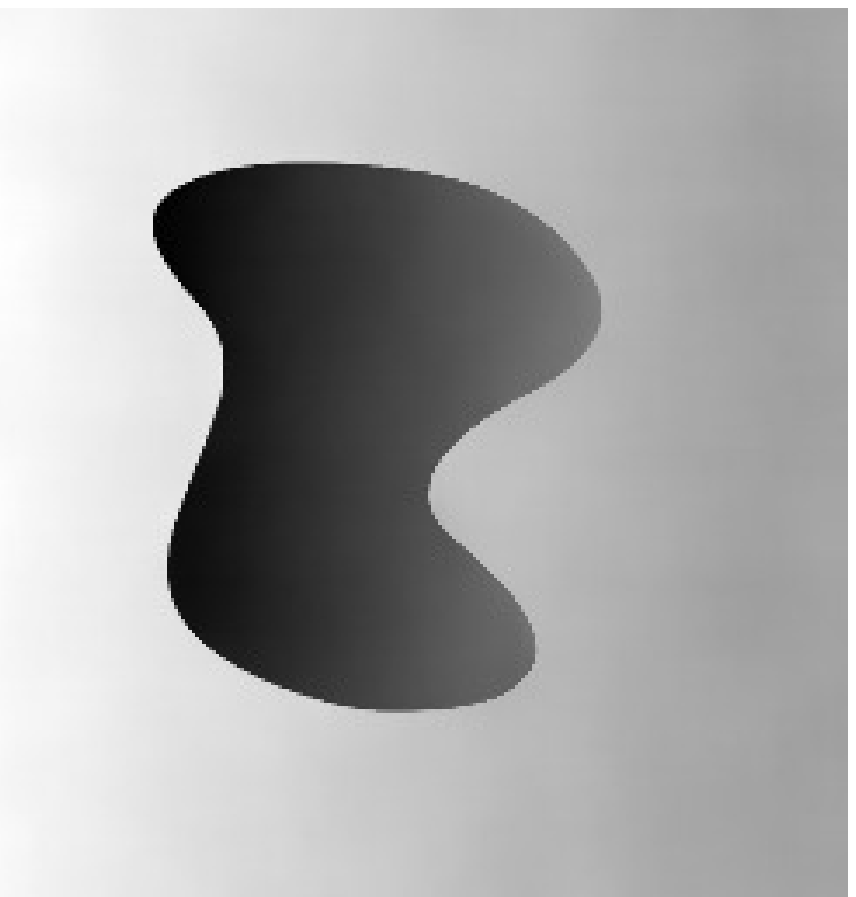}  \subfigure[\currentcaption]{\includegraphics[width=0.233\linewidth]{\currentname}}
\caption{Toy cartoon image (Blob)  corrupted Gaussian noise with $\sigma=20$. }
\label{fig:Bowl_sigma=20}
\end{figure}

\begin{figure}[hbt!]
\centering
\input{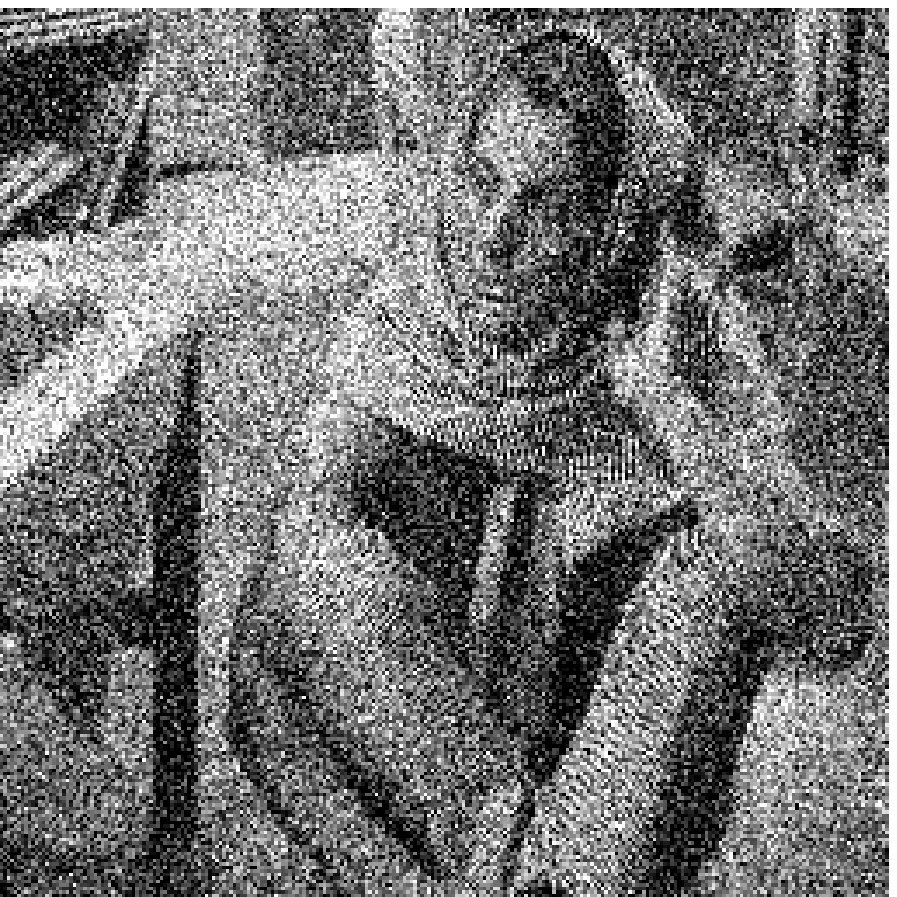}  \subfigure[\currentcaption]{\includegraphics[width=0.233\linewidth]{\currentname}}\hfill
\input{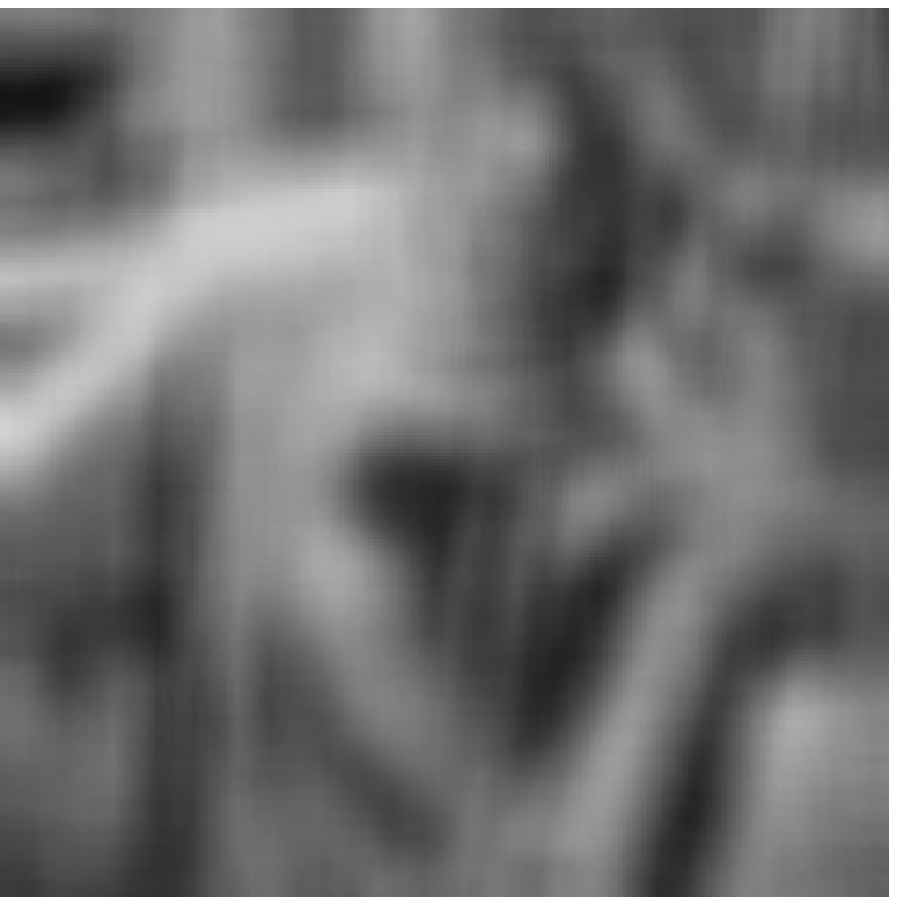}  \subfigure[\currentcaption]{\includegraphics[width=0.233\linewidth]{\currentname}}\hfill
\input{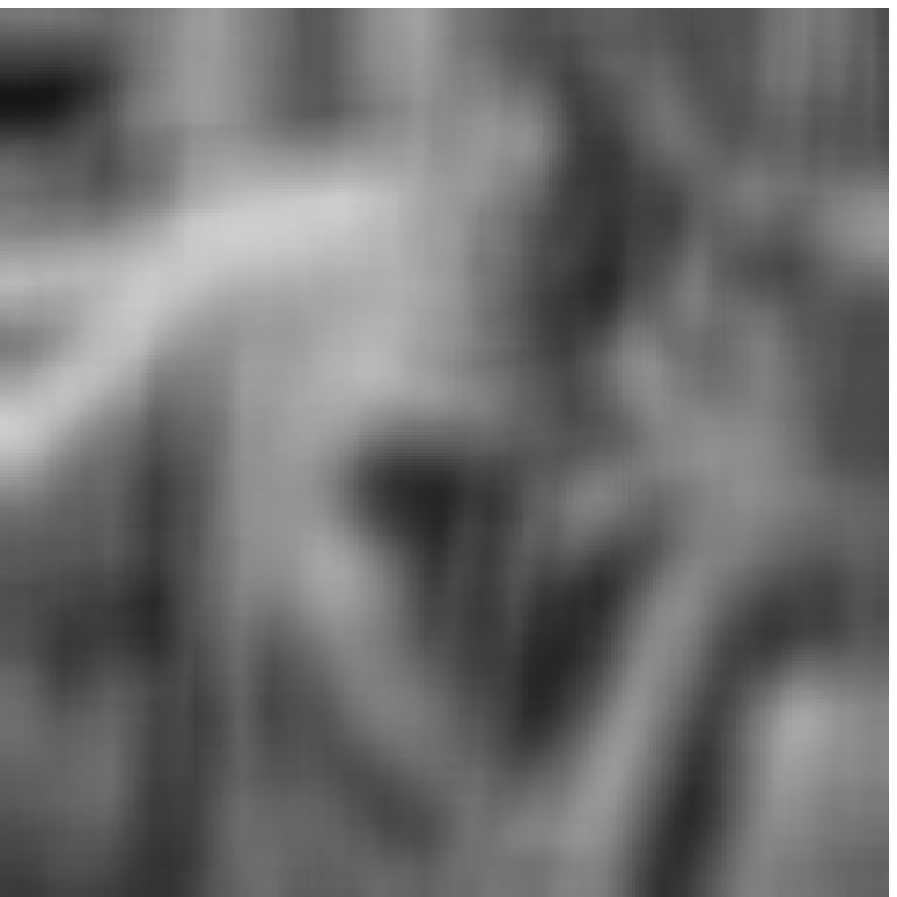}  \subfigure[\currentcaption]{\includegraphics[width=0.233\linewidth]{\currentname}}\hfill
\input{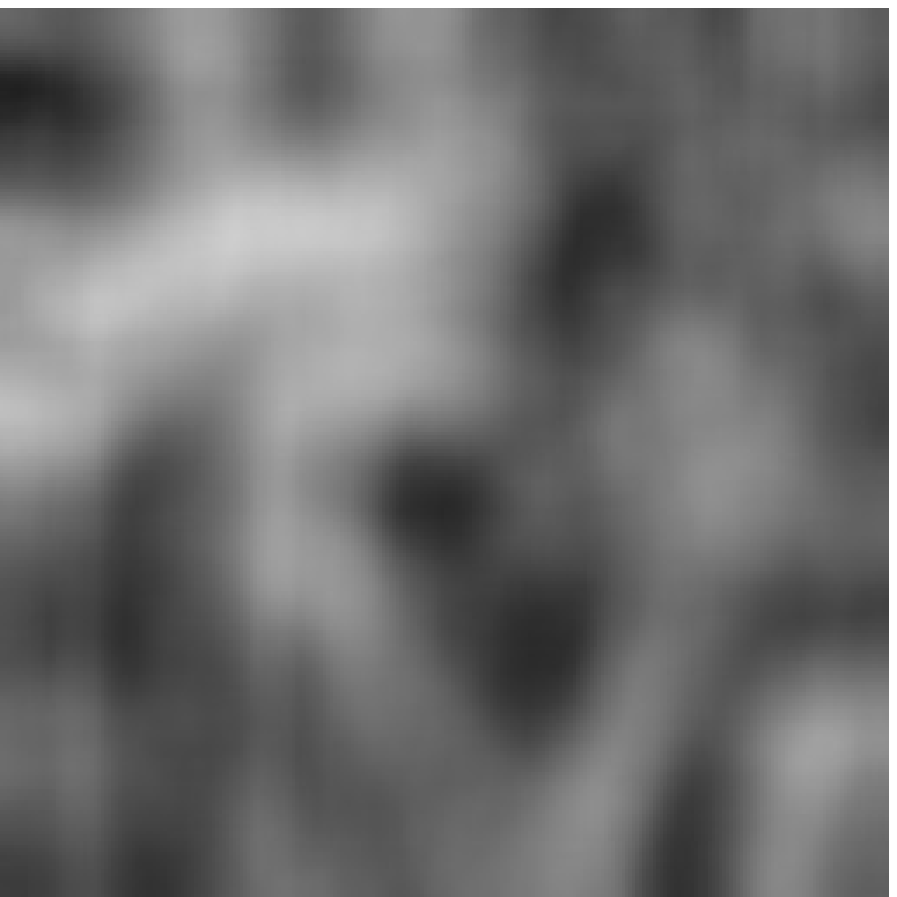}  \subfigure[\currentcaption]{\includegraphics[width=0.233\linewidth]{\currentname}}\hfill
\input{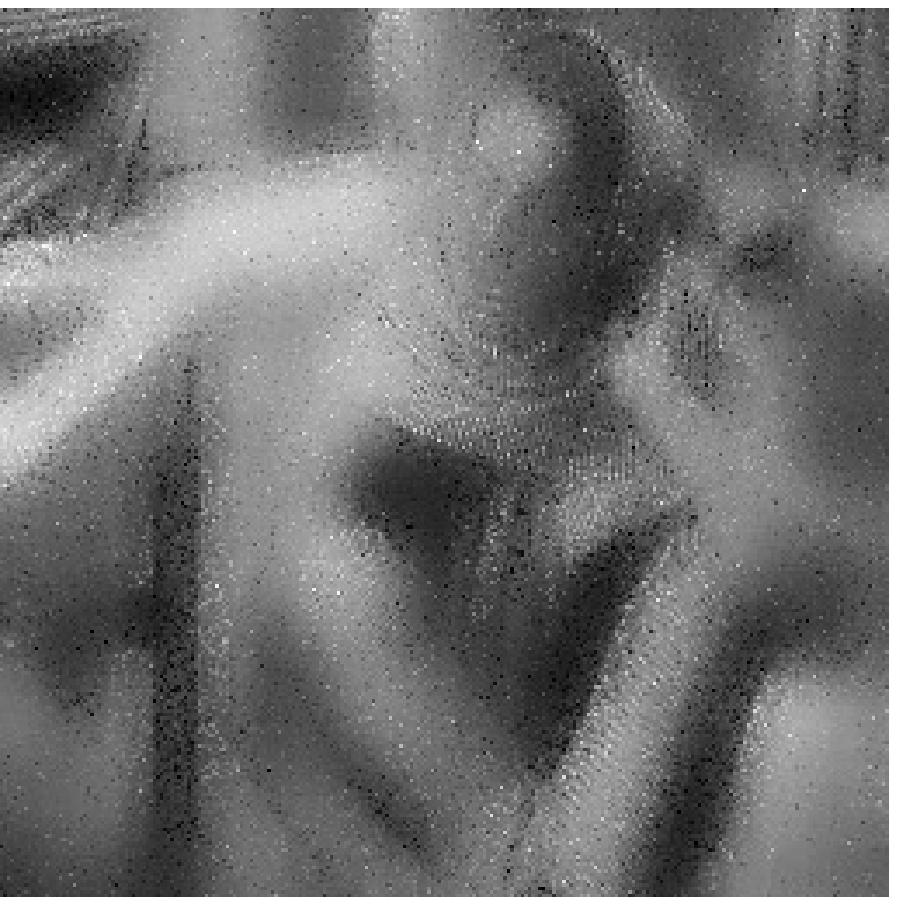}  \subfigure[\currentcaption]{\includegraphics[width=0.233\linewidth]{\currentname}}\hfill
\input{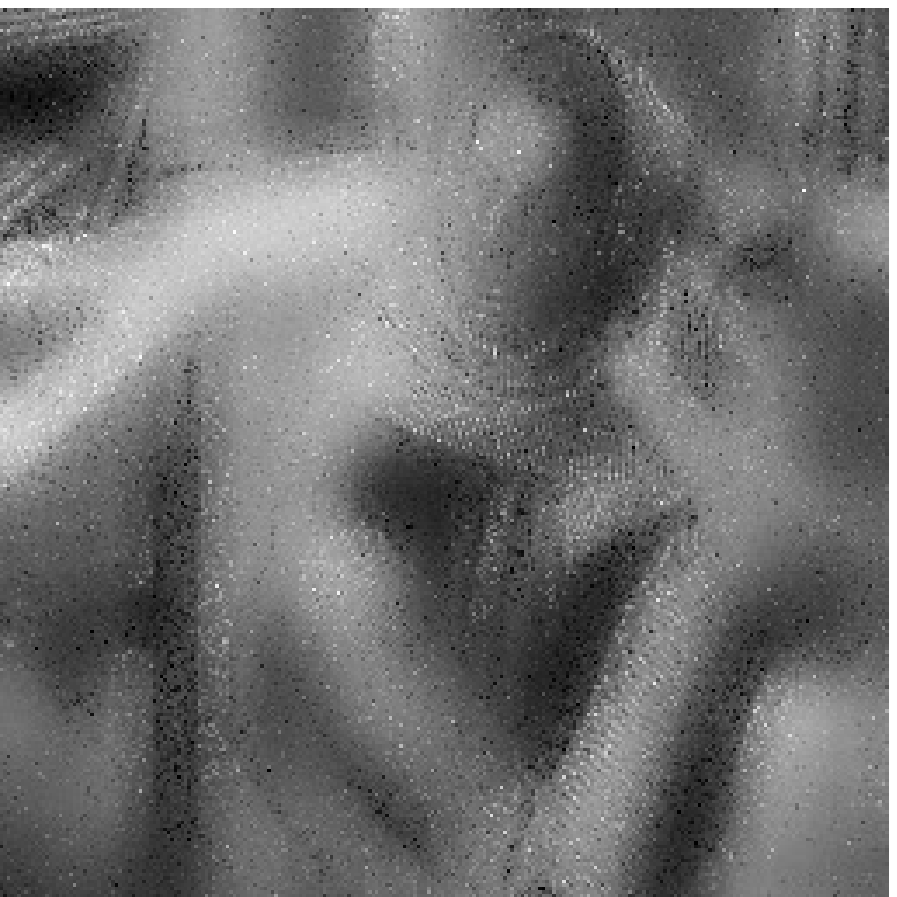}  \subfigure[\currentcaption]{\includegraphics[width=0.233\linewidth]{\currentname}}\hfill
\input{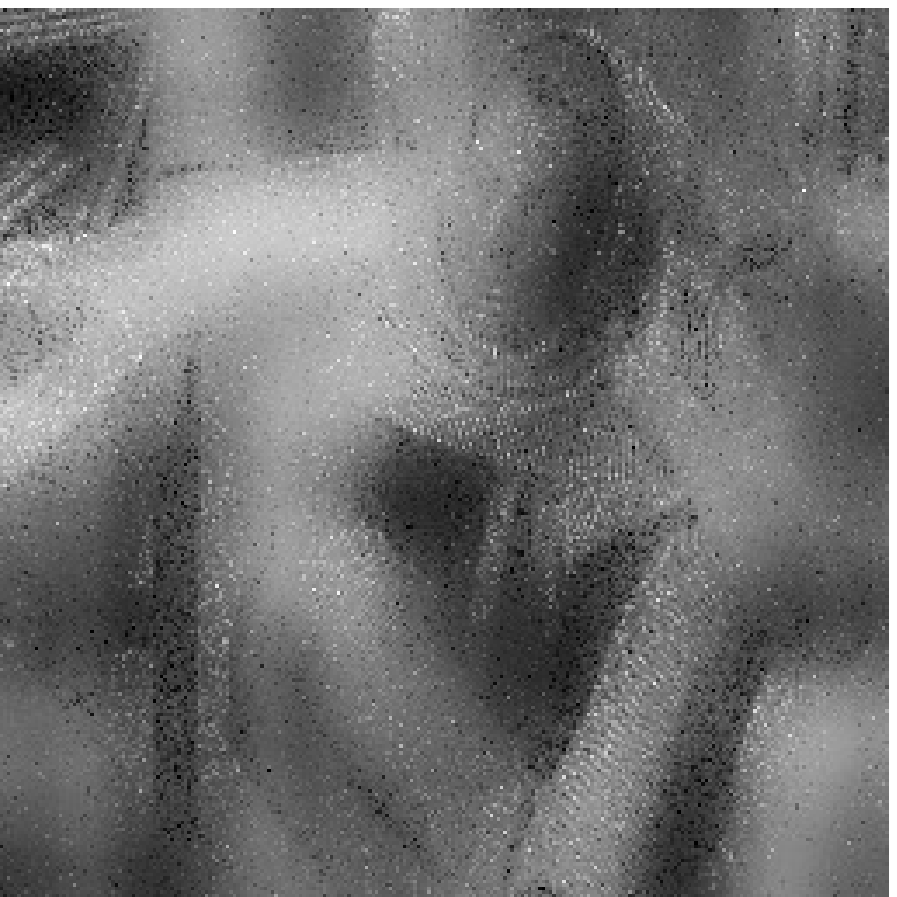}  \subfigure[\currentcaption]{\includegraphics[width=0.233\linewidth]{\currentname}}\hfill
\input{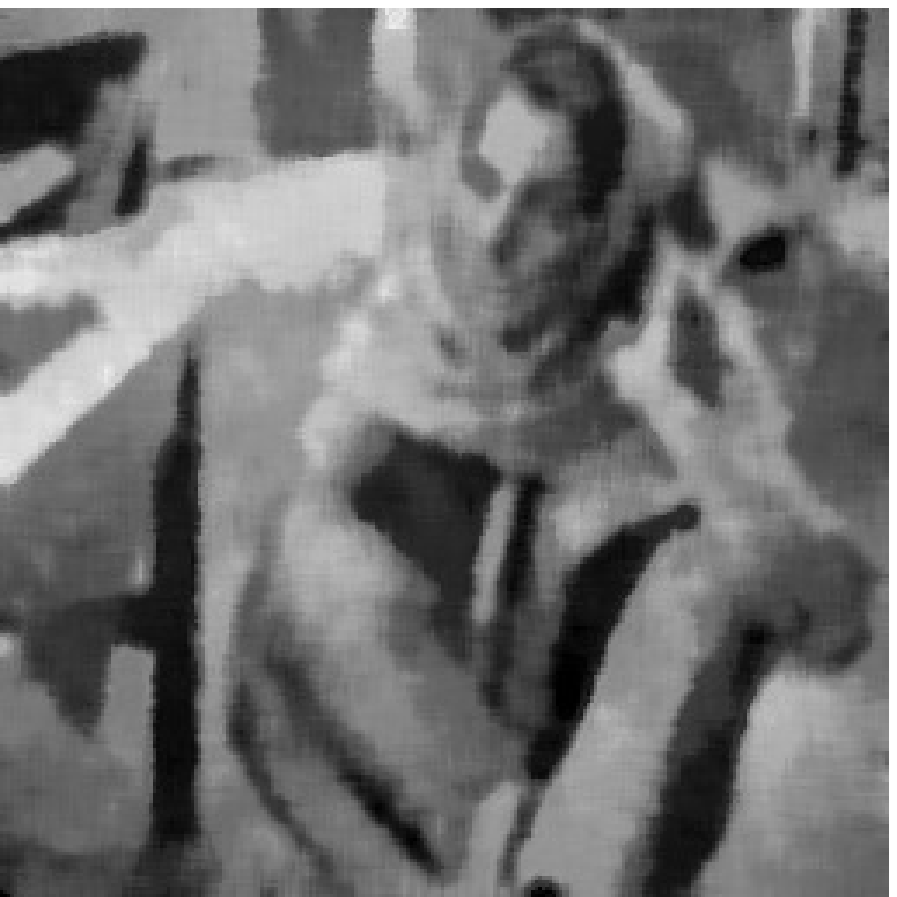}  \subfigure[\currentcaption]{\includegraphics[width=0.233\linewidth]{\currentname}}\hfill
\input{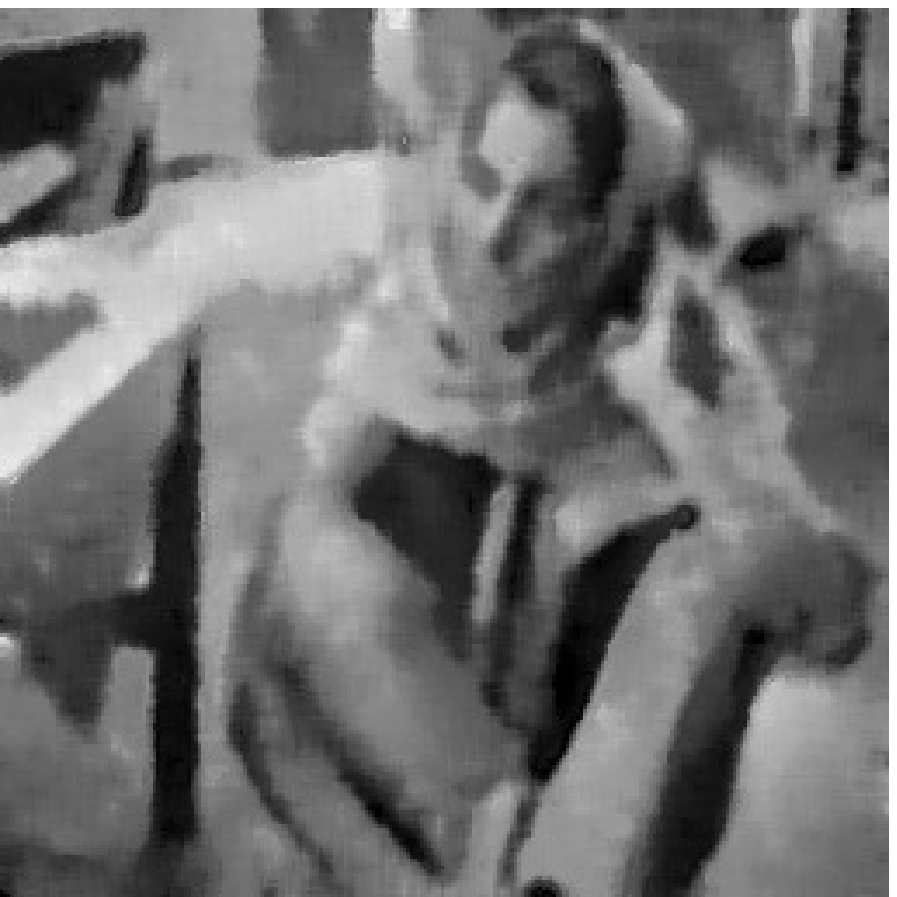}  \subfigure[\currentcaption]{\includegraphics[width=0.233\linewidth]{\currentname}}\hfill
\input{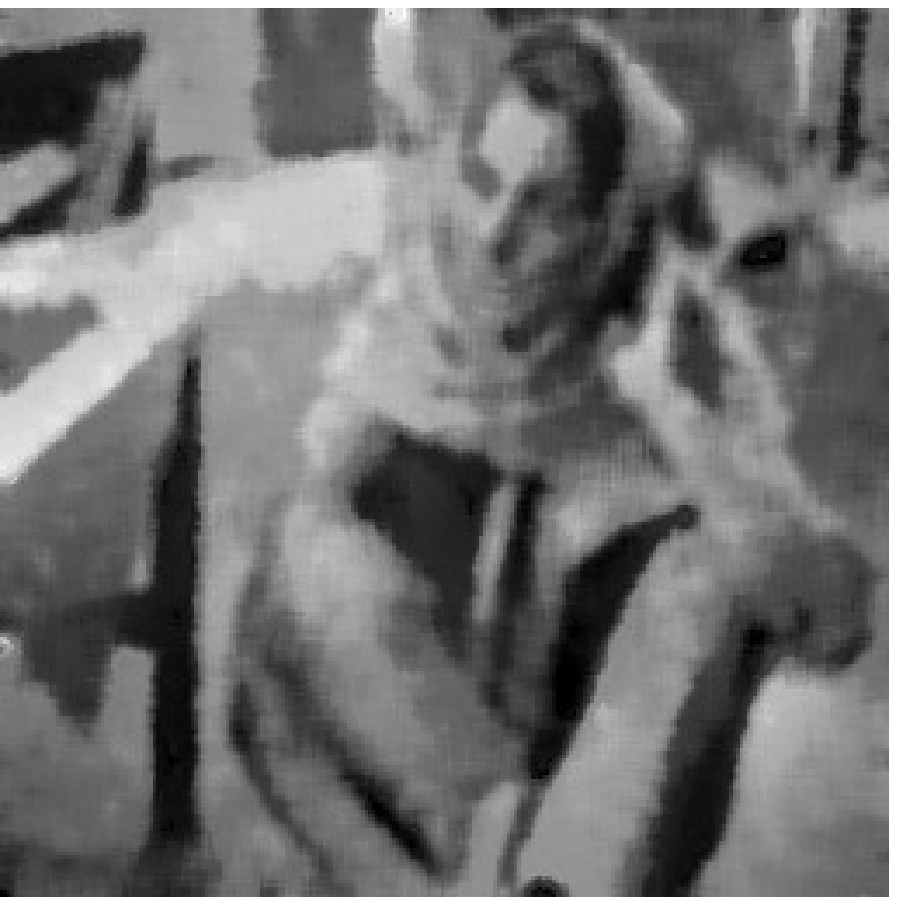}  \subfigure[\currentcaption]{\includegraphics[width=0.233\linewidth]{\currentname}}\hfill
\input{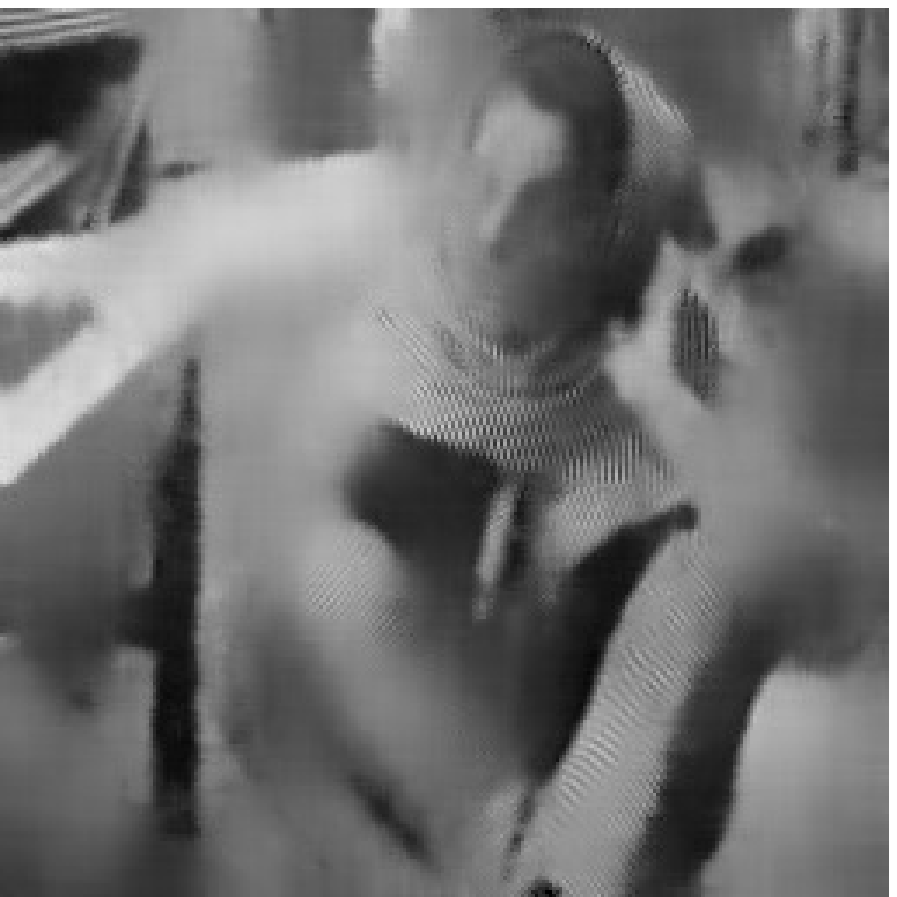}  \subfigure[\currentcaption]{\includegraphics[width=0.233\linewidth]{\currentname}}\hfill
\input{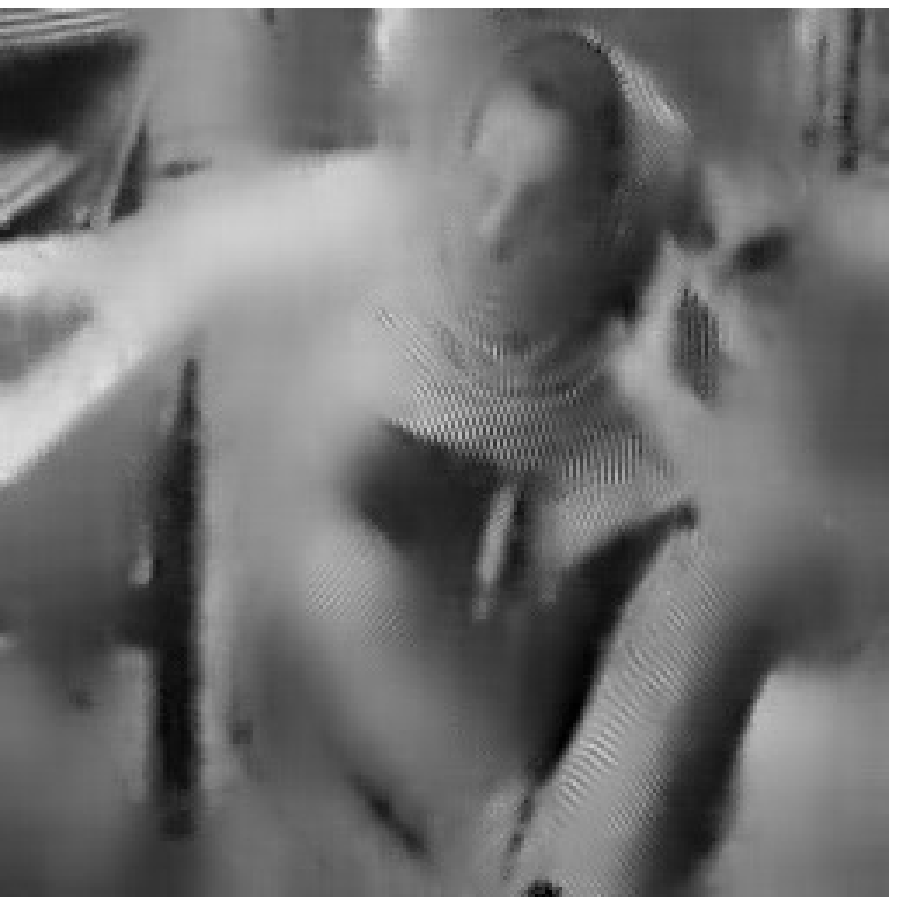}  \subfigure[\currentcaption]{\includegraphics[width=0.233\linewidth]{\currentname}}\hfill
\input{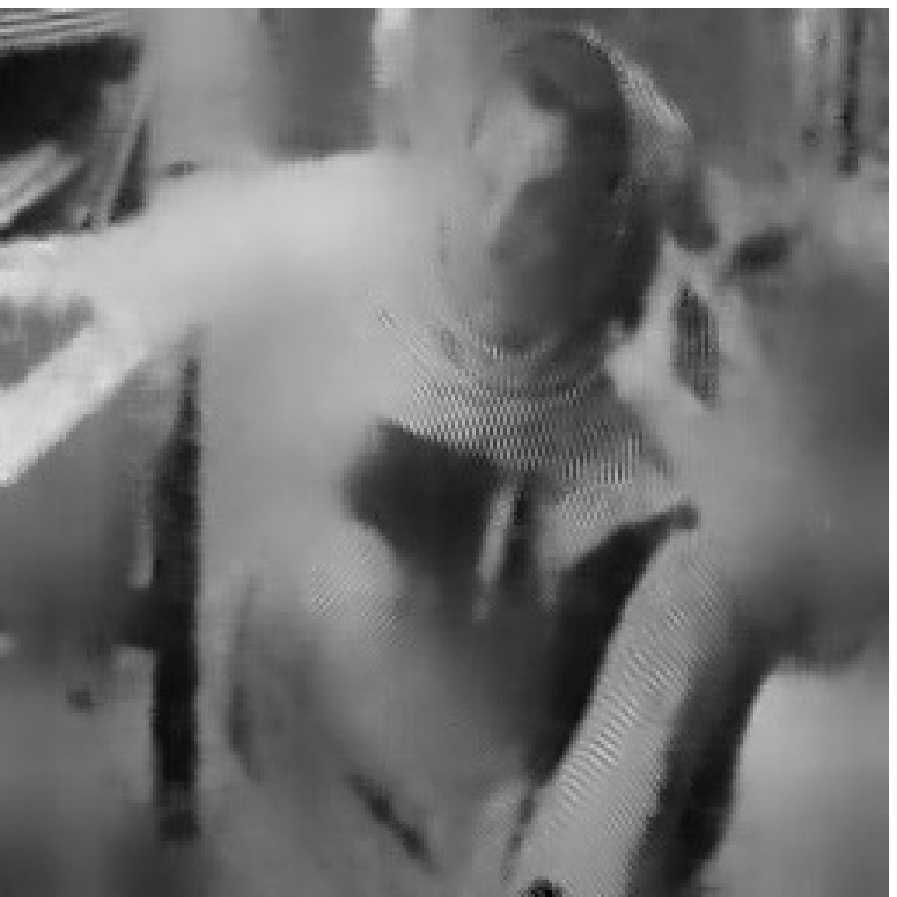}  \subfigure[\currentcaption]{\includegraphics[width=0.233\linewidth]{\currentname}}\hfill
\input{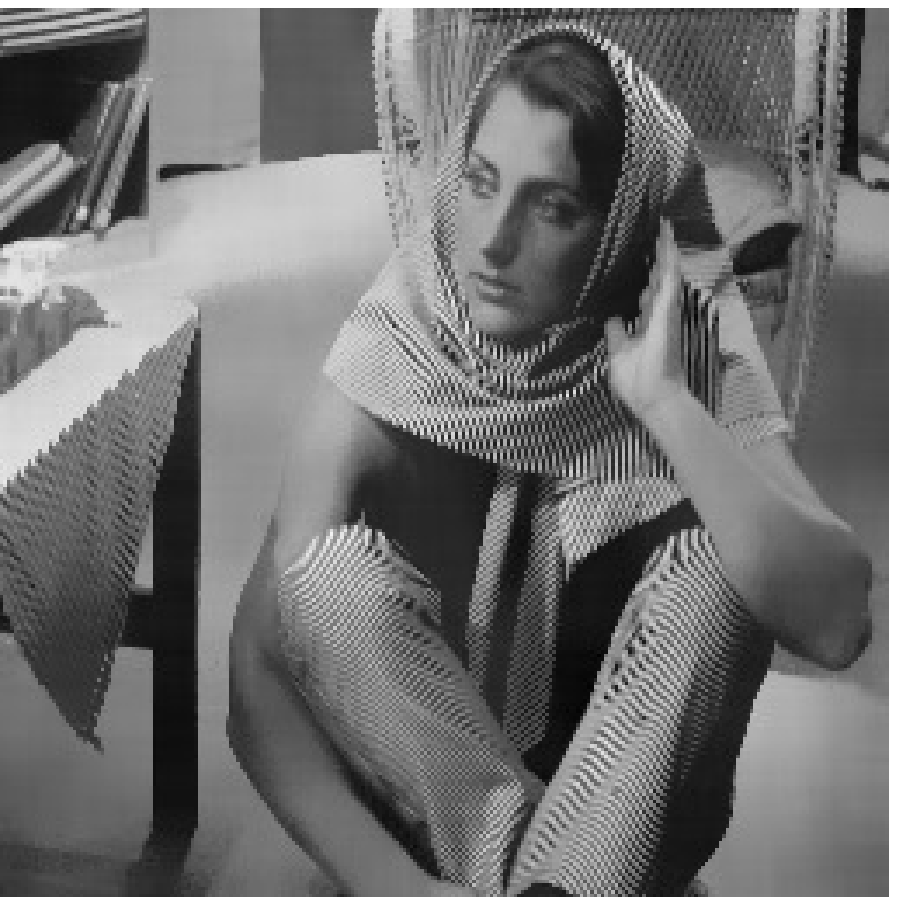}  \subfigure[\currentcaption]{\includegraphics[width=0.233\linewidth]{\currentname}}\hfill
\input{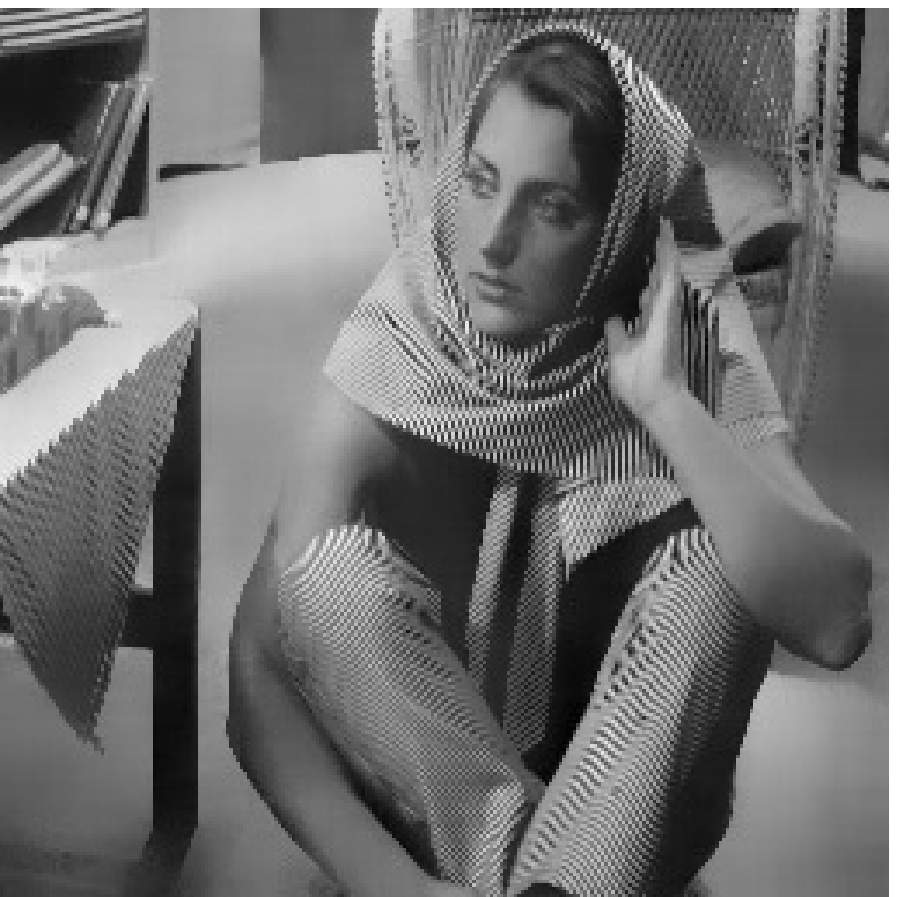}  \subfigure[\currentcaption]{\includegraphics[width=0.233\linewidth]{\currentname}}\hfill
\input{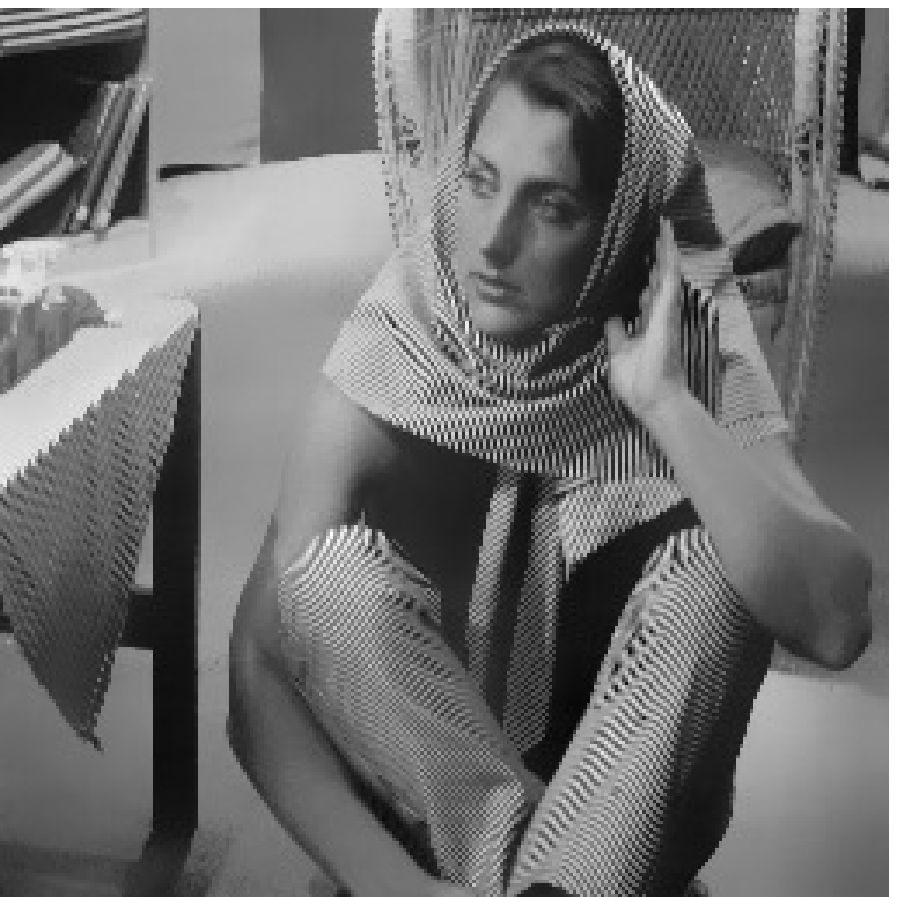}  \subfigure[\currentcaption]{\includegraphics[width=0.233\linewidth]{\currentname}}
\caption{Barbara image   corrupted Gaussian noise with $\sigma=50$. }
\label{fig:Barbara_sigma=100}
\end{figure}

\begin{figure}[hbt!]
\centering
\input{./latex_figures/Im_Noisy_100_Stripes}  \subfigure[\currentcaption]{\includegraphics[width=0.233\linewidth]{\currentname}}\hfill
\input{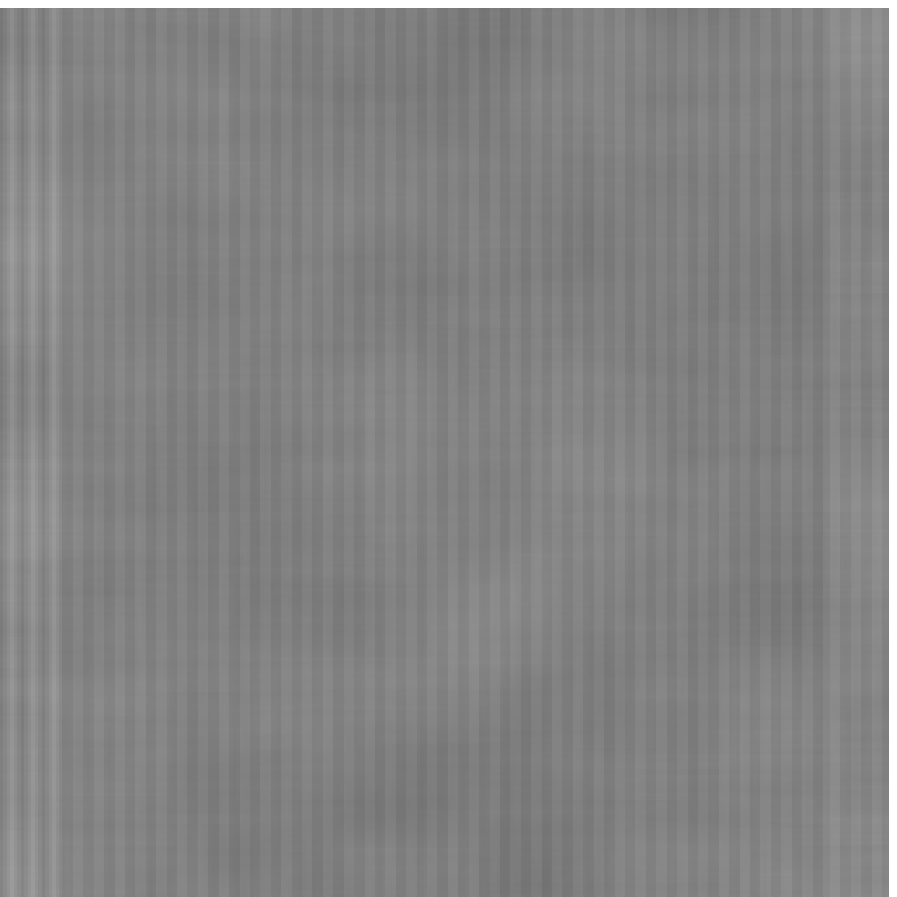}  \subfigure[\currentcaption]{\includegraphics[width=0.233\linewidth]{\currentname}}\hfill
\input{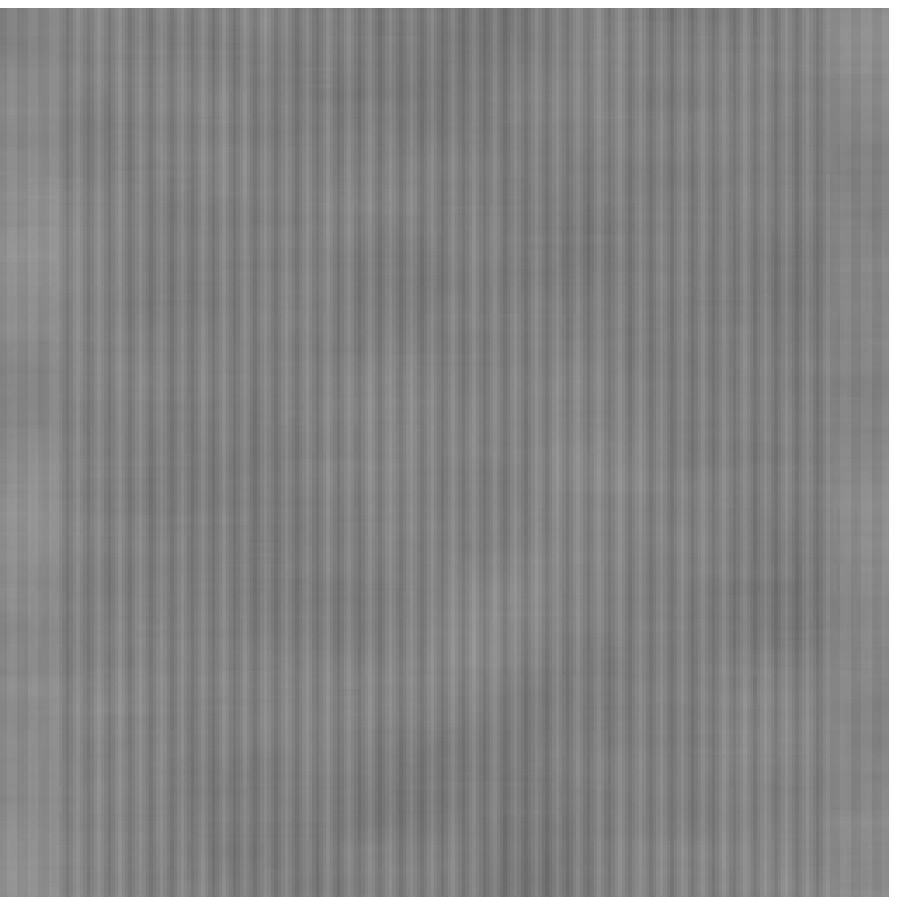}  \subfigure[\currentcaption]{\includegraphics[width=0.233\linewidth]{\currentname}}\hfill
\input{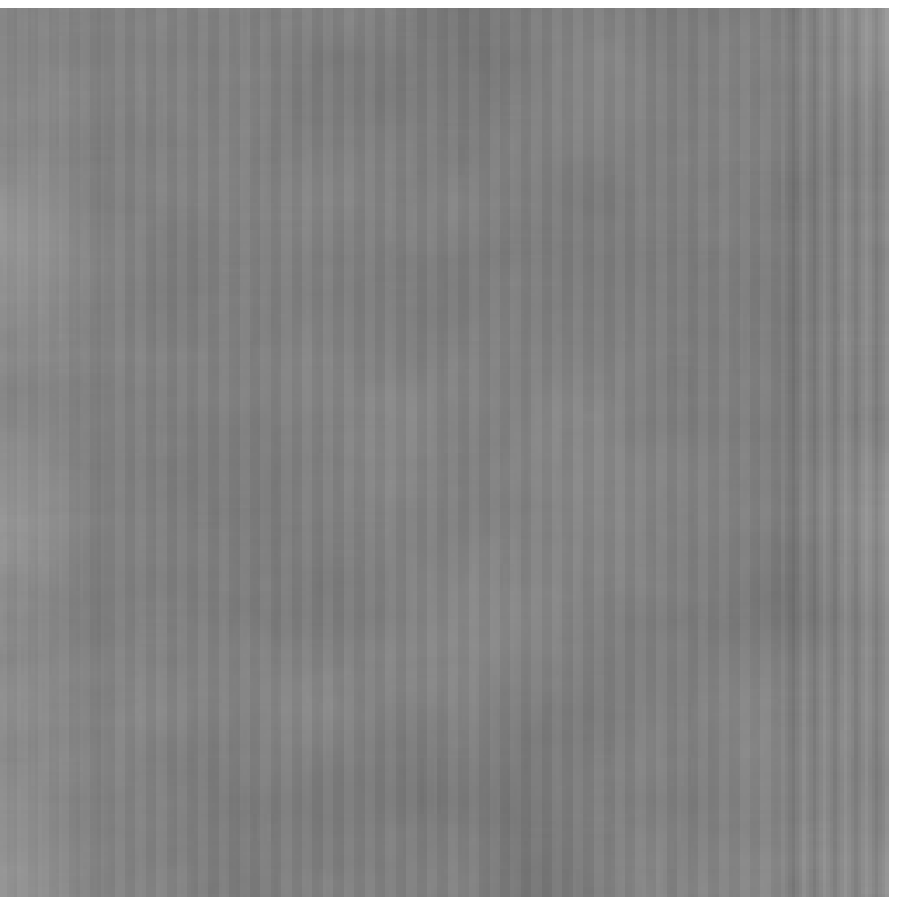}  \subfigure[\currentcaption]{\includegraphics[width=0.233\linewidth]{\currentname}}\hfill
\input{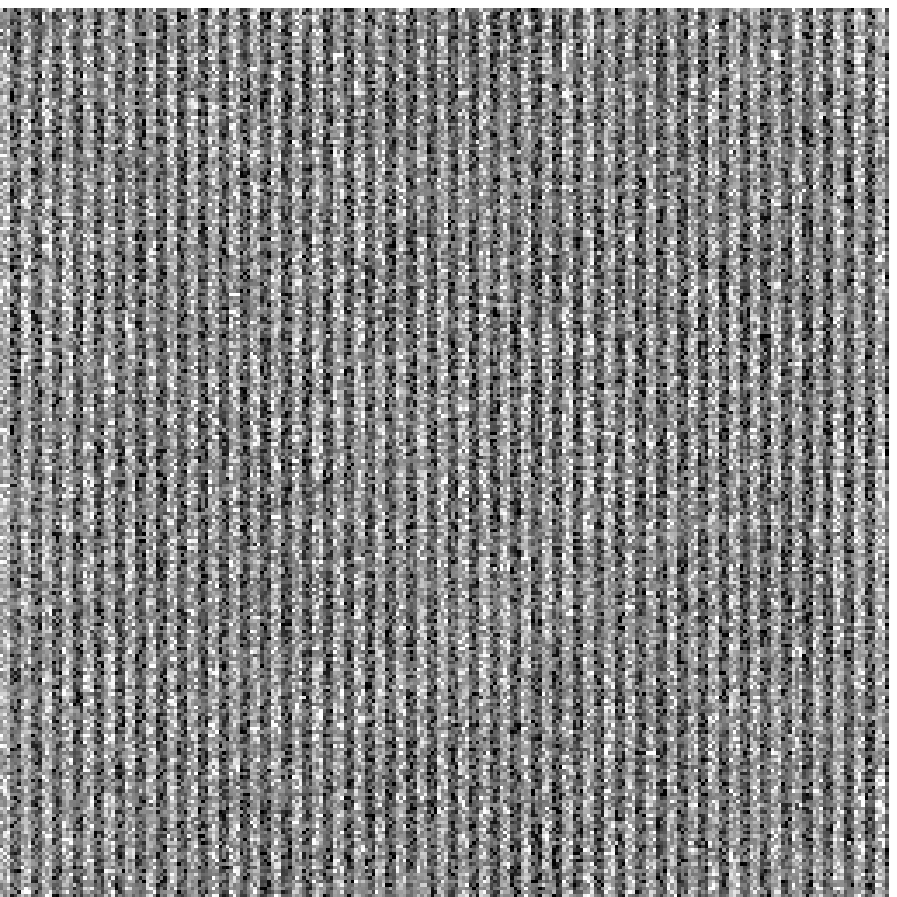}  \subfigure[\currentcaption]{\includegraphics[width=0.233\linewidth]{\currentname}}\hfill
\input{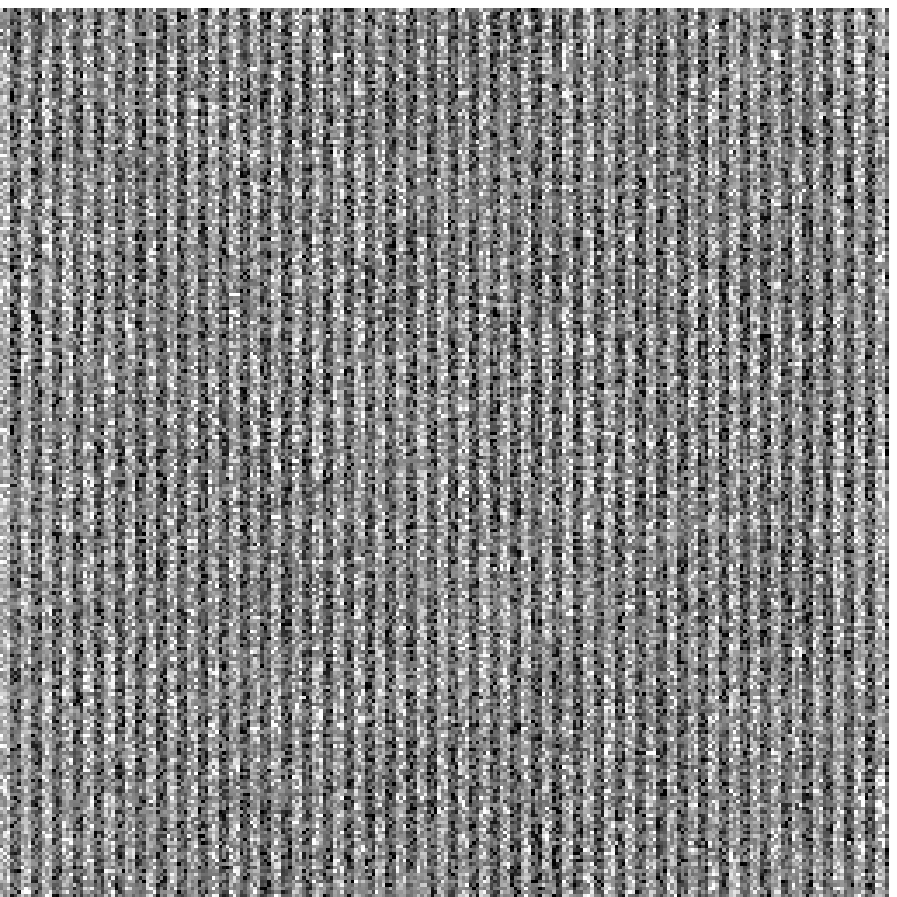}  \subfigure[\currentcaption]{\includegraphics[width=0.233\linewidth]{\currentname}}\hfill
\input{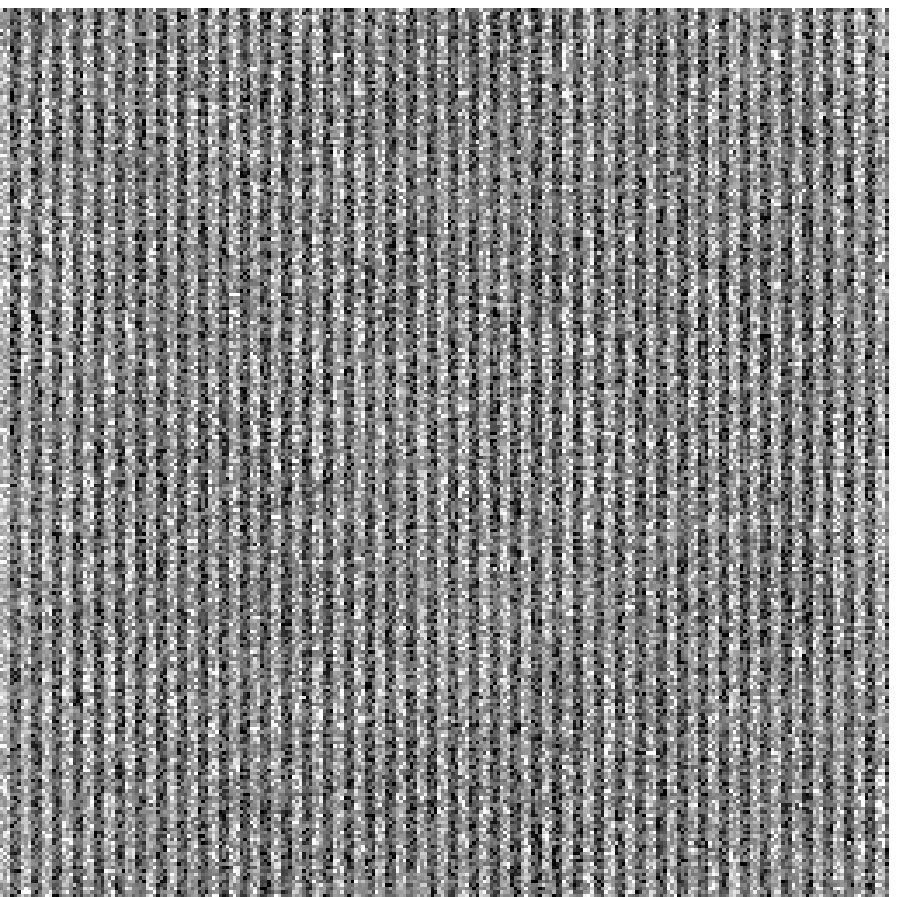}  \subfigure[\currentcaption]{\includegraphics[width=0.233\linewidth]{\currentname}}\hfill
\input{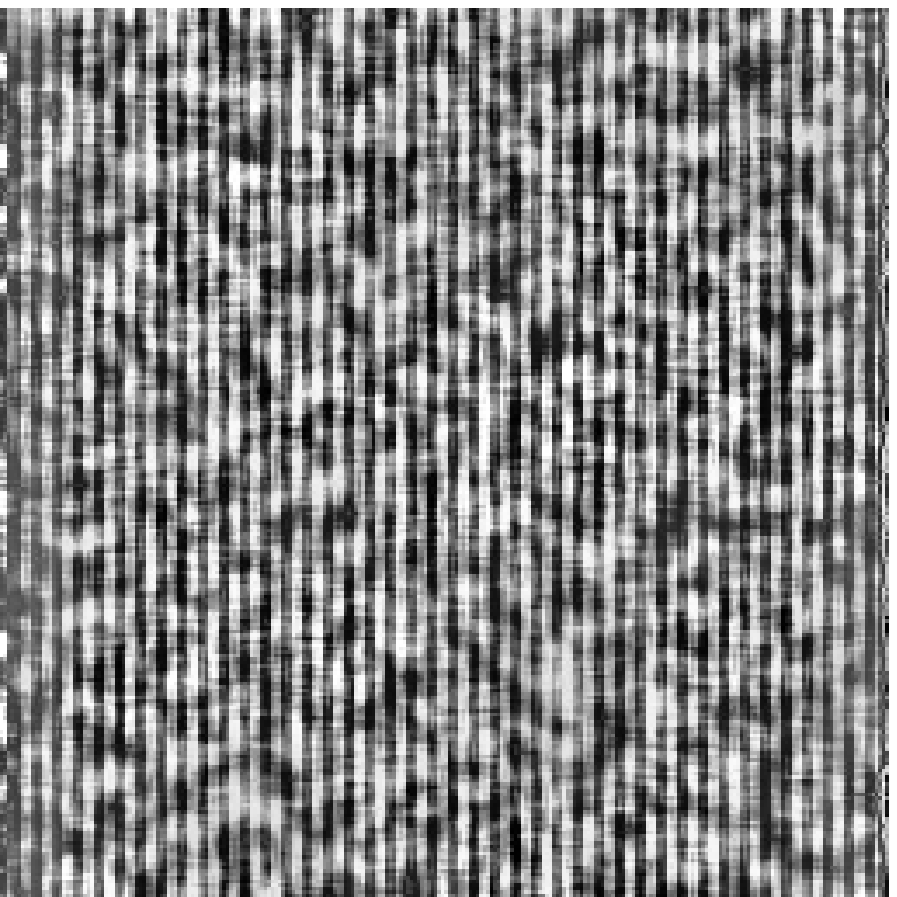}  \subfigure[\currentcaption]{\includegraphics[width=0.233\linewidth]{\currentname}}\hfill
\input{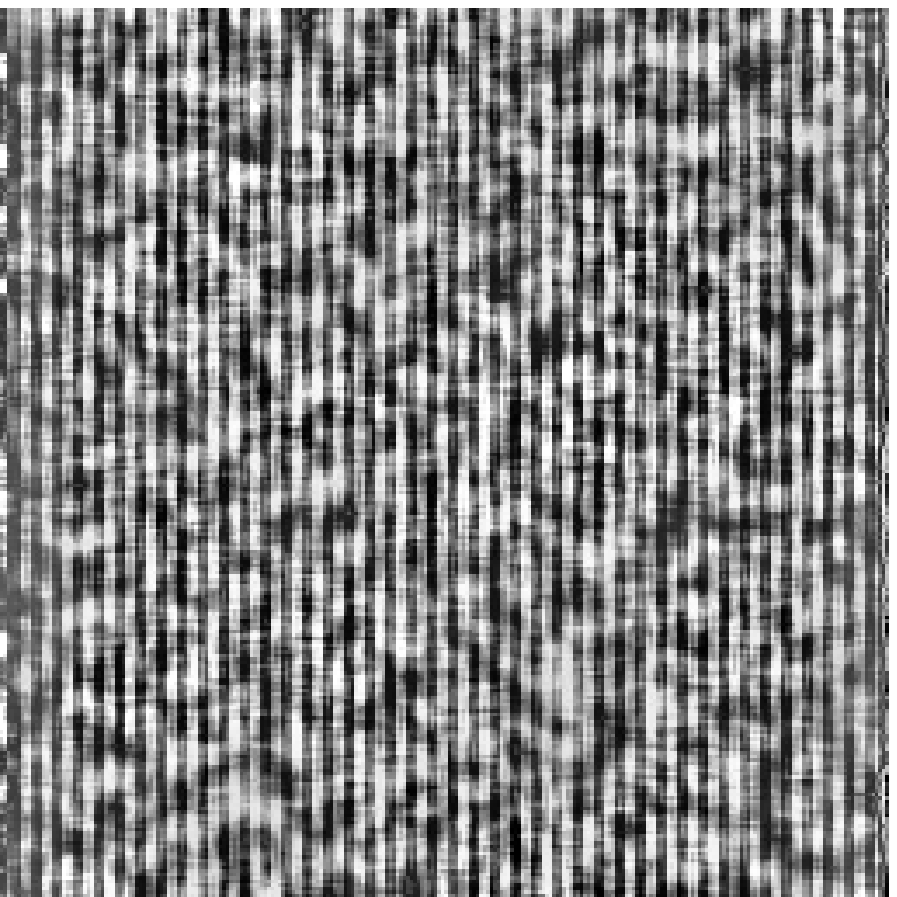}  \subfigure[\currentcaption]{\includegraphics[width=0.233\linewidth]{\currentname}}\hfill
\input{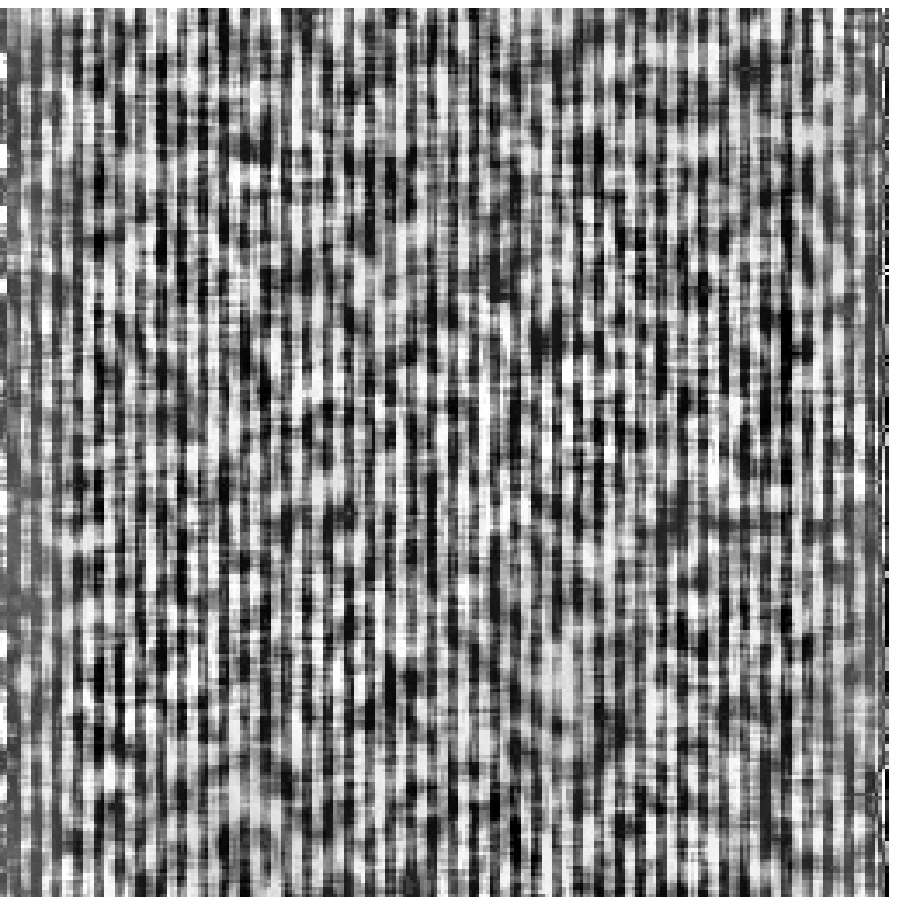}  \subfigure[\currentcaption]{\includegraphics[width=0.233\linewidth]{\currentname}}\hfill
\input{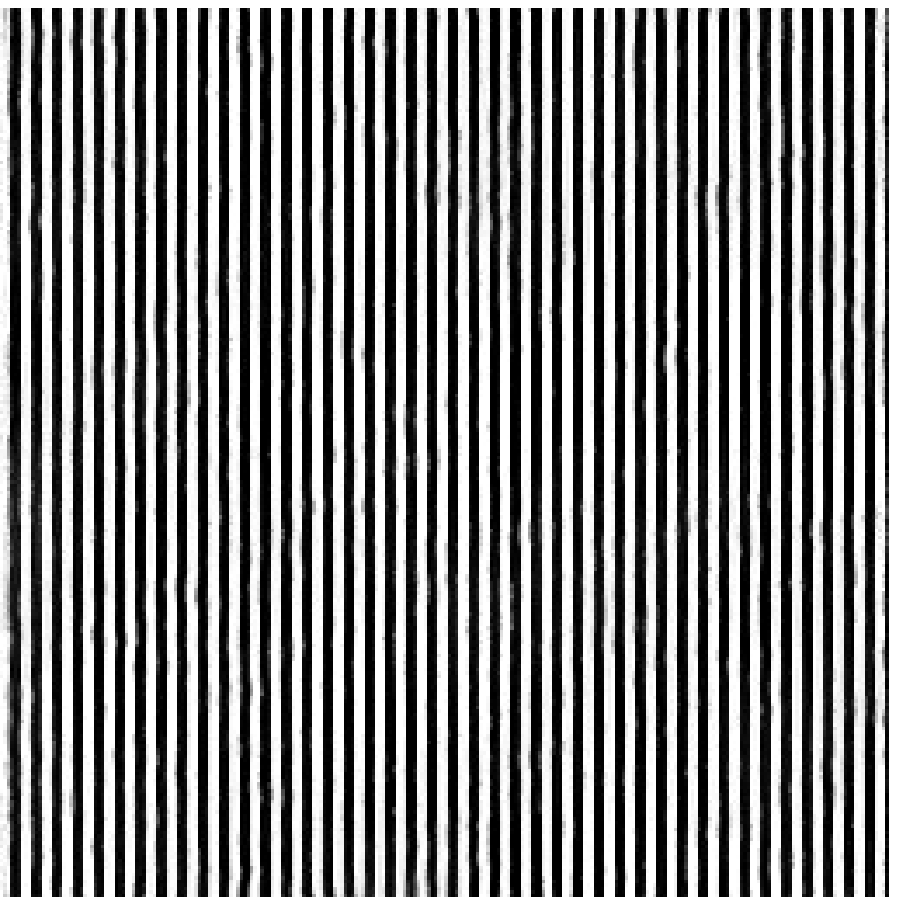}  \subfigure[\currentcaption]{\includegraphics[width=0.233\linewidth]{\currentname}}\hfill
\input{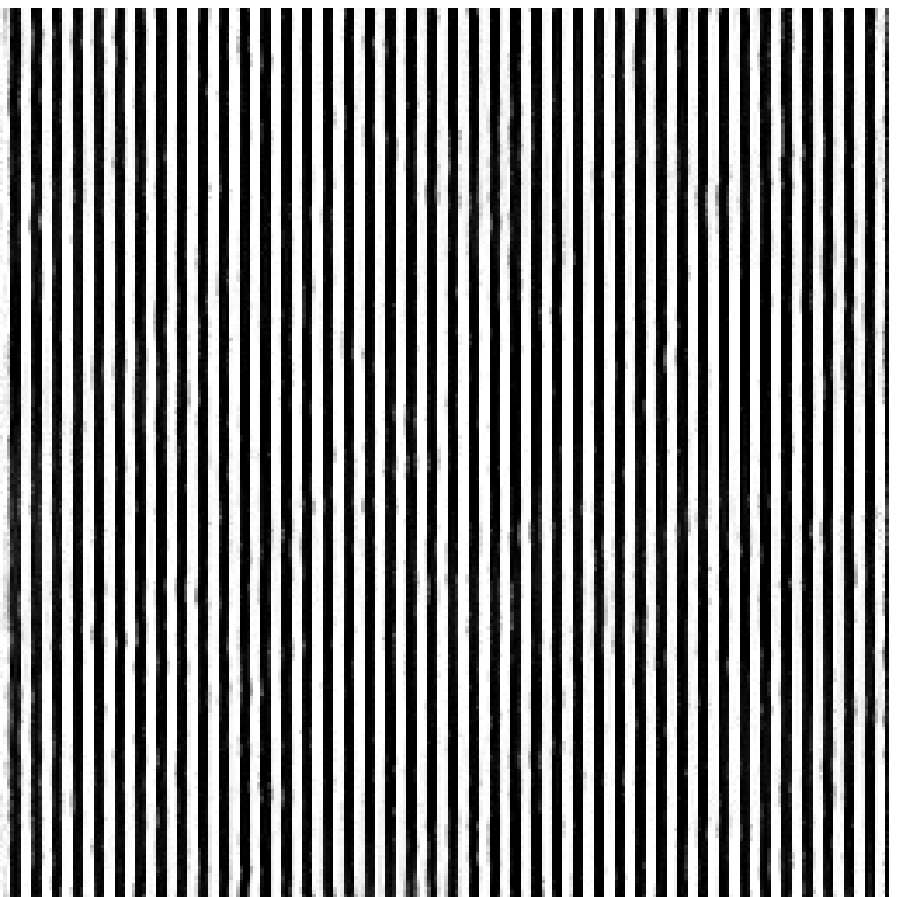}  \subfigure[\currentcaption]{\includegraphics[width=0.233\linewidth]{\currentname}}\hfill
\input{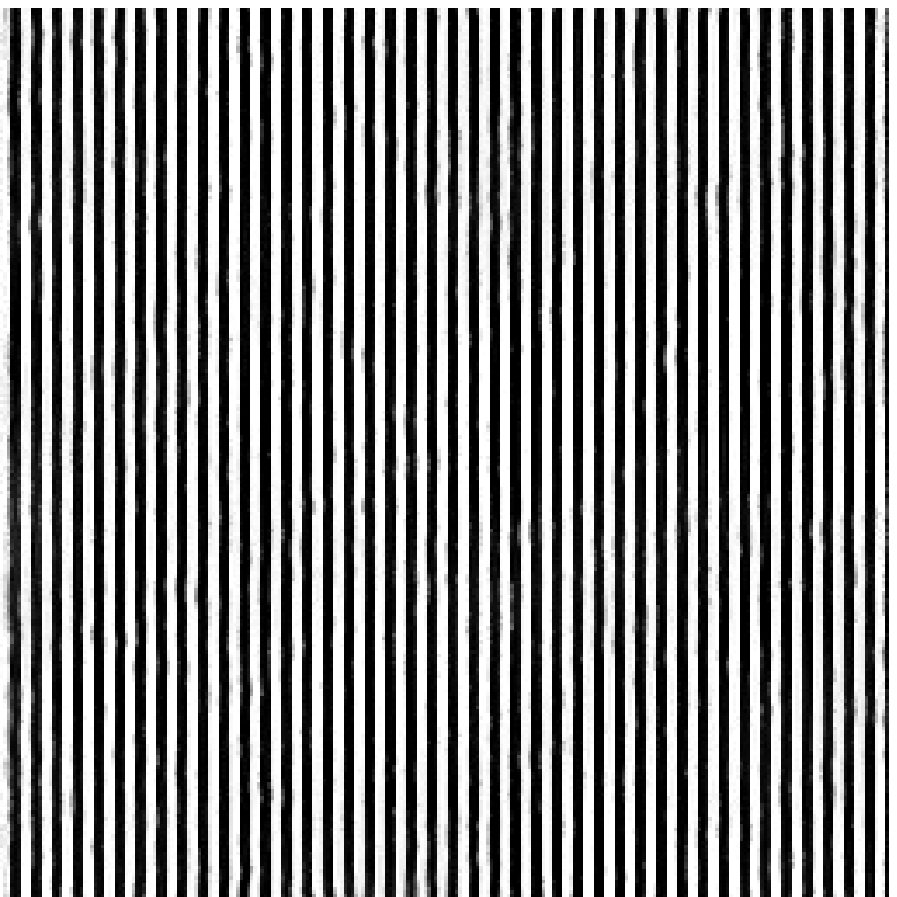}  \subfigure[\currentcaption]{\includegraphics[width=0.233\linewidth]{\currentname}}\hfill
\input{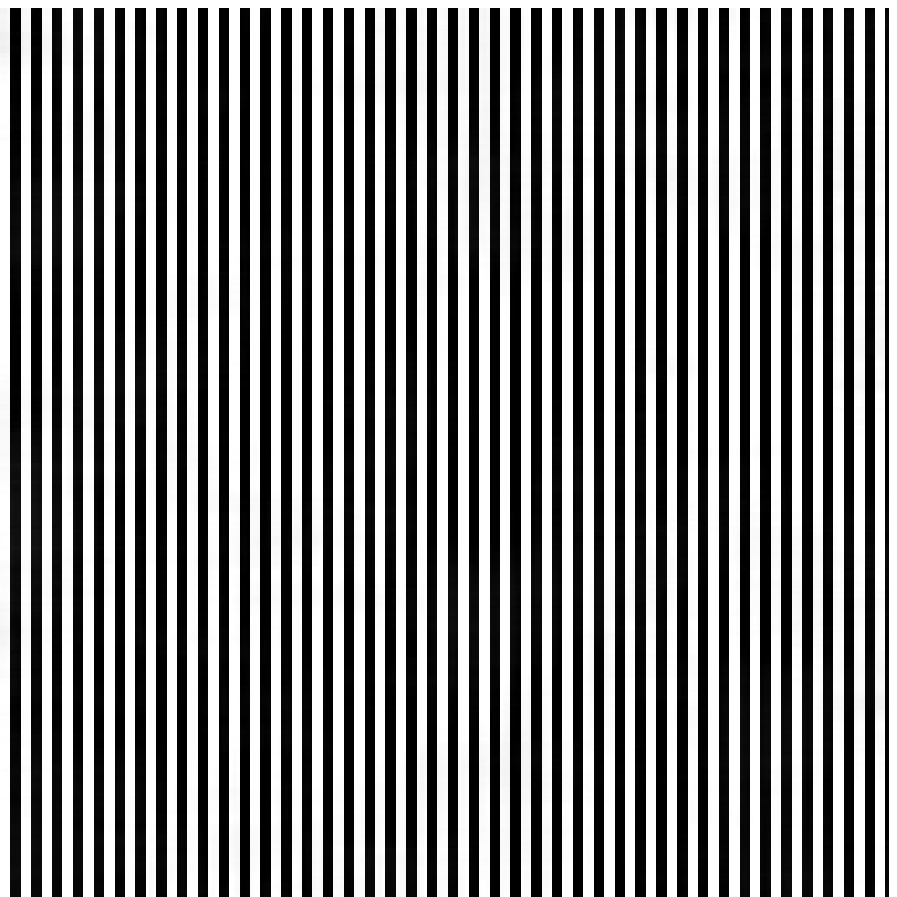}  \subfigure[\currentcaption]{\includegraphics[width=0.233\linewidth]{\currentname}}\hfill
\input{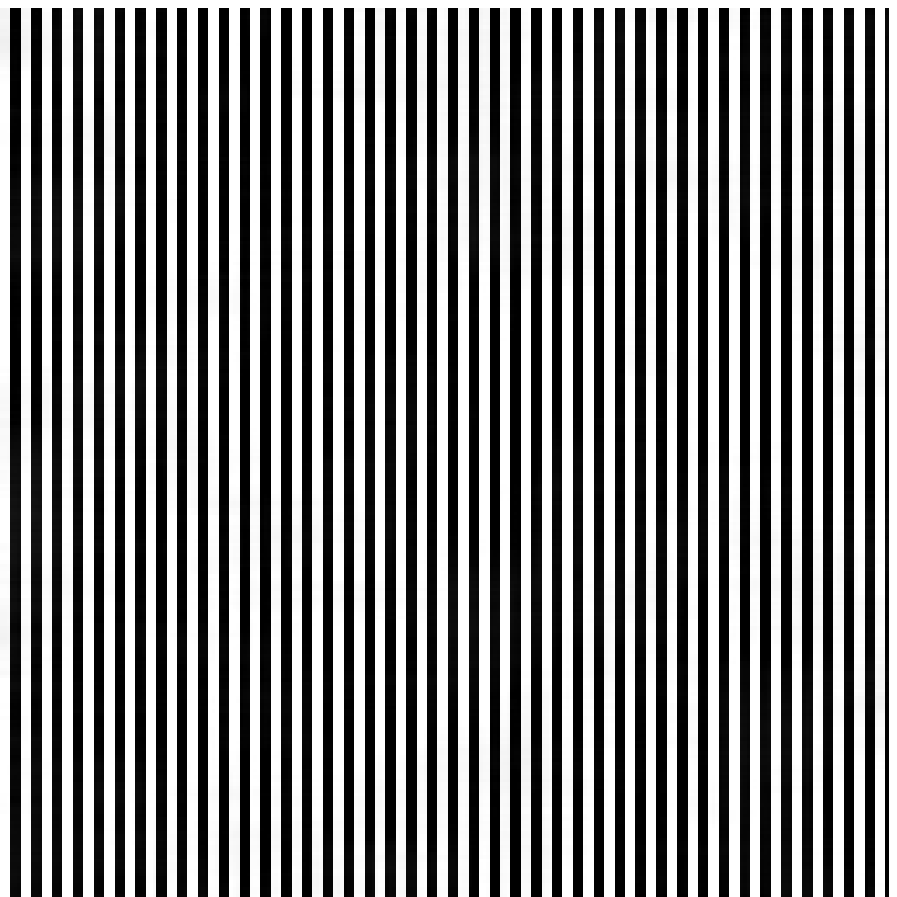}  \subfigure[\currentcaption]{\includegraphics[width=0.233\linewidth]{\currentname}}\hfill
\input{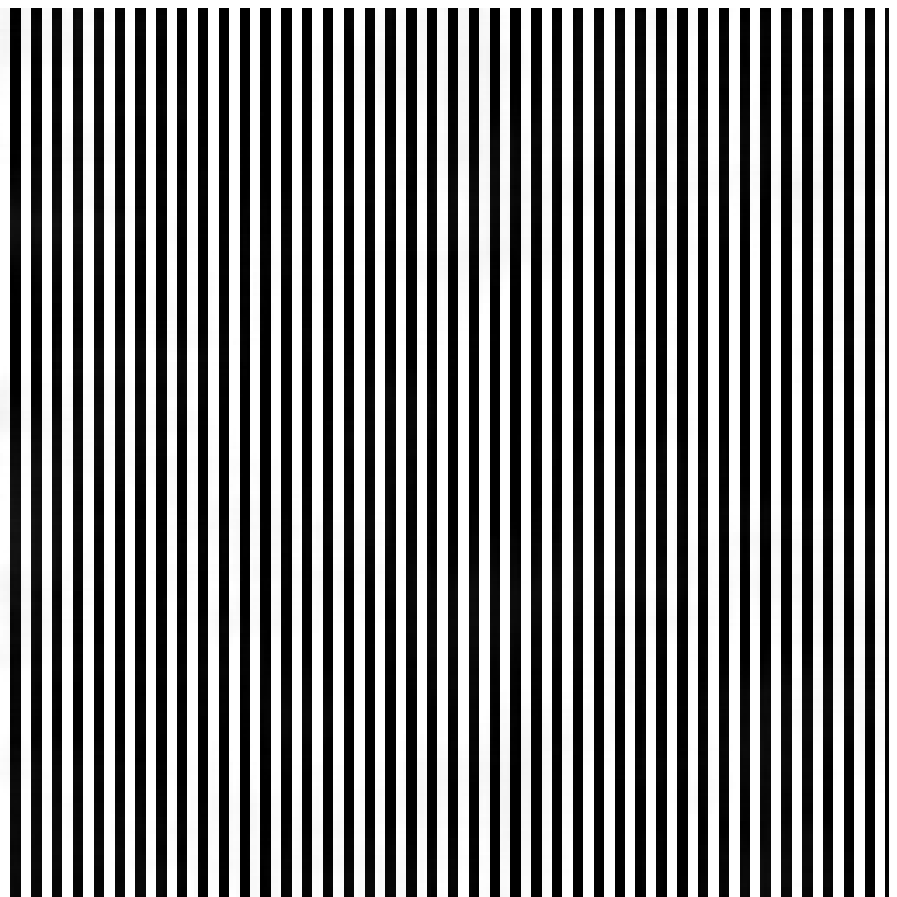}  \subfigure[\currentcaption]{\includegraphics[width=0.233\linewidth]{\currentname}}
\caption{Toy texture image (Stripes)  corrupted Gaussian noise with $\sigma=100$. }
\label{fig:Stripes_sigma=100}
\end{figure}

\section{Discussion}
\label{sec:discussion}
The theoretical results of the preceding sections are summarized
in Table~\ref{tab:summary}. 
\begin{table}[tbp]
\begin{center}
\begin{tabular}{|l|l|l|}
\hline
{\bf Image class} & {\bf Method} & {\bf Bound} \\ \hline
\multirow{5}{*}{$\cF^{\rm cartoon}$} & MO & $\cR_n \asymp \cR^{\rm MO} := (\sigma^2/n^d)^{2\alpha/(d+2\alpha)}$\\
 & LF & $\cR_n \asymp (\sigma^2/n^d)^{1/(d+1)}$\\
 & YF & $\cR_n \leq (1+o(1))\cR^{\rm MO}$ (for low noise)\\
 & NLM & $\cR_n \preceq \begin{cases}
(1+o(1))\cR^{\rm MO} & \mbox{ for low noise
}\\
[(\sigma^4 \log n)^{1/d}/n] \vee \cR^{\rm MO} & \mbox{ otherwise} \end{cases}$\\
 & NLM-Av. & $\cR_n \preceq \begin{cases}
(1+o(1))\cR^{\rm MO} & \mbox{ for low noise
}\\
[(\sigma^2 \log n)^{1/d}/n] \vee \cR^{\rm MO} & \mbox{ otherwise} \end{cases}$\\ \hline
\multirow{4}{*}{$\cF^{\rm pattern}$} & MO & $\cR_n \asymp \cR^{\rm MO}$\\
 & LF & $\cR_n \asymp 1$\\
 & YF & $\cR_n \preceq \cR^{\rm MO}$ (for low noise)\\
 & NLM & $\cR_n \preceq \begin{cases}
(1+o(1))\cR^{\rm MO} & \mbox{ for low noise
}\\
[(na)^d \cR^{\rm MO} & \mbox{ for ``distinct'' patterns} \end{cases}$\\ \hline
\multirow{5}{*}{$\cF^{\rm thin}$} & MO & $\E (\wh{f}_i - f_i)^2 \asymp (h^{\rm MO}/a)^{d-d_0}\cR^{\rm MO}$\\
 & LF & $\E (\wh{f}_i - f_i)^2 \asymp 1$\\
 & YF & $\E (\wh{f}_i - f_i)^2 \preceq (h^{\rm MO}/a)^{d-d_0}\cR^{\rm MO}$ (for low noise)\\
 & NLM & $\E (\wh{f}_i - f_i)^2 \preceq (h^{\rm
  MO}/a)^{d-d_0}\cR^{\rm MO}$\\
 & NLM-Av. & $\E (\wh{f}_i - f_i)^2 \preceq (h^{\rm MO}/a)^{d-d_0}\cR^{\rm MO}$\\
\hline
\end{tabular}
\end{center}
\caption{Summary of results}
\label{tab:summary}
\end{table}

As described in the Introduction, the bounds described in this paper
and in the independent work
\cite{Maleki_Narayan_Baraniuk11,Maleki_Narayan_Baraniuk11b} address
{\em fundamental performance limits of NLM and related photometric
  image filtering methods}.  These methods have an established history
of strong empirical performance on natural images, but until now
little was known about how these methods performed asymptotically,
especially with respect to related methods based on computational
harmonic analysis (\eg wavelet or curvelet denoising).

 Both our bounds and the bounds in
  \cite{Maleki_Narayan_Baraniuk11} suggest that NLM has some
  limitations for piecewise smooth images when the noise is not small
  (\ie when YF cannot perform effectively).  When the noise is small,
  we note that YF is a special case of NLM with a patch size of one
  pixel, and the performance of NLM hinges upon our ability to measure
  the similarity of two patches based on noisy observations. In low
  noise, this similarity can already be estimated quite accurately
  with a single pixel patch. In stronger noise, the similarity
  measured through larger patches is more robust to noise --- but
  larger patches also introduce some bias. This results in an elbow in
  the performance bounds for NLM. Recent empirical results suggest
  that these limitations can be mitigated by adapting the kernel shape
  \cite{Takeda_Farsiu_Milanfar07}, the patch shape
  \cite{Deledalle_Duval_Salmon11}, or spatial bandwidth $h$
  \cite{Kervrann_Boulanger06}; a theoretical understanding of this
  kind of adaptation remains an open problem.

There are several distinctions between our work and the
closely related work in \cite{Maleki_Narayan_Baraniuk11} that bear
mentioning.  First, we consider the {\em cartoon model}, where the
functions are piecewise H\"older smooth images with a discontinuity
set corresponding to a Lipschitz mapping of the unit ball, while
Maleki {\em et al.} \cite{Maleki_Narayan_Baraniuk11} consider the
{\em horizon edge model}, where the functions are piecewise constant
with a discontinuity set corresponding to the graph of a Lipschitz
function.  Though they actually consider smoother edges, their
analysis reduces to the case of Lipschitz smoothness.  We consider
images in arbitrary dimension --- showing that NLM behaves differently
when $d \ge 3$ --- while Maleki {\em et al.}~consider the case of
2D images ($d=2$).  Because they consider functions that are piecewise
constant, they use the weighted average version
\eqref{eq:neighbor_filter} without spatial localization (\ie $h=\infty$).
Because we focus on smooth --- not necessarily constant --- regions, we
need to localize both YF and NLM. Applying the more complex LPR
\eqref{local_poly} enables YF to adapt to the degree of
regularity in each smooth region.  We note that Maleki et
al.~\cite{Maleki_Narayan_Baraniuk11} do not consider the case of low
noise and simply show that YF achieves the same performance as LF ---
which is also our conclusion in strong noise.  

Moreover to simplify the
analysis, Maleki {\em et al.}~consider a slightly modified version
of  NLM and derive lower bounds for oracle versions of YF
and NLM.  The lower bounds for NLM were also challenging for us and we
only provide heuristics.  We mention that our results imply that the
simpler NLM-average achieves the same performance as NLM in the horizon
model.

Let us also comment on the work of Tsybakov \cite{Tsybakov89}.  This
paper suggests and analyzes (within the cartoon model) a method very
similar to \NLMm, based on medians rather than means.  The method is
based on non-overlapping patches.  This allows the method to be
applicable in situations where the noise distribution is heavy-tailed.
(We were not aware of this work when we prepared our paper.)

Our analysis of NLM for  image classes with thin features or
regular patterns is also a significant novel aspect of our work.  
Though we expect wavelets are near-optimal for cartoon images when the discontinuity is Lipschitz, 
NLM has a significant empirical advantage over wavelets for certain kinds of repeating textures. 
We develop a model for images
with these features, and note that it does not approach cartoon or
horizon model asymptotically.  For this image class, we demonstrate
that NLM performs as well as it does for the cartoon
class.

The current bounds are based on ideal bandwidths which depend on the
unknown smoothness parameter $\alpha$. Thus we have demonstrated that
the adaptive filtering techniques considered adapt to the
discontinuity $\Omega$, but not to $\alpha$. We anticipate that
adaptivity to $\alpha$ is indeed possible and leave that analysis for
future work.

We note that NLM is not the current state-of-the-art image
  denoising method in common use. More evolved patch-based methods
  utilize sparse representations of patches, adaptive kernel
  bandwidths, and adaptive patch shapes (\lcf
  \cite{Dabov_Foi_Katkovnik_Egiazarian07,Dabov_Foi_Katkovnik_Egiazarian09,Mairal_Bach_Ponce_Sapiro_Zisserman09,Deledalle_Duval_Salmon11b,Spira_Kimmel05,Takeda_Farsiu_Milanfar07}). While
  these aspects are not considered in our analysis, the theoretical
  insights provided by this paper may potentially lead to an improved
  understanding of a broad class of patch-based image denoising
  methods and subsequently better algorithms (\lcf \cite{Salmon_Willett_AriasCastro12} for possible directions).

\section{Proofs}
\label{sec:proofs}

In this section, $C, C_1, C_2, \dots$ denote finite positive constants
that do not change with $n$ and whose actual value may change with
each appearance.

\subsection{Preliminary results}

We first gather some basic results.

\subsubsection{Some analysis}

Functions in $\cH_d(\alpha, \cf)$ are uniformly well-approximated
locally by polynomials of degree $\afloor$, specifically their Taylor
expansions.  For $g \in \cH_d(\alpha, \cf)$ and $x \in [0,1]^d$, the
Taylor expansion of $g$ at $x$ of degree $t \in \N$ is defined as
follows:
$$T^r_x g(x') = \sum_{|s| \leq t} g^{(s)}(x)  \ \prod_{i=1}^d \frac{(x_i' - x_i)^{s_i}}{s_i!}.$$

\begin{lem} \label{lem:taylor} For any $g \in \cH_d(\alpha, \cf)$,
$$
|g(x') - T^{\afloor}_x g (x')| \leq c_{\alpha} C_0 \|x' - x\|_{\infty}^\alpha, \quad \forall x, x' \in [0,1]^d,
$$
where 
\[
c_{\alpha} := \sum_{s\in\N^d : |s| = \afloor} \frac{1}{s_1! \cdots s_d!}.
\]
\end{lem}

\begin{proof} Though this sort of result is well-known, we provide a
  proof for completeness.  A Taylor approximation of degree $\afloor$
  gives:
$$g(x') =T^{\afloor}_x g (x') + \sum_{|s| = \afloor} (g^{(s)}(z) - g^{(s)}(x))  \ \prod_{i=1}^d \frac{(x_i' - x_i)^{s_i}}{s_i!},$$
for some $z$ on the segment joining $x$ and $x'$.
Hence,
$$|g(x') - T^{\afloor}_x g (x')| \leq c_{\alpha} \|x' - x\|_{\infty}^{\afloor} \max_{|s| = \afloor} |g^{(s)}(z) - g^{(s)}(x)|.$$
Now apply \eqref{eq:deriv_max} and the fact that $\|z - x\|_{\infty} \leq \|x' - x\|_{\infty}$ to get
$$|g^{(s)}(z) - g^{(s)}(x)| \leq C_0 \|x' - x\|_{\infty}^{\alpha - \afloor}, \quad \forall s \in \N^d, |s| = \afloor.\, \vspace{-0.84cm} $$ 
\hspace{12.85cm}\end{proof} 

\subsubsection{Some geometry}
For a measurable set $A \subset \R^d$, \beq \label{rho} \rho(A) :=
\inf_{h \in (0,1)} \inf_{x \in A} \sup
\left\{\frac{\Vol(B(y,s))}{\Vol(B(x,h))} : 
  B(y,s) \subset B(x,h) \cap A\right\}.  \eeq
The quantity $\rho(A)$ provides some measure of how irregular the boundary of $A$ is.  
The following lemma
bounds $\rho$ from below for sets whose boundary is sufficiently
regular.

\begin{lem} \label{lem:rho} Let $\phi : \R^d \to \R^d$ be injective,
  with $\phi$ and $\phi^{-1}$ both $C$-Lipschitz.  Then for $\Omega =
  \phi(B(0,1))$ and $\rho$ defined in \eqref{rho}, we have
  $\min(\rho(\Omega), \rho(\Omega^c)) \geq (2 C)^{-d}$.
\end{lem}

\begin{proof} 
  Fix $x \in \Omega$ and $h > 0$.  Since $\phi$ is Lipschitz with
  constant $C$, we have $\phi(B(\phi^{-1}(x), h/C)) \subset B(x, h)$.
  Note that $z := \phi^{-1}(x) \in B(0,1)$ and, by the triangle
  inequality, $B(z, h/C) \cap B(0,1) \supset B(z', t)$, where $z' :=
  (1-h/(2C)) z$ and $t := h/(2C)$.  Because $\phi^{-1}$ is
  $C$-Lipschitz, we have $\phi^{-1}(B(\phi(z'), t/C)) \subset B(z',
  t)$, so that
\[
B(y, s) \subset \phi(B(z',t)) \subset \phi(B(z, h/C) \cap B(0,1)) \subset B(x,h) \cap \Omega,
\]
where $y := \phi(z')$ and $s := t/C$.
We obtain a lower bound for $\Omega^c$ 
in a
similar way.
\end{proof}  \medskip

Next is a result on the number of sample points within a certain
distance of a subset.  Let $\cX_{n^d}$ be the set of sample points,
that is, $\cX_{n^d} = \{x_i : i \in I_n^d\}$.

\begin{lem} \label{lem:count} For any subset $A \subset (0,1)^d$ of
  the form $A = B(A', \eta)$ for some $A' \subset (0,1)^d$ and $4/n
  \leq \eta \leq 1$,
$$
8^{-d} n^d \Vol(A) \leq |A \cap \cX_{n^d}| \leq 4^d n^d \Vol(A).
$$ 
\end{lem}

\begin{proof} Let $z_1, \dots, z_k \in (0,1)^d$ be a maximal
  $\eta$-packing for $A'$ (\ie the balls $ B(z_j, \eta/2)$ for
  $j=1,\dots,k$ are disjoint and included in $A' \subset A$, and for
  any $z\in A'$, there is $j$ such that $z \in B(z_j, \eta)$).
By the triangle inequality, we have
$$\bigcup_{j=1,\dots,k}  B(z_j, \eta/2) \subset A \subset \bigcup_{j=1,\dots,k}  B(z_j, 2 \eta).$$
On the one hand, taking volumes on all sides, we get $k \eta^d \leq
\Vol(A) \leq k 2^d (2 \eta)^d$, since the unit ($\|\cdot\|_{\infty}$) ball
has volume $2^d$.  This turns into $4^{-d} \Vol(A) \leq k \eta^d \leq
\Vol(A)$.  On the other hand, counting sample points on all sides, using
the fact that
\[
\eta^d n^d \leq |B(z, \eta) \cap \cX_{n^d}| \leq (2\eta)^d n^d, \quad \forall z \in (0,1)^d, \ \forall \eta \in (2/n, 1),
\]
we get
$$
k (\eta/2)^d n^d \leq \sum_{j=1}^k |B(z_j, \eta/2) \cap \cX_{n^d}| \leq |A \cap
\cX_{n^d}| \leq \sum_{j=1}^k |B(z_j, 2 \eta) \cap \cX_{n^d}| \leq k (4\eta)^d n^d.
$$ 
Combining these, we get the desired result. \end{proof} 

\begin{lem} \label{lem:vol} Suppose $1 \le d_0 \le d$ are integers and
  let $\phi : \R^{d_0} \to \R^d$ be injective, with $\phi$ and
  $\phi^{-1}$ (on the range of $\phi$) both $C$-Lipschitz with $C \ge
  1$.  Then there is another constant $C' > 1$ such that, for $A :=
  \phi((0,a)^{d_0})$ and $h \in (0, 1)$,
\[
\frac1{C'} a^{d_0}h^{d-d_0} \leq \Vol(B(A,h)) \leq C' a^{d_0}h^{d-d_0}.
\]
Consequently, if $\phi : \R^{d} \to \R^d$ is as above and $A :=
\phi(\partial B(0,1))$, the result holds with $d - d_0 = 1$.
\end{lem}

\begin{proof} 
  We first observe that, for any $z \in \bbR^{d_0}$ and $h > 0$, since
  $\phi$ is $C$-Lipschitz, \beq \label{phi-ub} \phi(B(z,h)) \subset
  B(\phi(z), C h).  \eeq 
Now, let $z_1, \dots, z_m$ denote a maximal
  $h$-packing of $(0,a)^{d_0}$.  Note that $m \asymp (a/h)^{d_0}$ when
  $h \leq 1$.  By definition $\|z_i - z_j\| \ge h$, so that
  $\|\phi(z_i) - \phi(z_j)\| \ge h/C$ since $\phi^{-1}$ is
  $C$-Lipschitz with $C \ge 1$.  Hence,
\[
\bigsqcup_{i=1,\dots,k} B(\phi(z_i), h/C) \subset B(A, h/C) \subset B(A,h),
\]
implying
\[
\sum_{i=1}^m \Vol(B(\phi(z_i), h/C)) \le \Vol(B(A, h)).
\]
We then conclude by the fact that $\sum_{i=1}^m \Vol(B(\phi(z_i), h/C))
\asymp m h^d \asymp a^{d_0}h^{d-d_0}$.
For the upper bound, we use the fact that $(0,a)^{d_0} \subset \cup_{i=1,\dots,k} 
B(z_i, h)$, so that
\[
A \subset \bigcup_{i=1,\dots,k} \phi(B(z_i, h)) \subset \bigcup_{i=1,\dots,k} B(\phi(z_i), C h),
\]
by \eqref{phi-ub}.  Hence, using the triangle inequality,
\[
\Vol(B(A, h)) \le \sum_{i=1}^m \Vol(B(\phi(z_i), C h + h)) \asymp m h^d \asymp a^{d_0}h^{d-d_0}.
\]
For the second part, we use the fact that $\partial B(0,1) =
\cup_\ell \phi_\ell((0,1)^{d-1})$ for a finite set of functions
$\phi_\ell$ satisfying the requirements and the fact that the composition $\phi \circ \phi_\ell$ is also Lipschitz.
\end{proof}

\subsubsection{Some statistics}

We establish here some bounds on the point-wise MSE \eqref{bias-var}
of LPR \eqref{local_poly}.  We mention that much finer results
exist in dimension $d=1$ for the case where the underlying function
$f$ is smooth; see~\cite{Fan_Gijbels96} and references therein.

\medskip
\begin{lem}[Variance] \label{lem:local_poly_var} 

For any sufficiently
  large constant $C>0$, depending only on  $d, r,$ the
  following is true.  Consider the LPR estimator of the form
  \eqref{local_poly}, with weights $\wnlm_{i,j} \in \{0,1\}$.  Assume
  that $B_i^{\rm in} \subset A_i := \{j: \wnlm_{i,j} = 1\} \subset
  B_{i,h} := \{j: x_j \in B(x_i, h)\}$, for some discrete ball
  $B_i^{\rm in}$ satisfying $|B_i^{\rm in}| \ge |B_{i,h}|/C$ for some constant $C$.

\beq \label{poly_var_lb} \frac1{C} \sigma^2 (nh)^{-d} \leq
  \Var(\wh{f}_i) \leq C \sigma^2 (nh)^{-d}.  
\eeq
\end{lem}

\begin{proof} We assume without loss of generality that $x_i = 0$ and
  drop the subscript $i$ for simplicity.  Below $C$ denotes a generic
  constant that may change with each appearance.
Let
\begin{equation}\label{eq:def_q}
 q = \sum_{s = 0}^r {s + d-1 \choose d-1}={r + d \choose d},
\end{equation}
which is the number of monomials in $d$ variables of degree $r$ or
less.
Let $\bX$ denote the $|A| \times q$ matrix with coefficients $(x_j^s:
j \in A, |s| \leq r)$.  By definition of the local polynomial
estimator \eqref{local_poly} and the usual least squares formula, we
have
\begin{equation}\label{eq:LPR_system}
\wh{f} = \be^T (\bX^T \bX)^{-1} \bX^T \by,
\end{equation}
where $\be = (1, 0, \dots, 0) \in \R^q$ and $\by := (y_j: j \in A)$
(assuming that $\bX$ is full-rank, which we prove further
down).  In particular,
\[
\Var(\wh{f}) = \sigma^2 \be^T (\bX^T \bX)^{-1} \be,
\]
since $y_j = f(x_j) + \eps_j$, with the noise $(\eps_j)$ being uncorrelated and having identical variance $\sigma^2$.  

Let $z_j = x_j/h$ and $\bZ = (z_j^s: j \in A, |s| \leq r)$, and also
let $\bH = \diag(h^{|s|}, |s| \leq r)$, so that $\bX = \bZ \bH$,
leading to \beq \label{poly_varZ} \Var(\wh{f}) = \sigma^2 \be^T
\bH^{-1} (\bZ^T \bZ)^{-1} \bH^{-1} \be = \sigma^2 \be^T (\bZ^T
\bZ)^{-1} \be, \eeq since $\bH^{-1} \be = \be$.  This is because $\bH$ is an invertible diagonal matrix with first element equal to 1.  The
reason we work with $\bZ$ instead of $\bX$ is that, under the
conditions assumed here, $z_j \in [-1,1]^d$
(because $j\in A$) and $(nh)^{-d} \bZ^T \bZ$
is bounded from above and below in terms of its spectrum.  Indeed,
define matrices $\bZ_1 = (z_j^s: j \in B^{\rm in}, |s| \leq r)$,
$\bZ_2 = (z_j^s: j \in A \setminus B^{\rm in}, |s| \leq r)$, $\bZ_3 =
(z_j^s: j \in B_{h}, |s| \leq r)$ and $\bZ_4 = (z_j^s: j \in B_{h}
\setminus A, |s| \leq r)$.  Let $\prec$ denote the ordering for
positive semi-definite matrices.  Since
\[
\bZ_1^T \bZ_1 \prec \bZ_1^T\bZ_1 + \bZ_2^T \bZ_2  = \bZ^T \bZ = \bZ_3^T\bZ_3 - \bZ_4^T \bZ_4 \prec \bZ_3^T \bZ_3,
\]
it suffices that we focus on proving a lower bound on the spectrum of
$\bZ_1^T \bZ_1$ and upper bound on the spectrum of $\bZ_3^T \bZ_3$.
Consider therefore the case where $A$ itself is a discrete ball, say
$A = \{j: x_j \in B(x, a h)\}$, where $a \in (C^{-1/d},1)$ by assumption.  
Let $z = x/h$.  First, assume that $a$ and $h$ remain fixed.  
Then for $s,t \in \N^d$ such that $|s| \vee |t| \leq r$, we
have 
\beq \label{mst} \frac1{(nh)^{d}} (\bZ^T \bZ)_{st} =
(nh)^{-d}\sum_{j \in A} z_j^{s+t} \to M_{st} := \int_{B(z, a)} u^{s+t}
du, \quad \text{ when } nh \to \infty, 
\eeq 
recognizing a Riemann sum
on the LHS.  So, if $\bM = (M_{st} : |s| \vee |t| \leq r)$, we have
the convergence $(nh)^{-d} \ \bZ^T \bZ \to \bM,$ when $nh \to \infty$.
$\bM$ is a well-defined positive semi-definite matrix since its
elements are bounded by 1 --- because $B(z, a) \subset B(0,1)$  ---
so we only need to show that it is positive uniformly over $a \in
(C^{-1/d},1)$.  Let $\lambda_{z,a}$ denote the smallest eigenvalue of
$\bM$ with integral over $B(z, a)$, with $z \in B(0,1)$ and $a \in
(C^{-1/d},1)$.  We want to show that $\lambda_{z,a}$ is bounded away
from 0.  Suppose this is not the case, that there are sequences $(z_m,
a_m)$ such that $\lambda_{z_m,a_m} \to 0$ as $m \to \infty$.  By
compacity, we may assume that $(z_m, a_m) \to (z_\infty, a_\infty) \in
\overline{B(0,1)} \times [C^{-1/d},1]$.
Then $\lambda_{z_\infty,
  a_\infty} = 0$, by continuity.  Let $\bM_\infty$ be the associated
matrix.  Then there is $b_\infty \in \R^q$ nonzero such that
\[
0 = b_\infty^T \bM_\infty b_\infty = \int_{B(z_\infty, a_\infty)} \sum_{s,t} b_{\infty_s} b_{\infty,t} u^{s+t} du = \int_{B(z_\infty, a_\infty)} \Big(\sum_{s} b_{\infty,s} u^{s} \Big)^2 du,
\]
where the sums are over $s \in \N^d$ such that $|s| \leq r$.  This
leads to a contradiction since the polynomial in the second integral
cannot be zero on a nonempty ball.  

So far, we assumed that $a$ and $z$ were fixed.  Assume this is not the case.  The upper bound on the largest eigenvalue of $\bZ^T \bZ$ is bounded in the exact same way, using the fact that $\|z_j\| \le 1$. For the lower bound we still have that
\[
\liminf \ (nh)^{-d}\sum_{j \in A} z_j^{s+t} \ge \inf_{z', a'} \int_{B(z', a')} u^{s+t},
du
\]
where the inf is over $z'$ and $a'$ such that $a' \in
(C^{-1/d},1)$ and $B(z', a') \subset B(0,1)$. Our arguments apply to the RHS. 
We conclude that there is $C_1 \in
(0,\infty)$ such that, for $nh$ large enough, \beq \label{ZZ}
\frac1{C_1} (nh)^{d} \leq \lambda_{\rm min}(\bZ^T \bZ) \leq
\lambda_{\rm max}(\bZ^T \bZ) \leq C_1 (nh)^{d}.  \eeq We then redefine
$C$ as $\max(C, C_1)$ and conclude with \eqref{poly_varZ}.
\end{proof} \medskip

\medskip

\begin{lem}[Bias: Upper Bound] \label{lem:local_poly_ub} Assume that
  $f \in \cF^{\rm cartoon}(\alpha, \cf)$, with foreground $\Omega$,
  and that the conditions of \lemref{local_poly_var} also
  hold. If moreover either $A_i \subset \Omega$ or $A_i \subset
  \Omega^c$, then, for some constant $C>0$, the following inequality holds
  \beq \label{poly_bias_ub} (\E \wh{f}_i - f_i)^2 \le \min(1, C h^{2
    \alpha}).  \eeq
\end{lem}

\begin{proof}
  We continue with the notation introduced in the proof of
  \lemref{local_poly_ub}.  WLOG, assume $A \subset \Omega$.  In
  that case, $f$ is smooth in the window, since $f = \fin$.  By
  \lemref{taylor}, $f(x_j) = T_0^{\afloor} f(x_j) + g(x_j)$, where
  $T_0^{\afloor} f$ is a polynomial of degree at most $\afloor \leq r$,
 and $|g(x_j)| \leq c_\alpha \cf h^\alpha$ for all $j \in A$.
Now, for a polynomial $p$ of degree at most $r$, let $\bp = (p(x_j) :
j \in A)$ so that $\bp = \bX \ba$ for some $\ba \in \R^q$, and we have
\[
\be^T (\bX^T \bX)^{-1} \bX^T \bp = \be^T \ba = a_0 = p(0).
\]
With this reproducing formula and the fact that $T_0^{\afloor} f(0) = f(0)$, 
\[
\E \wh{f} - f(0) = \be^T (\bX^T \bX)^{-1} \bX^T \bg = \be^T (\bZ^T \bZ)^{-1} \bZ^T \bg.
\]
Because of \eqref{ZZ}, we have 
\beq \label{Z-res} 
|\be^T (\bZ^T\bZ)^{-1} \bZ^T \bg| \leq \frac{\|\be\|_2 \cdot \|\bZ^T
  \bg\|_2}{\lambda_{\rm min}(\bZ^T \bZ)} \leq C_1 (nh)^{-d} \|\bZ^T
\bg\|_2,  
\eeq 
where $C_1$ is the constant of \lemref{local_poly_var}.
But the entries $(\bZ^T \bg)_s = \sum_{j \in A} g(x_j)
z_j^s$ are uniformly bounded by $|A| \cdot c_\alpha \cf h^\alpha =
O((nh)^d h^\alpha)$, so that the RHS in \eqref{Z-res} is of order
$O(h^\alpha)$.  This imply that $(\E \wh{f} - f(0))^2 \le C_2
h^{2\alpha}$ for some constant $C_2$, and we conclude by redefining $C$ as 
$\max(C, C_1, C_2)$.
\end{proof} \medskip

\medskip
\begin{lem}[Bias: Lower Bound] \label{lem:local_poly_lb} Let $f = {\bf
    1}_\Omega$, where $\Omega = (0, 1/2) \times (0,1)^{d-1}$ and
  linear filtering (meaning $\wnlm_{i,j} = 1$ if, and only if, $\|x_i
  - x_j\| \le h/2$).  Then there is a constant $C > 0$ such that, when
  $\dist(x_i, \partial \Omega) \le h/C$, we have
  \beq \label{poly_bias_lb} (\E \wh{f}_i - f_i)^2 \ge 1/C.  \eeq
\end{lem}

\begin{proof}
  We continue with the notation introduced in the proof of
  \lemref{local_poly_ub}.  In particular, we translate everything
    so that $x_i = 0$. WLOG, assume that $x_i \in \Omega$ and let
  $\delta = \dist(x_i, \partial \Omega)$.  Let $A_\Omega = \{j: x_j
  \in B_h \cap \Omega\}$ and define $A_{\Omega^c}$ similarly.
  Using the reproducing formula, we get 
  \beq \label{lb1} \E \wh{f}
  -f(x_i) = \be^T (\bZ^T \bZ)^{-1} \bZ^T {\bf 1} _{A_\Omega} - \be^T
  (\bZ^T \bZ)^{-1} \bZ^T {\bf 1}  = - \be^T (\bZ^T \bZ)^{-1} \bZ^T {\bf
    1}_{A_{\Omega^c}}.  
    \eeq 
    Assume that $\delta < h$, in which case
 $A_{\Omega^c} = \{j: x_j \in (\delta, h) \times
  (-h,h)^{d-1}\}$, in which case the RHS in \eqref{lb1} is equal to
  $-G(\delta/h)$, where
\[
G(a) := \be^T (\bZ^T \bZ)^{-1} \bZ^T  {\bf 1} _{\{j : \, z_j \in (a, 1) \times (-1,1)^{d-1}\}}.
\]  
It suffices to show that there is $C > 0$ such that $G(a) \ge 1/C$
when $a \le 1/C$.  Assume this is not the case, in which case there is
$a_m \to 0$ such that $G(a_m) \to 0$.  As in the proof of
\lemref{local_poly_ub}, recognizing Riemann sums we see that, as $m
\to \infty$,
\[
\frac1{(nh)^d} \bZ^T \bZ \to \bM := \int_{(-1,1)^d} \bp(z) \bp(z)^T dz, \quad \bp(z) := (z^{s}: |s| \le r), 
\]
and
\[
\frac1{(nh)^d} \bZ^T {\bf 1} _{\{j : \, z_j \in (a, 1) \times (-1,1)^{d-1}\}}  \to \bv := \int_{(0,1) \times (-1,1)^{d-1}} \bp(z) dz.
\]
Let
\[
\bu = \int_{(-1,0) \times (-1,1)^{d-1}} \bp(z) dz.
\]
We have $\bu + \bv = \bM \cdot {\bf 1} $ (where ${\bf 1} = (1, \dots, 1)$), so that $\be^T \bM^{-1} (\bu + \bv)
= 1$, and therefore $\be^T \bM^{-1} \bv = 1/2$ by symmetry.  Hence,
$G(a_m) \to 1/2$, which is a contradiction.
\end{proof} \medskip 

This lemma states a lower bound on the squared bias of linear
filtering when $f$ is an indicator function of a half hypercube.  The
result is actually much more general. Any function $f$ in the
cartoon class has a foreground $\Omega$ whose boundary is
well-approximated by a hyperplane --- since $\partial \Omega$ is
Lipschitz --- and $f$
is approximately locally piecewise constant ($f$ is smooth on $\Omega$ and $\Omega^c$).  Hence, near the
discontinuity, $f$ resembles the function in \lemref{local_poly_lb}.

\subsubsection{Some probability}

The following result asserts that the maximum of $m$ identically
distributed random variables with exponentially decaying tails is at
most a power of $\log m$.

\begin{lem} \label{lem:decay}
Suppose $X_1, \dots, X_m$ are such that for some $a,b,c > 0$,
\[
\pr{|X_r| > t} \leq c \exp(-(t/a)^b), \quad \forall t > c, \quad \forall r = 1, \dots, m.
\]
Then for $m$ sufficiently large,
\[
\pr{\max(|X_1|, \dots, |X_m|) >  a (2 \log m)^{1/b}} \leq c/m.
\]
\end{lem}

\begin{proof} Define $x_m =  a (2 \log m)^{1/b}$.  
By the union bound,
\begin{align*}
\pr{\max(|X_1|, \dots, |X_m|) > x_m} 
&\leq \pr{|X_1| > x_m} + \cdots \pr{|X_m| > x_m} \\
&\leq m c \exp(-(x_m/a)^b) \\ 
&= c m^{-1} \to 0. \\[-1.49cm]
\end{align*}
 \hspace{9.5cm}
\end{proof}
\vspace{0.5cm}

\begin{lem} \label{lem:normal-max}
For $X_i \sim \cN(0, \sigma_i^2)$ for $i =1, \dots, m$, and any $C > 0$, we have
\[
\pr{\max_{1\leq i\leq m} |X_i| > \max_{1\leq i\leq m} \sigma_i \sqrt{2 C \log m}} \leq m^{1-C}.
\]
\end{lem}

\begin{proof} Fix $t \geq 1$ and let $\sigma = \max_{1\leq i\leq m}
  \sigma_i$.  By the union bound and the fact that $\pr{\cN(0,1) > t}
  \leq \exp(-t^2/2)$, we have
\[
\pr{\max_{1\leq i\leq m} |X_i| > t} \leq \sum_{i=1}^m \pr{|X_i| > t} \leq \sum_{i=1}^{m} \exp(-t^2/(2 \sigma_i^2)) \leq m \exp(-t^2/(2 \sigma^2)).
\]
We then plug in $t = \sigma \sqrt{2 C \log m}$.
\end{proof} \medskip

\begin{lem} \label{lem:chi-max}
Suppose $X_i \sim \chi_k^2$ for $i = 1, \dots, m$.  There is a constant $C_1$ such that, for any $C > C_1$, if $k \ge (64/9) C \log(m)$, we have
\begin{align}
\pr{\max_{1 \leq i\leq m} X_i > k + 2 \sqrt{C k \log m}} \leq m^{1-C/2} \label{eq:chi2_max}\\
\pr{\min_{1 \leq i\leq m} X_i < k - 2 \sqrt{C k \log m}} \leq m^{1-C/2} \label{eq:chi2_min}
\end{align}
\end{lem}
\begin{proof}
Let us prove the first inequality. Since the moment generating function of a $\chi_k^2$ is
$t\to(1-2t)^{-k/2}\1{(t<1/2)}$,
Chernoff's bound gives
\[
\pr{X_i > t} \le \exp\left( - (t-k)/2 + (k/2) \log (t/k) \right), \quad \forall t > k.
\]
We then use the inequality $\log(1 + x) \leq x - x^2/2 + x^3/3$, valid for $x \in (0,1)$ and input $t = k + 2 \sqrt{C k \log m}$, to get
\beqn
-(t-k)/2 + (k/2) \log (t/k) &\leq& - (t-k)^2/(4k) + (t-k)^3/(6 k^2) \\
&=& -C \log m + (4/3) C^{3/2} \sqrt{\log(m)/k} \log m,
\eeqn
and bound the second term by $(C/2) \log m$.  We then obtain
\[
\pr{X_i > t} \le m^{-C/2},
\]
and apply the union bound as before.
The second inequality is proved in the same way considering that 
$\log(1 - x) \geq -x - x^2/2 - x^3/3$ holds for $x \in (0,1)$. 
\end{proof}
 \medskip

\begin{lem} \label{lem:chi-non} Suppose $X_i \sim
  \chi_k^2(\delta_i^2)$ (non-central chi-square) for $i = 1, \dots,
  m$.  There is a constant $C_1$ such that, for any $C > C_1$, if $k
  \ge 16 C \log(m)$ and $\delta_{\rm min} := \min_i \delta_i \geq
  2 \sqrt{C \log m}$ , we have
\[
\pr{\min_{1\leq  i\leq m} X_i < \delta_{\rm min}^2/4 + k - 3 \sqrt{C k \log m}} \leq 2 m^{1-C/2}.
\]
Similarly, if $\delta_{\rm max} = \max_i \delta_i \le \sqrt{C \log m}$, we have
\[
\pr{\max_{1\leq i\leq m} X_i > k + 3 \sqrt{C k \log m}} \leq m^{1-C/2}.
\]
\end{lem}
\begin{proof}
First, $X_i \equiv (Z_{i} + \delta_i)^2 + Y_i$, where $Z_{i} \sim \cN(0,1)$ and $Y_i \sim \chi_{k-1}^2$ are i.i.d. 
Hence, 
\[
\min_{1\leq  i\leq m} X_i \ge \min_{1\leq  i\leq m} (Z_i + \delta_i)^2 + \min_{1\leq  i\leq m} Y_i.
\]
Let $E_i = \displaystyle\{\max_{1\leq  i\leq m} |Z_i| \ge \sqrt{C \log m}\}$.  By \lemref{normal-max}, we have
\[
\pr{E_i} \le m^{1 - C/2}.
\]
Let 
\[
F_i = \displaystyle\{\min_{1\leq  i\leq m} Y_i \le  k-1 - 2 \sqrt{C (k-1) \log m}\} .
\] 
To control the $Y_i$'s, we apply inequality \eqref{eq:chi2_min} to get
\[
\pr{F_i} \le m^{1-C/2}.
\]
Under $E_i^c \cap F_i^c$, we have $\min_i X_i \ge \delta^2/4 + k - 3 \sqrt{C k \log m}$ and
\[
\pr{E_i^c \cap F_i^c} = 1 - \pr{E_i \cup F_i} \ge 1 - \pr{E_i} - \pr{F_i} \ge 1 - 2 m^{1-C/2}.
\]
This proves the bound on $\min_i X_i$; arguments for $\max_i X_i$ are similar and simpler.
\end{proof} \medskip

\subsection{Proofs of the main results}

\subsubsection{Proof of \thmref{linear}}
We start with the upper bound.  Fix $f \in \cF^{\rm cartoon}(\alpha,
\cf)$ with foreground $\Omega$.  Let $Q = \{i: \dist(x_i, \partial
\Omega) \leq h\}$.  For $i \in Q$, we use the fact that $|\wh{f}_i -
f_i| \le 1$, which implies $\E(\wh{f}_i - f_i)^2 \leq 1.$ For $i
\notin Q$, from \lemref{local_poly_var} and \lemref{local_poly_ub},
coupled with the bias-variance decomposition \eqref{bias-var}, we get
\[
\E(\wh{f}_i - f_i)^2 \leq C ( h^{2\alpha} + \sigma^2 (nh)^{-d} ).
\]
Using \lemref{count}, we have $|Q| \leq 4^d n^d |B(\partial \Omega,
h)|$, while and $|B(\partial \Omega, h)| = O(h)$ by \lemref{vol} and
the fact that $|\partial \Omega|$ is of order 1.  Summing over all $i
\in I_n^d$, we get
\[
\mse_f(\wh{f}) \leq \frac{n^d - |Q|}{n^d} \cdot C ( h^{2\alpha} + \sigma^2 (nh)^{-d} ) + \frac{|Q|}{n^d} \cdot C (1 + \sigma^2 (nh)^{-d} ) \leq C_1 (h + \sigma^2 (nh)^{-d}).
\]
Minimizing the RHS with respect to $h$ yields the upper bound in
\thmref{linear}.

For the lower bound, redefine $Q = \{i: \dist(x_i, \partial \Omega)
\leq h/C_1\}$, where $C_1$ is the constant of \lemref{local_poly_lb}.
For $i \notin Q$, we use \lemref{local_poly_var} and the
bias-variance decomposition \eqref{bias-var}, to get
\[
\E(\wh{f}_i - f_i)^2 \geq \frac1{C_2} \sigma^2 (nh)^{-d}.
\]
For $i \in Q$ we use \lemref{local_poly_lb} and the bias-variance
decomposition \eqref{bias-var}, to get
\[
\E(\wh{f}_i - f_i)^2 \geq \frac1{C_1}, 
\]
Using \lemref{count} again, we have the following lower bound on the
MSE (for $n$ large enough),
\[
\mse_f(\wh{f}) \geq \frac{n^d - |Q|}{n^d} \cdot \frac1{C_2} \sigma^2 (nh)^{-d} + \frac{|Q|}{n^d} \cdot \frac1{C_1} \geq C_3 (h + \sigma^2 (nh)^{-d}).
\]
Minimizing the RHS with respect to $h$ leads to the lower bound in \thmref{linear}.

\subsubsection{Proof of \thmref{MO}}
The proof for the upper bound is the same as that of \thmref{linear}
in smooth regions, leading to an upper bound on the MSE of the form
\[
\mse_f(\wh{f}) \leq C ( h^{2\alpha} + \sigma^2 (nh)^{-d} ).
\] 
Then minimizing this quantity over $h$ gives the stated result.  The
lower bound is a well-known minimax bound~\cite[Theorem 5.1.2,
p.\ 133]{Korostelev_Tsybakov93}. 
\medskip
\subsubsection{Proof of \thmref{BO}}
\label{sec:proof-BO}
Let $\delta(x) = \dist(x, \Omega)$.  The proof is similar to that of
\thmref{MO}, except that the variance varies by location.  The
point bias is of order $O(h^\alpha)$ everywhere, because the
smoothing window is of radius at most $h$, with all
points in the window being on the same side of the discontinuity.
However, the point variance is of order $O(\sigma^2 \lceil n
\delta(x_i) \rceil^{-d})$, since the window is of radius
$\delta(x_i)$ (immediate consequences of
\lemref{local_poly_var}).

Let us sum the point variances over all the pixels in the image.  The
situation is different according to the dimension.  We start with
$d=1$, so that $\Omega = (a,b) \subset (0,1)$.  For $\delta$ small
enough, there are exactly four points at distance {less than} $\delta$ from
$\partial \Omega$ (two on each side of the two jump locations).  Let's
consider the sample points $x_i \in [b,1)$, and let $j$ be such that
$x_{j-1} < b \leq x_j$.  Note that $j = bn (1 + o(1))$, and we assume
that $b$ is fixed.  For $i \in [j, j + nh]$, the variance is bounded
by $C \sigma^2/(i-j+1)$, while for $i \geq j + nh$ (in the smooth region),
the variance is of order $O(\sigma^2/(nh))$ as before.  Hence, summing
over $i \geq j$, the averaged variance in that region is bounded by
\[
\frac{C \sigma^2}{n-nh-j} \left(\sum_{i = j}^{j+[nh]} \frac1{i-j+1} +   \frac{n - nh-j}{nh}\right) = O(\frac{\sigma^2}{n})\left( \sum_{k=1}^{nh} \frac1k  +1\right) = O(\frac{\sigma^2 \log(n)}{n}).
\]
The same is true for all the other three regions.  

When $d \geq 2$, the story is just slightly different.  Define $Q_\ell
= \{i: \delta(x_i) \leq h 2^{-\ell}\}$ and let $\ell_0$ be such that
$h 2^{-\ell_0} < 2/n \leq h 2^{-\ell_0+1}$.  Stratifying, we have the
following bound on the averaged variance 
\beq 
  \frac{C \sigma^2}{n^d} \left(\sum_{\ell=0}^{\ell_0} \ \sum_{i \in
      Q_\ell \setminus Q_{\ell+1}} (nh 2^{-\ell-1})^{-d} + \sum_{i
      \notin Q_0} (nh)^{-d} \right)= \frac{C \sigma^2}{n^d}
  \sum_{\ell=0}^{\ell_0} \ |Q_\ell\setminus Q_{\ell+1}| (n h)^{-d}
  2^{d(\ell+1)} + \frac{C \sigma^2 }{(nh)^{d}}.  \nonumber
\eeq
By \lemref{count} and \lemref{vol}, we have $ |Q_\ell \setminus
  Q_{\ell+1}| \leq |Q_\ell| \leq C_1 n^d \cdot h 2^{-\ell}$, for some
constant $C_1$.  Hence, the first sum on the RHS of the last equation
is bounded by
\[
C_2 \sigma^2 h (n h)^{-d} \sum_{\ell=0}^{\ell_0} 2^{(d-1)\ell} \leq C_3 \sigma^2 h (n h)^{-d} \cdot 2^{(d-1) \ell_0} = O(\sigma^2/n).
\]

This leads to an upper bound on the MSE of the form
\beq \label{BO2}
\mse_f(\wh{f}) \leq C ( h^{2\alpha} + \sigma^2 A_n /n + \sigma^2 (nh)^{-d} ).
\eeq
Minimizing this quantity over $h$ gives the upper bound stated in \thmref{BO}.  

For the lower bound, we know from minimax results
underlying the lower bound in \thmref{MO} that there are functions $f$
in the cartoon class where the bias in the smooth regions is of order
at least $h^{\alpha}$.  As for the variance, our upper bound for the
averaged variance is easily seen to be lower bounded (up to a
multiplicative constant).  This leads to a lower bound identical to
\eqref{BO2} modulo a multiplicative constant, and optimizing it leads
to the lower bound in \thmref{BO}.  We omit details.
\medskip

\subsubsection{Proof \thmref{Y}}
\label{sec:proof-Y}
Fix $i \in I_n^d$ and let $\eta_i = \max_{j \in B(i,nh)} |\eps_j|$.
Then by the union bound and \eqref{F},
 \beq \label{eta-union}
\P(\eta_i \geq t) \leq |B(i,nh)| \max_{j \in B(i,nh)} \P(|\eps_j| \geq
t) \leq (2nh+1)^d \, (1-F(t/\sigma)) =: p.  \eeq Hence, with
probability at least $1-p$, the event $E_i := \{\eta_i \leq t\}$ holds
true.
WLOG, assume that $x_i \in \Omega$.  Since $\fin$ is $\cf$-Lipschitz,
 we have $|f(x_i) - f(x_j)| = |\fin(x_i) - \fin(x_j)|
\leq \cf h$ when $x_j \in \Omega \cap B(x_i, h)$, and by the triangle
inequality, $|y_i - y_j| \leq \cf h + |\eps_i - \eps_j| \leq \cf h + 2
t$ under $E_i$.
Suppose there is $x_j \in \Omega^c \cap B(x_i, h)$.  In that case,
there is $x \in \partial \Omega \cap B(x_i,h)$ and we have \beqn
|f(x_i) - f(x_j)| &\geq& |\fin(x) - \fout(x)| - |\fin(x_i) - \fin(x)| -|\fout(x_j) - \fout(x)| \\
& \geq& \mu(f) - 2\cf h \geq 1/\cf - 2 \cf h, \eeqn again by the
triangle inequality and the fact that $\fout$ is also $\cf$-Lipschitz, this implies that
\[
|y_i - y_j| \geq 1/\cf - 2 \cf h - |\eps_i - \eps_j| \geq 1/\cf - 2 \cf h - 2t,
\]
under $E_i$.  
We see that we need  $h_y \ge \cf h + 2 t$ to ensure that
sample points $x_j \in \Omega \cap B(x_i, h)$ are selected, while we
require that $h_y < 1/\cf - 2 \cf h - 2t$ so that points $x_j \in
\Omega^c \cap B(x_i, h)$ are disregarded.  These two inequalities
are, for example, satisfied when $h_y = 1/(3\cf)$ and $t = 1/(6\cf)$,
and $h$ sufficiently small --- by our assumptions, $h = o(1)$.
 Assume $h_y$ and $t$ are chosen that way.  Then, when $E_i$
holds, the photometric kernel in YF is able to
exactly mimic the MO.

We now turn to bounding the MSE.  First, we have
\[
\E (\wh{f}_i^{\rm YF} - f_i)^2 = \E[ (\wh{f}_i^{\rm YF} - f_i)^2 \one{E_i}] + \E[ (\wh{f}_i^{\rm YF} - f_i)^2 \one{E_i^c}].
\]
Since $\wh{f}_i^{\rm YF} = \wh{f}_i^{\rm MO}$ on $E_i$, 
\[
\E[ (\wh{f}_i^{\rm YF} - f_i)^2 \one{E_i}] = \E[ (\wh{f}_i^{\rm MO} - f_i)^2 \one{E_i}] \leq \E(\wh{f}_i^{\rm MO} - f_i)^2.
\]
And since $|\wh{f}_i^{\rm YF} - f_i| \leq 1$ because of our clipping, we have 
\[
\E[ (\wh{f}_i^{\rm YF} - f_i)^2 \one{E_i^c}] \le \P(E_i^c) \le p.
\]
It remains to check that $p$ is negligible compared to MO risk given
in \thmref{MO}.  Indeed, using the fact that $t \asymp 1$, that $h \leq 1$ and that $\sigma \leq (C' \log n)^{-1/b}$, we have
\[
p = O(nh)^{d} \exp[-(t/(C\sigma))^b] = \exp[ (d  - t^b (C'/C^b)) \log n] = o(\sigma^2/n^d)^{2\alpha/(d+2\alpha)},
\]
when $C'$ is sufficiently large, implicitly assuming that $\sigma$ is
at least a polynomial in $n$, for otherwise the trivial estimator
$\wh{y} = y$ is optimal. This concludes the proof.

\medskip

\noindent {\bf When the noise level is not small.}
Assume $\sigma$ is fixed, for simplicity.  Note YF is identical to LF
when $h_y \to \infty$ sufficiently fast.  Assume therefore that $h_y
\le h_0$ for some fixed $h_0 < \infty$.  We now argue that YF is
essentially useless when this is the case.  Concretely, assume the
reverse of \eqref{F}, meaning 
\beq\label{F-lb} 
\P(|\eps_i| \le t) \le F(t/\sigma), \ \forall t, \ \forall i \in I_n^d.  
\eeq
We show that, when
$F(2h_0/\sigma) < 1$, YF has an overall squared bias (and therefore
MSE) of order 1, which is comparable to the trivial estimator $\wh{f}
= y$. In other words, for large noise and
relatively small $h_y$, YF can perform {\em worse} than LF. For
example, the bias is at least $h_0$ at locations $i$ satisfying
$|\eps_i| \geq h_0 + h_y$.  Indeed, we are averaging over values $y_j$
such that $|y_j - y_i| \leq h_y$, so that $|\wh{f}_i - y_i| \leq h_y$
and therefore $|\wh{f}_i - f_i| \geq |\eps_i| - |\wh{f}_i - y_i| \geq
(h_0 + h_y) - h_y = h_0$.  Moreover, by \eqref{F-lb}
\[
\pr{|\eps_i| \geq h_0 + h_y} \geq 1 - F(2h_0/\sigma) > 0.
\]
Integrating the squared bias over these sample points alone
leads to a lower bound of order 1.

\medskip
\subsubsection{Proof of \thmref{NLM}}
\label{sec:proof-NLM}

For simplicity, we ignore boundary issues and in particular assume
that all patches are of same size, with $m_{\patch} \asymp (nh_{\patch})^d$
sample points each, and similarly for spatial windows, with $m_h
\asymp (nh)^d$ sample points each.

{\bf Upper bound for NLM-average.}
For $i \in I_n^d$ such that $\patch_i \cap \partial \Omega \neq \emptyset$, we use the fact that $|\wh{f}_i - f_i| \le 1$ to get
\[
\E (\wh{f}_i - f_i)^2 \leq 1.
\]
Consider $i$ with $\patch_i \cap \partial \Omega = \emptyset$.
WLOG,  assume  $\patch_i \subset
\Omega$.
Take any $j \in B(i, nh)$.  By definition,
\beq \label{patch-decomp}
\ypbarj - \ypbari = \fpbarj  - \fpbari  + \epbarj  - \epbari .
\eeq
For the noise part we have 
\beq \label{eps-diff}
\epbarj - \epbari \sim \cN(0,
\sigma^2 \Delta_{ij}/m_{\patch}^2), 
\eeq 
where $\Delta_{ij}$ is the number of sample points in the symmetric difference of $\patch_i$ and $\patch_j$.  By \lemref{normal-max}, we have
that 
\beq \label{Qeps_ub} \max_{j \in B(x_i, h)}
|\epbarj - \epbari | \leq \zeta
:= 2 \sigma  \sqrt{ C \log(m_h)/m_{\patch}}, 
\eeq 
with probability at
least $1 - m_h^{1 - C}$.  In the sequel, we fix $C$ large and denote
by $E_i$ the event \eqref{Qeps_ub}.
For the signal part, we have the following \beqn \fpbari - \fpbarj
&=& \frac1{m_{\patch}} \sum_{x_k \in \patch_i} f(x_k) - \frac1{m_{\patch}} \sum_{x_k \in \patch_j} f(x_k) \\
&=& \frac1{m_{\patch}} \sum_{x_k \in \patch_0} (f(x_k + x_i) -
f(x_k+x_j)), \eeqn where $\patch_0$ is a generic patch centered at 0.
If $x_j \in \Omega$ with $\patch_j \subset \Omega$, then, since $\fin$
is $\cf$-Lipschitz,
\begin{eqnarray}
|\fpbari  - \fpbarj | 
& \leq& \frac1{m_{\patch}} \sum_{x_k \in \patch_0} |\fin(x_k + x_i) - \fin(x_k + x_j)| \nonumber \\
& \leq& \cf |x_i - x_j| \leq \cf h.  \label{Qf_ub}
\end{eqnarray}
If $x_j \in \Omega^c$, then there is a point $x \in B(x_i, h)
\cap \partial \Omega$, and we have $f(x_k) =\fin(x)+ [\fin(x_k)
-\fin(x)] $ for $x_k \in \Omega$ and $f(x_k) = \fout(x)+[\fout(x_k)
-\fout(x)]$ for $x_k \in \Omega^c$, with $|\fin(x_k) -\fin(x)| \le \cf
h$, $|\fout(x_k) -\fout(x)| \le \cf h$ and $|\fin(x) - \fout(x)| \geq
1/\cf$.  Hence, \beqn \fpbari - \fpbarj
&=& \frac1{m_{\patch}} \sum_{x_k \in \patch_i} f(x_k) - \frac1{m_{\patch}} \sum_{x_k \in \patch_j} f(x_k) \\
&=& \fin(x) + \frac1{m_{\patch}} \sum_{x_k \in \patch_i} [\fin(x_k) -\fin(x)] \\
&& - \fin(x) \frac{|\patch_j \cap \Omega|}{|\patch_j|} -  \frac1{m_{\patch}} \sum_{x_k \in \patch_j \cap \Omega} [\fin(x_k) -\fin(x)] \\
&& - \fout(x) \frac{|\patch_j \cap \Omega^c|}{|\patch_j|} -  \frac1{m_{\patch}} \sum_{x_k \in \patch_j \cap \Omega^c} [\fout(x_k) -\fout(x)]\\
&=& (\fin(x) - \fout(x)) \frac{|\patch_j \cap \Omega^c|}{|\patch_j|} +
R, \eeqn where $|R| \le 2 \cf h$.  We now use \lemref{rho} to bound
the fraction above from below by $(2\cf)^{-d}$, to get
\beq \label{Qf_lb} |\fpbari - \fpbarj | \ge (2\cf)^{-d} \mu - 2
\cf h.  \eeq 

Using the decomposition \eqref{patch-decomp}, coupled with the
triangle inequality and \eqref{Qeps_ub}, \eqref{Qf_ub} and
\eqref{Qf_lb}, we see that we need to choose $h_y$ such that
\beq \label{hy1} 
\cf h + \zeta \leq h_{y} < (2\cf)^{-d} \mu - 2 \cf h
- \zeta.  
\eeq 
The lower bound is to ensure that all the points $x_j
\in B(x_i, h)$ such that $\patch_j \subset \Omega$ are included in the
neighborhood of $x_i$ (\ie $\wnlm_{ij} = 1$), while the upper
bound is to ensure that no points in $\Omega^c$ are included (under $E_i$). 
For points $x_j \in B(x_i, h)$ such that $\patch_j
\cap \Omega^c \neq \emptyset$, they may or may not be included,
depending on how large that intersection is.  Note that \eqref{hy1} is
satisfied when $h_y$ is a sufficiently small constant since we have $h
\to 0$, $\zeta \to 0$ and $\mu \asymp 1$ under our assumptions.  In
any case, we assume that \eqref{hy1} holds.

In terms of MSE, we proceed as follows.  Let $B_i = \{j: x_j \in B(x_i,
h)\}$, $B_i^0 = \{j: x_j \in B(x_i, h), \,\patch_j \subset \Omega\}$
and $A_i =\{j: \wnlm_{i,j} = 1\}$ --- the latter is a random subset of
$B_i$.  We saw that $A_i \supset B_i^0$ under $E_i$, which implies
\[
E_i \subset \{A_i \supset B_i^0\} \subset \bigcup_{B_i^0 \subset A \subset B_i} \{A_i = A\},
\]
leading to 
\beq \label{Ei}
\one{E_i} \le \sum_{B_i^0 \subset A \subset B_i} \one{\{A_i = A\}}.
\eeq
Using \eqref{Ei} and the fact that $|\wh{f}_i - f_i| \leq 1$, we have 
\beqn
\E (\wh{f}_i - f_i)^2 
&=& \E[ (\wh{f}_i - f_i)^2 \one{E_i}] + \E[ (\wh{f}_i - f_i)^2 \one{E_i^c}] \\
&\le& \sum_{B_i^0 \subset A \subset B_i} \P(A_i = A) \, \E[ (\wh{f}_{A} - f_i)^2] + \P(E_i^c),
\eeqn
where $\wh{f}_{A}$ is the local polynomial estimator based on $A\subset I_n^d$.    
For the second term, $\P(E_i^c) \le m_h^{1-C}$ by \eqref{Qeps_ub}.
For the first term, by \lemref{rho}, we know that $B(x_i, h) \cap
\Omega$ contains a ball of radius $C_1 h$ with $C_1 > 0$ depending
only on $\cf$ and $d$.  Hence, by the triangle inequality, $B(x_i, h)
\setminus B(\Omega^c, h_{\patch})$ contains a ball of radius $C_1 h - h_{\patch} \ge
C_1h/2$ (eventually), implying that $B_i^0$ contains a discrete ball
of radius at least $(C_1h/3) n \asymp nh$.  Therefore $|B_i^0|/|B_i|
\asymp 1$ and we may apply \lemref{local_poly_var} and
\lemref{local_poly_ub} to each $A$ in the sum above, to get
\[
\E (\wh{f}_A - f_i)^2 \le C_2 (h^{2 \alpha} + \sigma^2 (nh)^{-d}),
\]
for a constant $C_2$.  Hence, using the fact that $\sum_{B_i^0 \subset
  A \subset B_i} \P(A_i = A) \le 1$, we have
\[
\E (\wh{f}_i - f_i)^2 \le C_2 (h^{2 \alpha} + \sigma^2 (nh)^{-d}) + m_h^{1-C}.
\]
By our choice for $h$, $h^{2 \alpha} + \sigma^2 (nh)^{-d} \asymp
(\sigma^2/n^d)^{2\alpha/(d+2\alpha)}$, and we may choose $C$ large
enough so the last term on the RHS is negligible,
leading to an MSE at $i$ of order
$O(\sigma^2/n^d)^{2\alpha/(d+2\alpha)}$.

The MSE is of the same order when $x_i \in \Omega^c$, and
summing over all $i \in I_n^d$, we get
\[
\mse_f(\wh{f}) \leq \frac{|Q|}{n^d} + O(\sigma^2/n^d)^{2\alpha/(d+2\alpha)},
\]
where $Q := \{i: \patch_i \cap \partial \Omega \neq \emptyset\}$.
Since $Q \subset \{i: \dist(x_i, \partial \Omega) < h_{\patch}\}$, by
\lemref{count} and \lemref{vol}, we have $|Q| \leq C_2 n^d \cdot h_{\patch}$,
so that
\[
\mse_f(\wh{f}) \leq O(h_{\patch} + (\sigma^2/n^d)^{2\alpha/(d+2\alpha)}),
\]
Optimizing over $h_{\patch}$ subject to \eqref{hy1} being
  satisfied, we achieve the desired result.

\medskip
{\bf Upper bound for NLM.}  We follow the same arguments.  Here we
focus on $i \in I_n^d$ such that $\dist(x_i, \partial \Omega) > 2h_{\patch}$
(instead of $h_{\patch}$), and assume WLOG that $x_i \in \Omega$.  Take $j \in
B(i, nh)$ such that $\patch_j \cap \patch_i = \emptyset$.  Note that
this is true when $x_j \in \Omega^c$.  By definition,
\[
\bypj - \bypi = \bbfpj - \bbfpi + \bepj - \bepi.
\]
Since $\bepj - \bepi \sim \cN(0, 2 \sigma^2 \bI_\mp)$, we have
$\|\bypj - \bypi\|_2^2 \sim 2 \sigma^2 \chi_\mp^2(\|\bbfpj -
\bbfpi\|_2^2/(2 \sigma^2))$, with
\[
\|\bbfpj - \bbfpi\|_2^2 = \sum_{x_k \in \patch_0} (f(x_k + x_j) - f(x_k + x_i))^2.
\] 
If $x_j \in \Omega$ with $\patch_j \subset \Omega$, then 
\begin{eqnarray}
\|\bbfpj - \bbfpi\|_2^2 
&=& \sum_{x_k \in \patch_0} (\fin(x_k + x_i) - \fin(x_k + x_j))^2 \\
&\le& \mp \cf^2 \|x_i - x_j\|_2^2 \leq \mp \cf^2 h^2,  \label{f_ub}
\end{eqnarray}
since $\fin$ is $\cf$-Lipschitz.  By \lemref{chi-non}
and the fact that $ \mp \cf^2 h^2/\sigma^2 = o(1)$, we conclude that
\beq \label{Ei2} \max_j \|\bypj - \bypi\|_2^2 \le 2 \sigma^2 \mp +
\zeta_\chi, \quad \zeta_\chi := 6 \sigma^2 \sqrt{C \mp \log m_h}, \eeq
with probability at least $1 - m^{1-C/2}$, where the maximum is over
$j$ such that $x_j \in B(x_i, h)$ and $\patch_j \subset \Omega
\setminus \patch_i$.  
Let $E_i$ be this event.

If $x_j \in \Omega^c$, then there is a point $x \in B(x_i, h)
\cap \partial \Omega$.  Let $Q_j = \{x_k \in \patch_0: x_k + x_j \in
\Omega^c\}$.  For $x_k \in Q_j$, we use the decomposition 
\beqn
f(x_k + x_j) - f(x_k + x_i) &=& \fout(x_k + x_j) - \fout(x) \\
&& + \fout(x) - \fin(x) + \fin(x) - \fin(x_k+x_i),
\eeqn with the first and third differences bounded by $\cf h$ in
absolute value, and the second bounded from below by $\mu$ in absolute
value.  We therefore have
\begin{eqnarray}
\delta_{ij}^2 := \|\bbfpj - \bbfpi\|_2^2 
&\geq& \sum_{x_k \in Q_j} (\fin(x_k + x_i) - \fout(x_k + x_j))^2 \\
&\geq& |Q_j| (\mu - 2 \cf h)^2\ge \delta^2 := \mp (2\cf)^{-d} \mu^2/2, \label{f_lb}
\end{eqnarray}
where we used \lemref{rho} to bound $|Q_j|$ from
below and the fact that $\mu \asymp 1$ while $h = o(1)$.  Since
$\|\bypj - \bypi\|_2^2 \sim 2 \sigma^2
\chi_\mp^2(\delta_{ij}^2/(2\sigma^2))$ and $\delta_{ij} \ge \delta$,
with \lemref{chi-non} we see that \beq \label{Fi} \min_j \|\bypj -
\bypi\|_2^2 \ge \delta^2/4 + 2 \sigma^2 \mp - \zeta_\chi, \eeq with
probability at least $1 - m^{1-C/2}$, where the minimum is over $j$
such that $x_j \in \Omega^c \cap B(x_i, h)$.  Let $F_i$ denote this
event.

Assuming \eqref{Ei2} and \eqref{Fi} hold, we see that we need to
choose $h_y$ such that 
\beq \label{hy} 
2 \sigma^2 \mp + \zeta_\chi
\leq h_{y}^2 < \mp (2\cf)^{-d} \mu^2/8 + 2 \sigma^2 \mp -
\zeta_\chi.  
\eeq 
The lower bound is to ensure that all the points
$x_j \in B(x_i, h)$ such that $\patch_j \subset \Omega$ and $\patch_j
\cap \patch_i = \emptyset$ are included in the neighborhood of $x_i$
(meaning $\wnlm_{ij} = 1$), while the upper bound is to ensure that no
points $x_j \in \Omega^c$ are included (under $E_i$).  For
all other points $x_j \in \Omega \cap B(x_i, h)$, they may or may not
be included, depending on how large that intersection is.  Note that there is an $h_y$ satisfying \eqref{hy} if, and only if,
\[
\mp (2\cf)^{-d} \mu^2/8 > 2 \zeta_\chi \Leftrightarrow \mp > C_1 \sigma^4 \log n,
\]
for a constant $C_1$ which depends only on $d, \cf, \mu$. 
Assuming $\mp$ is that large, \eqref{hy} is satisfied when $h_y^2 = 2 (1+\eta)
\sigma^2 \mp$  with $\eta$ sufficiently small.  In any case, we assume
that \eqref{hy} holds and the rest of the proof is identical to the one for NLM-average.

{\bf Lower bound (heuristics).}  We discuss here the
lower bound and where the issues are.  Consider the important case where $\sigma$
is fixed and assume that $f = \one{\Omega}$,  where $\Omega = (0,1/2)
\times (0,1)^{d-1}$.
Consider direct neighbors (\ie points with distance $1/n^d$) $x_i \in \Omega$ and $x_j \in \Omega^c$.
For $x_k$ such that $\patch_k \cap (\patch_i \cup \patch_j) =
\emptyset$, we use \eqref{eps-diff} to arrive at
\[
\ypbark - \ypbari \sim \cN(\lambda_k, 2 \sigma^2/\mp),
\]
and
\[
\ypbark - \ypbarj \sim \cN(\lambda'_k, 2 \sigma^2/\mp),
\]
where $|\lambda_k - \lambda'_k| = |\patch_i \setminus \patch_j| =
\mp^{-1/d}$.   When $d \le 2$, the difference in means $\mp^{-1/d}$ is
of order at most that of the standard deviation $\mp^{-1/2}$, so that
these two distributions cannot be effectively separated.  {\em
  Heuristically}, this indicates that if the photometric kernel of
NLM-average includes $x_k$ in the neighborhood of $x_i$, it also includes it
in the neighborhood of $x_j$ with non-negligible probability.  This is
evidence that the squared bias is of order 1 at these points.  Since
there are order $(nh)^{d-1}$ such sample points, averaging over
them yields a lower bound on the squared bias (and therefore on the
MSE) of order $O(1/n)$.  The same heuristics could be applied to NLM.

The story changes for $d \ge 3$.  In fact, for any $f$ in the cartoon
model with similar foreground, NLM-average --- and NLM --- achieve a
much better risk.  To see this, fix $x_i \in \Omega$.  We already know
that NLM-average behaves well when $\patch_i \subset \Omega$; therefore
assume that $\patch_i \cap \Omega \neq \emptyset$.  If $x_j - x_i$  
is not parallel to $\partial \Omega$, then $|\fpbarj-\fpbari| \ge \mp^{-1/d}$,  so that under \eqref{eps-diff},
\[
|\ypbark - \ypbarj| \ge \mp^{-1/d} - \zeta \ge \mp^{-1/3} - \zeta.
\]
Noting that $\zeta \asymp \sqrt{\log(n)/\mp} = o(\mp^{-1/d})$,
if we choose $h_y \asymp \mp^{-2/5}$, then with high probability, the
neighborhood of $x_i$ only includes those $x_j \in B(x_i, h)$ such
that $x_j - x_i$ is parallel to $\partial \Omega$, perhaps excluding
those such that $\patch_j \cap \patch_i = \emptyset$.  There are order
$(n h)^{d-1}$ such $x_j$'s, which drives the variance of the local
polynomial estimator at $x_i$.  This applies to all $x_i$ with
$\patch_i \cap \Omega \neq \emptyset$, and there are order $n^d h$
such $x_i$'s.  The MSE over these points yields an MSE of order
\[
\frac1{n^d} (n^d h) \left(h^{2\alpha} + \frac{\sigma^2}{(n h)^{d-1}}\right) = h^{2\alpha+1} + (n h^2) \frac{\sigma^2}{(n h)^{d}}. 
\]     
We know that the MSE over the points away from the discontinuity is of
order $ h^{2\alpha} + \frac{\sigma^2}{(n h)^{d}}$,
so the overall MSE is of order $h^{2\alpha} + (n h^2 \vee 1) \frac{\sigma^2}{(n h)^{d}}$.
Minimizing over $h$ yields a lower bound of $(\sigma^2/n^{d-1})^{2\alpha/(2 \alpha + d-2)} \vee (\sigma^2/n^{d})^{2\alpha/(2 \alpha + d)}$, which is the MO rate if $d \ge 2 \alpha$.  

{\bf Non-local versions.}  We quickly argue that, without
spatial localization, YF, NLM, and NLM-average do not perform that
well (relative to the MO), unless the
underlying function is a polynomial (of degree at most $r$, where $r$ is the  chosen degree
for the polynomial fitting) or all jumps are greater than $h_y$.  
Let us look at what the methods do on noiseless data.  For a given photometric bandwidth
$h_y$, consider the function $f = h_y \one{\Omega}$, where $\Omega =
(0,1/2) \times (0,1)^{d-1}$. 
Then both YF and NLM-average output a
constant estimator equal everywhere to the local polynomial estimator
applied to the whole image.  Hence, the MSE is at least $h_y^2/4$.
Given that we take $h_y$ relatively large, this leads to a large MSE
(of order~1).

\medskip
\subsubsection{Proof of \thmref{thin}}
\label{sec:proof-thin}

The only difference with the
cartoon model is in the behavior of local polynomial regression.  Fix a
point $x_i \in \Omega$.  By \lemref{vol} (scaled by $h$), $\Vol(B(x_i,
h) \cap \Omega) \asymp h^{d_0} a^{d-d_0}$.  (In fact, this is slightly
easier here since $\Omega$ is a band around the graph of a function.)
Therefore, by \lemref{count} (scaled by $h$), we see that, $
\# \{j: x_j \in B(x_i, h) \cap \Omega\} \asymp n^d h^{d_0} a^{d-d_0}$.
This is the number of observations we are ``averaging'' over.

For LF, we prove a lower bound of order 1 for the squared bias at
$x_i$; we proceed as in \lemref{local_poly_lb} with only cosmetic
adjustments.

For BO, we apply LPR to the sample points $x_j$ belonging to the largest ball centered at $x_i$ which is contained in $\Omega$.  Since we only consider $x_i \in \Omega \setminus B(\partial \Omega,a/C)$, 
then this ball is of radius at least $a/C$.  We then conclude using the same argument bounding the risk of BO in the cartoon model detailed in
\secref{proof-BO}.

For MO, and its mimickers YF and NLM, we need to refine
\lemref{local_poly_var} because, in the case where $\Omega$ is a thin
band, the largest ball within it is not representative of the sample
size used in the local polynomial fit --- which is what drives the
variance.  We explain how to adapt the proof of
\lemref{local_poly_var} to show that, for a constant $C$,  
\[
\Var(\wh{f}_i) \leq \frac{C \sigma^2}{n^d h^{d_0} a^{d-d_0}},
\]
 Let $u_1, \dots, u_m$ be a maximal $a$-packing of $B(x_i, h) \cap \Omega$, with $m \asymp (h/a)^{d_0}$.  Then 
\[
\bigsqcup_{k=1}^m B(u_k, a) \subset \Omega. 
\]
Using the notation introduced in the proof of \lemref{local_poly_var}, we have 
\[
\bZ^T \bZ \succ \sum_{k=1}^m \bZ_k^T \bZ_k, 
\]
where $\bZ_k = (z_j^s: x_j \in B(u_k, a), |s| \le r)$. 
We can then use \eqref{ZZ} to obtain
\[
\lambda_{\rm min}(\bZ_k^T \bZ_k) \ge \frac1C (n a)^d,
\] 
implying
\[
\lambda_{\rm min}(\bZ^T \bZ) \succ \frac1C m (n a)^d.
\]
This gives the upper bound on the variance, and the
bias behaves as expected, meaning that \lemref{local_poly_ub} holds.
It is now straightforward to deduce that MO at $i$ has a
squared bias of order $O(h^{2\alpha})$ and a variance of $O(\sigma^2 n^{-d}
h^{-d_0} a^{-d+d_0})$.  Given that $h = h^{\rm MO}$ and $a = o(h^{\rm
  MO})$, the variance dominates and may be expressed as $(h/a)^{d-d_0}
O(\sigma^2/(nh)^{d})$, with $O(\sigma^2/(nh)^{d})$ being the order of magnitude of the
point risk of MO under the cartoon model.

YF is still able to perfectly mimic MO under the
conditions of \thmref{Y} (same exact arguments).

For points $x_i \in \Omega$ with $\dist(x_i, \partial \Omega) > h_{\patch}^{\rm NLM}$,
the analysis for NLM is again exactly the same, the difference here
being in the number of $j$'s such that $\patch_j \subset \Omega$, which
is of order $n^{d} h^{d_0} a^{d-d_0}$.  The rest is the same.

\medskip
\subsubsection{Proof of \prpref{pattern}}
\label{sec:proof-pattern}

Here $\Omega$ and $\Omega^c$ are interchangeable, so we focus on the
former WLOG and fix $x_i \in \Omega$.  Again, the only
difference with the cartoon model is in the behavior of local linear
regression and we need a stronger version of \lemref{local_poly_var}
 when $\Omega$ is a repeated pattern.  Using
notations introduced in the proof of \lemref{local_poly_var}, we have
\[
\bZ^T \bZ = \sum_v \bZ_v^T \bZ_v,
\]
where $\bZ_v = (z_j^s: x_j \in B(x_i, h) \cap (\Xi + v), |s| \le r)$.
Note that we may restrict the sum to those $v \in a\bbZ^d$ such that
$B(x_i, h) \cap (\Xi + v) \neq \emptyset$, and there are order
$(h/a)^d$ such $v$'s.  Since they are all translates of each other,
let us focus on $\Xi$, that is, $v = 0$.

We again express $\bZ_0^T \bZ_0$ as a sum of matrices by partitioning the $d$-dimensional subgrid $\{x_j \in \Xi\}$ into discrete 1D grids of the form 
\[
L_{j_1, \dots, j_{d-1}} := \{((j_1-1/2)/n, \dots, (j_{d-1}-1/2)/n, (j_d - 1/2)/n) \in \Xi: j_d = 1, \dots, [na]\},  
\]
where $j_1, \dots, j_{d-1} \in \{1, \dots, [na]\}$.  We therefore have
\[
\bZ_0^T \bZ_0 = \sum_{j'} \bZ_{(j')}^T \bZ_{(j')},
\]
where $\bZ_{(j')} := (z_k^s: x_k \in L_{j'} \cap \Omega, |s| \le r)$ for $j' \in \{1, \dots, [na]\}^{d-1}$.

Since $N_\Omega \ge (1/C) N_{\Omega^c}$, we also have that $N_{\Xi} \ge (1/C) N_{(0,a)^d \setminus \Xi}$, so that $\Xi$ contains at least the fraction $1/(C+1)$ of the sample points in $(0,a)^d$ and therefore 
\beq \label{L}
\sum_{j' \in \{1, \dots, [na]\}^{d-1}} |L_{j'}| \ge \frac{[na]^d}{C+1}.
\eeq
Let
\[
J' :=  \{j' \in \{1, \dots, [na]\}^{d-1} : |L_{j'}| \ge [na]/(2C+2)\}.
\]
Since $|L_{j'}| \le [na]$, we have
\[
\sum_{j' \in \{1, \dots, [na]\}^{d-1}} |L_{j'}| \le [na] |J'| + \frac{[na]}{2C+2} ([na]^{d-1} - |J'|),
\]
so that $|J'| \ge [na]^{d-1}/(2C+1)$ by \eqref{L}.
We focus on $\bZ_{(j')}$ with $j' \in J'$.  Notice that this reduces
the analysis to the one-dimensional case.  

\begin{lem} \label{lem:vandermonde}
There is a numeric constant $C > 0$ such that any polynomial regression matrix of the form $\bU = ((k/m)^s: 0 \le s \le r; k \in K)$, with $K \subset \{-m, \dots, m\}$ and $|K| \ge r+1$, satisfies $\lambda_{\rm min}(\bU^T\bU) \ge |K| (|K|/m)^{2r}/C$.     
\end{lem}

\begin{proof}
  Let $k_1 < \cdots < k_q$ be the elements of $K$.  Define $\ell_0 =
  [q/(r+2)]$ and for $\ell =1, \dots, \ell_0-1$, let $K_\ell =
  \{k_{\ell}, k_{\ell+\ell_0}, \dots, k_{\ell+(r+1)\ell_0}\}$.  Note
  that $|K_\ell| = r+1$ and $k_{\ell+ (j+1) \ell_0} - k_{\ell+ j
    \ell_0} \ge \ell_0$.  Now, the matrix $\bU_\ell = ((k/m)^s: 0 \le
  s \le r; k \in K_\ell)$ is a Vandermonde $(r+1) \times (r+1)$
  matrix.  It is well-known that $\bU_\ell$ is invertible, and more
  precisely, the main result in~\cite{Gautschi62} says that

\[
\|\bU_\ell^{-1}\|_\infty = \max_{1 \le i \le r+1} \prod_{j \in \{1, \dots, r+1\} \setminus \{i\}}  \frac{1 + |k_{\ell + j \ell_0}|/m}{|k_{\ell + j \ell_0}/m - k_{\ell + i \ell_0}/m|},
\]
where $\| (a_{ij}) \|_\infty := \max_i \sum_j |a_{ij}|$.  Hence, 
\[
\|\bU_\ell^{-1}\|_2 \le \sqrt{r+1} \|\bU_\ell^{-1}\|_\infty \le \sqrt{r+1}  (2m/\ell_0)^r,
\]
where $\| \cdot \|_2$ is the usual Euclidean operator norm.  Hence,
\[
\lambda_{\rm min}(\bU_\ell^T \bU_\ell) \ge \|\bU_\ell^{-1}\|_2^{-2} \ge (\ell_0/(2m))^{2r}/(r+1).
\]
Since the index sets $K_\ell$ do not overlap, we have
\[
\lambda_{\rm min}(\bU^T \bU) \ge \sum_{\ell = 1}^{\ell_0} \lambda_{\rm min}(\bU_\ell^T \bU_\ell).
\]
When $r$ is fixed, $\ell_0 \asymp q$, so the RHS $\asymp q (q/m)^{2r}$.  
\end{proof}

Let $C_1$ denote the constant of \lemref{vandermonde} and let $C_2 =
C_1 (2C+1)^{2r+1}$.  Applying this result under the assumption that $[na]/(2C+1)\ge r+1$, 
we find that $\lambda_{\rm min}(\bZ_{(j')}^T \bZ_{(j')}) \ge [na]/C_2$
for all $j' \in J'$. From here, we have
\[
\lambda_{\rm min}(\bZ_0^T \bZ_0) \ge (\# J') [na]/C_2 \ge \frac{[na]^d}{C_2 (2C+2)}, 
\]
and then
\[
\lambda_{\rm min}(\bZ^T \bZ) \ge (h/a)^d \lambda_{\rm min}(\bZ_0^T \bZ_0) \asymp (nh)^d.
\]
With this established, the bias behaves as in the cartoon model, and
the rest of the analysis for MO and YF is exactly as before.

For NLM, some additional arguments are required.  We need
to compare $\patch_i$ with other patches centered at $x_j \in
B(x_i, h)$.  First, suppose that $x_j - x_i \in a \bbZ^d$.  Then, by
the periodicity of $\Omega$, $x_j \in \Omega$ too, and also $x_k + x_i
\in \Omega$ if, and only if, $x_k + x_j \in \Omega$, for all $x_k \in
\patch_0$.  Hence, \beqn
\|\bbfpj - \bbfpi\|_2^2 &=& \sum_{x_k \in \patch_0 \cap \Omega} (\fin(x_k + x_j) - \fin(x_k + x_i))^2 \\
&& + \sum_{x_k \in \patch_0 \cap \Omega^c} (\fout(x_k + x_j) - \fout(x_k + x_i))^2 \\
&\le& \mp \cf^2 \|x_i - x_j\|^2 \leq \mp \cf^2 h^2.  \eeqn This is the
equivalent of \eqref{f_ub}.

Suppose now that $x_j \in \Omega^c$.  Using the fact that $\fin$ and $\fout$ are $\cf$-Lipschitz, we have
\[
\bbfpi = \fin(x_i) \1 (\patch_i \cap \Omega) + \fout(x_i) \1 (\patch_i \cap \Omega^c) + O(h),
\]
and similarly, 
\[
\bbfpj = \fin(x_i) \1 (\patch_j \cap \Omega) + \fout(x_i) \1 (\patch_j \cap \Omega^c) + O(h),
\]
since $\fin(x_i) - \fin(x_j) = O(h)$ and $\fout(x_i) - \fout(x_j) = O(h)$.
Hence,
\beqn
\|\bbfpj - \bbfpi\|_2^2 &\ge& (\fout(x_i) - \fin(x_i))^2 \\
&& \times  \|\1 (\patch_j \cap \Omega) - \1 (\patch_i \cap
  \Omega)\|_2^2 + O(\mp h^2)\\
&\ge& \mu^2 \mp/C' + O(\mp h^2)  
\eeqn  
by \eqref{pattern-NLM}.  This is the equivalent of
\eqref{f_lb}.

Arguing as the in the proof of \thmref{NLM}, we see that, with high
probability, the regression neighborhood of $x_i$ includes all $x_j$
such that $x_j - x_i \in a \bbZ^d$, $x_j \in B(x_i,h)$ and $\patch_j
\cap \patch_i = \emptyset$ --- those $x_j$'s are in $\Omega$ like
$x_i$ --- and excludes all $x_j \in \Omega^c$ such that $\patch_j \cap
\patch_i = \emptyset$.  There are of order $(na)^d$ such points. 
Using the same techniques as before, this leads to a bound on the
variance of order $(na)^d/(nh)^d$. The trade-off with the bias
for a choice of bandwidth $h^{\rm MO}$ leads to the $(na)^d \, \cR^{\rm MO}$
upper bound in the proposition.

In principle, an additional argument would be
needed to exclude those $x_j \in \Omega^c$ such that $\patch_j \cap
\patch_i \neq \emptyset$, since in that case $\|\bepj - \bepi\|_2^2$
is not chi-square as before.  However, it is not hard to see that even
if these are included in the regression neighborhood, it does not
change things much since their number is small --- of order $O(\log
n)$.

\section*{Acknowledgements}
We would like to thanks Alexandre Tsybakov for pointing out some
helpful references and Micha\"el Chichignoud for fruitful comments
that helped improving this work.

\bibliographystyle{spmpsci}
\bibliography{references_all}

\end{document}

%% file: notations_arxiv.tex
\newcommand{\currentcaption}{}
\newcommand{\currentname}{}

\newcommand{\nc}{\newcommand}

\newcommand{\patch}{\textsf{P}}
\newcommand{\wnlm}{\omega}

\makeatletter

\newcommand*{\wpatch}{%
 \protect \@ifnextchar\bgroup\w@patch\w@@patch
}

\newcommand*{\w@patch}[1]{
  \ensuremath{\boldsymbol S \left( #1 \right)}%
}

\newcommand*{\w@@patch}{
  \ensuremath{\boldsymbol S}%
}

\newcommand*{\wpatchoppose}{%
 \protect \@ifnextchar\bgroup\my@patch\my@@patch
}

\newcommand*{\my@patch}[1]{
  \ensuremath{\check{(\boldsymbol S}\left( #1 \right)}%
}

\newcommand*{\my@@patch}{
  \ensuremath{\check{\boldsymbol S}}%
}

\makeatother

\newcommand{\R}{\mathbb{R}}

\newcommand{\N}{\mathbb{N}}

\newcommand{\Z}{\mathbb{Z}}

\nc{\ind}{\mathds{1}}
\renewcommand{\P}{\mathbb{P}}
\newcommand{\E}{\mathbb{E}}

\DeclareMathOperator{\Var}{Var}

\newcommand{\1}{\mathds{1}}

\newtheorem{thm}{Theorem}[section]
\newtheorem{prp}{Proposition}[section]
\newtheorem{lem}{Lemma}[section]

\newcommand{\beq}{\begin{equation}}
\newcommand{\eeq}{\end{equation}}
\newcommand{\beqn}{\begin{eqnarray*}}
\newcommand{\eeqn}{\end{eqnarray*}}
\newcommand{\bitem}{\begin{itemize}}
\newcommand{\eitem}{\end{itemize}}
\newcommand{\benum}{\begin{enumerate}}
\newcommand{\eenum}{\end{enumerate}}
\newcommand{\bmult}{\begin{multline*}}
\newcommand{\emult}{\end{multline*}}
\newcommand{\bcenter}{\begin{center}}
\newcommand{\ecenter}{\end{center}}

\newcommand{\thmref}[1]{Theorem~\ref{thm:#1}}
\newcommand{\prpref}[1]{Proposition~\ref{prp:#1}}
\newcommand{\lemref}[1]{Lemma~\ref{lem:#1}}
\newcommand{\secref}[1]{Section~\ref{sec:#1}}
\newcommand{\figref}[1]{Figure~\ref{fig:#1}}
\nc{\argmin}{\mathop{\mathrm{arg\,min}}}
\nc{\argmax}{\mathop{\mathrm{arg\,max}}}

\DeclareMathOperator{\dist}{dist}

\newcommand{\cF}{\mathcal{F}}

\newcommand{\cH}{\mathcal{H}}

\newcommand{\cN}{\mathcal{N}}

\newcommand{\cR}{\mathcal{R}}

\newcommand{\cX}{\mathcal{X}}

\newcommand{\bH}{\mathbf{H}}
\newcommand{\bI}{\mathbf{I}}

\newcommand{\bM}{\mathbf{M}}

\newcommand{\bU}{\mathbf{U}}

\newcommand{\bX}{\mathbf{X}}

\newcommand{\bZ}{\mathbf{Z}}

\newcommand{\ba}{\mathbf{a}}

\newcommand{\be}{\mathbf{e}}
\newcommand{\bbf}{\mathbf{f}}
\newcommand{\bg}{\mathbf{g}}

\newcommand{\bp}{\mathbf{p}}

\newcommand{\bu}{\mathbf{u}}
\newcommand{\bv}{\mathbf{v}}

\newcommand{\by}{\mathbf{y}}

\newcommand{\beps}{{\boldsymbol\eps}}

\newcommand{\bbN}{\mathbb{N}}

\newcommand{\bbR}{\mathbb{R}}

\newcommand{\bbZ}{\mathbb{Z}}

\newcommand{\pr}[1]{\mathbb{P}\left(#1\right)}

\newcommand{\eps}{\varepsilon}

\newcommand{\Vol}{{\rm Vol}}

\definecolor{darkgreen}{rgb}{.0,.638,.035}
\definecolor{darkred}{rgb}{.638,.0,.035}
\definecolor{purple}{rgb}{0.4,.1,.9}

\newcommand{\mse}{{\rm MSE}}
\newcommand{\ie}{{\em i.e.,~}}
\newcommand{\eg}{{\em e.g.,~}}
\newcommand{\lcf}{{\em cf.\ }}

\newcommand{\ave}{{\rm Ave}}
\newcommand{\one}[1]{\mathds{1}_{\{ #1 \}}}
\newcommand{\fin}{f_\Omega}
\newcommand{\fout}{f_{\Omega^c}}

\newcommand{\afloor}{\lfloor \alpha \rfloor}
\newcommand{\NLMm}{NLM-average}

\newcommand{\cf}{C_0}
\newcommand{\wh}[1]{\widehat{#1}}
\renewcommand{\mp}{{m_\patch}}

\newcommand{\bypi}{{\by_{\patch_i}}}
\newcommand{\bypj}{{\by_{\patch_j}}}

\newcommand{\bepi}{{\beps_{\patch_i}}}
\newcommand{\bepj}{{\beps_{\patch_j}}}

\newcommand{\bbfpi}{{\bbf_{\patch_i}}}
\newcommand{\bbfpj}{{\bbf_{\patch_j}}}

\newcommand{\ypbark}{{\overline{y}_{\patch_k}}}
\newcommand{\ypbari}{{\overline{y}_{\patch_i}}}
\newcommand{\fpbari}{{\overline{f}_{\patch_i}}}
\newcommand{\epbari}{{\overline{\eps}_{\patch_i}}}

\newcommand{\ypbarj}{{\overline{y}_{\patch_j}}}
\newcommand{\fpbarj}{{\overline{f}_{\patch_j}}}
\newcommand{\epbarj}{{\overline{\eps}_{\patch_j}}}

%% file: mse_oracle_Bowl.tex
& $\sigma=5$
& $\sigma=20$
& $\sigma=50$
& $\sigma=100$
\\
\hline
$r=0$ 
& 7
& 13
& 23
& 35
\\
$r=1$ 
& 9
& 17
& 25
& 33
\\
$r=2$ 
& 23
& 41
& 59
& 61
\\

%% file: mse_replica5.tex
& Blob
& Sinusoid
& Bowl
& Ridges
& Stripes
& Barbara
& Cam.
\\
\hline
& \multicolumn{7}{|c|}{$\sigma = 5$}\\
\hline
LF0 
& 35.27
& 40.29
& 57.71
& 48.80
& 21077.49
& 408.13
& 437.96
\\
LF1 
& 48.50
& 55.50
& 74.18
& 110.08
& 22787.90
& 473.15
& 529.64
\\
LF2 
& 72.11
& 82.89
& 105.09
& 246.70
& 15424.99
& 586.47
& 663.59
\\
YF0 
& 1.45
& 1.67
& 1.74
& 13.70
& 1.68
& 20.48
& 13.59
\\
YF1 
& 1.11
& 1.30
& 1.01
& 8.57
& 1.21
& 20.03
& 13.39
\\
YF2 
& 0.94
& 1.17
& 0.87
& 7.94
& 0.69
& 20.79
& 13.48
\\
NLM-Av.0 
& 1.61
& 2.55
& 3.30
& 4.15
& 327.70
& 255.23
& 188.74
\\
NLM-Av.1 
& 1.49
& 1.79
& 2.33
& 4.08
& 202.88
& 202.96
& 99.13
\\
NLM-Av.2 
& 1.35
& 1.67
& 1.99
& 3.45
& 329.83
& 242.91
& 151.60
\\
NLM0 
& 1.55
& 1.74
& 1.86
& 3.88
& 3.89
& 19.54
& 13.52
\\
NLM1 
& 1.47
& 1.55
& 1.73
& 3.65
& 3.11
& 19.79
& 13.72
\\
NLM2 
& 1.51
& 1.51
& 1.59
& 3.59
& 1.68
& 20.27
& 13.74
\\
MO0 
& 1.58
& 1.77
& 0.97
& 16.16
& 1.23
& 35.10
& 28.60
\\
MO1 
& 1.11
& 1.26
& 0.52
& 19.13
& 0.81
& 37.84
& 28.99
\\
MO2 
& 0.97
& 1.31
& 0.35
& 12.41
& 0.36
& 39.65
& 29.75
\\
\hline
& \multicolumn{7}{|c|}{$\sigma = 20$}\\
\hline
LF0 
& 77.77
& 91.66
& 109.58
& 305.04
& 13956.38
& 607.11
& 684.54
\\
LF1 
& 104.25
& 134.27
& 141.41
& 533.25
& 14455.88
& 725.30
& 818.62
\\
LF2 
& 139.80
& 208.44
& 188.67
& 707.47
& 16996.42
& 901.23
& 994.78
\\
YF0 
& 15.43
& 17.05
& 11.61
& 118.99
& 9.77
& 189.69
& 104.70
\\
YF1 
& 17.93
& 20.80
& 11.46
& 158.55
& 7.81
& 219.63
& 113.33
\\
YF2 
& 18.97
& 24.66
& 14.34
& 174.83
& 6.59
& 242.48
& 122.42
\\
NLM-Av.0 
& 6.99
& 9.27
& 17.76
& 18.73
& 332.71
& 345.66
& 307.67
\\
NLM-Av.1 
& 6.02
& 7.72
& 14.12
& 19.61
& 406.11
& 334.80
& 275.87
\\
NLM-Av.2 
& 4.58
& 7.53
& 13.05
& 15.15
& 399.46
& 352.14
& 306.27
\\
NLM0 
& 5.76
& 6.37
& 12.66
& 20.44
& 11.54
& 121.17
& 92.36
\\
NLM1 
& 5.53
& 5.70
& 13.28
& 21.68
& 9.31
& 129.10
& 96.88
\\
NLM2 
& 5.02
& 4.53
& 13.16
& 19.44
& 5.92
& 137.95
& 101.09
\\
MO0 
& 4.00
& 4.41
& 5.03
& 31.78
& 4.65
& 41.67
& 34.24
\\
MO1 
& 2.96
& 3.25
& 2.82
& 35.75
& 2.89
& 44.92
& 34.06
\\
MO2 
& 2.26
& 2.74
& 1.88
& 33.60
& 1.83
& 45.68
& 34.88
\\
\hline
& \multicolumn{7}{|c|}{$\sigma = 50$}\\
\hline
LF0 
& 149.70
& 211.28
& 195.46
& 847.84
& 17633.15
& 900.24
& 997.56
\\
LF1 
& 162.93
& 232.97
& 211.39
& 939.06
& 15081.34
& 955.73
& 1048.22
\\
LF2 
& 209.56
& 290.03
& 273.74
& 1501.85
& 15705.82
& 1157.37
& 1221.38
\\
YF0 
& 112.17
& 138.30
& 146.13
& 591.42
& 857.69
& 652.97
& 523.88
\\
YF1 
& 129.07
& 155.85
& 164.73
& 655.67
& 722.93
& 699.37
& 574.87
\\
YF2 
& 146.05
& 178.84
& 199.35
& 998.40
& 741.85
& 811.85
& 629.87
\\
NLM-Av.0 
& 23.66
& 29.89
& 52.96
& 64.32
& 807.78
& 419.12
& 389.60
\\
NLM-Av.1 
& 21.51
& 27.56
& 36.86
& 69.56
& 770.17
& 414.81
& 372.68
\\
NLM-Av.2 
& 18.17
& 26.07
& 39.07
& 67.12
& 820.21
& 425.60
& 385.13
\\
NLM0 
& 21.64
& 27.35
& 36.32
& 162.17
& 40.92
& 367.48
& 230.35
\\
NLM1 
& 29.09
& 31.35
& 30.78
& 179.78
& 25.50
& 381.14
& 234.01
\\
NLM2 
& 25.33
& 30.60
& 30.15
& 245.86
& 20.72
& 398.52
& 243.81
\\
MO0 
& 7.68
& 8.32
& 11.23
& 48.99
& 10.90
& 50.50
& 42.64
\\
MO1 
& 7.72
& 7.98
& 9.20
& 57.78
& 8.70
& 55.67
& 44.46
\\
MO2 
& 5.56
& 6.11
& 6.43
& 67.90
& 6.01
& 49.81
& 41.51
\\
\hline
& \multicolumn{7}{|c|}{$\sigma = 100$}\\
\hline
LF0 
& 239.01
& 319.36
& 300.81
& 1340.28
& 17131.55
& 1198.68
& 1249.50
\\
LF1 
& 225.89
& 307.90
& 285.37
& 1277.60
& 17776.32
& 1159.50
& 1218.33
\\
LF2 
& 225.36
& 305.06
& 291.83
& 1550.89
& 17079.65
& 1188.50
& 1248.76
\\
YF0 
& 308.15
& 367.90
& 365.84
& 1206.35
& 8848.92
& 1108.23
& 1080.66
\\
YF1 
& 299.59
& 359.21
& 352.73
& 1156.57
& 9197.93
& 1077.16
& 1064.45
\\
YF2 
& 296.39
& 355.73
& 356.40
& 1375.29
& 8813.04
& 1099.59
& 1077.78
\\
NLM-Av.0 
& 64.41
& 76.19
& 118.55
& 202.15
& 8223.43
& 556.50
& 495.62
\\
NLM-Av.1 
& 66.78
& 74.22
& 94.12
& 204.95
& 8385.63
& 554.88
& 492.05
\\
NLM-Av.2 
& 66.27
& 73.42
& 98.31
& 224.59
& 8118.55
& 555.58
& 495.58
\\
NLM0 
& 91.67
& 131.36
& 167.44
& 819.97
& 91.49
& 911.60
& 628.08
\\
NLM1 
& 118.08
& 135.37
& 183.32
& 786.01
& 90.29
& 926.67
& 662.68
\\
NLM2 
& 101.83
& 127.34
& 171.13
& 956.96
& 88.01
& 918.08
& 646.76
\\
MO0 
& 14.19
& 15.12
& 22.74
& 80.33
& 19.90
& 61.09
& 54.41
\\
MO1 
& 18.31
& 17.96
& 23.06
& 91.72
& 22.38
& 76.21
& 65.36
\\
MO2 
& 17.76
& 17.65
& 18.95
& 88.60
& 22.09
& 72.00
& 62.53
\\

%% file: MinimaxNLM_arxiv_R1.bbl
\begin{thebibliography}{10}
\providecommand{\url}[1]{{#1}}
\providecommand{\urlprefix}{URL }
\expandafter\ifx\csname urlstyle\endcsname\relax
  \providecommand{\doi}[1]{DOI~\discretionary{}{}{}#1}\else
  \providecommand{\doi}{DOI~\discretionary{}{}{}\begingroup
  \urlstyle{rm}\Url}\fi

\bibitem{Alvarez_Mazorra94}
Alvarez, L., Mazorra, L.: Signal and image restoration using shock filters and
  anisotropic diffusion.
\newblock SIAM J. Numer. Anal. \textbf{31}(2), 590--605 (1994)

\bibitem{AriasCastro_Donoho09}
Arias-Castro, E., Donoho, D.L.: Does median filtering truly preserve edges
  better than linear filtering?
\newblock Ann. Statist. \textbf{37}(3), 1172--1206 (2009)

\bibitem{Awate_Whitaker06}
Awate, S.P., Whitaker, R.T.: Unsupervised, information-theoretic, adaptive
  image filtering for image restoration.
\newblock {IEEE} Trans. Pattern Anal. Mach. Intell. \textbf{28}(3), 364--376
  (2006)

\bibitem{Azzabou_Paragios_Guichard07}
Azzabou, N., Paragios, N., Guichard, F.: Image denoising based on adapted
  dictionary computation.
\newblock In: ICIP, pp. 109--112 (2007)

\bibitem{Buades}
Buades, A.: Image and movie denoising by non local means.
\newblock Ph.D. thesis, Universitat de les Illes Balears (2006)

\bibitem{Buades_Coll_Morel05}
Buades, A., Coll, B., Morel, J.M.: A review of image denoising algorithms, with
  a new one.
\newblock Multiscale Model. Simul. \textbf{4}(2), 490--530 (2005)

\bibitem{Chatterjee_Milanfar11}
Chatterjee, P., Milanfar, P.: Patch-based near-optimal image denoising.
\newblock submitted  (2011)

\bibitem{Criminisi_Perez_Toyama04}
Criminisi, A., P{\'e}rez, P., Toyama, K.: Region filling and object removal by
  exemplar-based image inpainting.
\newblock {IEEE} Trans. Image Process. \textbf{13}(9), 1200--1212 (2004)

\bibitem{Dabov_Foi_Katkovnik_Egiazarian07}
Dabov, K., Foi, A., Katkovnik, V., Egiazarian, K.O.: Image denoising by sparse
  {3-D} transform-domain collaborative filtering.
\newblock {IEEE} Trans. Image Process. \textbf{16}(8), 2080--2095 (2007)

\bibitem{Dabov_Foi_Katkovnik_Egiazarian09}
Dabov, K., Foi, A., Katkovnik, V., Egiazarian, K.O.: {BM3D} image denoising
  with shape-adaptive principal component analysis.
\newblock In: Proc. Workshop on Signal Processing with Adaptive Sparse
  Structured Representations (SPARS'09) (2009)

\bibitem{Deledalle_Duval_Salmon11b}
Deledalle, C.A., Duval, V., Salmon, J.: Anisotropic non-local means with
  spatially adaptive patch shapes.
\newblock In: SSVM (2011)

\bibitem{Deledalle_Duval_Salmon11}
Deledalle, C.A., Duval, V., Salmon, J.: Non-local methods with shape-adaptive
  patches ({NLM-SAP}).
\newblock J. Math. Imaging Vis. pp. 1--18 (2011)

\bibitem{Donoho_Johnstone94}
Donoho, D.L., Johnstone, I.M.: Ideal spatial adaptation by wavelet shrinkage.
\newblock Biometrika \textbf{81}(3), 425--455 (1994)

\bibitem{Donoho_Johnstone_Kerkyacharian_Picard95}
Donoho, D.L., Johnstone, I.M., Kerkyacharian, G., Picard, D.: Wavelet
  shrinkage: asymptopia?
\newblock J. Roy. Statist. Soc. Ser. B \textbf{57}(2), 301--369 (1995)

\bibitem{Duval_Aujol_Gousseau11}
Duval, V., Aujol, J.F., Gousseau, Y.: A bias-variance approach for the nonlocal
  means.
\newblock SIAM J. Imaging Sci. \textbf{4}(2), 760--788 (2011)

\bibitem{Efros_Leung99}
Efros, A.A., Leung, T.: Texture synthesis by non-parametric sampling.
\newblock In: ICCV, pp. 1033--1038 (1999)

\bibitem{Fan_Gijbels96}
Fan, J., Gijbels, I.: Local polynomial modelling and its applications,
  \emph{Monographs on Statistics and Applied Probability}, vol.~66.
\newblock Chapman \& Hall, London (1996)

\bibitem{Gautschi62}
Gautschi, W.: On inverses of {Vandermonde} and confluent vandermonde matrices.
\newblock Numerische Mathematik \textbf{4}, 117--123 (1962)

\bibitem{Gilboa_Osher07}
Gilboa, G., Osher, S.: Nonlocal linear image regularization and supervised
  segmentation.
\newblock Multiscale Model. Simul. \textbf{6}(2), 595--630 (2007)

\bibitem{Hastie_Tibshirani_Friedman09}
Hastie, T., Tibshirani, R., Friedman, J.: The elements of statistical learning,
  second edn.
\newblock Springer Series in Statistics. Springer, New York (2009)

\bibitem{Johnstone98}
Johnstone, I.M.: Oracle inequalities and nonparametric function estimation.
\newblock In: Proceedings of the {I}nternational {C}ongress of
  {M}athematicians, {V}ol. {III} ({B}erlin, 1998), Extra Vol. III, pp. 267--278
  (electronic) (1998)

\bibitem{Katkovnik99}
Katkovnik, V.: A new method for varying adaptive bandwidth selection.
\newblock {IEEE} Trans. Image Process. \textbf{47}(9), 2567--2571 (1999)

\bibitem{Katkovnik_Egiazarian_Astola02}
Katkovnik, V., Egiazarian, K.O., Astola, J.T.: Adaptive window size image
  de-noising based on intersection of confidence intervals ({ICI}) rule.
\newblock J. Math. Imaging Vis. \textbf{16}(3), 223--235 (2002)

\bibitem{Katkovnik_Foi_Egiazarian_Astola04}
Katkovnik, V., Foi, A., Egiazarian, K.O., Astola, J.T.: Directional varying
  scale approximations for anisotropic signal processing.
\newblock In: EUSIPCO, pp. 101--104 (2004)

\bibitem{Katkovnik_Foi_Egiazarian_Astola10}
Katkovnik, V., Foi, A., Egiazarian, K.O., Astola, J.T.: From local kernel to
  nonlocal multiple-model image denoising.
\newblock Int. J. Comput. Vision \textbf{86}(1), 1--32 (2010)

\bibitem{Kervrann_Boulanger06}
Kervrann, C., Boulanger, J.: Optimal spatial adaptation for patch-based image
  denoising.
\newblock {IEEE} Trans. Image Process. \textbf{15}(10), 2866--2878 (2006)

\bibitem{Korostelev_Tsybakov93}
Korostel{\"e}v, A.P., Tsybakov, A.B.: Minimax theory of image reconstruction,
  \emph{Lecture Notes in Statistics}, vol.~82.
\newblock Springer-Verlag, New York (1993)

\bibitem{Lee83}
Lee, J.S.: Digital image smoothing and the sigma filter.
\newblock Computer Vision, Graphics, and Image Processing \textbf{24}(2),
  255--269 (1983)

\bibitem{Lepski_Mammen_Spokoiny97}
Lepski, O.V., Mammen, E., Spokoiny, V.G.: Optimal spatial adaptation to
  inhomogeneous smoothness: an approach based on kernel estimates with variable
  bandwidth selectors.
\newblock Ann. Statist. \textbf{25}(3), 929--947 (1997)

\bibitem{Levin_Nadler11}
Levin, A., Nadler, B.: Natural image denoising: Optimality and inherent bounds.
\newblock In: CVPR (2011)

\bibitem{Mahmoudi_Sapiro05}
Mahmoudi, M., Sapiro, G.: Fast image and video denoising via nonlocal means of
  similar neighborhoods.
\newblock {IEEE} Signal Process. Lett. \textbf{12}, 839--842 (2005)

\bibitem{Mairal_Bach_Ponce_Sapiro_Zisserman09}
Mairal, J., Bach, F., Ponce, J., Sapiro, G., Zisserman, A.: Non-local sparse
  models for image restoration.
\newblock In: ICCV, pp. 2272--2279 (2009)

\bibitem{Maleki_Narayan_Baraniuk11b}
Maleki, A., Narayan, M., Baraniuk, R.G.: Anisotropic nonlocal means (2011).
\newblock Submitted to {\em Applied and Computational Harmonic Analysis}

\bibitem{Maleki_Narayan_Baraniuk11}
Maleki, A., Narayan, M., Baraniuk, R.G.: Suboptimality of nonlocal means for
  images with sharp edges (2011).
\newblock Submitted to {\em Applied and Computational Harmonic Analysis}

\bibitem{Mallat09}
Mallat, S.: A wavelet tour of signal processing.
\newblock Elsevier/Academic Press, Amsterdam (2009).
\newblock The sparse way, With contributions from Gabriel Peyr{\'e}

\bibitem{Milanfar12}
Milanfar, P.: A tour of modern image filtering.
\newblock IEEE Signal Processing Magazine, to appear as a feature article in
  IEEE signal processing magazine  (2012)

\bibitem{Muller_Stadtmuller87}
M{\"u}ller, H.G., Stadtm{\"u}ller, U.: Variable bandwidth kernel estimators of
  regression curves.
\newblock Ann. Statist. \textbf{15}(1), 182--201 (1987)

\bibitem{Nadaraya64}
Nadaraya, E.A.: On estimating regression.
\newblock Theory of Probability and its Applications \textbf{9}(1), 141--142
  (1964)

\bibitem{Perona_Malik90}
Perona, P., Malik, J.: Scale space and edge detection using anisotropic
  diffusion.
\newblock {IEEE} Trans. Pattern Anal. Mach. Intell. \textbf{12}, 629--639
  (1990)

\bibitem{Polzehl_Spokoiny00}
Polzehl, J., Spokoiny, V.G.: Adaptive weights smoothing with applications to
  image restoration.
\newblock J. R. Stat. Soc. Ser. B Stat. Methodol. \textbf{62}(2), 335--354
  (2000)

\bibitem{Polzehl_Spokoiny03}
Polzehl, J., Spokoiny, V.G.: Image denoising: pointwise adaptive approach.
\newblock Ann. Statist. \textbf{31}(1), 30--57 (2003)

\bibitem{Portilla_Strela_Wainwright_Simoncelli03}
Portilla, J., Strela, V., Wainwright, M., Simoncelli, E.P.: Image denoising
  using scale mixtures of gaussians in the wavelet domain.
\newblock {IEEE} Trans. Image Process. \textbf{12}(11), 1338--1351 (2003)

\bibitem{Salmon}
Salmon, J.: Agr\'egation d'estimateurs et m\'ethodes \`a patch pour le
  d\'ebruitage d'images num\'eriques.
\newblock Ph.D. thesis, Universit{\'e} Paris Diderot (2010)

\bibitem{Salmon10}
Salmon, J.: On two parameters for denoising with {Non-Local Means}.
\newblock {IEEE} Signal Process. Lett. \textbf{17}, 269--272 (2010)

\bibitem{Salmon_Strozecki12}
Salmon, J., Strozecki, Y.: Patch reprojections for {Non Local} methods.
\newblock Signal Processing \textbf{92}(2), 447--489 (2012)

\bibitem{Salmon_Willett_AriasCastro12}
Salmon, J., Willett, R., Arias-Castro, E.: A two-stage denoising filter: the
  preprocessed {Yaroslavsky} filter.
\newblock In: IEEE Statistical Signal Processing Workshop (2012)

\bibitem{Singer_Shkolniksy_Nadler09}
Singer, A., Shkolnisky, Y., Nadler, B.: {Diffusion interpretation of nonlocal
  neighborhood filters for signal denoising.}
\newblock SIAM J. Imaging Sci. \textbf{2}(1), 118--139 (2009)

\bibitem{Smith_Brady97}
Smith, S.M., Brady, J.M.: Susan-a new approach to low level image processing.
\newblock Int. J. Comput. Vision \textbf{23}(1), 45--78 (1997)

\bibitem{Spira_Kimmel05}
Spira, A., Kimmel, R.: Enhancing images painted on manifolds.
\newblock In: Scale Space and PDF Methods in Computer Vision, pp. 492--502
  (2005)

\bibitem{Spokoiny98}
Spokoiny, V.G.: Estimation of a function with discontinuities via local
  polynomial fit with an adaptive window choice.
\newblock Ann. Statist. \textbf{26}(4), 1356--1378 (1998)

\bibitem{Starck_Candes_Donoho02}
Starck, J.L., Cand{\`e}s, E.J., Donoho, D.L.: The curvelet transform for image
  denoising.
\newblock {IEEE} Trans. Image Process. \textbf{11}(6), 670--684 (2002)

\bibitem{Szlam_Maggioni_Coifman08}
Szlam, A.D., Maggioni, M., Coifman, R.R.: Regularization on graphs with
  function-adapted diffusion processes.
\newblock J. Mach. Learn. Res. \textbf{9}, 1711--1739 (2008)

\bibitem{Takeda_Farsiu_Milanfar07}
Takeda, H., Farsiu, S., Milanfar, P.: Kernel regression for image processing
  and reconstruction.
\newblock {IEEE} Trans. Image Process. \textbf{16}(2), 349--366 (2007)

\bibitem{Tasdizen09}
Tasdizen, T.: Principal neighborhood dictionaries for nonlocal means image
  denoising.
\newblock {IEEE} Trans. Image Process. \textbf{18}(12), 2649--2660 (2009)

\bibitem{Tomasi_Manduchi98}
Tomasi, C., Manduchi, R.: Bilateral filtering for gray and color images.
\newblock In: ICCV, pp. 839--846 (1998)

\bibitem{Tsybakov89}
Tsybakov, A.B.: Optimal orders of accuracy of the estimation of nonsmooth
  images.
\newblock Problemy Peredachi Informatsii \textbf{25}(3), 13--27 (1989)

\bibitem{Tsybakov09}
Tsybakov, A.B.: Introduction to nonparametric estimation.
\newblock Springer Series in Statistics. Springer, New York (2009)

\bibitem{VanDeVille_Kocher09}
{Van De Ville}, D., Kocher, M.: {SURE}-based {Non-Local Means}.
\newblock {IEEE} Signal Process. Lett. \textbf{16}, 973--976 (2009)

\bibitem{VanDeVille_Kocher11}
{Van De Ville}, D., Kocher, M.: Non-local means with dimensionality reduction
  and sure-based parameter selection.
\newblock {IEEE} Trans. Image Process. \textbf{99}(99), 1--1 (2011)

\bibitem{Watson64}
Watson, G.S.: Smooth regression analysis.
\newblock Sankhya: The Indian Journal of Statistics, Series A \textbf{26}(4),
  359--372 (1964)

\bibitem{Weissman_Ordentlich_Seroussi_Verdu_Weinberger05}
Weissman, T., Ordentlich, E., Seroussi, G., Verd{\'u}, S., Weinberger, M.J.:
  Universal discrete denoising: known channel.
\newblock {IEEE} Trans. Inf. Theory \textbf{51}(1), 5--28 (2005)

\bibitem{Yaroslavsky85}
Yaroslavsky, L.P.: Digital picture processing, \emph{Springer Series in
  Information Sciences}, vol.~9.
\newblock Springer-Verlag, Berlin (1985)

\bibitem{Zewail_Thomas09}
Zewail, A.H., Thomas, J.M.: 4D electron microscopy: imaging in space and time.
\newblock Imperial College Pr (2009)

\bibitem{Zontak_Irani11}
Zontak, M., Irani, M.: Internal statistics of a single natural image.
\newblock In: CVPR, pp. 977--984 (2011)

\end{thebibliography}
